\documentclass[10pt]{amsart}
\theoremstyle{plain}
\usepackage{amsmath, amssymb, amscd,graphicx}
\newtheorem{claim}{Claim}
\newtheorem{lemma}{Lemma}[section]
\newtheorem{prop}[lemma]{Proposition}
\newtheorem{theorem}[lemma]{Theorem}
\newtheorem{cor}[lemma]{Corollary}
\newtheorem{definition}[lemma] {Definition}
\newtheorem{conv}{Convention}[section]
\newtheorem{conj}[lemma]{Conjecture}

\newcommand{\Q}{{{\mathbb Q}}}
\newcommand{\R}{{{\mathbb R}}}
\newcommand{\Z}{{{\mathbb Z}}}

\newcommand{\psibar}{{{\overline{\psi}}}}

\renewcommand{\mod}{{{\mathrm{mod\;}}}}
\DeclareMathOperator*{\minq}{{{{\mathrm{min}}^q\!}}}
\DeclareMathOperator*{\maxq}{{{ {\mathrm{max}}^q\!}}}

\newcommand{\len}{{{\mathrm{len}}}}

\renewcommand{\epsilon}{{\varepsilon}}

\title{A Number Theoretic Result for Berge's Conjecture}
\author{Sarah Dean Rasmussen}
\address{Harvard University Dept. of Mathematics, Cambridge, MA 02138}
\email{sarah@math.harvard.edu}
\thanks{The author was supported by an NSF Graduate Research fellowship.}

\begin{document}
\begin{abstract}
(Original version of PhD thesis, submitted in Spring 2009 to Harvard University.
Provides a solution of the $p > k^2$ case, corresponding to Berge families I--VI, of
the ``Lens space realization problem'' later solved in entirety by Greene \cite{GreeneBerge}.)
In the 1980's, Berge proved that a certain collection of knots in $S^3$ admitted
lens space surgeries, a list which Gordon conjectured was exhaustive.
More recently \cite{Rasmussen}, J. Rasmussen used techniques from Heegaard Floer homology
to translate the related problem of classifying simple knots in lens spaces
admitting L-space homology sphere surgeries into a combinatorial number theory
question about the triple $(p,q,k)$ associated to a knot of homology class
$k \in H_1\!\left(L(p,q)\right)$ in the lens space $L(p,q)$.
In the following paper, we solve this number theoretic problem
in the case of $p > k^2$.
\end{abstract}

\maketitle

\section{Introduction}

One of the most basic methods of transforming an old 3-manifold
into a new 3-manifold is by performing surgery along a knot in
the old manifold.  Given a 3-manifold $Y$ and a knot $K \subset Y$
(which is just an embedding of $S^1$ into $Y$), one performs
Dehn surgery by cutting out a solid torus neigborhood
$n_K \subset Y$ of $K$, performing a Dehn twist on $n_K$,
and then gluing this Dehn-twisted version of $n_K$
back into $Y \setminus n_K$ to form a new 3-manifold, $Y^{\prime}$.
One can perform a similar procedure with a link in $Y$ (an embedding
of a disjoint union of $S^1$'s into $Y$), but this is equivalent
to the iterated procedure of performing surgery along each knot
that forms a component of the link.

Any 3-manifold can be obtained via integer Dehn surgery on a link in $S^3$.
Much is still unknown, however, about the problem of classifying
non-hyperbolic integer surgeries on knots in $S^3$.
One of the early lines of progress towards answering this question was
initiated by Berge, who constructed a list of knots in $S^3$ admitting
integer lens space surgeries \cite{Berge},
a list which Gordon later conjectured was exhaustive.

Reversing our point of view, we could equivalently ask which knots,
in which lens spaces, have integer $S^3$ surgeries.  This change of perspective
sends Berge's knots in $S^3$ to the knot cores of their associated lens space fillings.
It turns out that all of these resulting knots are {\em simple}, where a
simple knot in a lens space $L(p,q)$
is a knot obtained by placing two basepoints inside the standard
genus one Heegaard diagram of $L(p,q)$.
Since there is a unique simple
knot in each homology class $k \in H_1(L(p,q)) \cong \Z/p$,
any simple knot in a lens space can be described by a triple
$(p, q, k)$ with $k \in \Z/p$.

While the conjectured simplicity of all lens space knots admitting integer $S^3$
surgeries is a key part of the Berge-Gordon conjecture,
and a difficult question of ongoing interest (see, {\em e.g.}, \cite{BGH}),
we shall not address it in this paper.
Instead, we focus on classifying which simple knots in lens spaces
admit integer $S^3$ surgeries.  From this standpoint, the conjecture of the
Berge and Gordon could be phrased as follows.
\begin{conj}[Berge, Gordon]
\label{conj: berge}
If a knot $K$ in a lens space $L(p,q)$ is simple,
hence parameterized by the triple $(p,q,k)$, then
$K$ admits an integer $S^3$ surgery if and only if
$q \equiv k^2 (\mod p)$
and one or more of
the following is true:

{\noindent{Solutions when $p > k^2$:}}
\begin{align*}
\text{ I and II:}\;\;
&
\begin{cases}
   p \equiv ik \pm 1 \; (\mod k^2),
       & \gcd (i,k) = 1, 2
\end{cases}
                  \\
\text{ III:}\;\;
&
\begin{cases}
   p \equiv \pm d(2k+1) \; (\mod k^2),
       & d \vert k-1,\;  2\!\! \not\vert\, \frac{k-1}{d}
            \\
   p \equiv \pm d(2k-1) \; (\mod k^2),
       & d \vert k+1,\; 2\!\! \not\vert\, \frac{k+1}{d}
\end{cases}
                  \\
\text{ IV:}\;\;
&
\begin{cases}
   p \equiv \pm d(k+1) \; (\mod k^2),
       & d \vert 2k-1,\; 2\!\! \not\vert\, \frac{2k-1}{d}
             \\
   p \equiv \pm d(k-1) \; (\mod k^2),
       & d \vert 2k+1,\; 2\!\! \not\vert\, \frac{2k+1}{d}
\end{cases}
                  \\
\text{ V:}\;\;
&
\begin{cases}
   p \equiv \pm d(k+1) \; (\mod k^2),
       & d \vert k+1,\; 2\!\! \not\vert\, d
              \\
   p \equiv \pm d(k-1) \; (\mod k^2),
       & d \vert k-1,\; 2\!\! \not\vert\, d
\end{cases}
                  \\
\text{ VI:}\;\;
&
\begin{cases}
   \text{Special case of  V}.
\end{cases}
\end{align*}

{\noindent{Solutions when $p < k^2$:}}
\begin{align*}
\text{ VII and VIII:}
&
\begin{cases}
    k^2\pm k \pm 1 \equiv 0 \; (\mod p)
\end{cases}
                  \\
\text{ IX, X, XI, and XII:}
&
\begin{cases}
p = \frac{1}{11}(2k^2 + k + 1).
\end{cases}
\end{align*}
\end{conj}

In \cite{Rasmussen}
Jacob Rasmussen analyzed the above conjecture by studying
the Heegaard Floer homology of knots in lens spaces
with L-space homology-sphere surgeries.  An L-space
is a 3-manifold whose Heegaard Floer homology
has the smallest possible rank, in a certain precise sense.
A homology sphere has the same ordinary homology as a sphere.
The only known 3-manifolds satisfying both of these properties
are the Poincar{\'e} sphere and $S^3$.

Rasmussen showed in \cite{Rasmussen}
that if $K \subset L(p,q)$
admits an L-space surgery, then the unique simple knot $K^{\prime}$
in the same homology class as $K$ satisfies 
\begin{equation}
\label{genus bound}
\mathrm{genus}(K^{\prime}) < \frac{p+1}{2}.
\end{equation}

On the other hand, the genus of a simple knot is easily calculated
\cite{Rasmussen},
\cite{OSLens}.
The standard doubly-pointed Heegaard diagram for a simple knot
provides a simple presentation for the fundamental group of the knot,
from which one can compute its Alexander polynomial using Fox calculus.
The simple knot $K^{\prime} \subset L(p,q)$ 
of homology class $k \in H_1(L(p,q)) \cong \Z/p$
has Alexander polynomial
\begin{equation}
  {\Delta}_{K^{\prime}}
= \left(\frac{t - 1}{t^p - 1}\right) {\bar{\Delta}}_{K^{\prime}},
\end{equation}
where
\begin{equation}
{\bar{\Delta}}_{K^{\prime}} = \sum_{i=0}^{p-1} t^{f(i)},
\end{equation}
and $f(i)$ is defined recursively by
\begin{equation}
   f(i+1) - f(i) 
:= \begin{cases}
      k-p
         & iq \in \{0, \ldots, k-1\} \subset \Z/p
             \\
      k
         & \text{otherwise}
   \end{cases}\;\;.
\end{equation}
The genus of $K^{\prime}$ is then given by
\begin{equation}
  \mathrm{genus}(K^{\prime})
= \frac{\mathrm{deg}{\bar{\Delta}}_{K^{\prime}} - p + 1}{2},
\end{equation}
where
\begin{equation}
  \mathrm{deg}{\bar{\Delta}}_{K^{\prime}}
= \max_{i \in \Z/p} f(i) - \min_{i \in \Z/p} f(i).
\end{equation}
Letting $G(p,q,k)$ denote this degree $\mathrm{deg}{\bar{\Delta}}_{K^{\prime}}$,
one can then translate the condition (\ref{genus bound}) on the genus of
$K^{\prime}$ into a condition on $p$, $q$, and $k$:
\begin{equation}
\mathrm{genus}(K^{\prime}) < \frac{p+1}{2}
\;\;\;\;\Leftrightarrow\;\;\;\;
G(p,q,k) < 2p.
\end{equation}

On the other hand, a knot determined by
$p$, $q$, $k$ only admits a homology-sphere surgery
if $q \equiv k^2\;(\mod p)$.
Thus, the knot determined by $p$, $q$, and $k$
admits an L-space homology-sphere surgery if and only if
the following two conditions hold:
\begin{itemize}
\item[(i)]$G(p,q,k) < 2p$,
\item[(ii)]$q \equiv k^2\;(\mod p)$.
\end{itemize}
This translates the topological problem of classifying which
simple knots in $L(p,q)$ admitting $S^3$ surgeries into
the number theoretical and combinatorial problem on which
the present paper focuses.
In this paper, we classify all solutions $(p,q,k)$
of conditions (i) and (ii)
in the case of $p > k^2$, and observe that they
match with Berge's prediction.  That is, we prove:
\begin{theorem}
\label{thm: intro main thm}
Conjecture \ref{conj: berge} holds in the case of $p > k^2$.
\end{theorem}

In the case of $p < k^2$, however, the above strategy generates
nearly 20 different families of lens space knots with Poincar{\'e} sphere
surgeries, making our methods less tractable.

On the other hand,
the case of $p > k^2$ is interesting in its own right.
For example,
one way to construct a knot in a lens space with an
$S^3$ surgery is to start with a knot $K^{\prime} \in S^1 \times D^2$
with a nontrivial $S^1 \times D^2$ surgery.
If $\alpha \in H_1\!\left( \partial\!\left(S^1 \times D^2\right)\right)$
is the class which bounds in the surgery,
then any Dehn filling of $S^1 \times D^2$
along a curve $\beta$ with $\beta \cdot \alpha = \pm 1$
gives a knot in a lens space with an $S^3$ surgery.
In \cite{GabaiTorus} and \cite{BergeTorus},
Gabai and Berge
classified all knots in $S^1 \times D^2$ with
nontrivial $S^1 \times D^2$ surgeries.
As observed by Berge in \cite{Berge},
the knots obtained by this construction are precisely 
the Berge knots of types I through VI
listed in Conjecture \ref{conj: berge}.
Thus, Theorem \ref{thm: intro main thm} implies the
the following result.

\begin{cor}
If $K \subset L(p,q)$ has an integral $S^3$ surgery, and $p > k^2$, where
$k \in H_1(L(p,q)) \cong \Z/p$ is the homology class of $K$,
then the unique simple knot in the same homology class
is obtained as described above from a knot $K^{\prime} \in S^1 \times D^2$ which
admits a nontrivial $S^1 \times D^2$ surgery.
\end{cor}

In Section \ref{s:defs},
we define an invariant $\bar{G}\left(p,q ,k\right)$ of
$(p,q,k)$ which is easier to work with, satisfying
\begin{equation}
G\left(p,q,k\right) = \bar{G}\left(p,q^{-1}(\mod p),k\right).
\end{equation}
We also state a few simple properties of $\bar{G}\left(p,q ,k\right)$.

The $G$--triple $(p,q,k)$ then satisfies $q = k^2 \in \Z/p$
if and only if the corresponding $\bar{G}$--triple satisfies
$q = k^{-2} \in \Z/p$.
In Proposition \ref{prop:section q=k^-2, (p,k^-2,k) is like (k^2, p^-1, k)}
of Section \ref{s:q=k-2}, we prove that
if $p > k^2$, then the $\bar{G}$--triple $(p, k^{-2}, k)$ 
is a solution if and only if
the $\bar{G}$--triple $(k^2, p^{-1}(\mod k^2), k)$ is solution,
thereby reducing our classification problem to the study of
$\bar{G}$--triples of the form $(k^2, q, k)$.

Thus, the bulk of the work in this paper resides in
Section \ref{s:p=k2}, which classifies $\bar{G}$--triples of the form
$(k^2, q, k)$ satisfying $\bar{G}(k^2,q,k) < 2p$.
The calculation of $\bar{G}$ requires the bookkeeping of
certain marked elements of $\Z/p$.  The main strategy of
Section \ref{s:p=k2} is to use a special choice (\ref{eq: definition of z_i^j})
of ordering of these marked elements, in order to exploit some
useful combinatorial properties of the problem at hand.
             \\

\noindent{\bf Acknowledgements:}

I would like to thank Jacob Rasmussen for introducing me to this problem,
Cliff Taubes for advising me from Harvard, Zoltan Szab{\'o} for advising
me from Princeton, Mike Hopkins for advising my minor thesis,
and Richard Taylor for his endless academic and moral support
as director of graduate studies.
I am very grateful to my readers, Cliff Taubes,
Peter Kronheimer, and Noam Elkies,
and to the extremely patient referee who volunteered to check my paper.

I would also like to thank Irene Minder---for
repeatedly going above and
beyond the call of duty to carry out my administrative obligations for me
in my absence---and Susan Gilbert for all her work in the past year
to help me meet graduation requirements in absentia.

Many people have supported me during my graduate school years.
At Princeton, Zoltan's students took me in as one of their own.
In my year between Harvard and Princeton, the Rutgers string theorists
gave me an office and let me participate in their seminars and discussions.
At Harvard, the students in my alcove, in addition to ``honorary'' alcove members,
were quick to include me in social gatherings.  In my year at the physics
department at Stanford, Simeon Hellerman was generous about
answering questions and suggesting papers to read, while
Jacob Shapiro, Kevin Purbhoo, and Tom Coates helped me to retain my
identity as a mathematician.
In the mean time, Catherine and Christian Le Cocq
stood in as surrogate parents and kept me well fed.
Of course, I am also grateful
to the National Defense Science and Engineering Fellowship and the
National Science Foundation's Graduate Research Fellowship for
financing my graduate study.

As an undergraduate, I greatly benefited from interactions with the 
mathematicians and string theorists at Duke.  From the mathematics
department, I would especially like to thank Chad Schoen,
Robert Bryant, David Morrison, Bill Pardon, David Kraines, Eric Sharpe, and my advisor,
Paul Aspinwall.  I would also like to thank my dear friends
Ilarion Melnikov, Sven Rinke, and Ronen Plesser from the physics department.

From my precollegiate years of education, I would like to thank
John Kolena, John Goebel, Dan Teague, and Kevin Bartkovich
from the NC School of Science and Mathematics,
Ken Collins, Rudine Marlowe, Jeff Knull, and Caroline Huggins
from Charlotte Latin School,
Bill Cross and Dana Mackenzie from Duke T.I.P.,
and Harold Reiter from the Charlotte Math Club.

Last but not least, I want to thank my family.
My parents and brother nurtured my love of
mental challenge from an early age,
and have unfailingly supported me through the years.
More recently, my little daughter Katie
has reminded me of how much fun can be found in
exploration and discovery.
Most of all, I am indebted to my husband, Jake.
He took care of me during my long recovery from catching fire
while cooking (oops), and helped to clean up after me during five
months of constant morning sickness, not to mention
all the cooking, cleaning, dishwashing, and
(more recently) baby care he has always done without being asked.
I cannot express how much his love, support, 
and patience have sustained me these last years.

\newpage

\section{General Case: Definitions and Basic Properties}
\label{s:defs}

In the following, we define an invariant $\bar{G}$ related to
the invariant $G$---defined in the Introduction---associated
to simple knots in lens spaces.  We also provide a minor shortcut
for calculating $\bar{G}$ and observe some of its basic symmetries.

\begin{definition}
\label{def:v}
Suppose $p \in \Z$ with $p \geq 2$, and $k, q \in \Z/p$ with $k \neq 0$
and $q$ primitive.  
Then the triple $(p,q,k)$ determines a map
$v_{(p,q,k)} : \Z/p \times \Z/p \rightarrow  \Z$, given by
\begin{equation*}
v_{(p,q,k)}(x,y) := \#\!\left(\Z\cap \left\langle\tilde{x},\tilde{y}\right]\right)\left[k\right]_p \;\; - \;\;
                               \#\!\left(\tilde{Q}\cap \left\langle\tilde{x},\tilde{y}\right]\right)p,
\end{equation*} 
where $\left[k\right]_p$ denotes the representative of $k$ in $\{0, \ldots, p-1\}$,
$\tilde{Q} := {\pi}^{-1}(Q)$ is the preimage under 
$\Z \stackrel{\pi}{\to} \Z/p$ of
\begin{equation*}
Q := \left\{aq\in\Z/p \left|\; a \in \{0, \ldots, \left[k\right]_p-1\} \right. \right\},
\end{equation*}
and  $\tilde{x}, \tilde{y} \in \Z$ are any $\pi$-lifts of $x$ and $y$ 
satisfying $x < y$.
\end{definition}

Since the difference between any two lifts of $(x,y)$ contributes a multiple of
$p [k]_p - [k]_p p = 0$ to $v_{(p,q,k)}(x,y)$, this definition is independent
of the choice of lift of $(x,y)$.  Note that $v_{(p,q,k)}$ is antisymmetric:
\begin{equation}
v_{(p,q,k)}(x,y) = -v_{(p,q,k)}(y,x).
\end{equation}

\begin{definition}
The triple $(p,q,k)$ determines a positive integer $\bar{G}$, given by
\begin{equation*}
\bar{G}(p,q,k) := \max_{x,y \in \Z/p} v_{(p,q,k)}(x,y).
\end{equation*}
\end{definition}
\noindent A little thought shows that $\bar{G}$ is related to the invariant
$G$ defined in the Introduction by
\begin{equation}
\bar{G}(p,q,k) = G(p,q^{-1},k),
\end{equation}
so that $\bar{G}\left(p,q,k\right)$ gives the degree of the rescaled
Alexander polynomial of the simple knot in $L(p, q^{-1}(\mod p))$
of homology class $k \in H_1\!\left(L(p, q^{-1})\right) \cong \Z/p$.
Thus, misleadingly, the $q$ used in most of this paper is {\em not}
the $\tilde{q}$ defining the lens space $L(p,\tilde{q})$ harboring
the simple knot, but is rather the inverse modulo $p$ of that $\tilde{q}$.
If confusion is likely to occur, we shall distinguish between the
arguments of $\bar{G}$ and $G$ by calling them $\bar{G}$--triples
and $G$--triples, respectively.

The following proposition somewhat simplifies the computation of $\bar{G}$:
\begin{prop}
\label{prop:G from aqs}
\begin{equation*}
\bar{G}(p,q,k) =
   \max_{x,y \in Q} \left|v_{(p,q,k)}(x,y)\right|
   + p - [k]_p.
\end{equation*}
\end{prop}

\begin{proof}
For brevity, write $v$ for $v_{(p,q,k)}$.
First, note that since $v(x,y) = -v(y,x)$, the above proposition would be
equivalent if we removed the absolute value sign, but the absolute value sign
will be convenient in later arguments.

Let $(x_*,y_*)$ denote an element of $\Z/p \times \Z/p$ at which a maximum of
$v$ (and thus also of $|v|$) occurs.
There must then exist $a_1\in\{0,\ldots,[k]_p-1\}$ such that
$x_* \equiv {a_1}q\; (\mod p)$.  Otherwise $v(x_* -1,y_*) = v(x_*,y_*) + [k]_p$,
contradicting the maximality of $v(x_*,y_*)$.
Similarly, there must exist $a_2\in\{0,\ldots,[k]_p-1\}$ such that
$y_* \equiv {a_2}q - 1\; (\mod p)$.  Otherwise, 
$v(x_*,y_*+1) = v(x_*,y_*) + [k]_p$.  Thus,
\begin{equation}
   v(x_*,y_*) = v({a_1}q,{a_2}q - 1) = v({a_1}q, {a_2}q) +  p - [k]_p.
\end{equation}
\end{proof}

\begin{definition}
We say that a triple $(p,q ,k)$ is 
{\em genus-minimizing} if 
\begin{equation*}
\bar{G}(p,q,k) < 2p.
\end{equation*}
\end{definition}
\noindent The choice of the term ``genus-minimizing'' is intended to reflect the
fact that $\bar{G}(p,q,k) < 2p$ if and only if the knot
determined by the $\bar{G}$--triple $(p,q,k)$ satisfies the
genus bound necessary for the knot to admit an $S^3$ surgery.

When actually trying to determine if $(p, q, k)$ is genus-minimizing, we
often exploit Proposition \ref{prop:G from aqs} to use the following
simpler criterion.
\begin{cor}
\label{cor: intro defs, genus-minimizing if v < p + k}
The triple $(p, q, k)$ is genus-minimizing if and only if
\begin{equation*}
\max_{x,y \in Q} \left|v_{(p,q,k)}(x,y)\right| < p + [k]_p.
\end{equation*}
\end{cor}

Lastly, we observe a few basic symmetries of $\bar{G}$.
\begin{prop}
\label{prop:G properties}
The invariant $\bar{G}$ obeys the following three identities:
\begin{itemize}
\item[\em{(i)}] $\bar{G}(p,-q,k) = \bar{G}(p,q,k)$,
\item[\em{(ii)}] $\bar{G}(p,q,-k) = \bar{G}(p,q,k)$,
\item[\em{(iii)}] $\bar{G}(p,q^{-1},qk) = \bar{G}(p,q,k)$,
\end{itemize}
\end{prop}
\begin{proof}[Proof of (i)]

It is clear from the original definition of $v$ (Definition \ref{def:v}) that
\begin{equation}
v_{(p,q,k)}(x,y) = -v_{(p,-q,k)}(-x,-y) = v_{(p,-q,k)}(-y,-x).
\end{equation}
But $(x,y) \mapsto (-y,-x)$ is just an involution on $\Z/p \times \Z/p$,
so $\bar{G}(p,-q,k) = \bar{G}(p,q,k)$.
\end{proof}

\begin{proof}[Proof of (ii)]
We begin by remarking on the effect on $v$ of translating $Q$.
For any $a_0 \in \Z/p$, define $Q_k^0$ and $Q_k^{a_0}$ by
\begin{align}
    Q_k^0 
&:= \left\{aq\in\Z/p \left|\; a \in \{0, \ldots, \left[k\right]_p-1\} \right. \right\},
            \\
    Q_k^{a_0}
&:=  \left\{aq\in\Z/p \left|\; a \in \{a_0 + 0, \ldots, a_0+\left[k\right]_p-1\} \right. \right\},
\end{align}
so that $Q_k^{a_0} =  {a_0}q + Q_k^0$.  This translation of $Q_k^0$ has the
effect of translating the domain of $v$.  That is,
\begin{equation}
v^{Q_k^0}_{(p,q,k)}(x,y)  = v^{Q_k^{a_0}}_{(p,q,k)}(x+{a_0}q, y+{a_0}q),
\end{equation}
where $v^q$ denotes the result of replacing $Q$ with $q$
in the definition of $v$.

Returning to the problem at hand, set
\begin{equation}
Q_{p-k}^0 = \left\{ aq \in \Z/p | a \in \{0, \ldots, p-[k]_p - 1\}\right\}.
\end{equation}
Then $\Z/p = Q^0_{p-k} \coprod Q^{p-k}_k$, so if we let
${\tilde{Q}}^0_{p-k} := {\pi}^{-1}\left(Q^0_{p-k}\right)$ and 
${\tilde{Q}}^{p-k}_k := {\pi}^{-1}\left(Q^{p-k}_k\right)$
denote the preimages of $Q^0_{p-k}$ and $Q^{p-k}_k$ under 
$\Z \stackrel{\pi}{\to} \Z/p$, then we have
\begin{equation}
{\tilde{Q}}^0_{p-k}  = \Z \; \setminus \; {\tilde{Q}}^{p-k}_k.
\end{equation}
Applying this relation to the definition of $v$ gives
\begin{align}
v_{(p,q,-k)}(x,y) 
    :\!\!&=  \#\left(\Z\cap (\tilde{x},\tilde{y}]\right) (p\!-\![k]_p)\; - \;
                \#\!\left( {\tilde{Q}}^0_{p-k} \cap (\tilde{x},\tilde{y}]\right)p
                  \\ \nonumber
         &=   \#\left(\Z\cap (\tilde{x},\tilde{y}]\right)(p\!-\![k]_p) \; - \;
                 \left[ \#\left( \Z \cap (\tilde{x},\tilde{y}]\right) \;-\;
                          \#\left( {\tilde{Q}}^{p-k}_k \!\cap (\tilde{x},\tilde{y}]\right) \right]p
                  \\ \nonumber
         &= - \#\left(\Z\cap (\tilde{x},\tilde{y}]\right)[k]_p \; + \;
                 \#\!\left({\tilde{Q}}^{p-k}_k \!\cap (\tilde{x},\tilde{y}]\right)p
                  \\ \nonumber
         &= -v_{(p,q,k)}(x-(p\!-\!k)q,\;y-(p\!-\!k)q).
\end{align}
Then, since
\begin{equation}
\max_{x,y \in \Z/p} -v_{(p,q,k)}(x-(p\!-\!k)q,\;y-(p\!-\!k)q)
= \max_{x,y\in\Z/p} v_{(p,q,k)}(x,y),
\end{equation}
we have $\bar{G}(p,q,-k) = \bar{G}(p,q,k)$.
\end{proof}

\begin{proof}[Proof of (iii)]
In Lemma 2.5 of \cite{Rasmussen}, Jacob Rasmussen proves that
under the identification $L(p,q) \cong L(p,q^{-1})$
obtained by exchanging the roles of $\alpha$ and $\beta$
in the Heegaard diagram,
the $G$--triples $(p,q,k)$ and $(p,q^{-1}, q^{-1}k)$
specify the same simple knot.

This implies that the $G$--triples
$(p, q^{-1}, k)$ and $(p, q, qk)$ specify the same knot,
which in turn implies that the $\bar{G}$--triples
$(p, q, k)$ and $(p, q^{-1}, qk)$ specify the same knot
(recalling that $\bar{G}(p,q,k) = G(p,q^{-1},k)$ for all $p$, $q$, and $k$).
In particular, this means that the knots corresponding to the $\bar{G}$--triples
$(p, q, k)$ and $(p, q^{-1}, qk)$ have the same Alexander polynomial.
Thus $\bar{G}(p,q,k) = \bar{G}(p,q^{-1}, qk)$.
\end{proof}

From here on, we restrict our attention to special cases 
relevant to Berge's Conjecture.
Berge's Conjecture involves the classification of genus-minimizing
$\bar{G}$--triples for which $q = k^{-2}$ in $\Z/p$, and of those,
we are interested in the case in which $p > \left([k]_p\right)^2$.  As shown in
Section \ref{s:q=k-2}, this classification problem is equivalent to
classifying genus-minimizing $\bar{G}$--triples of the form $(k^2, q, k)$,
{\em i.e.}, with $p = k^2$, where here, we take $k$ to be a positive integer.
There is no loss of generality in taking $k$ to be positive, since
by Proposition \ref{prop:G properties}, $(k^2, q, -k)$ is genus-minimizing
if and only if $(k^2, q, k)$ is genus-minimizing.   We focus exclusively on
this latter case in Section \ref{s:p=k2}.

\section{Case $p=k^2$}
\label{s:p=k2}

For the entirety of this section, we take $k$ to be a fixed positive integer
and set $p = k^2$.  Then the fact that we need $k \in \{0, \ldots, p-1\}$
implies that $k < k^2$, so we know that $k \geq 2$.
Furthermore, any primitive element $q \in \Z/k^2$ defines an entire
triple $(k^2, q, k)$, so we shall often abbreviate notation
and write $v_q$ for $v_{(k^2, q, k)}$ and $\bar{G}(q)$ for
$\bar{G}(k^2, q, k)$.  We shall also say that $q$ is (or is not)
genus-minimizing to convey that the triple $(k^2, q, k)$
is (or is not) genus-minimizing.  Lastly, we shall write
$Q_q$ to denote the set 
$Q_q := \left\{ aq | a \in \{0, \ldots, k-1\}\right\} \subset \Z/k^2$,
and ${\tilde{Q}}_q := {\pi}^{-1}\left(Q_q\right) \subset \Z$ to denote the 
preimage of $Q_q$ under $\Z \stackrel{\pi}{\rightarrow} \Z/k^2$.

The goal of this section is to classify all genus-minimizing
elements $q \in \Z/k^2$.  To determine if a particular $q$ is genus-minimizing,
we shall apply the following criterion:
\begin{prop}
\label{prop: not minimizing if v >= k(k+1)}
For any primitive $q \in \Z/k^2$, 
$q$ is genus-minimizing if and only if
$v_q(x,y) \leq k(k-1)$ for all $x,y \in Q_q$, and 
$q$ is not genus-minimizing
if and only if
there exist $x_*, y_* \in Q_q$ such that $\left|v_q(x_*, y_*)\right| \geq k(k+1)$.
\end{prop}
\begin{proof}
Corollary \ref{cor: intro defs, genus-minimizing if v < p + k} states that a
triple $(p, q, k)$ is genus-minimizing if and only if
\begin{equation}
\max_{x,y \in Q} \left|v_{(p,q,k)}(x,y)\right| < p + k.
\end{equation}
Here, $p+k = k^2 + k$, and so it remains to show that $v_q(x,y) \neq k^2$ for all $x, y \in Q_q$.
This is true because if there exist $a_1, a_2 \in \{0, \ldots, k-1\}$ such that
$\frac{v_q(a_1 q, a_2 q)}{k} \equiv 0 \;(\mod k)$, then $a_2 q - a_1 q \equiv 0\; (\mod k)$,
implying $a_1 = a_2$, so that $v_q(a_1 q, a_2 q) = 0$.

\end{proof}

\subsection{Notation}
Before proceeding further, we should establish some modular
arithmetic notation.  Most of these notations are completely
conventional.  For $x \in \Q$, we use ${\lfloor}x{\rfloor}$ to
indicate the greatest integer less than or equal to $x$, and ${\lceil}x{\rceil}$
to indicate the least integer greater than or equal to $x$.
Equally conventional, for $x,y,N \in \Z$, is our use of
a vertical bar to indicate divisibility ($x|y$ denotes that
$x$ divides $y$), and the notation 
$x \equiv y\; (\mod N)$ to indicate that $N|x - y$.
What is less conventional is our choice of notation to pick out
a particular representative of a congruence class modulo $N$.
For any $N >0$, we define the map
\begin{equation}
\left[\cdot\right]_N : \Z/N \rightarrow \{0, \ldots, N-1\} \subset \Z
\end{equation}
by setting $\left[x\right]_N$ equal to the unique integer in $\{0, \dots, N-1\}$
such that $\left[x\right]_N \equiv x\; (\mod N)$.  We shall also sometimes 
use $[\cdot]_{N}$ to denote the composition of the quotient map
$\Z \rightarrow \Z/N$ with the representative selection map 
$\Z/N \stackrel{[\cdot]_N}{\rightarrow} \Z$.

\subsection{Definition of Parameters}
\label{ss: definition of d, c, m, mu, gamma, alpha, and types}
From now on, we take $k$ to satisfy $k > 100$,
to avoid special cases that arise when $k$ is small.
It is easy to check by computer
that the genus-minimizing solutions for $q$ match those in 
Proposition \ref{prop: bookkeeping for q genus-minimizing}
when $2 \leq k \leq 100$.

As will soon become clear in the proof of Proposition \ref{prop: G = 2k(k-1) and unique}
and in the definition of mobile points, we shall spend much more
time using $[\pm q^{-1}]_k q$ than $q$ itself.
Our goal is therefore the following.
After choosing an appropriate
sign $\xi \in \{\pm1\}$, we want to choose integers
$d \in \{ [\pm q^{-1}]_k \}$,
$m, c \in \{0, \ldots, k-1\}$, and
$\alpha, \gamma, \mu \in \{\pm1\}$ to parameterize $\xi q$, such that  
\begin{equation}
\left[d{\xi}q\right]_{k^2} = ({\mu}m+ {\gamma}c)k + \alpha < \frac{k^2}{2}.
\end{equation}
For notational convenience, we first
define the function 
${\sigma}_n : \{\pm 1 \in \Z/n \} \rightarrow \{\pm 1 \in \Z \}$,
for $n\in{\Z}_{>0}$,
by ${\sigma}_n(+1)=+1$ and ${\sigma}_n(-1)=-1$ when
$n>2$, by ${\sigma}_2 \equiv -1$ when $n=2$, and by
${\sigma}_1 \equiv +1$ when $n=1$.  We then parameterize $q$ as follows.

\begin{definition}
\label{def: parameters assigned to q}
Suppose that $k>100$.
We then take the following steps to associate
the parameters $d, c, m, \xi, \alpha, \gamma, \mu \in \Z$
to any given primitive $q \in \Z/k^2$.  Set
\begin{align}
\label{eq: parameter definitions, d}
d :=& \min\left\{ \left[q^{-1}\right]_k, \left[-q^{-1}\right]_k \right\},
               \\
\label{eq: parameter definitions, xi}
\xi :=& \begin{cases}
                +1     & \left[dq\right]_{k^2} < \textstyle{\frac{k^2}{2}}
                                \\
                -1     & \left[dq\right]_{k^2} > \textstyle{\frac{k^2}{2}}
            \end{cases}
            \;\;\;\text{(so that}\;
            \left[d{\xi}q\right]_{k^2} = \min\left\{ \left[dq\right]_{k^2},  \left[-dq\right]_{k^2} \right\}
            \text{)},
               \\
\label{eq: parameter definitions, alpha}
\alpha :=& {\sigma}_k(dq)\xi,
               \\
\label{eq: parameter definitions, c}
c :=& \min\left\{ \left[k^{-1}\right]_d, \left[-k^{-1}\right]_d \right\},
               \\
\label{eq: parameter definitions, gamma}
\gamma :=& {\sigma}_d(-ck)\alpha,
               \\
\label{eq: parameter definitions, mu}
\mu :=& \begin{cases}
                +1     & m^{\prime} \geq 0
                                                \\
                -1     & m^{\prime} < 0
                            \end{cases},
               \;\;\;\;\text{where}\;m^{\prime} := \left[\frac{d{\xi}q - \alpha}{k}\right]_k  - {\gamma}c,
               \\
\label{eq: parameter definitions, m}
m :=& \left[{\mu}m^{\prime}\right]_k.
\end{align}
\end{definition}
It is then straightforward to show that the above definitions imply that
\begin{equation}
\xi q = \alpha\gamma\mu \frac{ck+\alpha\gamma}{d}\left(mk+\alpha\mu\right) \;\;\in\Z/k^2
\end{equation}
and that $\left[d{\xi}q\right]_{k^2} = ({\mu}m+ {\gamma}c)k + \alpha$.
The definition of $\xi$ in  (\ref{eq: parameter definitions, xi}) then implies that
$\left[d{\xi}q\right]_{k^2} < \frac{k^2}{2}$.
Note that, since (\ref{eq: parameter definitions, gamma}) implies
$d$ divides $ck +\alpha\gamma$, the fraction $\frac{ck + \alpha\gamma}{d}$
denotes an integer.  This also implies that 
$c = 0$ if and only if $d = 1$, which is true if and only if
$q \equiv \pm 1\;(\mod k)$.

We next consider the possible values of $c, d,$ and $m$
when $c \neq 0$.  So far, we know that $c>0$ and $d > 1$.
Moreover, (\ref{eq: parameter definitions, gamma})
and the definition of ${\sigma}_2$ imply that $d=2$ only if $\alpha\gamma = -1$.
Since $k > 2$, (\ref{eq: parameter definitions, d}) tells us that
$d < \frac{k}{2}$.  Similarly, (\ref{eq: parameter definitions, c}) 
tells us that $c=1$ when $d=2$, and otherwise $1 \leq c < \frac{d}{2}$.
This fact, combined with (\ref{eq: parameter definitions, gamma}),
implies that $0 < \frac{ck+\alpha\gamma}{d} < \frac{k}{2}$.
In addition, since $1 < d < \frac{k}{2}$, we know that
$\frac{ck+\alpha\gamma}{d} \neq 1$, and so
$\frac{ck+\alpha\gamma}{d} \geq 2$.  The possible values for $m$
depend on $(\mu, \gamma) \in \{(1,1), (1,-1), (-1,1)\}$ (with $(-1,-1)$ excluded
because (\ref{eq: parameter definitions, mu}) implies $\mu = +1$ when
$\gamma = -1$).
Since (\ref{eq: parameter definitions, xi}) and (\ref{eq: parameter definitions, alpha}) 
ensure that $\left[\frac{d\xi q - \alpha}{k}\right]_k \leq \frac{k}{2}$, we have
$0 \leq m \leq \frac{k}{2} - c < \frac{k}{2}$
when $(\mu,\gamma) = (1,1)$, and
$1 \leq c \leq m \leq \frac{k}{2} + c < \frac{3k}{4}$
when $(\mu, \gamma) = (1,-1)$.
When
$(\mu, \gamma) = (-1,1)$, we have $m = c - \left[\frac{d\xi q - \alpha}{k}\right]_k > 0$,
and so $0 < m \leq c < \frac{k}{4}$.

It turns out that many properties of $v_q$ depend on whether
$c=0$ and on the value of $(\mu, \gamma)$, motivating the
following definitions.

\begin{definition}
Suppose that $k>100$ and $q$ is primitive in $\Z/k^2$.
If $q \equiv \pm 1\; (\mod k)$, or equivalently, if $c=0$,
then we say that $q$ is of {\em type 0}.
\end{definition}

\begin{definition}
Suppose that $k>100$, that $q$ is primitive in $\Z/k^2$, and that
$c \neq 0$, so that $q$ is not of type 0.
We then say that $q$ is of {\em positive type} if 
$q = +\xi q$, and of {\em negative type} if $q = -\xi q$.
In either case, $[d\xi q]_{k^2} = (\mu m+ \gamma c)k + \alpha$,
with $(\mu, \gamma) \in \{(1,1), (1,-1), (-1,1)\}$.
\end{definition}

Note that Proposition \ref{prop:G properties} implies that $q$ is
genus-minimizing if and only if $\xi q$ is genus-minimizing.

\begin{prop}
\label{prop: properties of parameters d, m, c, alpha, mu, gamma}
Suppose that $q$ is genus-minimizing and of positive or negative type.
Then $d$, $m$, $c$, $\mu$, $\gamma$, and $\alpha \in \Z$
satisfy the following properties.
\begin{itemize}
\item[(i)]
$\xi q = \alpha\gamma\mu \frac{ck+\alpha\gamma}{d}\left(mk+\alpha\mu\right) \in \Z/k^2$,
with $\frac{ck+\alpha\gamma}{d} \in \Z$.

\item[(ii)]
$\left[d\xi q\right]_{k^2} = ({\mu}m+ {\gamma}c)k + \alpha < \frac{k^2}{2}$.

\item[(iii)]
$2 \leq d < \frac{k}{2}$, and $d \geq 3$ when $\alpha\gamma = +1$.

\item[(iv)]
$1 \leq c < \frac{d}{2}$ (unless $d=2$, in which case $c=1$).

\item[(v)]
$2 \leq  \frac{ck+\alpha\gamma}{d} < \frac{k}{2}$.

\item[(vi)]
$1 \leq m \leq \frac{k}{2} - c < \frac{k}{2}$ when $(\mu,\gamma) = (1,1)$, 
{\em i.e.}, when $d\xi q = (m+c)k + \alpha$.

\noindent $1 \leq c <  m \leq \frac{k}{2} + c < \frac{3k}{4}$ when $(\mu,\gamma) = (1,-1)$,
{\em i.e.}, when $d\xi q = (m-c)k + \alpha$.

\noindent $1 \leq m < c < \frac{k}{4}$ when $(\mu,\gamma) = (-1,1)$,
{\em i.e.}, when $d\xi q = (c-m)k + \alpha$.
\end{itemize}
\end{prop}
\begin{proof}
We have already discussed the properties as listed above, except
for some changes made to the inequalities in (vi), due to the fact that
we now know that $q$ is genus-minimizing.  The only extra information
we added is the fact that
$m\neq c$ if $q$ is genus-minimizing and $(\mu, \gamma) \in \{(1,-1), (-1,1)\}$,
and the fact that $q$ is not genus-minimizing when $m=0$.

Suppose that $m=c$ and $(\mu, \gamma) \in \{(1,-1), (-1,1)\}$.
Then, setting $q^{\prime}:= \alpha\xi q$, we have
$(0q^{\prime}, dq^{\prime}, 2dq^{\prime}) = (0,1,2)$.  Thus
\begin{align}
    -v_{q^{\prime}}(0{q^{\prime}},2d{q^{\prime}})
&= -(2-0)k \;\;+\;\;
      \# \left({\tilde{Q}}_{q^{\prime}} \cap \left\langle 0,2\right] \right) k^2
      \\ \nonumber
&= -2k + 2k^2
      \\ \nonumber
&\geq k(k+1) \;\;\mathrm{when}\;\; k \geq 3,
\end{align}
and so, by Proposition 
\ref{prop: not minimizing if v >= k(k+1)},
$q^{\prime}$, hence $q$,
is not genus-minimizing when $m=c$ and $(\mu, \gamma) \in \{(1,-1), (-1,1)\}$.

Suppose $m = 0$.
Then setting $q^{\prime} := \gamma\xi q = \frac{ck+{\alpha\gamma}}{d}$ makes
$[q^{\prime}]_{k^2} < \frac{k}{2}$, and so
$[aq^{\prime}]_{k^2} = a[q^{\prime}]_{k^2}$
for all $a \in \{0, \ldots, k-1\}$.  Thus
\begin{align}
   \left|v_{q^{\prime}}(0{q^{\prime}},(k-1){q^{\prime}})\right|
&= \left|\left((k-1)[q^{\prime}]_{k^2} - 0[q^{\prime}]_{k^2}\right)k \;\;-\;\; 
      \# \left({\tilde{Q}}_q \cap \left\langle 0, (k-1)[q^{\prime}]_{k^2}\right] \right) k^2\right|
      \\ \nonumber
&= \left| (k-1)[q^{\prime}]_{k^2}k \;\;-\;\; (k-1) k^2\right|
      \\ \nonumber
&= k(k-1)(k-[q^{\prime}]_{k^2})
      \\ \nonumber
&> k(k-1)\textstyle{\frac{k}{2}}
      \\ \nonumber
&\geq k(k+1) \;\;\mathrm{when}\;\; k \geq 4,
\end{align}
and so, by Proposition 
\ref{prop: not minimizing if v >= k(k+1)},
$q^{\prime}$, hence $q$,
is not genus-minimizing when $m=0$.
\end{proof}

\subsection{Genus-Minimizing $q$ of Type $0$}
\label{ss: q of type 0}

We begin by classifying the genus-minimizing solutions for
$q$ of type 0, since this requires no additional machinery.
Recall that $q$ is of type 0 if and only if
$q \equiv \pm 1\; (\mod k)$.  Thus we may write
$q = nk \pm 1$, for some $n \in \{0, \ldots, k-1\} \subset \Z$.

\begin{prop}
\label{prop: type 0 genus-minimizing classification}
If $q$ is of type 0, so that we may write $q=nk\pm1$,
then $q$ is genus-minimizing if and only if $\gcd(n,k) \in \{1,2\}$.
If such $q$ is genus-minimizing, then $\bar{G}(q) = 2k(k-1)$,
and the maximum is attained uniquely.
\end{prop}
\begin{proof}
As observed in the proof of Proposition \ref{prop:G properties}.(i),
$v_{-q}(x,y) = v_q(-y,-x)$.  Thus $\bar{G}(q) := \max_{x,y \in \Z/k^2} v_q(x,y)$
satisfies $\bar{G}(q) = \bar{G}(-q)$, and the maximum is attained
uniquely for $q$ if and only if it is attained uniquely for $-q$.
It therefore suffices to take $q = nk + 1$,
since any $q^{\prime} = n^{\prime}k -1$ satisfies
$-q^{\prime} = (k-n^{\prime})k + 1$.

We first show that $q$ is not genus-minimizing when $\gcd(n,k) \geq 3$.
Let $\delta = \gcd(n,k)$, so that $q = \textstyle{\frac{n}{\delta}}{\delta}k + 1$,
and suppose that $\delta \geq 3$.  Then
\begin{equation}
    \left(0q, {\textstyle{\frac{k}{\delta}}}q, 2{\textstyle{\frac{k}{\delta}}}q \right)
=  \left(0, {\textstyle{\frac{k}{\delta}}}, 2{\textstyle{\frac{k}{\delta}}}\right)
\end{equation}
in $\left(\Z/k^2\right)^3$, and so
\begin{align}
       \left| v_q \left(0q, 2{\textstyle{\frac{k}{\delta}}}q \right) \right|
&=  \left|\left( 2{\textstyle{\frac{k}{\delta}}} - 0\right)k \;\;-\;\;
       \#\left({\tilde{Q}}_q \cap \left\langle 0,  2{\textstyle{\frac{k}{\delta}}}\right] \right)k^2\right|
           \\ \nonumber
&\geq  -2{\textstyle{\frac{k}{\delta}}}k \;\;+\;\; 2k^2
           \\ \nonumber
&=     k(2k - 2{\textstyle{\frac{k}{\delta}}})
           \\ \nonumber
&\geq   k(2k - (k-1)) \;\;\;(\mathrm{since}\; \delta \geq 3)
           \\ \nonumber
&=     k(k+1).
\end{align}
Thus, by Proposition \ref{prop: not minimizing if v >= k(k+1)},
$q$ is not genus-minimizing.

Next, suppose that $\gcd(n,k)=1$.
Since $n^{-1} \in \Z/k$ exists, we have
\begin{equation}
    {\left[ jn^{-1} \right]_k} q 
=  jk + \left[ jn^{-1} \right]_k \;\; \in \Z/k^2
   \forall j \in \{0, \ldots, k-1\}.
\end{equation}
Thus for each $j \in \{0, \ldots, k-1\}$, 
we observe that ${\tilde{Q}}_q \cap \left[jk, (j+1)k\right\rangle$ contains
precisely one element, namely, $\left[{\left[ jn^{-1} \right]_k}q\right]_{k^2}$.
This means that, for any $j_1 < j_2$ with $j_1, j_2 \in \{0, \ldots, k-1\}$,
\begin{equation}
\# \left({\tilde{Q}}_q \cap
   \left\langle \left[\left[{j_1}n^{-1} \right]_k q\right]_{k^2} ,
                \left[\left[{j_2}n^{-1} \right]_k q\right]_{k^2} \right] \right)
   \;=\; j_2 - j_1.
\end{equation}
We now compute $v_q\left({a_1}q, {a_2}q\right)$ for an arbitrary
pair $a_1, a_2 \in \{0, \ldots, k-1\}$, ordered such that
$\left[{a_1}n\right]_k < \left[{a_2}n\right]_k$:
\begin{align}
      v_q\left({a_1}q, {a_2}q\right)
&=    k  \cdot \left[ \left({\left[{a_2}n\right]_k}k + a_2\right) - 
                     \left({\left[{a_1}n\right]_k}k + a_1\right) \right]  \;\;-\;\;
      k^2\cdot \left({\left[{a_2}n\right]_k} - {\left[{a_1}n\right]_k} \right)
          \\ \nonumber
&= k\left(a_2-a_1\right).
\end{align}
Thus $\left|v_q\left({a_1}q, {a_2}q\right)\right| \leq k(k-1)$, and the maximum
value of $v_q\left({a_1}q, {a_2}q\right)$ is attained uniquely when 
$(a_1,a_2) = (0,k-1)$, yielding $v_q\left(0q, (k-1)q\right) = k(k-1)$.
Thus by Proposition \ref{prop:G from aqs},
\begin{equation}
\bar{G}(q) = k(k-1) + k^2 - k = 2k(k-1),
\end{equation}
and the maximum is attained uniquely.

Finally, suppose that $\gcd(n,k)=2$.
Let $s = \frac{n}{2}$, so that $q = 2sk + 1$.
Then $\gcd\left(s, \frac{k}{2}\right) = 1$ implies
$s^{-1} \in \Z/\frac{k}{2}$ exists, and so
for each $j \in \{0, \ldots, \frac{k}{2}-1\}$, we have
\begin{align}
          \left[ js^{-1} \right]_{\frac{k}{2}} q 
&\equiv   2jk + \left[ js^{-1} \right]_{\frac{k}{2}} \; (\mod k^2),
                 \\
          \left(\left[ js^{-1} \right]_{\frac{k}{2}} 
                + \textstyle{\frac{k}{2}} \right) q 
&\equiv   2jk + \left[ js^{-1} \right]_{\frac{k}{2}} + \textstyle{\frac{k}{2}} \; (\mod k^2).
\end{align}
Thus for each $j \in \{0, \ldots, \frac{k}{2}-1\}$, the set 
${\tilde{Q}}_q \cap \left[2jk, 2(j +1)k\right\rangle$ contains
precisely two elements, namely, 
$\left[ \left[ js^{-1} \right]_{\frac{k}{2}} q \right]_{k^2}$ and
$\left[ \left(\left[ js^{-1} \right]_{\frac{k}{2}} + \frac{k}{2} \right)q \right]_{k^2}$.
This is similar to the case in which $\gcd(n,k)=1$, but this time
there are two distinct types of element in each interval of length $2k$,
so when we go to compute $v_q(x,y)$ for arbitrary elements $x,y\in Q_q$,
there will be four types of pairs $(x,y)$ to consider.

Consider an arbitrary pair $a_1,a_2 \in \{0, \ldots, \frac{k}{2}-1\}$,
ordered such that 
$\left[{a_1}s\right]_{\frac{k}{2}} < \left[{a_2}s\right]_{\frac{k}{2}}$.
In order to exhaust all possible pairs of elements in $Q_q$, we need
to consider all four pairs,
\begin{equation}
\left({a_1}q, {a_2}q\right),\;
\left(\left({a_1} + \textstyle{\frac{k}{2}}\right)\!q, 
         \left({a_2} + \textstyle{\frac{k}{2}}\right)\!q\right),\;
\left({a_1}q, \left({a_2} + \textstyle{\frac{k}{2}}\right)\!q\right),\;\mathrm{and}\;
\left(\left({a_1} + \textstyle{\frac{k}{2}}\right)\!q, {a_2}q\right).
\end{equation}
By arguments similar to those used in the case of $\gcd(n,k)=1$, we have
\begin{equation}
   v_q\left({a_1}q, {a_2}q\right)
=  v_q\left(\left({a_1} + \textstyle{\frac{k}{2}}\right)\!q, 
            \left({a_2} + \textstyle{\frac{k}{2}}\right)\!q\right)
=  k\left(a_2-a_1\right),
\end{equation}
so that
\begin{equation}
\left|v_q\left({a_1}q, {a_2}q\right)\right|,
\left|v_q\left(\left({a_1} + \textstyle{\frac{k}{2}}\right)\!q, 
            \left({a_2} + \textstyle{\frac{k}{2}}\right)\!q\right)\right|
\leq k\left(\textstyle{\frac{k}{2}}-1\right).
\end{equation}
We compute the two remaining cases by hand.
\begin{align}
      v_q \left({a_1}q, \left({a_2} + \textstyle{\frac{k}{2}}\right)\!q\right)
&=    k\cdot\left[ \left( 2\left[{a_2}s\right]_{\frac{k}{2}}k 
                          + a_2 + \textstyle{\frac{k}{2}} \right)   -
                   \left( 2\left[{a_1}s\right]_{\frac{k}{2}}k + a_1\right) \right]
                    \\ \nonumber
&\;\;\;\;\;\;\;\;-\; k^2\cdot \left(2\left(\left[{a_2}s\right]_{\frac{k}{2}}
                         -  \left[{a_1}s\right]_{\frac{k}{2}}\right) + 1 \right)
      \\ \nonumber
&=   k\left(a_2 - a_1 - \textstyle{\frac{k}{2}} \right),
\end{align}
so that 
$\left|v_q \left({a_1}q, \left({a_2} + \textstyle{\frac{k}{2}}\right)\!q\right)\right| 
\leq k(k-1)$, and
\begin{align}
      v_q \left(\left({a_1} + \textstyle{\frac{k}{2}}\right)\!q, {a_2}q\right)
&=    k\cdot\left[ \left( 2\left[{a_2}s\right]_{\frac{k}{2}}k + a_2\right) - 
                   \left( 2\left[{a_1}s\right]_{\frac{k}{2}}k + a_1 
                     + \textstyle{\frac{k}{2}} \right) \right]
                    \\ \nonumber
&\;\;\;\;\;\;\;\;-\; k^2\cdot \left(2\left(\left[{a_2}s\right]_{\frac{k}{2}}
                         -  \left[{a_1}s\right]_{\frac{k}{2}}\right) - 1 \right)
      \\ \nonumber
&=   k\left(a_2 - a_1 + \textstyle{\frac{k}{2}} \right),
\end{align}
so that
$\left|v_q \left(\left({a_1} + \textstyle{\frac{k}{2}}\right)\!q, {a_2}q\right)\right| 
\leq k(k-1)$.

Thus $\left|v_q(x,y)\right| \leq k(k-1)$ for all $x,y \in Q_q$, and the
maximum value of $v_q(x,y)$ is attained uniquely when 
$(x,y) = \left(0q,(k-1)q\right)$.  Proposition \ref{prop:G from aqs} then gives
\begin{equation}
\bar{G}(q) = k(k-1) + k^2 - k = 2k(k-1).
\end{equation}
\end{proof}

\subsection{$\bar{G}(q) = 2k(k-1)$ for Genus-Minimizing $q$}

In the previous section, we learned that 
any genus-minimizing $q$ of type 0 satisfies
$\bar{G}(q) = 2k(k-1)$, and the corresponding maximum
is uniquely attained.  In fact, this result is true for
{\em any} genus-minimizing $q \in \Z/k^2$, regardless of 
the form $q$ or $-q$ takes.

\begin{prop}
\label{prop: G = 2k(k-1) and unique}
If $q$ is genus-minimizing, then $\bar{G} = 2k(k-1)$ and
the corresponding maximum is uniquely attained. 
\end{prop}
\begin{proof}
It is instructive to begin with a more general question:
Given $l \in \{0, \ldots, k-1\}$, what are all the pairs
$x,y \in Q_q$ for which
$\frac{v_q(x,y)}{k} \equiv l \;(\mod k)$?
The answer is straight-forward.
Write $(x,y) = \left({a_1}q, {a_2}q\right)$, with
$a_1, a_2 \in \{0, \ldots, k-1\}$.  Then
\begin{align}
         \frac{v_q\left({a_1}q,{a_2}q\right)}{k} 
&\equiv l \;(\mod k)
                \\
         {a_2}q - {a_1}q
&\equiv l \;(\mod k)
                \\
        {a_2}
&\equiv {a_1} + q^{-1}l \;(\mod k),
\end{align}
so there are exactly $k$ such pairs, 
$(x,y) \in
\left\{ \left. \left(aq, {\left[a +  q^{-1}l\right]_k}q\right) \right|
a \in \{0, \ldots, k-1\} \right\}$.

Note that this answer implies that if $\frac{v_q(x,y)}{k} \equiv 0\;(\mod k)$,
then $v_q(x,y) = 0$.
In particular, $k(k) \notin \left\{\left. v_q(x,y) \right| x,y \in Q_q\right\}$.

On the other hand, if $\frac{v_q(x,y)}{k} \equiv l \;(\mod k)$ for some
$l \neq 0$, and $q$ is genus-minimizing,
then $v_q(x,y) \in \{k(l), k(l-k)\}$.  If, in addition, $l$ is primitive in $\Z/k$,
then it is easy to determine how many of these pairs $(x,y)$ satisfy $v_q(x,y) = k(l)$
and how many satisfy $v_q(x,y) = k(l-k)$.  First, note that is clear
from the definition of $v$ that $v_q(u,v) + v_q(v,w) = v_q(u,w)$
for any $u,v,w \in \Z/k^2$.  Thus
\begin{align}
   \sum_{j\in \{0,\dots, k-1\}} 
   v_q\left( {\left[j(q^{-1}l)\right]_k}q, {\left[(j+1)(q^{-1}l)\right]_k}q  \right)
&= v_q\left( {\left[0(q^{-1}l)\right]_k}q, {\left[k(q^{-1}l)\right]_k}q \right)
          \\ \nonumber
&=  v_q(0,0)
          \\ \nonumber
&= 0, 
\end{align}
but we also know that $(x,y) \in
\left\{\left.\left( {\left[j(q^{-1}l)\right]_k}q, {\left[(j+1)(q^{-1}l)\right]_k}q \right)
           \right| j \in \{0, \ldots k-1\} \right\}$
if and only if $v_q(x,y) \in \{k(l), k(l-k)\}$.
Let $t:= \#\left\{(x,y)\in Q_q \times Q_q | v_q(x,y) = k(l)\right\}$ 
(which implies $k-t$ is the number of pairs $(x,y)$ with $v_q(x,y) = k(l-k)$).
Then we can rewrite the above sum as
\begin{equation}
t\cdot k(l) + (k-t)\cdot k(l-k) = 0.
\end{equation}
This linear equation has solution $t = k-l$.  Thus 
\begin{align}
\label{k-l guys = l}
   \#\left\{(x,y)\in Q_q \times Q_q | v_q(x,y) = k(l)\right\}
&= k-l,\;\;\; \mathrm{and}
        \\
   \#\left\{(x,y)\in Q_q \times Q_q | v_q(x,y) = k(l-k)\right\}
&= l.
\end{align}

At last, we address the proposition at hand.
Suppose $q$ is genus-minimizing.  Thus
by Proposition \ref{prop: not minimizing if v >= k(k+1)},
$v_q(x,y) \leq k(k-1)$ for all $x, y \in Q_q$.
Now, $\frac{k(k-1)}{k} = k-1$ is primitive in $\Z/k$,
so by Equation (\ref{k-l guys = l}) the number of pairs 
$(x,y) \in Q_q \times Q_q$ satisfying $v_q(x,y) = k(k-1)$
is precisely $k - (k-1) = 1$.  Thus
\begin{equation}
\max_{(x,y) \in Q_q \times Q_q} v_q(x,y) = k(k-1),
\end{equation}
this maximum is attained uniquely, and by Proposition \ref{prop:G from aqs},
\begin{equation}
\bar{G}(q) \;=\; k(k-1) \;+\; k^2 - k \;=\; 2k(k-1).
\end{equation}
\end{proof}

\subsection{Notation and Definitions for $q$ of Positive Type}
\label{ss: Notation and Definitions for q of Positive Type}

Since Proposition \ref{prop:G properties} implies that $q$ is
genus-minimizing if and only if $\xi q$ is genus-minimizing,
we lose no information by restricting ourselves to the case in which
$q = \xi q$, and we gain the advantage of being able to write
$q$ instead of $\xi q$ for the next several dozen pages.
We therefore take $q$ to be of positive type
({\em i.e.}, with $q = +\xi q$) for the following section.
We also take $q$ to be genus-minimizing.

The strategy in the preceding proof of focusing on $v_q(x,y)$ for pairs
\begin{equation}
\label{eq: pairs of [jq^-1]_k q terms}
(x,y) \in \left\{\left.\left( {\left[j(q^{-1}l)\right]_k}q, {\left[(j+1)(q^{-1}l)\right]_k}q \right)
           \right| j \in \{0, \ldots k-1\} \right\}
\end{equation}
turns out to be quite powerful.  Indeed, this idea serves as the foundation
for a combinatorial framework that helps us to classify the 
genus-minimizing solutions for $q$ of positive type.
We use a minor modification of (\ref{eq: pairs of [jq^-1]_k q terms}),
in that we replace $[q^{-1}]_k$ with the parameter $d$ assigned to $q$
in Definition \ref{def: parameters assigned to q};
that is, $d := \min\left\{\left[q^{-1}\right]_{\!k},\, \left[-q^{-1}\right]_{\!k}\right\}$.
It is for this reason that we chose to parameterize $q$ of in a way
that gives $[dq]_{k^2}$ the convenient form
\begin{equation}
[dq]_{k^2} = (\mu m + \gamma c)k + \alpha < \textstyle{\frac{k^2}{2}}
\end{equation}
when $q$ is of positive type.
See Definition \ref{def: parameters assigned to q}
for definitions of
$\mu, \gamma, \alpha \in \{\pm 1\}$ and $d,m,c \in\Z$,
and see Proposition \ref{prop: properties of parameters d, m, c, alpha, mu, gamma}
for a list of properties the parameters satisfy.

We now begin the construction
inspired by the proof of Proposition \ref{prop: G = 2k(k-1) and unique}.
We start by associating to $q$ a $k$-tuple
${\bf{z}} = \left(z_0, \ldots, z_{k-1}\right) \in \left(\Z/k^2\right)^k$,
defined by  $z_r := \left[rd\right]_k\! q \in \Z/k^2$.
Note that this makes $Q_q = \{z_0, \ldots, z_{k-1}\}$.
This k-tuple then satisfies
\begin{equation}
\label{v_qprime (z) sums to zero}
\sum_{r\in \{0, \ldots, k-1\}} v_q\!\left(z_r, z_{r+1}\right) = 0,
\end{equation}
where we define $z_k := z_0$.  Moreover, since
$dq \equiv \alpha\;(\mod k)$, we know that
\begin{equation}
\label{v_q = alpha k (mod k^2)}
v_q\!\left(z_r, z_{r+1}\right) \equiv \alpha k \; (\mod k^2)
\;\;\;
\forall \; r \in \{0, \ldots, k-1\}.
\end{equation}
Thus, following the reasoning
used in the proof of Proposition \ref{prop: G = 2k(k-1) and unique},
we deduce the following.

\begin{prop}
\label{prop: unique v_q = alpha(k - k^2), and the rest are v_q = alpha (k)}
When of positive type, $q$ is genus-minimizing 
if and only if there exists $r_* \in \{0, \ldots, k-1\}$ for which
\begin{align*}
   v_q\!\left(z_{r_*}, z_{r_*+1}\right) 
&= {\alpha} (k-k^2), \;\;\;\mathrm{but}
      \\
   v_q\!\left(z_r, z_{r+1}\right) 
&= {\alpha} (k) \;\;\;
\forall  r \in \{0, \ldots, k-1\} \setminus \left\{r_*\right\}.
\end{align*}
\end{prop}
\begin{proof}
The ``only if'' statement follows from the reasoning used in the proof of
Proposition \ref{prop: G = 2k(k-1) and unique}.  That is,
for all $r \in \{0, \ldots, k-1\}$,
Proposition \ref{prop: not minimizing if v >= k(k+1)} tells us that
$|v_q\!\left(z_r, z_{r+1}\right)| \leq k(k-1)$,
and (\ref{v_q = alpha k (mod k^2)}) tells us that
$v_q\!\left(z_r, z_{r+1}\right) \equiv \alpha k \; (\mod k^2)$.
From this, we deduce that, for all $r \in \{0, \ldots, k-1\}$, either
$v_q\!\left(z_r, z_{r+1}\right) = \alpha (k)$ or
$v_q\!\left(z_r, z_{r+1}\right) = \alpha(k - k^2)$.
Thus, equation (\ref{v_qprime (z) sums to zero}), and the fact that
it has $k$ summands, gives us two linear equations in two variables,
for which the solution is that $\alpha(k - k^2)$ occurs once, and
$\alpha (k)$ occurs $k-1$ times.

For the ``if'' statement, 
suppose that there exists such an $r_*$.
Then for any $z_{i_1}, z_{i_2} \in Q_q$ with, say, $i_1 < i_2$, we have
\begin{align}
   v_q\!\left( z_{i_1}, z_{i_2}\right)
&= \sum_{r=\{i_1, \ldots, i_2-1\}} v_q\!\left( z_r, z_{r+1}\right)
       \\
&= \begin{cases}
      k(i_2 - i_1)
      & r_* \notin \{i_1, \ldots, i_2-1\}
           \\
      k(i_2 - i_1 - k)
      & r_* \in \{i_1, \ldots, i_2-1\}
   \end{cases},
\end{align}
so that $\left\vert v_q\!\left( z_{i_1}, z_{i_2}\right)\right\vert \leq k(k-1)$.
Thus, by Proposition \ref{prop: not minimizing if v >= k(k+1)},
$q$ is genus-minimizing.

\end{proof}

Given such control over the value of $v_q\left(z_j, z_{j+1}\right)$,
it will be useful to focus our attention on the intervals
$\left\langle {\tilde{z}}_j, {\tilde{z}}_{j+1}\right]$, for appropriate lifts
${\tilde{z}}_j$ and ${\tilde{z}}_{j+1}$ of $z_j$ and $z_{j+1}$ to the integers.  
Indeed, we shall make such frequent use of intervals of
integer lifts of elements of $\Z/k^2$ that, for brevity, we establish
the following conventions for notational abuses.
\begin{conv}
For any $x, y \in \Z/k^2$, we shall write ``$\left\langle x, y \right\rangle$''
(or ``$\left\langle x, y \right]$'', ``$\left[ x, y \right\rangle$'', or ``$\left[ x, y \right]$'')
for the interval $\left\langle \tilde{x}, \tilde{y} \right\rangle$
(or $\left\langle \tilde{x}, \tilde{y} \right]$, et cetera), where $\tilde{x}, \tilde{y} \in \Z$
are respective lifts of $x$ and $y$ to the integers.  If the precise
lifts intended are not clear from context, then any lifts satisfying
$0 < \tilde{y}-\tilde{x} < k^2$ will suffice.
For $w \in \Z/k^2$, we shall write ``$w \in \left\langle x, y \right\rangle$"
(or ``$w \in \left\langle x, y \right]$," et cetera)
to signify that there exists a lift $\tilde{w} \in \Z$ of $w$ satisfying
$\tilde{w} \in \left\langle x, y \right\rangle$ (or $\tilde{w} \in \left\langle x, y \right]$, et cetera).
Similarly, for $w \in \Z/k^2$ and $s,t \in \Z$, we shall write
``$s < w < t$'' (or ``$s < w \leq t$'', et cetera) to signify that there exists a lift
$\tilde{w} \in \Z$ of $w$ satisfying $s < \tilde{w} < t$ (or $s < \tilde{w} \leq t$, et cetera).
If we only write ``$w < t$'' or ``$s < w$'', then the precise lift of $w$ intended
should be clear from context.
\end{conv}

Carrying on, we begin by examining the length of an interval
$\left\langle z_j, z_{j+1} \right]$.
For any $r \in \{0, \ldots, k-1\}$, we have
\begin{equation}
  z_{r+1} - z_r 
= \begin{cases}
    dq       &   \left[rd\right]_k <   k-d
                    \\
    (d-k)q   &   \left[rd\right]_k \ge k-d
  \end{cases},
\end {equation}
where 
$q= \alpha\gamma\mu\frac{ck + \alpha\gamma}{d}(mk + \alpha\mu)$
implies that $kq = \gamma\frac{ck + \alpha\gamma}{d}k$.

It is instructive to break $\bf{z}$ up into $d$ consecutive 
sub-tuples ${\bf{z}} = \left({\bf{z}}^0, \ldots, {\bf{z}}^{d-1}\right)$
in a manner that reflects the above structure of the differences
of consecutive entries of $\bf{z}$, {\em i.e.}, such that
\begin{align}
   z^j_{i+1} - z^j_i 
&= dq \;\;\; 
   \forall\; j \in \{0, \ldots, d-1\},\;
             i \in \{0, \ldots, \len({\bf{z}}^j) - 2 \};
        \\
   z^{j+1}_0 - z^j_{\len({\bf{z}}^j) - 1}
&= (d-k)q\;\;\;
   \forall \; j \in \{0, \ldots, d-2\}.
\end{align}
More explicitly, setting $\epsilon := \left[-k\right]_d$ (and noting that this
implies $c = \left[\alpha\gamma{\epsilon}^{-1}\right]_{d}$),
for each $j \in \{0, \ldots, d-1\}$, we define
$z^j_i := \left(\left[j\epsilon\right]_d + id\right)\!q \in \Z/k^2$,
where we restrict
the domain of $i$ so that 
$\left[j\epsilon\right]_d + id \in \{0, \ldots, k-1\}$.  Thus
\begin{align}
\label{eq: definition of z_i^j}
   {\bf{z}}^0
&= \left(\;  \left(\left[0\epsilon\right]_d + 0d\right)\!q,\;
             \left(\left[0\epsilon\right]_d + 1d\right)\!q,\; \ldots,\;
             \left(\left[0\epsilon\right]_d
              + \left({\left\lceil\!\frac{k - \left[0\epsilon\right]_d\!}{d}
                       \right\rceil} \!-\! 1\!\right)\!d
             \right)\! q
   \;\right),
         \\ \nonumber
   {\bf{z}}^1
&= \left(\;  \left(\left[1\epsilon\right]_d + 0d\right)\!q,\;
             \left(\left[1\epsilon\right]_d + 1d\right)\!q,\; \ldots,\;
             \left(\left[1\epsilon\right]_d
              + \left({\left\lceil\!\frac{k - \left[1\epsilon\right]_d\!}{d}
                       \right\rceil} \!-\! 1\!\right)\!d
             \right)\! q
   \;\right),
         \\ \nonumber
&\;\;\vdots
         \\ \nonumber
   {\bf{z}}^{d-1}
&= \left(\; \left(\left[(d\!-\!1)\epsilon\right]_d + 0d\right)\!q,\;
            \left(\left[(d\!-\!1)\epsilon\right]_d + 1d\right)\!q,\; \ldots,
   \right.
                         \\ \nonumber
   &\left.\;\;\;\;\;\;\;\;\;\;\;\;\;\;\;\;\;\;\;\;\;\;
          \;\;\;\;\;\;\;\;\;\;\;\;\;\;\;\;\;\;\;\;\;\;
             \left(\left[(d\!-\!1)\epsilon\right]_d
              + \left({\left\lceil\!\frac{k - \left[(d\!-\!1)\epsilon\right]_d\!}{d}
                       \right\rceil} \!-\! 1\!\right)\!d
             \right)\! q
   \;\right).
\end{align}
Such ${\bf{z}}^j$ then has length
\begin{align}
   \len({\bf{z}}^j)
&= \left(\left(\left\lceil \frac{k - \left[j\epsilon\right]_d}{d}\right\rceil - 1\right)
   - 0\right) + 1
          \\ \nonumber
&= \left\lceil \frac{k}{d} \right\rceil
   - \begin{cases}
        1   &\left[k\right]_d - \left[j\epsilon\right]_d  \leq   0
                \\
        0   &\left[k\right]_d - \left[j\epsilon\right]_d  > 0
     \end{cases}
          \\ \nonumber
&= \frac{k+\epsilon}{d}
   - \begin{cases}
        1   &\left[j\epsilon\right]_d  \geq   d - \epsilon
                \\
        0   &\left[j\epsilon\right]_d  < d - \epsilon
     \end{cases}
          \\ \nonumber
&= \frac{k+\epsilon}{d} - {\theta}^{d, \epsilon}(j),
\end{align}
where, for any relatively prime positive integers $d$ and $\epsilon$
with $d \geq 2$ and $\epsilon < d$, and for any $j \in \Z$,
${\theta}^{\cdot, \cdot}(\cdot)$ is a function
defined by
\begin{equation}
\label{eq: def of theta}
   {\theta}^{d, \epsilon}(j)
:= \begin{cases}
        1   &\left[j\epsilon\right]_d  \geq   d - \epsilon
                \\
        0   &\left[j\epsilon\right]_d  <      d - \epsilon
   \end{cases},
\end{equation}
or, equivalently, by
\begin{equation}
\label{eq: def 2 of theta}
    \theta^{d, \epsilon}(j)
=  \left\lfloor \frac{(j+1)\epsilon}{d}\right\rfloor
   -\left\lfloor \frac{j\epsilon}{d}\right\rfloor.
\end{equation}

Since the definition of ${\bf{z}}^j$ depends only on the value of
$j$ modulo $d$, it is natural to extend the definition
of ${\bf{z}}^j$ to be valid for all $j \in \Z$,
setting ${\bf{z}}^j := {\bf{z}}^{\left[j\right]_d}$.
We therefore henceforward regard the collection $\left\{{\mathbf{z}}^j\right\}_{j\in\Z}$ 
as being indexed by elements of $\Z/d$.
For any $j \in\Z/d$, we then compute
$z^{j+1}_0 - z^j_0 \in \Z/k^2$, as follows:
\begin{align}
   \label{eq:  z^{j+1}_0 - z^j_0}
   z^{j+1}_0 - z^j_0
&= \left(z^{j+1}_0 - z^j_{\len({\bf{z}}^j)-1}\right)
   + \left(z^j_{\len({\bf{z}}^j)-1} - z^j_0\right)
            \\ \nonumber
&= (d-k)q + \left(\len({\bf{z}}^j)-1\right)dq
            \\ \nonumber
&= -kq + \len({\bf{z}}^j)dq
            \\ \nonumber
&= -\gamma\frac{ck + \alpha\gamma}{d}k
   +\left(\frac{k+\epsilon}{d} - {\theta}^{d, \epsilon}(j)\right)\left(({\mu}m+{\gamma}c)k+\alpha\right)
            \\ \nonumber
&= {\mu}m\frac{k^2}{d}
   + \left(\frac{\epsilon}{d} - {\theta}^{d, \epsilon}(j)\right)\left(({\mu}m+{\gamma}c)k+\alpha\right)
            \\ \nonumber
&= {\mu}m\frac{k^2}{d}
   + \left(\frac{\epsilon}{d} - {\theta}^{d, \epsilon}(j)\right)[dq]_{k^2}
\end{align}
Of course, the summand ``$\mu m\frac{k^2}{d}$''
does not make much sense as an element of $\Z/k^2$.
To make sense of the last two lines of (\ref{eq:  z^{j+1}_0 - z^j_0}),
and of similar expressions, one must first interpret the summands as elements of $\frac{1}{d}\Z$,
next interpret their sum as an element of $\Z$, and then take the image
of this element of $\Z$ in $\Z/k^2$.
Equation (\ref{eq:  z^{j+1}_0 - z^j_0}) further abuses notation in its usage
of ${\theta}^{d, \epsilon}$.  Since ${\theta}^{d, \epsilon}$ is periodic modulo $d$,
it descends to a function on $\Z/d$, which we denote in the same way as the
original function on $\Z$.

Given, in addition, any $l \in \Z/d$, we can use
(\ref{eq:  z^{j+1}_0 - z^j_0})
to compute $z^{j+l}_0 - z^j_0 \in \Z/k^2$.
Let $\tilde{j}, \tilde{l} \in \Z$ denote arbitrary lifts of $j$ and $l$ to the integers.  Then
\begin{align}
\label{eq: z_0^j+l - z_0^j = mu ml k^2/d + xi dq}
   z^{j+l}_0 - z^j_0
&= \sum_{i=\tilde{j}}^{\tilde{j} + \tilde{l}-1} \left(z^{i+1}_0 - z^i_0\right)
         \\ \nonumber
&=\sum_{i=\tilde{j}}^{\tilde{j} + \tilde{l}-1}
     \left(  {\mu}m\frac{k^2}{d}
             + \left(\frac{\epsilon}{d} - {\theta}^{d, \epsilon}(i)\right) [dq]_{k^2}
     \right)
         \\ \nonumber
&= {{\mu}m\tilde{l}}\frac{k^2}{d}
   +\left(\frac{\tilde{l}\epsilon}{d}
   - \sum_{i=\tilde{j}}^{\tilde{j} + \tilde{l}-1} {\theta}^{d, \epsilon}(i)\right) [dq]_{k^2}
         \\ \nonumber
&= \left[{\mu}ml\right]_d\frac{k^2}{d}
   \;+\; \Xi^{d, \epsilon}_l(j)\, [dq]_{k^2},
\end{align}
where the function $\Xi^{\cdot, \cdot}_{\cdot}(\cdot)$ is defined such that,
for any relatively prime positive integers $d$ and $\epsilon$ with
$d \ge 2$ and $\epsilon < d$, and for any
$j, l \in \Z/d$, we have
\begin{equation}
         \Xi^{d, \epsilon}_l(j) 
\;:=\; \frac{\tilde{l}\epsilon}{d}
         -\sum_{i=\tilde{j}}^{\tilde{j} + \tilde{l}-1} {\theta}^{d, \epsilon}(i)\;\;
         \in \textstyle{\frac{1}{d}}\Z,
\end{equation}
where $\tilde{j}, \tilde{l} \in \Z$ denote arbitrary lifts of $j$ and $l$ to $\Z$.
The value of $\Xi^{d, \epsilon}_l(j)$ does not depend on the choice of lift.
Note that since $\sum_{j\in\Z/d} {\theta}^{d, \epsilon}(j) = \epsilon$,
$\Xi^{d, \epsilon}_l$ has mean value zero.
The following lemma shows that $\Xi^{d, \epsilon}_l(j)$
(and hence $z^{j+l}_0 - z^j_0$)
stays as close to its mean value as possible.

\begin{lemma}
\label{lemma: q of positive type, xi lemma}
Suppose that $d$ and $\epsilon$ are relatively prime positive integers with
$d \ge 2$ and $\epsilon < d$,
and that $l \in \Z/d$.  Then, for any $j \in \Z/d$, we have
\begin{equation}
\nonumber
     \Xi^{d, \epsilon}_l(j) 
\in  \left\{ \frac{[l\epsilon]_d}{d}, \frac{-[-l\epsilon]_d}{d}\right\}.
\end{equation}
When $l \neq 0$, the sets
$\left\{ j \in \Z/d \left\vert\; \Xi^{d, \epsilon}_l(j) = \frac{[l\epsilon]_d}{d}\right.\right\}$
and 
$\left\{ j \in \Z/d \left\vert\; \Xi^{d, \epsilon}_l(j) = \frac{-[-l\epsilon]_d}{d}\right.\right\}$
have $\left[-l\epsilon\right]_d$ elements and 
$\left[l\epsilon\right]_d$ elements, respectively.
\end{lemma} 

\begin{proof}
Let $\tilde{l}, \tilde{j} \in \Z$ denote any lifts of $j, l \in \Z/d$
to the integers.  Then, using (\ref{eq: def 2 of theta}), we have
\begin{align}
         \Xi^{d, \epsilon}_l(j) 
&=   \frac{\tilde{l}\epsilon}{d}
         -\sum_{i=\tilde{j}}^{\tilde{j} + \tilde{l}-1}
         \left(\left\lfloor \frac{(i+1)\epsilon}{d}\right\rfloor
               -\left\lfloor \frac{i\epsilon}{d}\right\rfloor\right)
                          \\ \nonumber
&=   \frac{\tilde{l}\epsilon}{d}
         -\left(\left\lfloor \frac{(\tilde{j} + \tilde{l})\epsilon}{d}\right\rfloor
                 -\left\lfloor \frac{\tilde{j}\epsilon}{d}\right\rfloor\right)
                          \\ \nonumber
&\in \left\{  \frac{\tilde{l}\epsilon}{d} - 
                    \left\lfloor \frac{\tilde{l}\epsilon}{d}\right\rfloor,\;
                    \frac{\tilde{l}\epsilon}{d} - 
                    \left\lceil\frac{\tilde{l}\epsilon}{d}\right\rceil \right\}
                          \\ \nonumber
&=  \left\{ \frac{[l\epsilon]_d}{d}, \frac{-[-l\epsilon]_d}{d}\right\}.
\end{align}

Next, suppose that $l \neq 0$, and let
$n_+$ (respectively, $n_-$) denote the number of
values of $j \in \Z/d$ for which
$\Xi^{d, \epsilon}_l(j) = \frac{[l\epsilon]_d}{d}$
(respectively, $\Xi^{d, \epsilon}_l(j) = \frac{-[-l\epsilon]_d}{d}$).
Then $n_+$ and $n_-$ satisfy two independent linear equations,
\begin{align}
          d 
&\;=\; n_+ \;+\; n_-, \; \text{and}
                 \\
          0
&\;=\; \sum_{j \in \Z/d} \Xi^{d, \epsilon}_l(j)
   \;=\;       n_+ \!\left\lfloor\frac{l\epsilon}{d}\right\rfloor
          \;+\; n_- \!\left\lceil\frac{l\epsilon}{d}\right\rceil,
\end{align}
with unique solution
$n_+ = \left[-l\epsilon\right]_d$ and
$n_-   = \left[l\epsilon\right]_d$.

\end{proof}

Before proceeding further, we need some new notation.
First, we introduce the operations $\minq$ and $\maxq$,
which are only defined on two-element subsets of $\Z/k^2$
differing by $dq$.  For any $x \in \Z/k^2$, we say that
\begin{align}
      \minq\, \left\{x,\; x + dq\right\} 
&= x,
           \\
      \maxq\, \left\{x,\; x + dq\right\}
&= x + dq.
\end{align}
Next, for brevity, set
\begin{equation}
n_j := \len({\mathbf{z}}^j) - 1 = \left\lfloor\frac{k}{d}\right\rfloor - {\theta}^{d, \epsilon}(j)
\end{equation}
for each $j\in \Z/d$, so that ${\mathbf{z}}^j = (z_0^j, \ldots, z_{n_j}^j)$.
That is, for each $j\in\Z/d$, $n_j$ counts the number of intervals 
$\left\langle z_i^j, z_{i+1}^j\right]$ for which $z_i^j$ and $z_{i+1}^j$ are defined.
The combination of 
(\ref{eq: z_0^j+l - z_0^j = mu ml k^2/d + xi dq})
and Lemma \ref{lemma: q of positive type, xi lemma}
then provides the following result.

\begin{cor}
\label{cor: q of positive type: combo of lemma and difference eq}
If $q$ is of positive type, then for any $l \in \Z/d$ with $l \neq 0$,
the sets $\left\{z_0^{j+l} - z_i^j\right\}_{j\in\Z/d}$ and 
$\left\{z_{n_{j-l}}^{j-l} - z^j_{n_j - i}\right\}_{j \in \Z/d}$
each contain precisely two elements, which differ by $dq$.
More specifically,
\begin{align*}
      \minq_{j\in\Z/d} \left(z_0^{j+l} - z_i^j\right)
&= \left[{\mu}ml\right]_d\frac{k^2}{d} + 
      \left(\frac{\left[l\epsilon\right]_d}{d} - i - 1\right)[dq]_{k^2},
              \\
      \maxq_{j\in\Z/d} \left(z_{n_{j-l}}^{j-l} - z^j_{n_j - i}\right)
&= -\left[{\mu}ml\right]_d\frac{k^2}{d} -
      \left(\frac{\left[l\epsilon\right]_d}{d} - i - 1\right)[dq]_{k^2}.
\end{align*}
For any $l \in \Z/d$ with $l \neq 1$, the sets
$\left\{z_0^{j+l} - z_{n_j - (i+1)}^j\right\}_{j \in \Z/d}$ and
$\left\{z_{n_{j-l}}^{j-l} - z^j_{i+1}\right\}_{j \in \Z/d}$
each contain precisely two elements, which differ by $dq$.
More specifically,
\begin{align*}
      \minq_{j\in\Z/d} \left(z_0^{j+l} - z_{n_j - (i+1)}^j\right)
&= \left[{\mu}m(l-1)\right]_d\frac{k^2}{d} + 
      \left(\frac{\left[(l-1)\epsilon\right]_d}{d} + i\right)[dq]_{k^2}  + \psi,
              \\
      \maxq_{j\in\Z/d} \left(z_{n_{j-l}}^{j-l} - z^j_{i+1}\right)
&= -\left[{\mu}m(l-1)\right]_d\frac{k^2}{d} - 
      \left(\frac{\left[(l-1)\epsilon\right]_d}{d} + i\right)[dq]_{k^2}  - \psi,
\end{align*}
where $\psi := \left[z^{j+1}_0 - z_{n_j}^j\right]_{k^2} 
= \left[dq - kq\right]_{k^2} =
 \left[(({\mu}m + {\gamma}c)k + \alpha) - \gamma\frac{ck+\alpha\gamma}{d}k\right]_{k^2}$.
\end{cor}
This result also uses the fact that
$\frac{-[-l\epsilon]_d}{d} = \frac{[l\epsilon]_d}{d}-1$ when $l \neq 0$.
Again, note that to make sense of the four above right-hand expressions
as elements of $\Z/k^2$, one
must first interpret each expression as an integer, and then
take the image of that integer under the quotient
$\Z \rightarrow \Z/k^2$.

The above corollary plays an important role in studying any set of the form
${\tilde{Q}}_q \cap \left\langle z_i^j, z_{i+1}^j\right]$ or
${\tilde{Q}}_q \cap \left\langle z_{n_j - (i+1)}^j, z_{n_j - i}^j\right]$,
for fixed $i \in \{0, \ldots, \left\lfloor\frac{k}{d}\right\rfloor -2\}$,
as a function of $j \in \Z/d$.  (We shall also use the
corollary to study sets of the form ${\tilde{Q}}_q \cap \left\langle z_{n_{j-1}}^{j-1}, z_0^j\right]$
or ${\tilde{Q}}_q \cap \left\langle z_0^j, z_{n_{j-1}}^{j-1} \right]$
as functions of $j \in \Z/d$, but more on that later.)
Suppose that for some $j^{\prime} \in \Z/d$ and
$i \in \{0, \ldots, n_{j^{\prime}} \}$ there is an element of 
${\tilde{Q}}_q$ contained in the interval
$\left\langle z_i^{j^{\prime}}, z_{i+1}^{j^{\prime}}\right]$.
Then we can write $z_{i^{\prime}}^{j^{\prime}+l}$ for the image of this element in ${\tilde{Q}}_q$,
for some $l \in \Z/d$ and $i^{\prime} \in \{0, \ldots, n_{j^{\prime}+l} \}$.
Let $x := z_{i^{\prime}}^{j^{\prime}+l} - z_i^{j^{\prime}}$, so that
$z_{i^{\prime}}^{j^{\prime}+l} = z_i^{j^{\prime}} + x$.
Then if $i^{\prime} \notin \{0, n_{j^{\prime}+l} \}$,
Corollary \ref{cor: q of positive type: combo of lemma and difference eq}
implies that
$z_i^j + x \in \{z_{i-1}^j, z_i^j, z_{i+1}^j\} \subset {\tilde{Q}}_q$ for all $j \in \Z/d$.
That is, an element of ${\tilde{Q}}_q$ is contained
at the same relative position in $\left\langle z_i^j, z_{i+1}^j\right]$
for every $j \in \Z/d$.
Thus, if {\em every} $z_{i^{\prime}}^{j+l}$ contained in
$\left\langle z_i^j, z_{i+1}^j \right]$
satisfied $i^{\prime} \notin \{0, n_{j+l}\}$, for every $j \in \Z/d$, then
$\#\!\left({\tilde{Q}}_q \cap \left\langle z_i^j, z_{i+1}^j \right]\right)$
would be constant in $j \in \Z/d$.

This means that the only way for $\#\!\left({\tilde{Q}}_q \cap \left\langle z_i^j, z_{i+1}^j \right]\right)$
to change as $j$ varies in $\Z/d$ is for there to exist some
$l \in \Z/d$ for which either $z_0^{j^{\prime}+l} \in
\left\langle z_i^{j^{\prime}}, z_{i+1}^{j^{\prime}}\right]$
or $z_{n_{j^{\prime}-l}}^{j^{\prime}-l} \in
\left\langle z_i^{j^{\prime}}, z_{i+1}^{j^{\prime}}\right]$, for some $j^{\prime} \in \Z/d$.
This is the idea behind {\em mobile points}, which we define as follows.
For any $i \in \{0, \ldots, \left\lfloor\frac{k}{d}\right\rfloor -2\}$
and $l \in \Z/d$ with $l \neq 0$,
we say that $z_0^{j+l}$ is an R-mobile point in the interval
$\left\langle z_i^j, z_{i+1}^j \right]$ if
\begin{equation}
0 < \minq_{j\in\Z/d} \left(z_0^{j+l} - z_i^j\right) < \left[dq\right]_{k^2},
\end{equation}
and that
$z_{n_{j-l}}^{j-l}$ is an L-mobile point in the interval
$\left\langle z_{n_j - (i+1)}^j, z_{n_j - i}^j \right]$ if
\begin{equation}
-\left[dq\right]_{k^2} < \maxq_{j\in\Z/d} \left(z_{n_{j-l}}^{j-l} - z_{n_j - i}^j\right) < 0.
\end{equation}

Note that since we take $i$ to be constant in $j$, we must demand
$i \in \{0, \ldots, \left\lfloor\frac{k}{d}\right\rfloor -2\}$,
as opposed to, for example, $i \in \{0, \ldots, n_{j^{\prime}}-1\}$.
The ``R'' and ``L'' correspond to the fact that when an R-mobile point
``escapes,'' it does so by a distance of $dq$ to the right, whereas, when an
L-mobile point ``escapes,'' it does so by a distance of $dq$ to the left.
We say that the R-mobile point described in the above paragraph is
R-mobile ``rel $z_0^j$'', since positions are measured relative to $z_0^j$.
Likewise, we say that the above-described L-mobile point is L-mobile
``rel $z_{n_j}^j$'', since positions are measured relative to $z_{n_j}^j$.
Recall that Corollary  \ref{cor: q of positive type: combo of lemma and difference eq} shows that
\begin{equation}
\label{eq: mirror relation 1}
    \minq_{j\in\Z/d} \left(z_0^{j+l} - z_i^j\right)
=  -\maxq_{j\in\Z/d} \left(z_{n_{j-l}}^{j-l} - z^j_{n_j - i}\right)
\end{equation}
for all $l\in \Z/d$ with $l\neq 0$.  This equation establishes an isomorphism
which we call a {\em mirror relation}
between R-mobile points rel $z_0^j$ and L-mobile points rel $z_{n_j}^j$.
In particular, 
$z_0^{j+l}$ is an R-mobile point in $\left\langle z_i^j, z_{i+1}^j \right]$
if and only if $z_{n_{j-l}}^{j-l}$ is L-mobile point in
$\left\langle z_{n_j - (i+1)}^j, z_{n_j - i}^j \right]$.
We call such pairs of mobile points {\em mirror} mobile points.

One can also define R-mobile points rel $z_{n_j}^j$ and L-mobile points 
rel $z_0^j$, as follows.  For any $l \in \Z$/d with $l \neq 1$, we say that $z_0^{j+l}$ 
is R-mobile in
$\left\langle z_{n_j - (i+1)}^j, z_{n_j - i}^j \right]$ 
(hence is R-mobile rel $z_{n_j}^j$) if
\begin{equation}
0 < \minq_{j\in\Z/d} \left(z_0^{j+l} - z_{n_j - (i+1)}^j\right) < \left[dq\right]_{k^2}.
\end{equation}
Likewise, $z_{n_{j-l}}^{j-l}$ is L-mobile in
$\left\langle z_i^j, z_{i+1}^j \right]$ (hence is L-mobile rel $z_0^j$) if
\begin{equation}
-\left[dq\right]_{k^2} < \maxq_{j\in\Z/d} \left(z_{n_{j-l}}^{j-l} - z_{i+1}^j\right) < 0.
\end{equation}
Again, Corollary \ref{cor: q of positive type: combo of lemma and difference eq} shows that
\begin{equation}
\label{eq: mirror relation 2}
      \minq_{j\in\Z/d} \left(z_0^{j+l} - z_{n_j - (i+1)}^j\right)
= - \maxq_{j\in\Z/d} \left(z_{n_{j-l}}^{j-l} - z^j_{i+1}\right)
\end{equation}
for all $l \in \Z/d$ with $l \neq 1$.  Thus 
$z_0^{j+l}$ is R-mobile in $\left\langle z_{n_j - (i+1)}^j, z_{n_j - i}^j \right]$ 
if and only if $z_{n_{j-l}}^{j-l}$ is L-mobile in
$\left\langle z_i^j, z_{i+1}^j \right]$.  We call these pairs mirror as well.

We shall use the term ``mobile point'' to describe a point which is either
R-mobile or L-mobile.  We say that a point is mobile (respectively
R-mobile or L-mobile) in ${\bf{z}}^j$ if it is mobile (respectively
R-mobile or L-mobile) rel $z_0^j$ or rel $z_{n_j}^j$.  Note that
it is possible for a single mobile point to have both a ``rel $z_0^j$''
and a ``rel $z_{n_j}^j$'' description.  Indeed, this is always the case
when $i \notin \left\{0, \left\lfloor\frac{k}{d}\right\rfloor - 2\right\}$ and $l \notin \{0,1\}$.

Let us pause here to explain what we mean by
measuring positions relative to $z_0^j$ or to $z_{n_j}^j$.
Suppose that for some nonzero $l \in\Z/d$, we know that
$z_0^{j+l}$ is R-mobile rel $z_0^j$.  Then there exists some
$i \in \{0, \ldots, \left\lfloor\frac{k}{d}\right\rfloor -2\}$ such that
$z_0^{j+l}$ is R-mobile in $\left\langle z_i^j, z_{i+1}^j\right]$, and so
\begin{equation}
  i\left[dq\right]_{k^2}
< \minq_{j\in\Z/d} \left(z_0^{j+l} - z_0^j\right) 
< (i+1)\left[dq\right]_{k^2}.
\end{equation}
Note that if $\left\lfloor\frac{k}{d}\right\rfloor \left((m\!+\!c)k\!+\!1\right) > k^2$,
then more than one value of $i$ could satisfy this property.
Likewise, if for some nonzero $l\in\Z/d$ we know that
$z_{n_{j-l}}^{j-l}$ is L-mobile rel $z_{n_j}^j$, then there exists some
$i \in \{0, \ldots, \left\lfloor\frac{k}{d}\right\rfloor -2\}$ such that
$z_{n_{j-l}}^{j-l}$ is L-mobile in
$\left\langle z_{n_j-(i+1)}^j, z_{n_j-i}^j \right]$, and so
\begin{equation}
  -(i+1)\left[dq\right]_{k^2}
< \maxq_{j\in\Z/d} \left(z_{n_{j-l}}^{j-l} - z_{n_j}^j\right)
< -i\left[dq\right]_{k^2}.
\end{equation}
Again, it is possible for more than one value of $i$ to satisfy this property.

One could justifiably argue that a ``mobile point''
should really be called a ``collection of points,''
but it is better to think of a mobile point $z_0^{j+l}$
as an element of $\Z/k^2$ that is a function of $j$,
just as the position $x(t)$ of a particle on a line
is viewed as an element of $\R$ that is a function of $t$.
Indeed, it is perhaps useful to think of $j$ as a discrete
time variable.  For example, suppose that $z_0^{j+l}$ is
R-mobile rel $z_0^j$, hence R-mobile in some interval
$\left\langle z_i^j, z_{i+1}^j \right]$.  This means
that as $j$ varies, the point $z_0^{j+l} - z_0^j$
hops back and forth between $x$ and $x + dq$,
for some fixed $x \in 
\left\langle i\,dq, (i+1)dq \right\rangle$.
If $z_0^{j+l}$ is R-mobile rel $z_{n_j}^j$,
hence R-mobile in some interval
$\left\langle z_{n_j-(i+1)}^j, z_{n_j-i}^j \right]$,
then as $j$ varies, the point $z_0^{j+l} - z_{n_j}^j$
hops back and forth between $x$ and $x+dq$
for some fixed $x \in
\left\langle -(i+1)dq, -i\,dq \right\rangle$.
This is the type of motion indicated in the term ``mobile,''
and the type of ``worldline'' viewpoint by which we call the collection
$\{z_0^{j+l}\}_{j\in{\Z/d}}$ a ``point.''

We shall use the terms {\em active} and {\em inactive} to describe
whether such an R-mobile point is occupying position $x$ or position
$x + dq$ at a particular time $j=j^{\prime}\in\Z/d$.
That is, we say that an R-mobile point 
 $z_0^{j+l}$ in $\left\langle z_i^j, z_{i+1}^j \right]$
 is active at time $j=j^{\prime} \in \Z/d$ if
\begin{equation}
      z_0^{j^{\prime}+l} - z_i^{j^{\prime}}
\;=\; \minq_{j\in\Z/d} \left(z_0^{j+l} - z_i^j\right)
\end{equation}
and inactive otherwise.  Likewise, we say that
an L-mobile point $z_{n_{j-l}}^{j-l}$ in $\left\langle z_{n_j - (i+1)}^j, z_{n_j - i}^j \right]$
is active at time $j = j^{\prime} \in \Z/d$ if
\begin{equation}
z_{n_{j^{\prime}-l}}^{j^{\prime}-l} - z_{n_{j^{\prime}} - i}^{j^{\prime}}
= \maxq_{j\in\Z/d} \left(z_{n_{j-l}}^{j-l} - z_{n_j - i}^j\right)
\end{equation}
and is inactive otherwise.
The obvious analogous definitions for active and inactive
hold for an R-mobile point rel $z_{n_j}^j$
and an L-mobile point rel $z_0^j$.

The fact that $[dq]_{k^2} < \frac{k^2}{2}$ for $q$ of positive type
implies that an R-mobile point $z_0^{j+l}$ in $\left\langle z_i^j, z_{i+1}^j \right]$
satisfies
\begin{equation}
\label{eq: active means in for R-mobile}
z_0^{j+l}\;\text{is active at time}\;j=j^{\prime}
\;\;\;\;\;\Leftrightarrow\;\;\;\;\;
z_0^{j^{\prime}+l} \in \left\langle z_i^{j^{\prime}}, z_{i+1}^{j^{\prime}} \right]
\end{equation}
for all $j^{\prime}\in\Z/d$, and that an L-mobile point 
$z_{n_{j-l}}^{j-l}$ in $\left\langle z_{n_j - (i+1)}^j, z_{n_j - i}^j \right]$ satisfies
\begin{equation}
\label{eq: active means in for L-mobile}
z_{n_{j-l}}^{j-l}\;\text{is active at time}\;j=j^{\prime}
\;\;\;\;\;\Leftrightarrow\;\;\;\;\;
z_{n_{j^{\prime}-l}}^{j^{\prime}-l}
\in \left\langle z_{n_{j^{\prime}}-(i+1)}^{j^{\prime}}, z_{n_{j^{\prime}}-i}^{j^{\prime}} \right]
\end{equation}
for all $j^{\prime}\in\Z/d$ (and similarly for an R-mobile point rel $z_{n_j}^j$
or an L-mobile point rel $z_0^j$).  This is because, for example, if $z_0^{j+l}$ is
R-mobile in $\left\langle z_i^j, z_{i+1}^j\right]$, then
$\minq_{j\in\Z/d}\left(z_0^{j+l}-z_0^j\right) \in \left\langle 0, [dq]_{k^2} \right\rangle$,
implying $\maxq_{j\in\Z/d}\left(z_0^{j+l}-z_0^j\right) \in 
\left\langle [dq]_{k^2}, 2[dq]_{k^2} \right\rangle$.
Ordinarily, there would be nothing to prevent the intersection
$\left\langle [dq]_{k^2}, 2[dq]_{k^2} \right\rangle \cap
\left\langle 0, [dq]_{k^2} \right\rangle$ from being nonempty,
but since $[dq]_{k^2} < \frac{k^2}{2}$, we know that
$\maxq_{j\in\Z/d}\left(z_0^{j+l}-z_0^j\right) \notin \left\langle 0, [dq]_{k^2} \right\rangle$.
Thus $z_0^{j^{\prime}+l} \in \left\langle z_i^{j^{\prime}}, z_{i+1}^{j^{\prime}} \right]$
when $z_0^{j+l}$ is inactive at time $j=j^{\prime}$.

The notion of active versus inactive is important because it determines how
a mobile point contributes to the intersection of ${\tilde{Q}}_q$ with the
interval in question at a particular time.
\begin{prop}
For any $i \in \left\{0, \ldots, \left\lfloor\frac{k}{d}\right\rfloor-2\right\}$,
there is some $C_i \in \Z$, constant in $j^{\prime} \in \Z/d$, such that
\begin{equation}
\label{eq: contribution of active mobile points to Q_q cap < z_i^j, z_{i+1}^j ] }
      \#\left({\tilde{Q}}_q \cap \left\langle z_i^{j^{\prime}}, 
               z_{i+1}^{{j^{\prime}}} \right] \right)
\;=\; C_i \;+\; 
      \#\!\left\{\begin{array}{cc}
            \mathrm{mobile\; points\; in}\; \left\langle z_i^j, z_{i+1}^j \right]
                 \\
            \mathrm{which\; are\; active\; at\; time}\;j^{\prime}
          \end{array}
      \right\}
\end{equation}
for each $j^{\prime} \in \Z/d$.
Likewise, 
for any $i \in \left\{0, \ldots, \left\lfloor\frac{k}{d}\right\rfloor-2\right\}$,
there is some ${\bar{C}}_i \in \Z$, constant in $j^{\prime} \in \Z/d$, such that
\begin{equation}
\label{eq: contribution of active mobile points to Q_q cap < z_{n_j-(i+1)}^j, etc}
      \#\left({\tilde{Q}}_q \cap \left\langle z_{n_{j^{\prime}}-(i+1)}^{j^{\prime}}, 
               z_{n_{j^{\prime}}-i}^{j^{\prime}} \right] \right)
\;=\; {\bar{C}}_i \;+\; 
      \#\!\left\{\begin{array}{cc}
            \mathrm{mobile\; points\; in}\; \left\langle z_{n_j-(i+1)}^j, z_{n_j - i}^j \right]
                 \\
            \mathrm{which\; are\; active\; at\; time}\;j^{\prime}
          \end{array}
      \right\}
\end{equation}
for each $j^{\prime} \in \Z/d$.
\end{prop}
\begin{proof}
We restrict ourselves to the proof of
(\ref{eq: contribution of active mobile points to Q_q cap < z_i^j, z_{i+1}^j ] }),
since
(\ref{eq: contribution of active mobile points to Q_q cap < z_{n_j-(i+1)}^j, etc})
then follows from a mirror argument.
Equation (\ref{eq: contribution of active mobile points to Q_q cap < z_i^j, z_{i+1}^j ] })
counts the number of elements of $Q_q$
contained in the interval $\left\langle  z_i^{j^{\prime}}, z_{i+1}^{j^{\prime}} \right]$,
for any given $j^{\prime} \in \Z/d$.  Thinking of this interval as
$z_i^{j^{\prime}} + \left\langle 0, dq \right]$,
we can describe all the elements of $\left\langle 0, dq \right] \cap \Z$
as belonging to one of three categories.  For any
$x \in \left\langle 0, dq \right] \cap \Z$, one of the following is true:
\begin{itemize}
\item[(i)]
   {$z_i^{j^{\prime}} + x$ is never an element of $Q_q$,
  regardless of the value of $j^{\prime} \in \Z/d$.}

\item[(ii)]
   {$z_i^{j^{\prime}} + x$ is always an element of $Q_q$,
  for any value of $j^{\prime} \in \Z/d$.}

\item[(iii)]
  {$z_i^{j^{\prime}} + x$ is sometimes an element of $Q_q$,
   and sometimes not, depending on $j^{\prime} \in \Z/d$.}
\end{itemize}
Equation (\ref{eq: contribution of active mobile points to Q_q cap < z_i^j, z_{i+1}^j ] })
then records the contribution of $x$ to
$\#\left({\tilde{Q}}_q \cap \left\langle z_i^{j^{\prime}}, z_{i+1}^{{j^{\prime}}} \right] \right)$
as follows.
There are no contributions from $x$ of type (i);
we define $C_i$ to count the 
contributions from $x$ of type (ii); and
as explained below, the second summand
of (\ref{eq: contribution of active mobile points to Q_q cap < z_i^j, z_{i+1}^j ] })
counts the contributions from $x$ of type (iii).

The point of Corollary 3.10 is that, generically speaking, most $x$ which are not of
type (i) are of type (ii).  That is, suppose, for a given
$x \in \left\langle 0, dq \right] \cap \Z$, that there exists some $j_0 \in \Z/d$ for which
         $z_i^{j_0} + x  \in  Q_q$.
Then there exist $l \in \Z/d$ and $i_0 \in \left\{0, ..., n_{j_0+l}\right\}$ for which
         $z_i^{j_0} + x = z_{i_0}^{j_0+l}$.
The following argument, which holds in all but two exceptional cases
described below, shows that $x$ must be of type (ii).
By Corollary 3.10, the fact that $z_i^{j_0} + x = z_{i_0}^{j_0+l}$ implies that either
   $x = \minq_{j \in \Z/d}\left( z_{i_0}^{j+l} - z_i^j\right)$,
in which case
       $z_i^{j^{\prime}} + x  \in \left\{ z_{i_0}^{j^{\prime}+l}, z_{i_0-1}^{j^{\prime}+1} \right\}$
for all $j^{\prime} \in \Z/d$, or
   $x = \maxq_{j \in \Z/d}\left( z_{i_0}^{j+l} - z_i^j\right)$,
in which case
       $z_i^{j^{\prime}} + x \in \left\{ z_{i_0}^{j^{\prime}+l}, z_{i_0+1}^{j^{\prime}+1} \right\}$
for all $j^{\prime} \in \Z/d$.
Thus  $z_i^{j^{\prime}} + x \in Q_q$ for all $j^{\prime} \in \Z/d$,
and so $x$ is of type (ii).

The above argument fails in precisely two cases,
in each of which, $x$ is of type (iii).
\begin{itemize}
\item[Case R.]
{If $i_0 = 0$, and $x = \minq_{j \in \Z/d}\left( z_{i_0}^{j+l} - z_i^j\right)$,
then $z_{0-1}^{j^{\prime}-1}$ does not exist.  Thus
       $z_i^{j^{\prime}} + x \in Q_q$
if and only if
       $z_i^{j^{\prime}} + x = z_0^{j^{\prime}+l}$,
which is true if and only if
   $z_0^{j^{\prime}+l} - z_i^{j^{\prime}}
             = \minq_{j \in \Z/d}\left( z_0^{j+l} - z_i^j\right)$,
which, by definition, is true if and only if
   $z_0^{j+l}$ is active as an R-mobile point
in $\left\langle z_i^j, z_{i+1}^j \right]$ at time $j=j^{\prime}$.}

\item[Case L.]
{If $i_0 = n_{j_0+l}$, and $x = \maxq_{j \in \Z/d}\left( z_{i_0}^{j+l} - z_i^j\right)$,
and there exists $j''$ for which
     $z_i^{j''} + x = z_{n_{j''+l}}^{j''+l} + dq$ (which is not in $Q_q$).
We then say that $z_{n_{j+l}}^{j+l}$ is L-mobile
in $\left\langle z_i^j, z_{i+1}^j \right]$.  In this case,
       $z_i^{j^{\prime}} + x \in Q_q$
if and only if
       $z_i^{j^{\prime}} + x = z_{n_{j^{\prime}+l}}^{j^{\prime}+l}$,
which is true if and only if
  $z_{n_{j^{\prime}+l}}^{j^{\prime}+l} - z_i^{j^{\prime}}
             = \maxq_{j \in \Z/d}\left( z_{i_0}^{j+l} - z_i^j\right)$,
which, by definition, is true if and only if
   $z_{n_{j+l}}^{j+l}$ is active as an L-mobile point
in $\left\langle z_i^j, z_{i+1}^j \right]$ at time $j=j^{\prime}$.}
\end{itemize}
Thus, the second summand of
(\ref{eq: contribution of active mobile points to Q_q cap < z_i^j, z_{i+1}^j ] })
counts all $x \in \left\langle 0, dq \right] \cap \Z$ of type (iii).

Note that the above argument makes no use of equations
(\ref{eq: active means in for R-mobile}) and 
(\ref{eq: active means in for L-mobile}).
In particular, the above argument would hold
even if $[dq]_{k^2}$ failed to satisfy the condition $[dq]_{k^2} < \frac{k^2}{2}$.
\end{proof}

Corollary \ref{cor: q of positive type: combo of lemma and difference eq}
implies that for each mobile point, there must be at least one time when
the mobile point is active and at least one time when the mobile point is inactive.
In fact, we can specify precisely how many times a given mobile point is active.

\begin{prop}
\label{prop: mobile point is active [lepsilon] times or [(l-1)epsilon] times}
Suppose that $q$ of positive type is genus-minimizing.
If $z_0^{j+l}$ is R-mobile in
$\left\langle z_i^j, z_{i+1}^j \right]$ (and hence
$z_{n_{j-l}}^{j-l}$ is L-mobile in
$\left\langle z_{n_j-(i+1)}^j, z_{n_j-i}^j \right]$), for some $l \neq 0 \in \Z/d$
and $i \in \left\{0, \ldots, \left\lfloor\frac{k}{d}\right\rfloor - 2\right\}$,
then each of the two mobile points is active precisely
$[l\epsilon]_d$ times.
If $z_0^{j+l}$ is R-mobile in
$\left\langle z_{n_j-(i+1)}^j, z_{n_j-i}^j \right]$
(and hence $z_{n_{j-l}}^{j-l}$ is L-mobile in
$\left\langle z_i^j, z_{i+1}^j \right]$), for some $l \neq 1 \in \Z/d$
and $i \in \left\{0, \ldots, \left\lfloor\frac{k}{d}\right\rfloor - 2\right\}$,
then each of the two mobile points is active precisely
$[(l-1)\epsilon]_d$ times.
\end{prop}
\begin{proof}
Suppose the hypothesis of the first statement is true.
Then for all $j\in \Z/d$, 
(\ref{eq: z_0^j+l - z_0^j = mu ml k^2/d + xi dq}) implies that
\begin{align}
\label{prop: mobile active le times. eq: explicit equation for z_0^j+l -z_0^j}
      z^{j+l}_0 - z^{j}_i
&= \left[{\mu}ml\right]_d\frac{k^2}{d}
       \;+\; \left(\Xi^{d, \epsilon}_l(j)-i\right) [dq]_{k^2},
                  \\ \nonumber
      z_{n_{j-l}}^{j-l} - z_{n_j-i}^j
&= -\left[{\mu}ml\right]_d\frac{k^2}{d}
       \;-\; \left(\Xi^{d, \epsilon}_l(j-l+1)-i\right) [dq]_{k^2},
\end{align}
where, by
Lemma \ref{lemma: q of positive type, xi lemma},
$\Xi^{d, \epsilon}_l(j) \in \left\{ \frac{[l\epsilon]_d}{d},  \frac{[l\epsilon]_d}{d}-1 \right\}$
for all $j\in\Z/d$, with $\frac{[l\epsilon]_d}{d}$ occurring
$[-l\epsilon]_d$ times and $\frac{[l\epsilon]_d}{d} - 1$
occurring $[l\epsilon]_d$ times.
Since (\ref{prop: mobile active le times. eq: explicit equation for z_0^j+l -z_0^j})
implies $z_0^{j+l}$ is active in
$\left\langle z_i^j, z_{i+1}^j\right]$
(respectively $z_{n_{j-l}}^{j-l}$ is active in
$\left\langle z_{n_j-(i+1)}^j, z_{n_j-i}^j \right]$)
at time $j = j^{\prime}$ if and only if
$\Xi^{d, \epsilon}_l(j^{\prime}) = \frac{[l\epsilon]_d}{d}-1$
(respectively $\Xi^{d, \epsilon}_l(j^{\prime}-l+1) = \frac{[l\epsilon]_d}{d}-1$),
we conclude that each of the two mobile points is active
precisely $[l\epsilon]_d$ times.

Next, suppose the hypothesis of the second statement is true.
Then for all $j\in \Z/d$, 
(\ref{eq: z_0^j+l - z_0^j = mu ml k^2/d + xi dq}) implies that
\begin{align}
\label{prop: mobile active le times. eq: explicit equation for z_0^j+l -z_n_j^j}
      z^{j+l}_0 - z_{n_j -(i+1)}^j
&= \left[{\mu}m(l-1)\right]_d\frac{k^2}{d}
       \;+\; \left(\Xi^{d, \epsilon}_{l-1}(j+1)+(i+1)\right) [dq]_{k^2} + \psi,
                  \\ \nonumber
      z_{n_{j-l}}^{j-l} - z_{i+1}^j
&= -\left[{\mu}m(l-1)\right]_d\frac{k^2}{d}
       \;-\; \left(\Xi^{d, \epsilon}_{l-1}(j-l+1)+(i+1)\right) [dq]_{k^2} - \psi.
\end{align}
Here,
$\Xi^{d, \epsilon}_{l-1}(j) \in \left\{ \frac{[(l-1)\epsilon]_d}{d},  \frac{[(l-1)\epsilon]_d}{d}-1 \right\}$
for all $j\in\Z/d$, with $\frac{[(l-1)\epsilon]_d}{d}$ occurring
$[-(l-1)\epsilon]_d$ times and $\frac{[(l-1)\epsilon]_d}{d} - 1$
occurring $[(l-1)\epsilon]_d$ times.
Since (\ref{prop: mobile active le times. eq: explicit equation for z_0^j+l -z_n_j^j})
implies $z_0^{j+l}$ is active in
$\left\langle z_{n_j-(i+1)}^j, z_{n_j-i}^j \right]$
(respectively $z_{n_{j-l}}^{j-l}$ is active in
$\left\langle z_i^j, z_{i+1}^j\right]$)
at time $j = j^{\prime}$ if and only if
$\Xi^{d, \epsilon}_{l-1}(j^{\prime}+1) = \frac{[(l-1)\epsilon]_d}{d}-1$
(respectively $\Xi^{d, \epsilon}_{l-1}(j^{\prime}-l+1) = \frac{[(l-1)\epsilon]_d}{d}-1$),
we conclude that each of the two mobile points is active
precisely $[(l-1)\epsilon]_d$ times.

\end{proof}

We next define what it means for a mobile point to be {\em neutralized}.
An R-mobile point $z_0^{j+l}$ in 
$\left\langle z_i^j, z_{i+1}^j\right]$, for some $l \neq 0 \in \Z/d$
(respectively in $\left\langle z_{n_j-(i+1)}^j, z_{n_j-i}^j\right]$,
for some $l \neq 1 \in \Z/d$),
is called neutralized if
$z_{n_{j+l-1}}^{j+l-1}$ is also mobile in 
$\left\langle z_i^j, z_{i+1}^j\right]$
(respectively in $\left\langle z_{n_j-(i+1)}^j, z_{n_j-i}^j\right]$).
Likewise, an L-mobile point $z_{n_{j-l}}^{j-l}$ in
$\left\langle z_{n_j-(i+1)}^j, z_{n_j-i}^j\right]$, for some $l \neq 0 \in \Z/d$
(respectively in $\left\langle z_i^j, z_{i+1}^j\right]$,
for some $l \neq 1 \in \Z/d$),
is called neutralized if
$z_0^{j-l+1}$ is also mobile in 
$\left\langle z_{n_j-(i+1)}^j, z_{n_j-i}^j\right]$
(respectively in $\left\langle z_i^j, z_{i+1}^j\right]$).
A pair $z_0^{j+l}$, $z_{n_{j+l-1}}^{j+l-1}$
of neutralized mobile points in
$\left\langle z_i^j, z_{i + 1}^j\right]$
(respectively in $\left\langle z_{n_j-(i+1)}^j, z_{n_j-i}^j\right]$)
 is so called because,
as shown in Proposition \ref{prop: active iff inactive true iff neutralized},
\begin{equation}
\label{main prop, part (i), neutralizing def, active iff inactive}
z_{n_{j+l}}^{j+l}\;\text{is active}
\;\;\;\;\;\Leftrightarrow\;\;\;\;\;
z_0^{j+l+1}\;\text{is inactive}
\end{equation}
in $\left\langle z_i^j, z_{i + 1}^j\right]$
(respectively in $\left\langle z_{n_j-(i+1)}^j, z_{n_j-i}^j\right]$)
at every time $j=j^{\prime} \in \Z/d$.
Thus, their combined contribution to
$\#\!\left({\tilde{Q}}_q \cap \left\langle z_i^{j^{\prime}}, 
               z_{i+1}^{j^{\prime}} \right] \right)$
(respectively to $\#\!\left({\tilde{Q}}_q \cap \left\langle z_{n_{j^{\prime}}-(i+1)}^{j^{\prime}}, 
               z_{n_{j^{\prime}}-i}^{j^{\prime}} \right] \right)$)
is always constant in $j^{\prime} \in \Z/d$.
In other words, each member of the neutralized pair
``neutralizes'' the nonconstancy of the other member's contribution
to $\#\!\left({\tilde{Q}}_q \cap \left\langle z_i^{j}, 
z_{i+1}^{j} \right] \right)$
(respectively to 
$\#\!\left({\tilde{Q}}_q \cap \left\langle z_{n_j-(i+1)}^{j}, 
z_{n_j-i}^{j} \right] \right)$).

Property 
(\ref{main prop, part (i), neutralizing def, active iff inactive})
is also sufficient condition for two mobile points to form a neutralized pair,
as we now demonstrate.
\begin{prop}
\label{prop: active iff inactive true iff neutralized}
Suppose that $q$ is of positive type, and that
$z_0^{j+l_1}$ and $z_{n_{j-l_2}}^{j-l_2}$ are mobile in
$\left\langle z_i^j, z_{i+1}^j\right]$,
for some $l_1 \neq 0, l_2\neq 1 \in \Z/d$
(respectively in $\left\langle z_{n_j-(i+1)}^j, z_{n_j-i}^j\right]$,
for some $l_1 \neq 1, l_2 \neq 0 \in \Z/d$).  Then
$z_0^{j+l_1}$ and $z_{n_{j-l_2}}^{j-l_2}$ form a neutralized pair
({\em i.e.}, $l_1 + l_2 = 1$) if and only if they satisfy
\begin{equation}
\nonumber
z_{n_{j-l_2}}^{j-l_2}\;\text{is active}
\;\;\;\;\;\Leftrightarrow\;\;\;\;\;
z_0^{j+l_1}\;\text{is inactive}
\end{equation}
in $\left\langle z_i^j, z_{i+1}^j\right]$
(respectively in $\left\langle z_{n_j-(i+1)}^j, z_{n_j-i}^j\right]$)
at all times $j=j^{\prime} \in \Z/d$.
\end{prop}

\begin{proof}
We begin with the ``only if'' statement.
Suppose, for some $l \neq 0 \in \Z/d$,
that $z_0^{j+l}$ and $z_{n_{j+l-1}}^{j+l-1}$ are mobile in
$\left\langle z_i^j, z_{i+1}^j\right]$.
Then for any $j^{\prime}\in\Z/d$, we have
\begin{align}
      z_{n_{j^{\prime}+l-1}}^{j^{\prime}+l-1} - z_{i+1}^{j^{\prime}}
&= \left(z_0^{j^{\prime}+l}-\psi\right) 
     -\left(z_i^{j^{\prime}}+ dq\right)
               \\ \nonumber
&= \left(z_0^{j^{\prime}+l} - z_i^{j^{\prime}}\right)
     -\psi - dq.
\end{align}
In particular, $\left(z_{n_{j^{\prime}+l-1}}^{j^{\prime}+l-1} - z_{i+1}^{j^{\prime}}\right)
-\left(z_0^{j^{\prime}+l} - z_i^{j^{\prime}}\right)$ is constant in $j^{\prime}\in \Z/d$.
Thus, at any time $j=j^{\prime} \in \Z/d$,
\begin{equation}
z_{n_{j^{\prime}+l-1}}^{j^{\prime}+l-1} - z_{i+1}^{j^{\prime}}
\!= \maxq_{j\in\Z/d}\left(z_{n_{j+l-1}}^{j+l-1} - z_{i+1}^j\right)
\;\;\;\Leftrightarrow\;\;\;
   z_0^{j^{\prime}+l} - z_i^{j^{\prime}}
\!= \maxq_{j\in\Z/d}\left(z_0^{j+l}\! - z_i^j\right),
\end{equation}
which means that at any time $j=j^{\prime} \in \Z/d$,
$z_{n_{j+l-1}}^{j+l-1} $ is active in $\left\langle z_i^j, z_{i+1}^j\right]$
if and only if $z_0^{j+l}$ is inactive
in $\left\langle z_i^j, z_{i+1}^j\right]$.

Next, we prove the ``if'' statement.
Suppose we know, for every $j^{\prime} \in \Z/d$, that
$z_{n_{j-l_2}}^{j-l_2}$ is active
in $\left\langle z_i^j, z_{i+1}^j\right]$
at time $j=j^{\prime}$ if and only if
$z_0^{j+l_1}$ is inactive
in $\left\langle z_i^j, z_{i+1}^j\right]$
at time $j=j^{\prime}$.
Then at any given time $j=j^{\prime}\in\Z/d$, either
$z_0^{j+l_1}$ is active and $z_{n_{j-l_2}}^{j-l_2}$ is inactive, so that
\begin{align}
       z_0^{j^{\prime}+l_1}  - z_{n_{j^{\prime}-l_2}}^{j^{\prime}-l_2}
&=  \left(z_0^{j^{\prime}+l_1} - z_i^{j^{\prime}}\right)
     - \left(z_{n_{j^{\prime}-l_2}}^{j^{\prime}-l_2} - z_{i+1}^{j^{\prime}}\right)
     - \left(z_{i+1}^{j^{\prime}} - z_i^{j^{\prime}}\right)
               \\ \nonumber
&=  \minq_{j\in\Z/d} \left(z_0^{j+l_1} - z_i^j\right)
      -\minq_{j\in\Z/d} \left(z_{n_{j-l_2}}^{j-l_2} - z_{i+1}^j\right)
      - dq,
\end{align}
or $z_0^{j+l_1}$ is inactive and $z_{n_{j-l_2}}^{j-l_2}$ is active, so that
\begin{align}
       z_0^{j^{\prime}+l_1}  - z_{n_{j^{\prime}-l_2}}^{j^{\prime}-l_2}
&=  \left(z_0^{j^{\prime}+l_1} - z_i^{j^{\prime}}\right)
     - \left(z_{n_{j^{\prime}-l_2}}^{j^{\prime}-l_2} - z_{i+1}^{j^{\prime}}\right)
     - \left(z_{i+1}^{j^{\prime}} - z_i^{j^{\prime}}\right)
               \\ \nonumber
&=  \maxq_{j\in\Z/d} \left(z_0^{j+l_1} - z_i^j\right)
      -\maxq_{j\in\Z/d} \left(z_{n_{j-l_2}}^{j-l_2} - z_{i+1}^j\right)
      - dq
               \\ \nonumber
&=  \minq_{j\in\Z/d} \left(z_0^{j+l_1} - z_i^j\right)
      -\minq_{j\in\Z/d} \left(z_{n_{j-l_2}}^{j-l_2} - z_{i+1}^j\right)
      - dq.
\end{align}
Thus  $z_0^{j+l_1}  - z_{n_{j-l_2}}^{j-l_2}$ is constant in $j\in\Z/d$,
which means that $z_0^{j+l_1}  - z_0^{j-l_2+1}$ is constant in $j\in\Z/d$,
which, by Corollary \ref{cor: q of positive type: combo of lemma and difference eq},
implies that $(j+l_1)  - (j-l_2+1) = 0$, and so $l_1+l_2 = 1$.
             \\

The analogous proof for mobile points in
$\left\langle z_{n_j-(i+1)}^j, z_{n_j-i}^j\right]$ is the same, but with
$l_1 \neq 0$, $l_2 \neq 1$, $z_i^j$, $z_{i+1}^j$, $z_i^{j^{\prime}}$, and
$z_{i+1}^{j^{\prime}}$ replaced with
$l_1 \neq 1$, $l_2 \neq 0$, $z_{n_j-(i+1)}^j$, $z_{n_j-i}^j$, 
$z_{n_{j^{\prime}}-(i+1)}^{j^{\prime}}$, and
$z_{n_{j^{\prime}}-i}^{j^{\prime}}$, respectively.
\end{proof}

Lastly, we introduce the notion of the {\em pseudomobile point},
which plays a role analogous to that of an ordinary mobile point,
but in the interval
$\left\langle z_{n_{j-1}}^{j-1}, z_0^j \right]$.
{\em A pseudomobile point is not a mobile point!}
The term mobile point {\em only} pertains to intervals of the form
$\left\langle z_i^j, z_{i+1}^j \right]$ or
$\left\langle z_{n_j - (i+1)}^j, z_{n_j - i}^j \right]$
for some $i \in \{0, \ldots, \left\lfloor\frac{k}{d}\right\rfloor - 2 \}$.
Thus, the question of whether ${\bf{z}}^j$ is said to have mobile points
has nothing to do with whether
$\left\langle z_{n_{j-1}}^{j-1}, z_0^j \right]$, which straddles 
${\bf{z}}^{j-1}$ and ${\bf{z}}^j$, is said to have
pseudomobile points.

With that caveat out of the way, we commence with a definition.
First, for the remainder of Section \ref{s:p=k2}, we fix
\begin{align}
   \psi 
:=& \left[z_0^j - z_{n_{j-1}}^{j-1}\right]_{k^2}\;\;\;\text{(which is constant in}\;j\in\Z/d\text{)}
        \\ \nonumber
=&  \left[dq - kq\right]_{k^2}
        \\ \nonumber
=& \left[(({\mu}m+{\gamma}c)k+\alpha) 
        \;-\;  \textstyle{\gamma\frac{ck+\alpha\gamma}{d}}k\right]_{k^2}.
\end{align}
Then, for any $l \neq 0 \in \Z/d$, we say that
$z_0^{j+l}$ is an R-pseudomobile point in
$\left\langle z_{n_{j-1}}^{j-1}, z_0^j \right]$ if
\begin{equation}
      0
\;<\; \minq_{j\in\Z/d} \left( z_0^{j+l} - z_{n_{j-1}}^{j-1} \right)
\;<\; \psi,
\end{equation}
and that $z_{n_{j-1-l}}^{j-1-l}$ is an
L-pseudomobile point in
$\left\langle z_{n_{j-1}}^{j-1}, z_0^j \right]$ if
\begin{equation}
      -\psi
\;<\; \maxq_{j\in\Z/d} \left( z_{n_{j-1-l}}^{j-1-l}- z_0^j \right)
\;<\; 0.
\end{equation}
We say that a point is pseudomobile in
$\left\langle z_{n_{j-1}}^{j-1}, z_0^j \right]$ if it is
R-pseudomobile or L-pseudomobile in
$\left\langle z_{n_{j-1}}^{j-1}, z_0^j \right]$.

Using Corollary \ref{cor: q of positive type: combo of lemma and difference eq},
it is easy to calculate that
\begin{align}
      \minq_{j\in\Z/d}\left(z_0^{j+l} - z_{n_{j-1}}^{j-1}\right)
&= [{\mu}ml]_d\frac{k^2}{d} + \left(\frac{[l\epsilon]_d}{d}-1\right)[dq]_{k^2}   +  \psi,
                 \\
      \maxq_{j\in\Z/d}\left(z_{n_{j-1-l}}^{j-1-l}-z_0^j\right)
&= -[{\mu}ml]_d\frac{k^2}{d} - \left(\frac{[l\epsilon]_d}{d}-1\right)[dq]_{k^2}   -  \psi,
\end{align}
for all nonzero $l\in\Z/d$.  Thus
\begin{equation}
    \minq_{j\in\Z/d} \left( z_0^{j+l} - z_{n_{j-1}}^{j-1} \right)
= -\maxq_{j\in\Z/d} \left( z_{n_{j-1-l}}^{j-1-l} - z_0^j \right)
\end{equation}
for all nonzero $l\in\Z/d$.
In particular, $z_0^{j+l}$ is R-pseudomobile in
$\left\langle z_{n_{j-1}}^{j-1}, z_0^j \right]$ if and only if
$z_{n-1-l}^{j-1-l}$ is L-pseudomobile in
$\left\langle z_{n_{j-1}}^{j-1}, z_0^j \right]$.
We therefore say that $z_0^{j+l}$ and
$z_{n_{j-1-l}}^{j-1-l}$ are mirror pseudomobile points
in $\left\langle z_{n_{j-1}}^{j-1}, z_0^j \right]$.

Next, we discuss the notion of active and inactive
pseudomobile points.
If $z_0^{j+l}$ is R-pseudomobile in
$\left\langle z_{n_{j-1}}^{j-1}, z_0^j \right]$,
then we say that $z_0^{j+l}$ is active
at time $j=j^{\prime} \in \Z/d$ if 
\begin{equation}
      z_0^{j^{\prime}+l} - z_{n_{j^{\prime}-1}}^{j^{\prime}-1}
\;=\; \minq_{j\in\Z/d} \left( z_0^{j+l} - z_{n_{j-1}}^{j-1} \right),
\end{equation}
and is inactive otherwise.  Likewise,
if $z_{n_{j-1-l}}^{j-1-l}$ is L-pseudomobile in
$\left\langle z_{n_{j-1}}^{j-1}, z_0^j \right]$,
then $z_{n_{j-1-l}}^{j-1-l}$ is active at time $j = j^{\prime}\in\Z/d$ if
\begin{equation}
      z_{n_{j^{\prime}-1-l}}^{j^{\prime}-1-l} - z_0^{j^{\prime}}
\;=\; \maxq_{j\in\Z/d} \left( z_{n_{j-1-l}}^{j-1-l} - z_0^j \right),
\end{equation}
and is inactive otherwise.
Similar to the case of mobile points,
there exists $C \in \Z$ constant in $j^{\prime} \in \Z/d$
such that, for any $j^{\prime} \in \Z/d$, we have
\begin{equation}
      \#\left({\tilde{Q}}_q \cap \left\langle z_{n_{j^{\prime}-1}}^{j^{\prime}-1}, 
               z_0^{{j^{\prime}}} \right] \right)
\;=\; C \;+\; 
      \#\!\left\{\begin{array}{cc}
         \text{pseudomobile points in}\;
         \left\langle z_{n_{j-1}}^{j-1}, z_0^j \right]
                  \\
         \text{which are active at time}\; j=j^{\prime} \end{array}
      \right\}.
\end{equation}
The following proposition specifies how many times a given
pseudomobile point is active.

\begin{prop}
\label{prop: pseudomobile point is active [lepsilon] times}
Suppose that $q$ of positive type is genus-minimizing.
If $z_0^{j+l}$ is R-pseudomobile (and hence
$z_{n_{j-1-l}}^{j-1-l}$ is L-pseudomobile)
in $\left\langle z_{n_{j-1}}^{j-1}, z_0^j \right]$,
then each of the two pseudomobile points is active precisely
$[l\epsilon]_d$ times.
\end{prop}
\begin{proof}
For all $j\in \Z/d$, 
(\ref{eq: z_0^j+l - z_0^j = mu ml k^2/d + xi dq}) implies that
\begin{align}
\label{prop: pseudomobile active le times. eq: explicit equation for z_0^j+l -z_0^j}
      z^{j+l}_0 - z_{n_{j-1}}^{j-1}
&= \left[{\mu}ml\right]_d\frac{k^2}{d}
       \;+\; \left(\Xi^{d, \epsilon}_l(j)\right) [dq]_{k^2} +\psi,
                  \\ \nonumber
      z_{n_{j-1-l}}^{j-1-l} - z_0^j
&= -\left[{\mu}ml\right]_d\frac{k^2}{d}
       \;-\; \left(\Xi^{d, \epsilon}_l(j-l)\right) [dq]_{k^2} - \psi,
\end{align}
where, by
Lemma \ref{lemma: q of positive type, xi lemma},
$\Xi^{d, \epsilon}_l(j) \in \left\{ \frac{[l\epsilon]_d}{d},  \frac{[l\epsilon]_d}{d}-1 \right\}$
for all $j\in\Z/d$, with $\frac{[l\epsilon]_d}{d}$ occurring
$[-l\epsilon]_d$ times and $\frac{[l\epsilon]_d}{d} - 1$
occurring $[l\epsilon]_d$ times.
Since (\ref{prop: pseudomobile active le times. eq: explicit equation for z_0^j+l -z_0^j})
implies $z_0^{j+l}$ (respectively $z_{n_{j-1-l}}^{j-1-l}$)
is active in $\left\langle z_{n_{j-1}}^{j-1}, z_0^j \right]$
at time $j = j^{\prime}$ if and only if
$\Xi^{d, \epsilon}_l(j^{\prime}) = \frac{[l\epsilon]_d}{d}-1$
(respectively $\Xi^{d, \epsilon}_l(j^{\prime}-l) = \frac{[l\epsilon]_d}{d}-1$),
we conclude that each of the two pseudomobile points is active
precisely $[l\epsilon]_d$ times.
\end{proof}

Finally, we say that an R-pseudomobile point $z_0^{j+l}$ in
$\left\langle z_{n_{j-1}}^{j-1}, z_0^j \right]$
is {\em neutralized} if
$z_{n_{j+l-1}}^{j+l-1}$ is L-pseudomobile in
$\left\langle z_{n_{j-1}}^{j-1}, z_0^j \right]$,
and that an L-pseudomobile point $z_{n_{j+l}}^{j+l}$ in
$\left\langle z_{n_{j-1}}^{j-1}, z_0^j \right]$
is neutralized if
$z_0^{j+l+1}$ is R-pseudomobile in
$\left\langle z_{n_{j-1}}^{j-1}, z_0^j \right]$.
We say a pseudomobile point is {\em non-neutralized}
if it is not neutralized.
The reason for this terminology is as follows.
As proven in
Proposition \ref{prop: active iff inactive means neutralized for pseudomobile}
below, if $z_0^{j+l}$ and $z_{n_{j+l-1}}^{j+l-1}$ are pseudomobile in
$\left\langle z_{n_{j-1}}^{j-1}, z_0^j \right]$, then
at any time $j^{\prime} \in \Z/d$,
\begin{equation}
z_0^{j^{\prime}+l}\;\text{is active}
\;\;\;\;\;\Leftrightarrow\;\;\;\;\;
z_{n_{j^{\prime}+l-1}}^{j^{\prime}+l-1}\;\text{is inactive}.
\end{equation}
Thus, their combined associated contribution to
$\#\!\left({\tilde{Q}}_q \cap \left\langle z_{n_{j^{\prime}-1}}^{j^{\prime}-1}, 
               z_0^{{j^{\prime}}} \right] \right)$
is always constant in $j^{\prime} \in \Z/d$.
              \\

\begin{prop}
\label{prop: active iff inactive means neutralized for pseudomobile}
Suppose that $q$ of positive type is genus-minimizing, and
that $z_{n_{j+l_1}}^{j+l_1}$ and $z_0^{j+l_2}\!$ are pseudomobile in
$\left\langle z_{n_{j-1}}^{j-1}, z_0^j \right]$
for some $l_1, l_2 \in \Z/d$.  Then
$z_{n_{j+l_1}}^{j+l_1}$ and $z_0^{j+l_2}$ form a neutralized pair
({\em i.e.}, $l_1 + 1 = l_2$) if and only if they satisfy
\begin{equation}
\nonumber
z_{n_{j+l_1}}^{j+l_1}\;\text{is active}
\;\;\;\;\;\Leftrightarrow\;\;\;\;\;
z_0^{j+l_2}\;\text{is inactive}
\end{equation}
in $\left\langle z_{n_{j-1}}^{j-1}, z_0^j \right]$ at all times $j=j^{\prime} \in \Z/d$.
\end{prop}
\begin{proof}
We begin with the ``only if'' statement.
For any $j^{\prime}\in\Z/d$, since
\begin{align}
      z_{n_{j^{\prime}+l}}^{j^{\prime}+l} - z_0^{j^{\prime}}
&= \left(z_0^{j^{\prime}+l+1}-\psi\right) 
     -\left(z_{n_{j^{\prime}-1}}^{j^{\prime}-1}+\psi\right)
               \\ \nonumber
&= \left(z_0^{j^{\prime}+l+1} - z_{n_{j^{\prime}-1}}^{j^{\prime}-1}\right) -2\psi,
\end{align}
we know that
\begin{equation}
z_{n_{j^{\prime}\!+l}}^{j^{\prime}\!+l}\! - z_0^{j^{\prime}}
\!= \maxq_{j\in\Z/d}\left(z_{n_{j+l}}^{j+l} - z_0^{j}\right)
\;\;\;\Leftrightarrow\;\;\;
   z_0^{j^{\prime}\!+l+1}\! - z_{n_{j^{\prime}-1}}^{j^{\prime}-1}
\!= \maxq_{j\in\Z/d}\left(z_0^{j+l+1}\! - z_{n_{j-1}}^{j-1}\right).
\end{equation}
Thus, at any time $j=j^{\prime} \in \Z/d$,
$z_{n_{j+l}}^{j+l}$ is active in $\left\langle z_{n_{j-1}}^{j-1}, z_0^j \right]$
if and only if $z_0^{j+l+1}$ is inactive
in $\left\langle z_{n_{j-1}}^{j-1}, z_0^j \right]$.

Next, we prove the ``if'' statement.
Suppose that at any given time $j=j^{\prime}\in\Z/d$, either
$z_0^{j+l_2}$ is active and $z_{n_{j+l_1}}^{j+l_1}$ is inactive
in $\left\langle z_{n_{j-1}}^{j-1}, z_0^j \right]$, so that
\begin{align}
       z_0^{j^{\prime}+l_2}  - z_{n_{j^{\prime}+l_1}}^{j^{\prime}+l_1}
&=  \left(z_0^{j^{\prime}+l_2} - z_{n_{j^{\prime}-1}}^{j^{\prime}-1}\right)
     - \left(z_{n_{j^{\prime}+l_1}}^{j^{\prime}+l_1} - z_0^{j^{\prime}}\right)
     - \left(z_0^{j^{\prime}} - z_{n_{j^{\prime}-1}}^{j^{\prime}-1}\right)
               \\ \nonumber
&=  \minq_{j\in\Z/d} \left(z_0^{j+l_2} - z_{n_{j-1}}^{j-1}\right)
      -\minq_{j\in\Z/d} \left(z_{n_{j+l_1}}^{j+l_1} - z_0^j\right)
      - \psi,
\end{align}
or $z_0^{j+l_2}$ is inactive and $z_{n_{j+l_1}}^{j+l_1}$ is active, so that
\begin{align}
       z_0^{j^{\prime}+l_2}  - z_{n_{j^{\prime}+l_1}}^{j^{\prime}+l_1}
&=  \left(z_0^{j^{\prime}+l_2} - z_{n_{j^{\prime}-1}}^{j^{\prime}-1}\right)
     - \left(z_{n_{j^{\prime}+l_1}}^{j^{\prime}+l_1} - z_0^{j^{\prime}}\right)
     - \left(z_0^{j^{\prime}} - z_{n_{j^{\prime}-1}}^{j^{\prime}-1}\right)
               \\ \nonumber
&=  \maxq_{j\in\Z/d} \left(z_0^{j+l_2} - z_{n_{j-1}}^{j-1}\right)
      -\maxq_{j\in\Z/d} \left(z_{n_{j+l_1}}^{j+l_1} - z_0^j\right)
      - \psi
               \\ \nonumber
&=  \minq_{j\in\Z/d} \left(z_0^{j+l_2} - z_{n_{j-1}}^{j-1}\right)
      -\minq_{j\in\Z/d} \left(z_{n_{j+l_1}}^{j+l_1} - z_0^j\right)
      - \psi,
\end{align}
Thus  $z_0^{j+l_2}  - z_{n_{j+l_1}}^{j+l_1}$ is constant in $j\in\Z/d$,
which means that $z_0^{j+l_2}  - z_0^{j+l_1+1}$ is constant in $j\in\Z/d$,
which, by Corollary \ref{cor: q of positive type: combo of lemma and difference eq},
implies that $(j+l_2) - (j+l_1+1) = 0$, and so $l_1+1 = l_2$.
\end{proof}

\subsection{Properties of ${\mathbf{z}}^j$ for Genus-Minimizing $q$ of Positive Type}
\label{ss: Properties of z^j for Genus-Minimizing q of Positive Type}

We have finally introduced enough terminology to be able to
state some results.  For an initial reading, the reader might
obtain a clearer picture of the overall argument if he or she skips 
all treatments of the special case in which
$\left\lfloor\frac{k}{d}\right\rfloor = 2$.


\begin{prop}
\label{prop: positive type, main prop}
Suppose $q$ of positive type is genus-minimizing,
and let $x_*, y_*$ denote the unique elements of $Q_q$ for which
$v_q(x_*, y_*) = {\alpha}(k-k^2)$.  Then the following are true:
\begin{itemize}
\item[(i)]
If an interval $\left\langle z_i^j, z_{i+1}^j\right]$
(respectively $\left\langle z_{n_j-(i+1)}^j, z_{n_j - i}^j\right]$)
has any mobile points, then there exists a unique $j_* \in \Z/d$
such that $\left(z_i^{j_*}, z_{i+1}^{j_*}\right) = \left(x_*, y_*\right)$
(respectively $\left(z_{n_{j_*}-(i+1)}^{j_*}, z_{n_{j_*} - i}^{j_*}\right) = \left(x_*, y_*\right)$).

\item[(i$\psi$)]
If the interval
$\left\langle z_{n_{j-1}}^{j-1}, z_0^j \right]$
has any non-neutralized pseudomobile points, then
there exists a unique $j_* \in \Z/d$ such that
$\left(z_{n_{j_*-1}}^{j_*-1}, z_0^{j_*}\right) = \left(x_*, y_*\right)$.

\item[(ii)]
If $v_q(z_{n_{j-1}}^{j-1}, z_0^j)$ is constant in $j\in\Z/d$, then
there are precisely two mobile points in ${\bf{z}}^j$,
namely, $z_0^{j+l}$ R-mobile in $\left\langle z_{i_*}^j, z_{i_*+1}^j \right]$ and
$z_0^{j-l}$ L-mobile in
$\left\langle z_{n_j - (i_*+1)}^j, z_{n_j - i_*}^j \right]$, for some nonzero $l\in\Z/d$,
where $i_*$ is the unique element of  
$\left\{0, \ldots, \left\lfloor\frac{k}{d}\right\rfloor -2 \right\}$ satisfying
$\left(x_*,y_*\right) = \left(z_{i_*}^{j_*},z_{i_*+1}^{j_*}\right) 
= \left(z_{n_{j_*} - (i_*+1)}^{j_*}, z_{n_{j_*} - i_*}^{j_*}\right)$.

\item[(ii$\psi$)]
If $v_q(z_{n_{j-1}}^{j-1}, z_0^j)$ is nonconstant in $j\in\Z/d$, then
$\left\langle z_{n_{j-1}}^{j-1}, z_0^j \right]$ has precisely one
non-neutralized R-pseudomobile point and precisely one
non-neutralized L-pseudomobile point, namely,
$z_0^{j+l}$ and $z_{n_{j-1-l}}^{j-1-l}$ for some nonzero $l\in \Z/d$.

\item[(iii)]
If $(\mu,\gamma) = (1,1)$, then
$v_q(z_{n_{j-1}}^{j-1}, z_0^j) \equiv {\alpha}k$ for all $j \in \Z/d$.

\item[(iv)]
All mobile points are non-neutralized.
Moreover, $\psi > 2\left[dq\right]_{k^2}$
unless $(\mu,\gamma) = (1,1)$, $\alpha = -1$, $m=2$, $c=1$,
and $d= \frac{k-1}{2} \equiv 0\; (\mod 2)$. 
\end{itemize}
\end{prop}

Before commencing with the proof,
we pause for a brief discussion that will facilitate the proofs
of Parts (i), (i$\psi$), (ii), and (ii$\psi$).

We introduce some notation used to make lists of R-mobile points and L-mobile points,
and sublists of mobile points which are active at a given time.  In particular,
we show in Claim \ref{claim: structure of L_R^j and L_L^j}
that an ordering on the list of R-mobile points endows the list of
active R-mobile points with a particular structure, and similarly for 
L-mobile points.  The tools and terminology introduced in this discussion
are not needed outside the proofs of Parts (i), (i$\psi$), (ii), and (ii$\psi$),
since the results of (i), (i$\psi$), (ii), and (ii$\psi$)
obviate their utility.

For the following discussion and the proofs of Parts (i) and (ii), 
fix any $i \in \{0, \ldots, \left\lfloor\frac{k}{d}\right\rfloor -2\}$,
and let $a_j$ denote one of two functions of $j\in\Z/d$; either
$a_j \equiv i$ for all $j \in \Z/d$, or $a_j = n_j - (i+1)$ for each $j \in \Z/d$.
We do this so that we can speak of the interval
$\left\langle z_{a_j}^j, z_{a_j + 1}^j\right]$
without needing to specify whether we are measuring positions 
rel $z_0^j$ or rel $z_{n_j}^j$.
Note that when $a_j \equiv i$,
any mobile point $z_0^{j+l}$ (respectively $z_{n_{j-l}}^{j-l}$) in
$\left\langle z_{a_j}^j, z_{a_j + 1}^j\right]$ must satisfy $l \neq 0$
(respectively $l\neq 1$),
whereas when $a_j = n_j - (i+1)$ for each $j \in \Z/d$,
any mobile point $z_0^{j+l}$ (respectively $z_{n_{j-l}}^{j-l}$) in
$\left\langle z_{a_j}^j, z_{a_j + 1}^j\right]$ must satisfy $l \neq 1$
(respectively $l\neq 0$).

Let us begin by listing the mobile points, if any, in 
$\left\langle z_{a_j}^j, z_{a_j + 1}^j\right]$.
Define $L_{\mathrm{R}}$ and $L_{\mathrm{L}}$ by
\begin{align}
L_{\mathrm{R}} := \left\{l \in \Z/d \left\vert\;
       z_0^{j+l}\;\text{is R-mobile in}\;
       \left\langle z_{a_j}^j, z_{a_j + 1}^j\right]\right.\right\},
              \\ \nonumber
L_{\mathrm{L}} := \left\{l \in \Z/d \left\vert\;
       z_{n_{j+l}}^{j+l}\;\text{is L-mobile in}\;
       \left\langle z_{a_j}^j, z_{a_j + 1}^j\right]\right.\right\}.
\end{align}
We then define an order relation $<_{\mathrm{R}}$ on $L_{\mathrm{R}}$ as follows.
For any two distinct elements $l_1, l_2 \in L_{\mathrm{R}}$, we say that
$l_1 <_{\mathrm{R}} l_2$ if there exists $j^{\prime} \in \Z/d$ such that
$z_0^{j+l_1}$ is active and
$z_0^{j+l_2}$ is inactive in $\left\langle z_{a_j}^j, z_{a_j + 1}^j\right]$
at time $j=j^{\prime}$, or in other words, if
\begin{align}
      z_0^{j^{\prime\!}+l_1} - z_{a_{j^{\prime}}}^{j^{\prime}} 
&= \minq_{j\in\Z/d} \left( z_0^{j+l_1}\! - z_{a_j}^j\right),
              \\ \nonumber
      z_0^{j^{\prime\!}+l_2} - z_{a_{j^{\prime}}}^{j^{\prime}} 
&= \maxq_{j\in\Z/d} \left( z_0^{j+l_2}\! - z_{a_j}^j\right).
\end{align}
Similarly, for any two distinct elements $l_1, l_2 \in L_{\mathrm{L}}$, we say that
$l_1 <_{\mathrm{L}} l_2$ if there exists $j^{\prime} \in \Z/d$ such that
$z_{n_{j+l_1}}^{j+l_1}$ is inactive and 
$z_{n_{j+l_2}}^{j+l_2}$ is active
in $\left\langle z_{a_j}^j, z_{a_j + 1}^j\right]$ at time $j=j^{\prime}$.

We claim that the relations $<_{\mathrm{R}}$ and $<_{\mathrm{L}}$ define
valid orderings on $L_{\mathrm{R}}$ and $L_{\mathrm{L}}$, respectively.
We begin by showing that $<_{\mathrm{R}}$ and $<_{\mathrm{L}}$
are well defined.
Focusing on the case of $<_{\mathrm{R}}$, assume that
$\left| L_{\mathrm{R}}\right| \geq 2$ (since otherwise the claim is trivial),
and choose two arbitrary distinct elements $l_1, l_2 \in L_{\mathrm{R}}$.
Suppose that there is no time $j=j^{\prime}\in\Z/d$ at which
$z_0^{j+l_1}$ is active and
$z_0^{j+l_2}$ is inactive in $\left\langle z_{a_j}^j, z_{a_j+1}^j \right]$, 
and that there is also no time
$j=j^{\prime}\in\Z/d$ at which
$z_0^{j+l_1}$ is inactive and
$z_0^{j+l_2}$ is active in $\left\langle z_{a_j}^j, z_{a_j+1}^j \right]$.
Then for every $j^{\prime} \in \Z/d$, either
$z_0^{j+l_1}$ and $z_0^{j+l_2}$ are both active
in $\left\langle z_{a_j}^j, z_{a_j+1}^j \right]$ at time $j=j^{\prime}$, in which case
\begin{align}
      z_0^{j^{\prime}\!+l_1} - z_0^{j^{\prime}\!+l_2} 
&=  \left(z_0^{j^{\prime}\!+l_1} - z_{a_{j^{\prime}}}^{j^{\prime}}\right)
      -\left(z_0^{j^{\prime}\!+l_2} - z_{a_{j^{\prime}}}^{j^{\prime}}\right)
              \\ \nonumber
&=  \minq_{j\in\Z/d} \left(z_0^{j+l_1} - z_{a_j}^j\right)
     - \minq_{j\in\Z/d} \left(z_0^{j+l_2} - z_{a_j}^j\right),
\end{align}
or $z_0^{j+l_1}$ and $z_0^{j+l_2}$ are both inactive
in $\left\langle z_{a_j}^j, z_{a_j+1}^j \right]$ at time $j=j^{\prime}$, in which case
\begin{align}
      z_0^{j^{\prime}\!+l_1} - z_0^{j^{\prime}\!+l_2} 
&=  \left(z_0^{j^{\prime}\!+l_1} - z_{a_{j^{\prime}}}^{j^{\prime}}\right)
      -\left(z_0^{j^{\prime}\!+l_2} - z_{a_{j^{\prime}}}^{j^{\prime}}\right)
              \\ \nonumber
&=  \maxq_{j\in\Z/d} \left(z_0^{j+l_1} - z_{a_j}^j\right)
     - \maxq_{j\in\Z/d} \left(z_0^{j+l_2} - z_{a_j}^j\right)
              \\ \nonumber
&=  \minq_{j\in\Z/d} \left(z_0^{j+l_1} - z_{a_j}^j\right)
     - \minq_{j\in\Z/d} \left(z_0^{j+l_2} - z_{a_j}^j\right).
\end{align}
Thus 
$z_0^{j^{\prime}\!+l_1} - z_0^{j^{\prime}\!+l_2}$
is constant in $j^{\prime}\in\Z/d$,
contradicting Corollary \ref{cor: q of positive type: combo of lemma and difference eq}.

On the other hand, suppose that there exist $j_1, j_2 \in \Z/d$ such that
$z_0^{j+l_1}$ is active and
$z_0^{j+l_2}$ is inactive 
in $\left\langle z_{a_j}^j, z_{a_j+1}^j \right]$ at time $j=j_1$,
but
$z_0^{j+l_1}$ is inactive and
$z_0^{j+l_2}$ is active
in $\left\langle z_{a_j}^j, z_{a_j+1}^j \right]$ at time $j=j_2$.
Then the fact that $z_0^{j+l_1}$ is active when $j=j_1$
but inactive when $j=j_2$ implies that
\begin{equation}
        \left(z_0^{j_1+l_1} - z_{a_{j_1}}^{j_1}\right)
       -\left(z_0^{j_2+l_1} - z_{a_{j_2}}^{j_2}\right)
= -dq,
\end{equation}
whereas the fact that $z_0^{j+l_2}$ is inactive when $j=j_1$
but active when $j=j_2$ implies that
\begin{equation}
        \left(z_0^{j_1+l_2} - z_{a_{j_1}}^{j_1}\right)
       -\left(z_0^{j_2+l_2} - z_{a_{j_2}}^{j_2}\right)
= dq.
\end{equation}
Subtracting these two equations gives
\begin{equation}
  \left(z_0^{j_1+l_1} - z_0^{j_1+l_2}\right)
- \left(z_0^{j_2+l_1} - z_0^{j_2+l_2}\right)
= -2dq,
\end{equation}
contradicting Corollary \ref{cor: q of positive type: combo of lemma and difference eq}.
Thus $<_{\mathrm{R}}$ is well-defined.
A similar argument shows that $<_{\mathrm{L}}$ is well-defined.

Next, we show that $<_{\mathrm{R}}$ defines a valid order relation on $L_{\mathrm{R}}$.
Suppose there exist distinct $l_1, l_2, l_3 \in L_{\mathrm{R}}$ such that
$l_1 <_{\mathrm{R}} l_2$, $l_2 <_{\mathrm{R}} l_3$,
and $l_3 <_{\mathrm{R}} l_1$.
Then there exists $j^{\prime} \in \Z$ such that
$z_0^{j+l_3}$ is active and $z_0^{j+l_1}$ is inactive 
in $\left\langle z_{a_j}^j, z_{a_j+1}^j \right]$ at time $j=j^{\prime}$.
Now, if $z_0^{j+l_2}$ is active in $\left\langle z_{a_j}^j, z_{a_j+1}^j \right]$
when $j=j^{\prime}$, then this
contradicts our assumption that $l_1 <_{\mathrm{R}} l_2$,
but if $z_0^{j+l_2}$ is inactive in $\left\langle z_{a_j}^j, z_{a_j+1}^j \right]$
when $j=j^{\prime}$,
then this contradicts our assumption that $l_2 <_{\mathrm{R}} l_3$.
Thus, if $l_1 <_{\mathrm{R}} l_2$ and $l_2 <_{\mathrm{R}} l_3$, 
then we must have $l_1 <_{\mathrm{R}} l_3$,
and so $<_{\mathrm{R}}$ defines an ordering on $L_{\mathrm{R}}$.
A similar argument shows that
$<_{\mathrm{L}}$ defines an ordering on $L_{\mathrm{L}}$.

We therefore write 
$l_1^{\mathrm{R}}, l_2^{\mathrm{R}}, \ldots, l_{|L_{\mathrm{R}}|}^{\mathrm{R}}$
for the elements of $L_{\mathrm{R}}$ such that
$l_1^{\mathrm{R}} 
<_{\mathrm{R}} l_2^{\mathrm{R}}
<_{\mathrm{R}} \ldots <_{\mathrm{R}} l_{|L_{\mathrm{R}}|}^{\mathrm{R}}$,
and likewise write
$l_1^{\mathrm{L}}, l_2^{\mathrm{L}}, \ldots, l_{|L_{\mathrm{L}}|}^{\mathrm{L}}$
for the elements of $L_{\mathrm{L}}$ such that
$l_1^{\mathrm{L}} 
<_{\mathrm{L}} l_2^{\mathrm{L}}
<_{\mathrm{L}} \ldots <_{\mathrm{L}} l_{|L_{\mathrm{L}}|}^{\mathrm{L}}$.
Thus when, for each ${j^{\prime}} \in \Z/d$, we define
\begin{align}
    L^{{j^{\prime}}}_{\mathrm{R}}
&:= \left\{l \in L_{\mathrm{R}} \left\vert\; 
     z_0^{j+l}\;\text{is active in}
     \;\left\langle z_{a_j}^j, z_{a_j+1}^j \right]\;
     \text{when}\;j=j^{\prime}
    \right.\right\},
              \\ \nonumber
    L^{{j^{\prime}}}_{\mathrm{L}}
&:= \left\{l \in L_{\mathrm{L}} \left\vert\; 
     z_{n_{j+l}}^{j+l}\;\text{is active in}
     \;\left\langle z_{a_j}^j, z_{a_j+1}^j \right]\;
     \text{when}\;j=j^{\prime}
    \right.\right\},
\end{align}
the order relations $<_{\mathrm{R}}$ on $L_{\mathrm{R}}$ and
$<_{\mathrm{L}}$ on $L_{\mathrm{L}}$ make the elements of
$L^{j^{\prime}}_{\mathrm{R}}$ and $L^{j^{\prime}}_{\mathrm{L}}$
easy to enumerate:
\begin{equation}
  L^{j^{\prime}}_{\mathrm{R}}
= \left\{ l_1^{\mathrm{R}}, \ldots, l_{|L_{\mathrm{R}}^{j^{\prime}}|}^{\mathrm{R}} \right\}
     \;\;\mathrm{and}\;\;
  L^{j^{\prime}}_{\mathrm{L}}
= \left\{ l^{\mathrm{L}}_{|L_{\mathrm{L}}| -  |L_{\mathrm{L}}^{j^{\prime}}|+1} \ldots,
          l_{|L_{\mathrm{L}}|}^{\mathrm{L}} \right\}
\end{equation}
for any $j^{\prime}\in\Z/d$.  In particular, the following is true.
\begin{claim}
\label{claim: structure of L_R^j and L_L^j}
For any $j^{\prime} \in \Z/d$ and $s\in\{1, \ldots, |L_{\mathrm{R}}|\}$
such that $ l_s^{\mathrm{R}} \in  L^{{j^{\prime}}}_{\mathrm{R}}$, we have
\begin{equation*}
l_1^{\mathrm{R}}, \ldots, l_s^{\mathrm{R}}  \in  L^{{j^{\prime}}}_{\mathrm{R}}.
\end{equation*}
For any $j^{\prime} \in \Z/d$ and $s\in\{1, \ldots, |L_{\mathrm{L}}|\}$
such that $ l_s^{\mathrm{L}} \in  L^{{j^{\prime}}}_{\mathrm{L}}$, we have
\begin{equation*}
l_s^{\mathrm{L}}, \ldots, l_{|L_{\mathrm{L}}|}^{\mathrm{L}}
\in L^{{j^{\prime}}}_{\mathrm{L}}.
\end{equation*}
\end{claim}
\begin{proof}
The result follows directly from the definitions of $<_{\mathrm{R}}$ and $<_{\mathrm{L}}$.
\end{proof}

\begin{proof}[Proof of (i)]
The proof of Part (i) relies on the result of Part (iv)
that neutralized mobile points do not exist when
$q$ of positive type is genus-minimizing.
We therefore prove a modified result, which we call (i'),
and which satisfies the property that the combined statements
of (i') and (iv) imply Part (i).
\begin{itemize}
\item[(i')]
{\em If the interval
$\left\langle z_{a_j}^j, z_{a_j + 1}^j\right]$
has any non-neutralized mobile points, then
there exists $j_* \in \Z/d$ such that
$\left(z_{a_{j_*}}^{j_*}, z_{a_{j_*+1}}^{j_*}\right) = (x_*, y_*)$.}
\end{itemize}

\noindent{\em{Proof of (i').}}\;
Suppose that $v_q(z_{a_j}^j, z_{a_j + 1}^j)$ is constant in $j \in \Z/d$.
This implies that there is some constant $M \in \Z_{\geq 0}$ for which
\begin{equation}
         \left\vert L^{j^{\prime}}_{\mathrm{R}}\right\vert
       + \left\vert L^{j^{\prime}}_{\mathrm{L}}\right\vert
\equiv M\;\;\;\;\text{for all}\; {j^{\prime}}\in \Z/d.
\end{equation}
Now, since $z_0^{j+l_{|L_{\mathrm{R}}|}}$
is R-mobile in $\left\langle z_{a_j}^j, z_{a_j + 1}^j\right]$,
there must be some time $j=j_{\mathrm{R}} \in \Z/d$ at which
$z_0^{j+l_{|L_{\mathrm{R}}|}}$ is active 
in $\left\langle z_{a_j}^j, z_{a_j+1}^j \right]$.
Thus $l^{\mathrm{R}}_{|L_{\mathrm{R}}|} \in L_{\mathrm{R}}^{j_{\mathrm{R}}}$,
which, by Claim \ref{claim: structure of L_R^j and L_L^j}, means that
$l_1^{\mathrm{R}}, \ldots, 
l^{\mathrm{R}}_{|L_{\mathrm{R}}|} \in L_{\mathrm{R}}^{j_{\mathrm{R}}}$,
or in other words, $L_{\mathrm{R}}^{j_{\mathrm{R}}} = L_{\mathrm{R}}$,
implying
$M \geq \left\vert L^{j_{\mathrm{R}}}_{\mathrm{R}}\right\vert
   =    \left\vert L_{\mathrm{R}}\right\vert$.
On the other hand, since $z_0^{j+l_1}$ is R-mobile in
$\left\langle z_{a_j}^j, z_{a_j + 1}^j\right]$,
we know there must be some time $j=j_{\mathrm{R}}^{\emptyset} \in \Z/d$
at which $z_0^{j+l_1}$ is inactive in $\left\langle z_{a_j}^j, z_{a_j+1}^j \right]$.
Thus $l^{\mathrm{R}}_1 \notin L_{\mathrm{R}}^{j_{\mathrm{R}}^{\emptyset}}$,
which, by the contrapositive of Claim \ref{claim: structure of L_R^j and L_L^j}, 
means that $l_1^{\mathrm{R}}, \ldots, 
l^{\mathrm{R}}_{|L_{\mathrm{R}}|} \notin L_{\mathrm{R}}^{j_{\mathrm{R}}}$,
or in other words,
$L_{\mathrm{R}}^{j_{\mathrm{R}}^{\emptyset}} = \emptyset$,
implying
$M =    \left\vert L^{j_{\mathrm{R}}^{\emptyset}}_{\mathrm{L}}\right\vert
   \leq \left\vert L_{\mathrm{L}}\right\vert$.
By similar reasoning, there exists $j_{\mathrm{L}} \in \Z/d$ such that
$L_{\mathrm{L}}^{j_{\mathrm{L}}} = L_{\mathrm{L}}$, implying
$M \geq \left\vert L^{j_{\mathrm{L}}}_{\mathrm{L}}\right\vert
   =    \left\vert L_{\mathrm{L}}\right\vert$,
and there exists $j_{\mathrm{L}}^{\emptyset} \in \Z/d$ such that
$L_{\mathrm{L}}^{j_{\mathrm{L}}^{\emptyset}} = \emptyset$, implying
$M =    \left\vert L^{j_{\mathrm{L}}^{\emptyset}}_{\mathrm{R}}\right\vert
   \leq \left\vert L_{\mathrm{R}}\right\vert$.
Thus
\begin{equation}
     \left\vert L_{\mathrm{R}} \right\vert
\leq M
\leq \left\vert L_{\mathrm{R}} \right\vert
\;\;\;\mathrm{and}\;\;\;
     \left\vert L_{\mathrm{L}} \right\vert
\leq M
\leq \left\vert L_{\mathrm{L}} \right\vert,
\end{equation}
and so
$\left\vert L_{\mathrm{R}} \right\vert = \left\vert L_{\mathrm{L}} \right\vert = M$.

Thus, recalling that 
$\left\vert L^{j^{\prime}}_{\mathrm{R}}\right\vert
+ \left\vert L^{j^{\prime}}_{\mathrm{L}}\right\vert = M$ for all $j^{\prime}\in \Z/d$,
we deduce that for any $j^{\prime} \in \Z/d$, we have
\begin{align}
   l_1^{\mathrm{L}}, \ldots, 
      l_{\left\vert L_{\mathrm{R}}^{j^{\prime}} \right\vert}^{\mathrm{L}}
      \notin L^{j^{\prime}}_{\mathrm{L}},&
&l_1^{\mathrm{R}}, \ldots, 
      l_{\left\vert L_{\mathrm{R}}^{j^{\prime}} \right\vert}^{\mathrm{R}}
       \in L^{j^{\prime}}_{\mathrm{R}},&
              \\
  l_{\left\vert L_{\mathrm{R}}^{j^{\prime}} \right\vert + 1}^{\mathrm{L}},
      \ldots, l_M^{\mathrm{L}} \in L^{j^{\prime}}_{\mathrm{L}},&
&l_{\left\vert L_{\mathrm{R}}^{j^{\prime}} \right\vert + 1}^{\mathrm{R}},
      \ldots, l_M^{\mathrm{R}} \notin L^{j^{\prime}}_{\mathrm{R}}.&
\end{align}
This means that, at any give time $j=j^{\prime}\in\Z/d$
and for any $s \in \{1, \ldots, M\}$,
$z_{n_{j+l_s^{\mathrm{L}}}}^{j+l_s^{\mathrm{L}}}$ is active
in $\left\langle z_{a_j}^j, z_{a_j+1}^j \right]$
if and only if 
$z_0^{j+l_s^{\mathrm{R}}}$
is inactive, and so
Proposition \ref{prop: active iff inactive true iff neutralized} tells us that
$z_{n_{j+l_s^{\mathrm{L}}}}^{j+l_s^{\mathrm{L}}}$ and 
$z_0^{j+l_s^{\mathrm{R}}}$ form a neutralized pair in 
$\left\langle z_{a_j}^j, z_{a_j+1}^j \right]$.
In other words, all mobile points in $\left\langle z_{a_j}^j, z_{a_j+1}^j \right]$
must be neutralized.

We have shown that if $v_q(z_{a_j}^j, z_{a_j + 1}^j)$ is constant
in $j \in \Z/d$, then all mobile points in
$\left\langle z_{a_j}^j, z_{a_j + 1}^j\right]$, if any,
are neutralized.  Thus, if
$\left\langle z_{a_j}^j, z_{a_j + 1}^j\right]$ has any
{\em{non}}-neutralized mobile points, then 
$v_q(z_{a_j}^j, z_{a_j + 1}^j)$ is {\em not} constant
in $j \in \Z/d$, and so, by
Proposition \ref{prop: unique v_q = alpha(k - k^2), and the rest are v_q = alpha (k)},
there must exist a unique
$j_* \in \Z/d$ for which
$v_q(z_{a_{j_*}}^{j_*}, z_{a_{j_*} + 1}^{j_*}) = \alpha(k-k^2)$.
\end{proof}

\begin{proof}[Proof of (i$\psi$)]
The proof of Part (i') works for Part (i$\psi$), with only a few minor changes.
One must replace the word ``mobile'' with the word ``pseudomobile''
and replace $z_{a_j}^j$ and $z_{a_{j+1}}^j$ with
$z_{n_{j-1}}^{j-1}$ and $z_0^j$, respectively---both in the actual proof,
and in the discussion of $L_{\mathrm{R}}$, $L_{\mathrm{L}}$, {\em et cetera},
preceding the proof of Part (i).  One must also replace the reference to
Proposition \ref{prop: active iff inactive true iff neutralized}
with a reference to 
Proposition \ref{prop: active iff inactive means neutralized for pseudomobile}.
\end{proof}

\begin{proof}[Proof of (ii)]
Here again, the result of Part (ii) relies on the
nonexistence of neutralized mobile points for genus-minimizing
$q$ of positive type, proven in Part (iv).
We therefore prove a modified claim, Part (ii'), which
does not rely on Part (iv), but which when combined with Part (iv)
implies (ii).
\begin{itemize}
\item[(ii')]
{\em If $v_q(z_{n_{j-1}}^{j-1}, z_0^j)$ is constant in $j\in\Z/d$, then
there are precisely two non-neutralized mobile points in ${\bf{z}}^j$,
namely, $z_0^{j+l}$ non-neutralized R-mobile in 
$\left\langle z_{i_*}^j, z_{i_*+1}^j \right]$ and
$z_0^{j+l}$ non-neutralized L-mobile in
$\left\langle z_{n_j - (i_*+1)}^j, z_{n_j - i_*}^j \right]$, for some $l\in\Z/d$,
where $i_*$ is the unique element of  
$\left\{0, \ldots, \left\lfloor\frac{k}{d}\right\rfloor -2 \right\}$ satisfying
$(x_*,y_*) = \left(z_{i_*}^{j_*},z_{i_*+1}^{j_*}\right) 
= \left(z_{n_{j_*} - (i_*+1)}^{j_*}, z_{n_{j_*} - i_*}^{j_*}\right)$.}
\end{itemize}

\noindent{\em{Proof of (ii').}}\;
Since $v_q\!\left(z_{n_{j-1}}^{j-1}, z_0^j\right)$ is constant in $j \in \Z/d$,
Proposition \ref{prop: unique v_q = alpha(k - k^2), and the rest are v_q = alpha (k)} tells us that
there exist unique $j_* \in \Z/d$ and $i_* \in \{0, \ldots, n_{j_*}-1\}$ for which
$v_q\!\left(z_{i_*}^{j_*}, z_{i_*+1}^{j_*}\right) = {\alpha}(k-k^2)$.
Thus, for the function $a_j$ of $j \in \Z/d$, we choose either
$a_j \equiv i_*$ for all $j \in \Z/d$,
or $a_j = n_j - (n_{j_*} - i_*)$ for each $j \in \Z/d$,
since these are the only two valid choices for $a_j$ for which
$\left(z_{i_*}^{j_*}, z_{i_*+1}^{j_*}\right) \in 
\left\{\left( z_{a_j}^j, z_{a_j+1}^j\right)\right\}_{j \in \Z/d}$.
We then have
\begin{equation}
  v_q\!\left(z_{a_j}^j, z_{a_j+1}^j\right)
= \begin{cases}
    {\alpha}(k-k^2)
       & j =j_*
             \\
    {\alpha}k
       & j \neq j_*
  \end{cases},
\end{equation}
and so there exists $M \in \Z$ with $M \geq 1$ such that
\begin{equation}
  \left\vert L_{\mathrm{R}}^j \right\vert
+ \left\vert L_{\mathrm{L}}^j \right\vert
= \begin{cases}
    M + \alpha
       & j = j_*
             \\
    M
       & j \neq j_*
  \end{cases}.
\end{equation}
We next argue that $\left\langle z^j_{a_j}, z^j_{a_j + 1} \right]$ has at most
one non-neutralized R-mobile point and at most one non-neutralized L-mobile point,
treating the cases of $\alpha = +1$ and $\alpha = -1$ separately.
        \\

{\noindent\bf{Case \!$\mathbf{\alpha = +1}$.}\;}
By the same reasoning as used in Part (i'),
one can show there exist
$j_{\mathrm{R}}\in\Z/d$ such that $L_{\mathrm{R}}^{j_{\mathrm{R}}} = L_{\mathrm{R}}$,
$j_{\mathrm{R}}^{\emptyset}\in\Z/d$ such that 
$L_{\mathrm{R}}^{j_{\mathrm{R}}^{\emptyset}} = \emptyset$,
$j_{\mathrm{L}}\in\Z/d$ such that $L_{\mathrm{L}}^{j_{\mathrm{L}}} = L_{\mathrm{L}}$,
and $j_{\mathrm{L}}^{\emptyset}\in\Z/d$ such that 
  $L_{\mathrm{L}}^{j_{\mathrm{L}}^{\emptyset}} = \emptyset$,
from which we conclude, respectively, that
$M+1 \geq \left\vert L_{\mathrm{R}} \right\vert$,
$M \leq \left\vert L_{\mathrm{L}} \right\vert$,
$M+1 \geq \left\vert L_{\mathrm{L}} \right\vert$, and
$M \leq \left\vert L_{\mathrm{R}} \right\vert$.
Now, if 
$\left\vert L_{\mathrm{R}} \right\vert = \left\vert L_{\mathrm{L}} \right\vert = M+1$,
then $j_{\mathrm{R}} = j_{\mathrm{L}} = j_*$,
but this implies that
$M+1 = \left\vert L_{\mathrm{R}}^{j_*} \right\vert
+ \left\vert L_{\mathrm{L}}^{j_*} \right\vert = 2(M+1)$, a contradiction.
Thus we are left with three possibilities:
\begin{equation}
\left\vert L_{\mathrm{R}} \right\vert = \left\vert L_{\mathrm{L}} \right\vert = M;
\;\;\;\;\;\;
\left\vert L_{\mathrm{R}} \right\vert = M+1,\; \left\vert L_{\mathrm{L}} \right\vert = M;
\;\;\;\;\mathrm{or}\;\;\;\;
\left\vert L_{\mathrm{R}} \right\vert = M,\; \left\vert L_{\mathrm{L}} \right\vert = M+1.
\end{equation}

First, consider the case in which
$\left\vert L_{\mathrm{R}} \right\vert 
= \left\vert L_{\mathrm{L}} \right\vert = M$.
Since
$ \left\vert L_{\mathrm{R}}^{j_*} \right\vert
+ \left\vert L_{\mathrm{L}}^{j_*} \right\vert = M+1 > M$,
there must exist some $s_0 \in \{1, \ldots, M\}$ for which
$l_{s_0}^{\mathrm{R}} \in L^{j_*}_{\mathrm{R}}$ and
$l_{s_0}^{\mathrm{L}} \in L^{j_*}_{\mathrm{L}}$.
Claim \ref{claim: structure of L_R^j and L_L^j} then tells us that
$l_1^{\mathrm{R}}, \ldots, l_{s_0}^{\mathrm{R}} \in L^{j_*}_{\mathrm{R}}$ and
$l_{s_0}^{\mathrm{L}}, \ldots, l_M^{\mathrm{L}} \in L^{j_*}_{\mathrm{L}}$.  But
\begin{equation}
    \left|\left\{ l_1^{\mathrm{R}}, \ldots, l_{s_0}^{\mathrm{R}}\right\}\right|
 + \left|\left\{  l_{s_0}^{\mathrm{L}}, \ldots, l_M^{\mathrm{L}}\right\}\right|
= M+1,
\end{equation}
and so we must have 
$L^{j_*}_{\mathrm{R}} = \left\{l_1^{\mathrm{R}}, \ldots, l_{s_0}^{\mathrm{R}}\right\}$ and
$L^{j_*}_{\mathrm{L}} = \left\{l_{s_0}^{\mathrm{L}}, \ldots, l_M^{\mathrm{L}}\right\}$, implying
$\left|L^{j_*}_{\mathrm{R}}\right| = s_0$ and
$\left|L^{j_*}_{\mathrm{R}}\right| = M+1-s_0$.
Note that this argument can also be used to show that if there is any
$s \in \{1, \ldots, M\}$ and any $j^{\prime} \in \Z/d$ for which
$l_{s}^{\mathrm{R}} \in L^{j^{\prime}}_{\mathrm{R}}$ and
$l_{s}^{\mathrm{L}} \in L^{j^{\prime}}_{\mathrm{L}}$,
then $\left|L^{j^{\prime}}_{\mathrm{R}}\right| = s$ and
$\left|L^{j^{\prime}}_{\mathrm{R}}\right| = M+1-s$,
so that $j^{\prime} = j_*$.  
There are also, however, no $s\in \{1, \ldots, M\}$ and $j^{\prime} \in \Z/d$ for which 
$l_{s}^{\mathrm{R}} \notin L^{j^{\prime}}_{\mathrm{R}}$ and
$l_{s}^{\mathrm{L}} \notin L^{j^{\prime}}_{\mathrm{L}}$, 
since this would imply
$\left\vert L_{\mathrm{R}}^{j^{\prime}} \right\vert
+ \left\vert L_{\mathrm{L}}^{j^{\prime}} \right\vert
< M$ when $j^{\prime} \neq j_*$, and since this would contradict the fact that 
$L^{j_*}_{\mathrm{R}} = \left\{l_1^{\mathrm{R}}, \ldots, l_{s_0}^{\mathrm{R}}\right\}$ and
$L^{j_*}_{\mathrm{L}} = \left\{l_{s_0}^{\mathrm{L}}, \ldots, l_M^{\mathrm{L}}\right\}$
when $j^{\prime} = j_*$.
Thus, for any $s \in \{1, \ldots, M\} \setminus \{s_0\}$,
and at any time $j=j^{\prime}\in\Z/d$, we know that
$z_{n_{j+l_s^{\mathrm{L}}}}^{j+l_s^{\mathrm{L}}}$
is active in $\left\langle z_{a_j}^j, z_{a_j+1}^j \right]$
if and only if
$z_0^{j+l_s^{\mathrm{R}}}$ is inactive.
Proposition \ref{prop: active iff inactive true iff neutralized}
then tells us that 
$z_{n_{j+l_s^{\mathrm{L}}}}^{j+l_s^{\mathrm{L}}}$ and 
$z_0^{j+l_s^{\mathrm{R}}}$ form a neutralized pair in
$\left\langle z_{a_j}^j, z_{a_j+1}^j \right]$, and so
all mobile points in 
$\left\langle z_{a_j}^j, z_{a_j + 1}^j \right]$
are neutralized except
except $z_0^{j+l_{s_0}^{\mathrm{R}}}$ and
$z_{n_{j+l_{s_0}^{\mathrm{L}}}}^{j+l_{s_0}^{\mathrm{L}}}$.
Thus, we have 1 non-neutralized R-mobile point
and 1 non-neutralized L-mobile point in
$\left\langle z_{a_j}^j, z_{a_j + 1}^j \right]$
when $\left\vert L_{\mathrm{R}} \right\vert 
= \left\vert L_{\mathrm{L}} \right\vert = M$.

Next, consider the case in which
$\left\vert L_{\mathrm{R}} \right\vert = M+1$ and 
$\left\vert L_{\mathrm{L}} \right\vert = M$.
Since $z_0^{j + l^{\mathrm{R}}_{M+1}}$ is R-mobile in
$\left\langle z_{a_j}^j, z_{a_j+1}^j\right]$, there must be some time 
when it is active, so 
there must exist some $j_{L_{\mathrm{R}}} \in \Z/d$ for which
$l_{M+1}^{\mathrm{R}} \in L_{\mathrm{R}}^{j_{L_{\mathrm{R}}}}\!$.
But this implies $L_{\mathrm{R}}^{j_{L_{\mathrm{R}}}} = L_{\mathrm{R}}$,
which means 
$\left\vert L_{\mathrm{R}}^{j_{L_{\mathrm{R}}}} \right\vert 
+ \left\vert L_{\mathrm{L}}^{j_{L_{\mathrm{R}}}} \right\vert \geq M+1$,
and so it must be the case that $j_{L_{\mathrm{R}}} = j_*$.
Thus $L_{\mathrm{R}}^{j_*} = L_{\mathrm{R}} = 
\left\{l_1^{\mathrm{R}}, \ldots, l_{M+1}^{\mathrm{R}}\right\}$ and 
$L_{\mathrm{L}}^{j_*} = \emptyset$.
Now, suppose there exist $s\in \{1, \ldots, M\}$ and $j^{\prime} \in \Z/d$ for which 
$l_{s}^{\mathrm{R}} \in L^{j^{\prime}}_{\mathrm{R}}$ and
$l_{s}^{\mathrm{L}} \in L^{j^{\prime}}_{\mathrm{L}}$.
Then, just as in the $\left\vert L_{\mathrm{R}} \right\vert 
= \left\vert L_{\mathrm{L}} \right\vert = M$ case, we have
$l_1^{\mathrm{R}}, \ldots, l_{s}^{\mathrm{R}} \in L^{j^{\prime}}_{\mathrm{R}}$ and
$l_{s}^{\mathrm{L}}, \ldots, l_M^{\mathrm{L}} \in L^{j^{\prime}}_{\mathrm{L}}$, and then since
$ \left|\left\{ l_1^{\mathrm{R}}, \ldots, l_{s}^{\mathrm{R}}\right\}\right|
 + \left|\left\{  l_{s}^{\mathrm{L}}, \ldots, l_M^{\mathrm{L}}\right\}\right|
= M+1$,
we conclude that $j^{\prime} = j_*$.
But this contradicts the fact that $L_{\mathrm{L}}^{j_*} = \emptyset$.
Thus there are no $s\in \{1, \ldots, M\}$ and $j^{\prime} \in \Z/d$ for which 
$l_{s}^{\mathrm{R}} \in L^{j^{\prime}}_{\mathrm{R}}$ and
$l_{s}^{\mathrm{L}} \in L^{j^{\prime}}_{\mathrm{L}}$.
There are also, however, no $s\in \{1, \ldots, M\}$ and $j^{\prime} \in \Z/d$ for which 
$l_{s}^{\mathrm{R}} \notin L^{j^{\prime}}_{\mathrm{R}}$ and
$l_{s}^{\mathrm{L}} \notin L^{j^{\prime}}_{\mathrm{L}}$, 
since this would imply
$\left\vert L_{\mathrm{R}}^{j^{\prime}} \right\vert
+ \left\vert L_{\mathrm{L}}^{j^{\prime}} \right\vert
< M$ when $j^{\prime} \neq j_*$, and since this would contradict
the fact that $L_{\mathrm{R}}^{j_*} = L_{\mathrm{R}} = 
\left\{l_1^{\mathrm{R}}, \ldots, l_{M+1}^{\mathrm{R}}\right\}$
when $j^{\prime} = j_*$.
Thus, for any $s \in \{1, \ldots, M\}$,
and at any time $j=j^{\prime}\in\Z/d$, we know that
$z_{n_{j+l_s^{\mathrm{L}}}}^{j+l_s^{\mathrm{L}}}$
is active in $\left\langle z_{a_j}^j, z_{a_j+1}^j \right]$
if and only if
$z_0^{j+l_s^{\mathrm{R}}}$ is inactive.
Proposition \ref{prop: active iff inactive true iff neutralized}
then tells us that 
$z_{n_{j+l_s^{\mathrm{L}}}}^{j+l_s^{\mathrm{L}}}$ and 
$z_0^{j+l_s^{\mathrm{R}}}$ form a neutralized pair in
$\left\langle z_{a_j}^j, z_{a_j+1}^j \right]$, and so
all mobile points in 
$\left\langle z_{a_j}^j, z_{a_j + 1}^j \right]$
are neutralized except $z_0^{j+l_{M+1}^{\mathrm{R}}}$.
Thus, we have 1 non-neutralized R-mobile point
and 0 non-neutralized L-mobile points in
$\left\langle z_{a_j}^j, z_{a_j + 1}^j \right]$
when $\left\vert L_{\mathrm{R}} \right\vert = M+1$ and 
$\left\vert L_{\mathrm{L}} \right\vert = M$.

Lastly, if
$\left\vert L_{\mathrm{R}} \right\vert = M$ and 
$\left\vert L_{\mathrm{L}} \right\vert = M+1$,
then an argument similar to that used in the preceding paragraph
shows that all mobile points in
$\left\langle z_{a_j}^j, z_{a_j + 1}^j \right]$
are neutralized except for one L-mobile point.
Thus, in all of the above three cases,
there is at most one non-neutralized R-mobile point
and at most one non-neutralized L-mobile point
in $\left\langle z_{a_j}^j, z_{a_j + 1}^j \right]$.
           \\

{\noindent\bf{Case \!$\mathbf{\alpha = -1}$.}\;}
Following a strategy similar to that used in the case of $\alpha = +1$,
we observe that, since $z_0^{j+l_{|L_{\mathrm{R}}|}}$
is R-mobile in $\left\langle z_{a_j}^j, z_{a_j + 1}^j\right]$,
there must be some time $j=j_{\mathrm{R}} \in \Z/d$ at which
$z_0^{j+l_{|L_{\mathrm{R}}|}}$ is active 
in $\left\langle z_{a_j}^j, z_{a_j+1}^j \right]$.
Thus $l^{\mathrm{R}}_{|L_{\mathrm{R}}|} \in L_{\mathrm{R}}^{j_{\mathrm{R}}}$,
which, by Claim \ref{claim: structure of L_R^j and L_L^j}, means that
$l_1^{\mathrm{R}}, \ldots, 
l^{\mathrm{R}}_{|L_{\mathrm{R}}|} \in L_{\mathrm{R}}^{j_{\mathrm{R}}}$,
or in other words, $L_{\mathrm{R}}^{j_{\mathrm{R}}} = L_{\mathrm{R}}$, implying
$M \geq \left\vert L^{j_{\mathrm{R}}}_{\mathrm{R}}\right\vert
   =    \left\vert L_{\mathrm{R}}\right\vert$.
On the other hand, since $z_0^{j+l_1}$ is R-mobile in
$\left\langle z_{a_j}^j, z_{a_j + 1}^j\right]$,
we know there must be some time $j=j_{\mathrm{R}}^{\emptyset} \in \Z/d$
at which $z_0^{j+l_1}$ is inactive in $\left\langle z_{a_j}^j, z_{a_j+1}^j \right]$.
Thus $l^{\mathrm{R}}_1 \notin L_{\mathrm{R}}^{j_{\mathrm{R}}^{\emptyset}}$,
which, by the contrapositive of Claim \ref{claim: structure of L_R^j and L_L^j}, 
means that $l_1^{\mathrm{R}}, \ldots, 
l^{\mathrm{R}}_{|L_{\mathrm{R}}|} \notin L_{\mathrm{R}}^{j_{\mathrm{R}}}$,
or in other words,
$L_{\mathrm{R}}^{j_{\mathrm{R}}^{\emptyset}} = \emptyset$,
implying
$M-1 =    \left\vert L^{j_{\mathrm{R}}^{\emptyset}}_{\mathrm{L}}\right\vert
   \leq \left\vert L_{\mathrm{L}}\right\vert$.
By similar reasoning, there exist
$j_{\mathrm{L}}\in\Z/d$ such that $L_{\mathrm{L}}^{j_{\mathrm{L}}} = L_{\mathrm{L}}$,
and $j_{\mathrm{L}}^{\emptyset}\in\Z/d$ such that 
  $L_{\mathrm{L}}^{j_{\mathrm{L}}^{\emptyset}} = \emptyset$,
from which we conclude, respectively, that
$M \geq \left\vert L_{\mathrm{L}} \right\vert$, and
$M-1 \leq \left\vert L_{\mathrm{R}} \right\vert$.
Now, if
$\left\vert L_{\mathrm{R}} \right\vert = \left\vert L_{\mathrm{L}} \right\vert = M-1$,
then $j_{\mathrm{R}} = j_{\mathrm{L}} = j_*$.  But this implies
$M-1 = \left\vert L_{\mathrm{R}}^{j_*} \right\vert
+ \left\vert L_{\mathrm{L}}^{j_*} \right\vert = 2(M-1)$, which means that
$\left\vert L_{\mathrm{R}} \right\vert = \left\vert L_{\mathrm{L}} \right\vert = M-1 = 0$,
contradicting our assumption that at least one of
$L_{\mathrm{R}}$ and $L_{\mathrm{L}}$ is nonempty.
Thus, we are left with three possibilities:
\begin{equation}
\left\vert L_{\mathrm{R}} \right\vert = \left\vert L_{\mathrm{L}} \right\vert = M;
\;\;\;\;\;\;
\left\vert L_{\mathrm{R}} \right\vert = M - 1,\; \left\vert L_{\mathrm{L}} \right\vert = M;
\;\;\;\;\mathrm{or}\;\;\;\;
\left\vert L_{\mathrm{R}} \right\vert = M,\; \left\vert L_{\mathrm{L}} \right\vert = M - 1.
\end{equation}

First, consider the case in which
$\left\vert L_{\mathrm{R}} \right\vert 
= \left\vert L_{\mathrm{L}} \right\vert = M$.
Since
$ \left\vert L_{\mathrm{R}}^{j_*} \right\vert
+ \left\vert L_{\mathrm{L}}^{j_*} \right\vert = M-1<M$,
there must exist some $s_0 \in \{1, \ldots, M\}$ for which
$l_{s_0}^{\mathrm{R}} \notin L^{j_*}_{\mathrm{R}}$ and
$l_{s_0}^{\mathrm{L}} \notin L^{j_*}_{\mathrm{L}}$.
The contrapositive of Claim \ref{claim: structure of L_R^j and L_L^j}
then tells us that
$l_{s_0}^{\mathrm{R}}, \ldots, l_M^{\mathrm{R}} \notin L^{j_*}_{\mathrm{R}}$ and
$l_1^{\mathrm{L}}, \ldots, l_{s_0}^{\mathrm{L}} \notin L^{j_*}_{\mathrm{L}}$.  But
\begin{equation}
    \left|\left\{ l_1^{\mathrm{R}}, \ldots, l_{s_0-1}^{\mathrm{R}}\right\}\right|
 + \left|\left\{  l_{s_0+1}^{\mathrm{L}}, \ldots, l_M^{\mathrm{L}}\right\}\right|
= M-1,
\end{equation}
and so we must have 
$L^{j_*}_{\mathrm{R}} = \left\{l_1^{\mathrm{R}}, \ldots, l_{s_0-1}^{\mathrm{R}}\right\}$ and
$L^{j_*}_{\mathrm{L}} = \left\{l_{s_0+1}^{\mathrm{L}}, \ldots, l_M^{\mathrm{L}}\right\}$, implying
$\left|L^{j_*}_{\mathrm{R}}\right| = s_0-1$ and
$\left|L^{j_*}_{\mathrm{R}}\right| = M-s_0$.
Note that this argument can also be used to show that if there is any
$s \in \{1, \ldots, M\}$ and any $j^{\prime} \in \Z/d$ for which
$l_{s}^{\mathrm{R}} \notin L^{j^{\prime}}_{\mathrm{R}}$ and
$l_{s}^{\mathrm{L}} \notin L^{j^{\prime}}_{\mathrm{L}}$,
then $\left|L^{j^{\prime}}_{\mathrm{R}}\right| = s-1$ and
$\left|L^{j^{\prime}}_{\mathrm{R}}\right| = M-s$,
so that $j^{\prime} = j_*$.
There are also, however, no $s\in \{1, \ldots, M\}$ and $j^{\prime} \in \Z/d$ for which 
$l_{s}^{\mathrm{R}} \in L^{j^{\prime}}_{\mathrm{R}}$ and
$l_{s}^{\mathrm{L}} \in L^{j^{\prime}}_{\mathrm{L}}$, 
since this would imply
$\left\vert L_{\mathrm{R}}^{j^{\prime}} \right\vert
+ \left\vert L_{\mathrm{L}}^{j^{\prime}} \right\vert
> M$ when $j^{\prime} \neq j_*$, 
and since this would contradict the fact that
$L^{j_*}_{\mathrm{R}} = \left\{l_1^{\mathrm{R}}, \ldots, l_{s_0-1}^{\mathrm{R}}\right\}$ and
$L^{j_*}_{\mathrm{L}} = \left\{l_{s_0+1}^{\mathrm{L}}, \ldots, l_M^{\mathrm{L}}\right\}$
when $j^{\prime} = j_*$.
Thus, for any $s \in \{1, \ldots, M\} \setminus \{s_0\}$,
and at any time $j=j^{\prime}\in\Z/d$, we know that
$z_{n_{j+l_s^{\mathrm{L}}}}^{j+l_s^{\mathrm{L}}}$
is active in $\left\langle z_{a_j}^j, z_{a_j+1}^j \right]$
if and only if
$z_0^{j+l_s^{\mathrm{R}}}$ is inactive.
Proposition \ref{prop: active iff inactive true iff neutralized}
then tells us that 
$z_{n_{j+l_s^{\mathrm{L}}}}^{j+l_s^{\mathrm{L}}}$ and 
$z_0^{j+l_s^{\mathrm{R}}}$ form a neutralized pair in
$\left\langle z_{a_j}^j, z_{a_j+1}^j \right]$, and so
all mobile points in 
$\left\langle z_{a_j}^j, z_{a_j + 1}^j \right]$
are neutralized except
except $z_0^{j+l_{s_0}^{\mathrm{R}}}$ and
$z_{n_{j+l_{s_0}^{\mathrm{L}}}}^{j+l_{s_0}^{\mathrm{L}}}$.
Thus, we have 1 non-neutralized R-mobile point
and 1 non-neutralized L-mobile point in
$\left\langle z_{a_j}^j, z_{a_j + 1}^j \right]$
when $\left\vert L_{\mathrm{R}} \right\vert 
= \left\vert L_{\mathrm{L}} \right\vert = M$.

Next, consider the case in which
$\left\vert L_{\mathrm{R}} \right\vert = M-1$ and 
$\left\vert L_{\mathrm{L}} \right\vert = M$.
Since 
$z_{n_{j + l^{\mathrm{L}}_M}}^{j + l^{\mathrm{L}}_M}$
is L-mobile in
$\left\langle z_{a_j}^j, z_{a_j+1}^j\right]$,
there must be some time when it is inactive, so 
there must exist some $j_{L_{\mathrm{L}}}^{\emptyset} \in \Z/d$ for which
$l_M^{\mathrm{L}} \notin L_{\mathrm{L}}^{j_{L_{\mathrm{L}}}^{\emptyset}}\!$.
The contrapositive of Claim \ref{claim: structure of L_R^j and L_L^j}
then tells us that $L_{\mathrm{L}}^{j_{L_{\mathrm{L}}}^{\emptyset}} = \emptyset$,
which means that
$\left\vert L_{\mathrm{R}}^{j_{L_{\mathrm{L}}}^{\emptyset}} \right\vert 
+ \left\vert L_{\mathrm{L}}^{j_{L_{\mathrm{L}}}^{\emptyset}} \right\vert \leq M-1$,
and so it must be the case that $j_{L_{\mathrm{L}}}^{\emptyset} = j_*$.
Thus $L_{\mathrm{L}}^{j_*} = \emptyset$,
and $l_M^{\mathrm{L}} \in L_{\mathrm{L}}^{j^{\prime}}$ whenever $j^{\prime}\neq j_*$.
Now, suppose there exist $s\in \{1, \ldots, M-1\}$ and $j^{\prime} \in \Z/d$ for which 
$l_{s}^{\mathrm{R}} \notin L^{j^{\prime}}_{\mathrm{R}}$ and
$l_{s}^{\mathrm{L}} \notin L^{j^{\prime}}_{\mathrm{L}}$.
The contrapositive of Claim \ref{claim: structure of L_R^j and L_L^j}
then tells us that 
$l_{s}^{\mathrm{R}}, \ldots, l_{M-1}^{\mathrm{R}} \notin L^{j^{\prime}}_{\mathrm{R}}$ and
$l_1^{\mathrm{L}}, \ldots, l_{s}^{\mathrm{L}} \notin L^{j^{\prime}}_{\mathrm{L}}$,
and then since
$ \left|\left\{ l_1^{\mathrm{R}}, \ldots, l_{s-1}^{\mathrm{R}}\right\}\right|
 + \left|\left\{  l_{s+1}^{\mathrm{L}}, \ldots, l_M^{\mathrm{L}}\right\}\right|
= M-1$,
we conclude that $j^{\prime} = j_*$.
But this contradicts the fact that $L_{\mathrm{L}}^{j_*} = \emptyset$.
Thus there are no $s\in \{1, \ldots, M-1\}$ and $j^{\prime} \in \Z/d$ for which 
$l_{s}^{\mathrm{R}} \notin L^{j^{\prime}}_{\mathrm{R}}$ and
$l_{s}^{\mathrm{L}} \notin L^{j^{\prime}}_{\mathrm{L}}$.
There are also, however, no $s\in \{1, \ldots, M-1\}$ and $j^{\prime} \in \Z/d$ for which 
$l_{s}^{\mathrm{R}} \in L^{j^{\prime}}_{\mathrm{R}}$ and
$l_{s}^{\mathrm{L}} \in L^{j^{\prime}}_{\mathrm{L}}$, 
since this would imply
$\left\vert L_{\mathrm{R}}^{j^{\prime}} \right\vert
+ \left\vert L_{\mathrm{L}}^{j^{\prime}} \right\vert
> M$ when $j^{\prime} \neq j_*$,
and since this would contradict
the fact that $L_{\mathrm{L}}^{j_*} = \emptyset$
when $j^{\prime} = j_*$.
Thus, for any $s \in \{1, \ldots, M-1\}$,
and at any time $j=j^{\prime}\in\Z/d$, we know that
$z_{n_{j+l_s^{\mathrm{L}}}}^{j+l_s^{\mathrm{L}}}$
is active in $\left\langle z_{a_j}^j, z_{a_j+1}^j \right]$
if and only if
$z_0^{j+l_s^{\mathrm{R}}}$ is inactive.
Proposition \ref{prop: active iff inactive true iff neutralized}
then tells us that 
$z_{n_{j+l_s^{\mathrm{L}}}}^{j+l_s^{\mathrm{L}}}$ and 
$z_0^{j+l_s^{\mathrm{R}}}$ form a neutralized pair in
$\left\langle z_{a_j}^j, z_{a_j+1}^j \right]$, and so
all mobile points in 
$\left\langle z_{a_j}^j, z_{a_j + 1}^j \right]$
are neutralized except $z_{n_{j+l_M^{\mathrm{L}}}}^{j+l_M^{\mathrm{L}}}$.
Thus, we have 0 non-neutralized R-mobile points
and 1 non-neutralized L-mobile point in
$\left\langle z_{a_j}^j, z_{a_j + 1}^j \right]$
when $\left\vert L_{\mathrm{R}} \right\vert = M-1$ and 
$\left\vert L_{\mathrm{L}} \right\vert = M$.

Lastly, if
$\left\vert L_{\mathrm{R}} \right\vert = M$ and 
$\left\vert L_{\mathrm{L}} \right\vert = M-1$,
then an argument similar to that used in the preceding paragraph
shows that all mobile points in
$\left\langle z_{a_j}^j, z_{a_j + 1}^j \right]$
are neutralized except for one R-mobile point.
In conclusion, whether $\alpha = +1$ or $\alpha = -1$, there is always
at most one non-neutralized R-mobile point and at most one
non-neutralized L-mobile point in $\left\langle z_{a_j}^j, z_{a_j + 1}^j \right]$.
      \\

We next consider the existence of non-neutralized mobile points 
in all of ${\bf{z}}^j$, as opposed to in just an interval in ${\bf{z}}^j$.
Since $v_q(z_{n_{j-1}}^{j-1}, z_0^j)$ is constant in $j\in\Z/d$, we know
by Part (i') that  ${\bf{z}}^j$ must have at least one non-neutralized
mobile point.  However, the existence of a non-neutralized R-mobile point
implies the existence of its mirror non-neutralized L-mobile point,
and {\em vice versa}, and so, in fact, ${\bf{z}}^j$ must have at least one 
non-neutralized R-mobile point and at least one non-neutralized L-mobile point.

Write $z_0^{j+l}$ for an arbitrary non-neutralized R-mobile point in ${\bf{z}}^j$,
and suppose that $z_0^{j+l}$ is {\em not} non-neutralized R-mobile rel $z_0^j$.
Then $z_0^{j+l}$ is non-neutralized R-mobile in 
$\left\langle z_{n_j - (i_*+1)}^j, z_{n_j - i_*}^j \right]$
for some $i_* \in \left\{0, \ldots, \left\lfloor\frac{k}{d}\right\rfloor -2 \right\}$,
and $l \neq 1$.  Consider first the case in which $l \neq 0$.
If $i_* > 0$, then
$\left\lfloor\frac{k}{d}\right\rfloor - (i_*+1),
\left\lfloor\frac{k}{d}\right\rfloor - (i_*+2)
\in \left\{0, \ldots, \left\lfloor\frac{k}{d}\right\rfloor -2 \right\}$, and so
$z_0^{j+l}$ is also R-mobile in either
$\left\langle z_{\left\lfloor\!\frac{k}{d}\!\right\rfloor - (i_*+1)}^j,
z_{\left\lfloor\!\frac{k}{d}\!\right\rfloor - i_*}^j\right]$ or
$\left\langle z_{\left\lfloor\!\frac{k}{d}\!\right\rfloor - (i_*+2)}^j,
z_{\left\lfloor\!\frac{k}{d}\!\right\rfloor - (i_*+1)}^j\right]$,
hence is R-mobile rel $z_0^j$, a contradiction.
Thus we must have $i_*=0$,
with $z_0^{j+l}$ R-mobile in 
$\left\langle z_{n_j - 1}^j, z_{n_j}^j \right]$.
This, in turn, implies the existence of 
a mirror non-neutralized L-mobile point, namely,
$z_{n_{j-l}}^{j-l}$ in $\left\langle z_0^j, z_1^j \right]$.
From Part (i'), we know that the existence of these non-neutralized 
mobile points implies that the sets
$\left\{\left(z_{n_j - 1}^j, z_{n_j}^j\right)\right\}_{j\in\Z/d}$
and
$\left\{\left(z_0^j, z_1^j\right)\right\}_{j\in\Z/d}$
each contain the pair $(x_*, y_*) \in Q_q \times Q_q$ for which
$v_q\!\left(x_*, y_*\right) = \alpha(k-k^2)$.
In particular, these two sets must intersect, 
and so there exists some $j_*\in \Z/d$ for which
$n_{j_*} - 1 = 0$.  Thus $n_{j_*}=1$ and
$\left\lfloor\frac{k}{d}\right\rfloor = 2$.
We can therefore express $\Z/d$ as the disjoint union of
$J_0$, $J_2$, and $J_3$, where
\begin{align}
       J_0
&:= \left\{j^{\prime} \in \Z/d \left|\;  
       n_{j^{\prime}} = 1;\;
       z_0^{j+l}\;\text{is inactive in}\;\!
       \left\langle z_{n_j-1}^j, z_{n_j}^j\right]
       \!\;\text{when}\;j=j^{\prime}
       \right.\!\right\}\!,
           \\
       J_2
&:= \left\{j^{\prime} \in \Z/d \left|\;  
       n_{j^{\prime}} = 2;\;
       z_0^{j+l}\;\text{is inactive in}\;\!
       \left\langle z_{n_j-1}^j, z_{n_j}^j\right]
       \!\;\text{when}\;j=j^{\prime}
       \right.\!\right\}\!,
           \\
       J_3
&:= \left\{j^{\prime} \in \Z/d \left|\;  
       n_{j^{\prime}} = 2;\;
       z_0^{j+l}\;\text{is active in}\;\!
       \left\langle z_{n_j-1}^j, z_{n_j}^j\right]
       \!\;\text{when}\;j=j^{\prime}
       \right.\!\right\}\!.
\end{align}
Note that we omitted the only other possibility,
\begin{equation}
       J_1
:= \left\{j^{\prime} \in \Z/d \left|\;  
       n_{j^{\prime}} = 1;\;
       z_0^{j+l}\;\text{is active in}\;\!
       \left\langle z_{n_j-1}^j, z_{n_j}^j\right]
       \!\;\text{when}\;j=j^{\prime}
       \right.\!\right\}\!,
\end{equation}
because the nonemptiness of $J_1$ would imply that 
$z_0^{j+l}$ was R-mobile in $\left\langle z_0^j, z_1^j\right]$.
We know, however, that $J_0$, $J_2$, and $J_3$ are nonempty.
$J_0$ is nonempty because $n_{j^{\prime}}=1$ only if $j^{\prime} \in J_0$;
$J_2$ is nonempty because otherwise $z_0^{j+l}-z_0^j$ would be constant
in $j\in\Z/d$, contradicting our assumption that $l \neq 0$; and
$J_3$ is nonempty because otherwise $z_0^{j+l}$ would never be active
in $\left\langle z_{n_j-1}^j, z_{n_j}^j\right]$.
Now, the fact that $z_0^{j+l}$ is not R-mobile in
$\left\langle z_0^j, z_1^j \right]$ implies that its mirror mobile point,
$z_{n_{j-l}}^{j-l}$, is not L-mobile in $\left\langle z_{n_j-1}^j, z_{n_j}^j\right]$,
and so $z_0^{j+l}$ is the only non-neutralized
mobile point in $\left\langle z_{n_j-1}^j, z_{n_j}^j\right]$.
Since $v_q\!\left(z_{n_{j_*}-1}^{j_*}, z_{n_{j_*}}^{j_*}\right) = \alpha(k-k^2)$,
and $v_q\!\left(z_{n_{j^{\prime}}-1}^{j^{\prime}}, z_{n_{j^{\prime}}}^{j^{\prime}}\right) = \alpha(k)$
for all $j^{\prime} \neq j_* \in \Z/d$, 
and since any non-neutralized mobile point active in 
$\left\langle z_{n_j - (i_*+1)}^j, z_{n_j - i_*}^j \right]$
at time $j=j^{\prime}$ contributes $-k^2$ to
$v_q\!\left(z_{n_{j^{\prime}}-1}^{j^{\prime}}, z_{n_{j^{\prime}}}^{j^{\prime}}\right)$,
this means that
$z_0^{j+l}$ is inactive in $\left\langle z_{n_j-1}^j, z_{n_j}^j\right]$
precisely once when $\alpha = -1$, and active precisely once
when $\alpha = +1$.  The fact that $J_0$ and $J_2$ are each nonempty
makes it impossible for $z_0^{j+l}$ to be inactive in 
$\left\langle z_{n_j-1}^j, z_{n_j}^j\right]$ precisely once, so we must have $\alpha = +1$,
implying $j_* \in J_3$, since that is the only time when $z_0^{j+l}$ is active.
But this contradicts the fact that $n_{j_*} = 1$.
Thus, when $l \neq 0$, any non-neutralized R-mobile point $z_0^{j+l}$ in
${\mathbf{z}}^j$ must be non-neutralized R-mobile rel $z_0^j$.

That leaves us with the case in which $l=0$: suppose that
$z_0^j$ is non-neutralized R-mobile in ${\mathbf{z}}^j$,
hence is non-neutralized R-mobile in
$\left\langle z_{n_j - (i_*+1)}^j, z_{n_j - i_*}^j \right]$
for some $i_* \in \left\{0, \ldots, \left\lfloor\frac{k}{d}\right\rfloor -2 \right\}$.
Then its mirror mobile point, $z_{n_j}^j$, is non-neutralized
L-mobile in $\left\langle z_{i_*}^j, z_{i_*+1}^j \right]$
and is not L-mobile in $\left\langle z_{n_j - (i_*+1)}^j, z_{n_j - i_*}^j \right]$,
so that $z_0^j$ is the only non-neutralized mobile point in
$\left\langle z_{n_j - (i_*+1)}^j, z_{n_j - i_*}^j \right]$.
Again, Part (i') implies there exists some $j_* \in \Z/d$ for which
\begin{equation}
\label{main prop, part (ii'), eq: x,y = z_i,z_i+1, etc}
\left(z_{n_{j_*} - (i_*+1)}^{j_*}, z_{n_{j_*} - i_*}^{j_*}\right)
= \left(z_{i_*}^{j_*}, z_{i_*+1}^{j_*}\right) = (x_*, y_*),
\end{equation}
so that $v_q\!\left(z_{n_{j_*} - (i_*+1)}^{j_*}, z_{n_{j_*} - i_*}^{j_*}\right) = \alpha(k-k^2)$, and
$v_q\!\left(z_{n_{j^{\prime}} - (i_*+1)}^{j^{\prime}}, z_{n_{j^{\prime}} - i_*}^{j^{\prime}}\right) = \alpha(k)$
for all $j^{\prime} \neq j_* \in \Z/d$.
If $\alpha = -1$, then
$z_0^j$ is inactive in $\left\langle z_{n_j - (i_*+1)}^j, z_{n_j - i_*}^j \right]$
at time $j=j_*$, and so
$z_0^{j_*} \in \left\langle z_{i_*}^{j_*}+dq, z_{i_*+1}^{j_*}+dq \right]$.
Now, the fact that 
$2[dq]_{k^2} < k^2$ implies that $i_* \geq 2$,
and (\ref{main prop, part (ii'), eq: x,y = z_i,z_i+1, etc})
implies that $n_{j_*}-(i_*+1)=i_*$.  Thus
$i_*+2 \leq 2i_* = n_{j_*}-1 \leq \left\lfloor\frac{k}{d}\right\rfloor - 1$,
so that $z_{i_*+2}^j$ is defined for all $j\in \Z/d$,
and so
$z_0^{j_*} \in \left\langle z_{i_*+1}^{j_*}, z_{i_*+2}^{j_*} \right]$
if $\alpha = -1$.
On the other hand, if $\alpha = +1$, then
$z_0^j$ is active in $\left\langle z_{n_j - (i_*+1)}^j, z_{n_j - i_*}^j \right]$
at time $j=j_*$, and so
$z_0^{j_*} \in \left\langle z_{i_*}^{j_*}, z_{i_*+1}^{j_*} \right]$.
We can summarize these two statements by saying that
$z_0^{j_*}\in \left\langle z_{i_*+\frac{1-\alpha}{2}}^{j_*}, z_{i_*+\frac{3-\alpha}{2}}^{j_*}\right]$,
which implies that
$z_0^{j^{\prime}}\in 
\left\langle z_{i_*+\frac{1-\alpha}{2}}^{j^{\prime}}, z_{i_*+\frac{3-\alpha}{2}}^{j^{\prime}}\right]$
for all $j^{\prime}\in \Z/d$.
Thus, for any $j_0 \in \Z/d$,
\begin{equation}
\label{main prop, part (ii'), all of Q_q between the z_0^js}
\left[z_0^{j_0}\right]_{k^2}
\;\;<\;\; x \;\;<\;\; \left[z_0^{j_0}\right]_{k^2} + \left(\textstyle{{i_*+\frac{3-\alpha}{2}}}\right)[dq]_{k^2}
\end{equation}
for {\em every} $x \neq z_0^{j^{\prime}} \in Q_q$.
In particular, (\ref{main prop, part (ii'), all of Q_q between the z_0^js})
holds for all $x \in \left\{ z_{n_{j-1}}^{j-1} \right\}_{j\in\Z/d}$.
Since $z_{n_{j-1}}^{j-1} -z_0^j$ is constant in $j\in\Z/d$, this implies that
$z_{n_{j^{\prime}-1}}^{j^{\prime}-1}
\in \left\langle z_{i^{\prime}}^{j^{\prime}}, z_{{i^{\prime}}+1}^{j^{\prime}} \right]$
for some $i^{\prime} \in \left\{0, \ldots, i_*+\frac{1-\alpha}{2} \right\}$
and for all $j^{\prime} \in \Z/d$,
where we recall that 
$i_*+\frac{1-\alpha}{2} \leq \left\lfloor\frac{k}{d}\right\rfloor-2$.
If $i^{\prime} \neq 0$, this implies that $z_{n_{j-1}}^{j-1}$ is
L-mobile rel $z_{n_j}^j$.
If $i^{\prime}=0$, then this implies that
$z_{n_{j^{\prime}-1}}^{j^{\prime}-1}
\in \left\langle z_{i_0}^{j^{\prime}}, z_{i_0}^{j^{\prime}} +dq \right]$
for some $i_0 \in \left\{i_* + \frac{1-\alpha}{2}, i_* + \frac{3-\alpha}{2} \right\}$,
so that, again, $z_{n_{j-1}}^{j-1}$ is L-mobile rel $z_{n_j}^j$.
In either case,
$z_{n_{j-1}}^{j-1}$ is L-mobile rel $z_{n_j}^j$, hence is L-mobile
in $\left\langle z_{n_j-(i+1)}^j, z_{n_j-i}^j\right]$
for some $i \in  \left\{0, \ldots, \left\lfloor\frac{k}{d}\right\rfloor -2 \right\}$.
Thus either $i=i_*$, so that
$z_0^j$ is {\em neutralized} R-mobile in
$\left\langle z_{n_j - (i_*+1)}^j, z_{n_j - i_*}^j \right]$,
or $i \neq i_*$, so that
$z_{n_{j-1}}^{j-1}$ is non-neutralized L-mobile
in $\left\langle z_{n_j-(i+1)}^j, z_{n_j-i}^j\right]$,
contradicting Part (i').

Thus any non-neutralized R-mobile point in ${\bf{z}}^j$
is non-neutralized R-mobile rel $z_0^j$, which also means that
the mirror statement must be true.  That is, 
any non-neutralized L-mobile point in ${\bf{z}}^j$
is non-neutralized L-mobile rel $z_{n_j}^j$.

We therefore can express the unique non-neutralized R-mobile point
in ${\bf{z}}^j$ as $z_0^{j+l}$ in $\left\langle z_{i_*}^j, z_{i_*+1}^j\right]$,
for some $l \neq 0 \in \Z/d$, where, as discussed earlier, 
$i_* \in \left\{0, \ldots, \left\lfloor\frac{k}{d}\right\rfloor -2 \right\}$
satisfies $v_q\!\left(z_{i_*}^{j_*}, z_{i_*+1}^{j_*}\right) = {\alpha}(k-k^2)$
for some unique $j_* \in \Z/d$.  The unique non-neutralized L-mobile point
is then the mirror mobile point,
$z_{n_{j-l}}^{j-l}$ non-neutralized L-mobile in
$\left\langle z_{n_j - (i_*+1)}^j, z_{n_j - i_*}^j \right]$.
Part (i') then tells us that
$(x_*,y_*) = \left(z_{i_*}^{j_*},z_{i_*+1}^{j_*}\right) 
= \left(z_{n_{j_*} - (i_*+1)}^{j_*}, z_{n_{j_*} - i_*}^{j_*}\right)$.

\end{proof}

\begin{proof}[Proof of (ii$\psi$)]
First, to prove that there is at most one non-neutralized R-pseudomobile point
and at most one non-neutralized L-pseudomobile point in
$\left\langle z_{n_{j-1}}^{j-1}, z_0^j \right]$, take the proof of the comparable
statement in Part (ii'), and make the following adaptations.
Replace the word ``mobile'' with the word ``pseudomobile,''
and replace $z_{a_j}^j$ and $z_{a_{j+1}}^j$ with
$z_{n_{j-1}}^{j-1}$ and $z_0^j$, respectively---both in the actual proof of Part (ii'),
and in the discussion of $L_{\mathrm{R}}$, $L_{\mathrm{L}}$, {\em et cetera},
preceding the proof of Part (i).  In addition,
replace all references in Part (ii') to
Proposition \ref{prop: active iff inactive true iff neutralized} with
references to 
Proposition \ref{prop: active iff inactive means neutralized for pseudomobile}.

Next, suppose there is a non-neutralized R-pseudomobile point,
say $z_0^{j+l}$ in $\left\langle z_{n_{j-1}}^{j-1}, z_0^j \right]$
for some $l\neq 0 \in \Z/d$.  Then its mirror pseudomobile point,
$z_{n_{j-1-l}}^{j-1-l}$, is also non-neutralized pseudomobile in
$\left\langle z_{n_{j-1}}^{j-1}, z_0^j \right]$.
A similar result holds if we start with a non-neutralized L-pseudomobile point.
Thus $\left\langle z_{n_{j-1}}^{j-1}, z_0^j \right]$ has precisely one
non-neutralized R-pseudomobile point and precisely one
non-neutralized L-pseudomobile point, namely,
$z_0^{j+l}$ and $z_{n_{j-1-l}}^{j-1-l}$ for some nonzero $l \in \Z/d$.
\end{proof}

\begin{proof}[Proof of (iii)]
First, recall that since $(\mu,\gamma)=(1,1)$, we have
$\left[dq\right]_{k^2} = (m\!+\!c)k\!+\!\alpha$, and
$\psi = \left[(m\!+\!c)k\!+\!\alpha - \frac{ck + \alpha}{d}k\right]_{k^2}$,
where we recall that 
$\psi:= \left[z_0^{j+1} - z_{n_j}^j\right]_{k^2} = \left[dq-kq\right]_{k^2}$.

Consider first the case in which 
$\psi > 2\left[dq\right]_{k^2}$,
which, as proven in Part (iv), implies that
all mobile points are non-neutralized.
This means in particular that Part (i)---as
opposed to merely Part (i')---holds.

Since $v_q\!\left(z_{n_{j-1}}^{j-1}, z_0^j\right) = -v_q\!\left(z_0^j, z_{n_{j-1}}^{j-1}\right)$,
it is sufficient to show that $\# \left( Q_q \cap \left\langle z_0^j, z_{n_{j-1}}^{j-1}\right] \right)$
is constant in $j\in \Z/d$.  To do this, our strategy will be to cover
$\left\langle z_0^j, z_{n_{j-1}}^{j-1}\right]$ with sets for which we understand
how their intersection with $Q_q$ changes as $j$ varies in $\Z/d$.
We begin by examining the length of $\left\langle z_0^j, z_{n_{j-1}}^{j-1}\right]$.
Since $\psi > 2\left[dq\right]_{k^2}$,
we know that $(m+c)k+\alpha \,-\, \frac{ck + \alpha}{d}k < 0$.  Thus,
for all $j\in\Z/d$, we have
\begin{align}
     \left[z_{n_{j-1}}^{j-1} - z_0^j\right]_{k^2}
&= \textstyle{\frac{ck+\alpha}{d}}k - ((m\!+\!c)k\!+\!\alpha)
            \\ \nonumber
&=((m\!+\!c)k\!+\!\alpha) \!\left(\textstyle{\frac{k}{d}} - 1\right)  \,-\, m\textstyle{\frac{k^2}{d}}
            \\ \nonumber
&< \textstyle{\left\lfloor \frac{k}{d} \right\rfloor} \left[dq\right]_{k^2}.
\end{align}
This implies that
\begin{equation}
\left[ z^{j-1}_{n_{j-1} - 1} - z_0^j\right]_{k^2} = \left[z^{j-1}_{n_{j-1}} - z_1^j\right]_{k^2}
< \left(\textstyle{\left\lfloor \frac{k}{d}\right\rfloor} -1\right) \left[dq\right]_{k^2},
\end{equation}
so that
\begin{align}
\left\langle z_0^j, z^{j-1}_{n_{j-1} - 1} \right] 
&\subset 
\left\langle z_0^j, z_{{\left\lfloor \frac{k}{d}\right\rfloor} -1}^j \right],
          \\
\left\langle z_1^j, z^{j-1}_{n_{j-1}} \right] 
&\subset 
\left\langle z_{n_{j-1} - \left({\left\lfloor \frac{k}{d}\right\rfloor} -1\right)}^{j-1}, z_{n_{j-1}}^{j-1} \right].
\end{align}
Thus, since
$\left\langle z_0^j, z_{n_{j-1}}^{j-1}\right] =
\left\langle z_0^j, z^{j-1}_{n_{j-1} - 1} \right] \cup
\left\langle z_1^j, z^{j-1}_{n_{j-1}} \right]$
for all $j \in \Z/d$, we know that
\begin{equation}
\label{eq: main prop: part (iii), covering}
   \left\langle z_0^j, z_{n_{j-1}}^{j-1}\right]
\;\;\subset\;\;
   \left\langle z_0^j, z_{{\left\lfloor \frac{k}{d}\right\rfloor} -1}^j \right]
   \cup
   \left\langle z_{n_{j-1} - \left({\left\lfloor \frac{k}{d}\right\rfloor} -1\right)}^{j-1}, z_{n_{j-1}}^{j-1} \right]
\end{equation}
for all $j \in \Z/d$.

Suppose that $v_q\!\left( z_{n_{j-1}}^{j-1}, z_0^j \right)$ is not constant
in $j \in \Z/d$.  Then since $q$ is genus-minimizing, we know that
for any $i \in \left\{0, \ldots, \left\lfloor\frac{k}{d}\right\rfloor - 2\right\}$,
both $v_q\!\left( z_i^j, z_{i+1}^j \right)$ and
$v_q\!\left( z_{n_j - (i+1)}^j, z_{n_j - i}^j \right)$ are constant
in $j\in \Z/d$.  Thus by Part (i), neither 
$\left\langle z_i^j, z_{i+1}^j \right]$ nor
$\left\langle z_{n_j - (i+1)}^j, z_{n_j - i}^j \right]$
has any mobile points.
Thus, for all $j \in \Z/d$ and $i \in \left\{0, \ldots, \left\lfloor\frac{k}{d}\right\rfloor - 2\right\}$,
we know that
\begin{align}
      Q_q \cap \left\langle z_i^j, z_{i+1}^j \right]
&=  Q_q \cap \left\langle z_i^0, z_{i+1}^0 \right]+ \left(z_0^j - z_0^0\right),
             \\
      Q_q \cap \left\langle z_{n_{j-1}-(i+1)}^{j-1}, z_{n_{j-1}-i}^{j-1} \right]
&= Q_q \cap \left\langle z_{n_{-1}-(i+1)}^{-1}, z_{n_{-1}-i}^{-1} \right] 
      +\left(z_{n_{j-1}}^{j-1} - z_{n_{-1}}^{-1}\right),
\end{align}
where we arbitrarily chose $j=0$ as a reference point.
Taking the union over  $i \in \left\{0, \ldots, \left\lfloor\frac{k}{d}\right\rfloor - 2\right\}$,
and using the fact that 
$z_0^j-z_0^0 = \left(z_{n_{j-1}}^{j-1}+\psi\right) - \left(z_{n_{-1}}^{-1} + \psi\right)
= z_{n_{j-1}}^{j-1} - z_{n_{-1}}^{-1}$, we then obtain
\begin{align}
      Q_q \cap \left\langle z_0^j, z_{{\left\lfloor \frac{k}{d}\right\rfloor} -1}^j \right]
&= Q_q \cap \left\langle z_0^0, z_{{\left\lfloor \frac{k}{d}\right\rfloor} -1}^0 \right]
      + \left(z_0^j - z_0^0\right)
             \\
      Q_q \cap \left\langle z_{n_{j-1} - \left({\left\lfloor \frac{k}{d}\right\rfloor} -1\right)}^{j-1}, 
                         z_{n_{j-1}}^{j-1} \right]
&= Q_q \cap \left\langle z_{n_{-1} - \left({\left\lfloor \frac{k}{d}\right\rfloor} -1\right)}^{-1},
                         z_{n_{-1}}^{-1} \right]
      + \left(z_0^j - z_0^0\right),
\end{align}
and so, by equation (\ref{eq: main prop: part (iii), covering}), we have
\begin{equation}
   Q_q \cap \left\langle z_0^j, z_{n_{j-1}}^{j-1}\right]
= Q_q \cap \left\langle z_0^{0}, z_{n_{-1}}^{-1}\right] + \left(z_0^j - z_0^0\right).
\end{equation}
Thus,
$\#\left(Q_q \cap \left\langle z_0^j, z_{n_{j-1}}^{j-1}\right]\right)$, and hence
$v_q\!\left(z_{n_{j-1}}^{j-1}, z_0^j\right)$,
is constant in $j\in\Z/d$.
              \\

We next show that it is algebraically impossible for $\psi$ to satisfy
$\left[dq\right]_{k^2} < \psi < 2\left[dq\right]_{k^2}$ when 
$(\mu,\gamma) = (1,1)$.  Suppose that 
$\left[dq\right]_{k^2} < \psi < 2\left[dq\right]_{k^2}$,
so that 
$\psi = (m+c)k+\alpha \;-\; \frac{ck+\alpha}{d}k + k^2$.
Then, since $0 < \psi - \left[dq\right]_{k^2} < \left[dq\right]_{k^2}$,
we have
\begin{equation}
0 \;<\; k^2 - \textstyle{\frac{ck+\alpha}{d}}k  \;<\; (m+c)k+\alpha.
\end{equation}
On the other hand, since $(m\!+\!c)k\!+\!\alpha \;-\; \frac{ck+\alpha}{d}k < 0$,
we know that
$(m\!+\!c)k\!+\!\alpha \;<\; \textstyle{\frac{ck+\alpha}{d}}k$.
Combining these two facts gives
\begin{equation}
k^2 - \textstyle{\frac{ck+\alpha}{d}}k   < \textstyle{\frac{ck+\alpha}{d}}k,
\end{equation}
which implies that
$\frac{ck+\alpha}{d} > \frac{k}{2}$.
This, however, contradicts the constraint $\frac{ck+\alpha\gamma}{d} < \frac{k}{2}$ from
Proposition \ref{prop: properties of parameters d, m, c, alpha, mu, gamma}.
Thus our initial supposition must be false, and so
$\psi \notin \left\langle \left[dq\right]_{k^2} , 2\left[dq\right]_{k^2}\right\rangle$.
        \\

This leaves us with the case in which
$\psi < \left[dq\right]_{k^2}$.
Suppose that
$v_q\!\left(z_{n_{j-1}}^{j-1}, z_0^j\right)$ is not constant in $j \in \Z/d$,
so that Part (ii$\psi$) guarantees the existence of 
a non-neutralized R-pseudomobile point, say, $z_0^{j+l}$ in
$\left\langle z_{n_{j-1}}^{j-1}, z_0^j \right]$,
for some $l \neq 0 \in \Z/d$.
Recall that this implies that
$\minq_{j\in\Z/d}\left(z_0^{j+l}- z_{n_{j-1}}^{j-1}\right) \in \left\langle 0, \psi\right\rangle$,
or equivalently, that
$\minq_{j\in\Z/d}\left(z_0^{j+l}- z_0^j\right) \in \left\langle -\psi, 0\right\rangle$.
This, in turn, implies that
\begin{align}
\label{eq: main prop (iii), max ineq for L-mobile}
     \maxq_{j\in\Z/d}\left(z_{n_{j+l-1}}^{j+l-1}- z_0^j\right) 
&\in \left\langle -\psi -\psi + dq,\; 0 -\psi + dq \right\rangle
\subset \left\langle -\psi, dq\right\rangle,\;\mathrm{and}
          \\
\label{eq: main prop (iii), max ineq for R-mobile}
     \maxq_{j\in\Z/d}\left(z_0^{j+l}- z_1^j\right)
&\in \left\langle -\psi - dq + dq,\; 0 -dq + dq\right\rangle
\subset \left\langle -dq, 0\right\rangle.
\end{align}
Line (\ref{eq: main prop (iii), max ineq for L-mobile}) then tells us that
either $z_{n_{j+l-1}}^{j+l-1}$ is L-pseudomobile in
$\left\langle z_{n_{j-1}}^{j-1}, z_0^j \right]$---contradicting
the supposition that $z_0^{j+l}$ is non-neutralized pseudomobile in
$\left\langle z_{n_{j-1}}^{j-1}, z_0^j \right]$---or 
$z_{n_{j+l-1}}^{j+l-1}$ is L-mobile in
$\left\langle z_0^j, z_1^j \right]$.
Line (\ref{eq: main prop (iii), max ineq for R-mobile}),
on the other hand, implies that
$\minq_{j\in\Z/d}\left(z_0^{j+l}- z_0^j\right)
\in \left\langle -dq, 0\right\rangle$,
which, since $2[dq]_{k^2} < k^2$, has no intersection
with $\left\langle 0, dq \right\rangle$, so that
$z_0^{j+l}$ is not R-mobile in $\left\langle z_0^j, z_1^j \right]$.
This means that if
$z_{n_{j+l-1}}^{j+l-1}$ is L-mobile in
$\left\langle z_0^j, z_1^j \right]$, then
it is non-neutralized L-mobile in
$\left\langle z_0^j, z_1^j \right]$,
contradicting our supposition that
$v_q\!\left(z_{n_{j-1}}^{j-1}, z_0^j\right)$ is not constant in $j \in \Z/d$.
Thus, our original supposition must be false, 
and so $v_q\!\left(z_{n_{j-1}}^{j-1}, z_0^j\right)$ must be constant in $j \in \Z/d$.
         \\

We have now shown that 
$v_q\!\left(z_{n_{j-1}}^{j-1}, z_0^j\right)$ is constant in $j \in \Z/d$
for all possible values of $\psi \in \Z/k^2$.
Thus, by Proposition \ref{prop: unique v_q = alpha(k - k^2), and the rest are v_q = alpha (k)},
$v_q\!\left(z_{n_{j-1}}^{j-1}, z_0^j\right) = \alpha(k)$ for all $j \in \Z/d$.
\end{proof}

\begin{proof}[Proof of (iv)]
Suppose that
$\psi < \left[dq\right]_{k^2}$.
Then, first of all, the conclusion of Part (iii) holds, 
because the $\psi < \left[dq\right]_{k^2}$ case of the proof of Part (iii)
does not use the hypothesis that $(\mu,\gamma) = (1,1)$.
Thus $v_q(z_{n_{j-1}}^{j-1}, z_0^j)$ is constant in $j \in \Z/d$,
and so by Part (ii'), we know that
there exist unique $l\neq 0 \in \Z/d$ and
$i_* \in \left\{ 0, \ldots, \left\lfloor \frac{k}{d} \right\rfloor - 2 \right\}$
for which $z_0^{j+l}$ is non-neutralized R-mobile in
$\left\langle z_{i_*}^j, z_{{i_*}+1}^j \right]$,
and $z_{n_{j-l}}^{j-l}$ is non-neutralized L-mobile in
$\left\langle z_{n_j - (i_*+1)}^j, z_{n_j - i_*}^j \right]$, with
$\left(z_{i_*}^{j_*}, z_{{i_*}+1}^{j_*}\right) 
= \left(z_{n_j - (i_*+1)}^{j_*}, z_{n_j - i_*}^{j_*}\right)
= (x_*,y_*)$, where $x_*, y_* \in Q_q$ are the unique elements of $Q_q$
satisfying $v_q(x_*,y_*) = \alpha(k-k^2)$.
Note that this implies $i_* = \frac{n_{j_*}-1}{2}$.

Since $z_0^{j+l}$ is R-mobile in
$\left\langle z_{i_*}^j, z_{{i_*}+1}^j \right]$, we know that
$\minq_{j\in\Z/d}\left(z_0^{j+l}-z_{i_*}^j\right)
\in \left\langle 0, \left[dq\right]_{k^2}\right\rangle$.
The fact that
$\psi < \left[dq\right]_{k^2}$ then implies
$\minq_{j\in\Z/d}\left(z_{n_{j+l-1}}^{j+l-1}-z_{i_*}^j\right)
\in \left\langle -[dq], \left[dq\right]_{k^2}\right\rangle$.
However, since $z_0^{j+l}$ is non-neutralized R-mobile in
$\left\langle z_{i_*}^j, z_{{i_*}+1}^j \right]$, we know that
$z_{n_{j+l-1}}^{j+l-1}$ is not L-mobile in
$\left\langle z_{i_*}^j, z_{{i_*}+1}^j \right]$, and so
$\minq_{j\in\Z/d}\left(z_{n_{j+l-1}}^{j+l-1}-z_{i_*}^j\right)
\notin \left\langle -[dq], 0\right\rangle$.  Thus
\begin{equation}
\label{eq: main prop (iv), type (1,1), chaperoned R-mobile, L min in}
\minq_{j\in\Z/d}\left(z_{n_{j+l-1}}^{j+l-1}-z_{i_*}^j\right)
\;\in\; \left\langle 0, \left[dq\right]_{k^2}\right\rangle.
\end{equation}

Suppose that ${i_*} < \left\lfloor\frac{k}{d}\right\rfloor - 2$.
Then (\ref{eq: main prop (iv), type (1,1), chaperoned R-mobile, L min in}) implies
$\maxq_{j\in\Z/d}\left(z_{n_{j+l-1}}^{j+l-1}-z_{{i_*}+2}^j\right)
\in \left\langle -\left[dq\right]_{k^2}, 0\right\rangle$,
so that $z_{n_{j+l-1}}^{j+l-1}$ is L-mobile in
$\left\langle z_{{i_*}+1}^j, z_{{i_*}+2}^j \right]$.
Moreover, the fact that
$\minq_{j\in\Z/d}\left(z_0^{j+l}-z_{i_*}^j\right)
\in \left\langle 0, \left[dq\right]_{k^2}\right\rangle$ implies
$\minq_{j\in\Z/d}\left(z_0^{j+l}-z_{i_*+1}^j\right)
\in  \left\langle -\left[dq\right]_{k^2}, 0\right\rangle$, which, since
$2[dq]_{k^2} < k^2$, has no intersection with
$\left\langle 0, \left[dq\right]_{k^2}\right\rangle$,
and so $z_0^{j+l}$ is {\em not} R-mobile in
$\left\langle z_{{i_*}+1}^j, z_{{i_*}+2}^j \right]$.
Thus $z_{n_{j+l-1}}^{j+l-1}$ is non-neutralized L-mobile in
$\left\langle z_{{i_*}+1}^j, z_{{i_*}+2}^j \right]$,
contradicting the fact that
$\left\langle z_{i_*}^j, z_{{i_*}+1}^j \right]$
already has a non-neutralized mobile point.
Thus, ${i_*} = \left\lfloor\frac{k}{d}\right\rfloor - 2$,
implying $\left\lfloor\frac{k}{d}\right\rfloor \in \{2,3\}$,
and $z_0^{j+l}$ is non-neutralized R-mobile in
$\left\langle z_{\left\lfloor\!\frac{k}{d}\!\right\rfloor - 2}^j, 
z_{\left\lfloor\!\frac{k}{d}\!\right\rfloor - 1}^j \right]$.

We next determine which configurations of ${\bf{z}}^j$
admit such an R-mobile point.  We begin by partitioning
the values of
$j^{\prime}\in\Z/d$ partition into four sets,
according to whether 
$n_{j^{\prime}}=\left\lfloor\frac{k}{d}\right\rfloor$
or $n_{j^{\prime}}=\left\lfloor\frac{k}{d}\right\rfloor - 1$,
and to whether the $z_0^{j+l}$ is active or inactive in
$\left\langle z_{\left\lfloor\!\frac{k}{d}\!\right\rfloor - 2}^j,
z_{\left\lfloor\!\frac{k}{d}\!\right\rfloor - 1}^j \right]$ at time $j=j^{\prime}$
(or equivalently, to whether
$z_0^{j^{\prime}+l}
\in \left\langle z_{\left\lfloor\!\frac{k}{d}\!\right\rfloor - 2}^{j^{\prime}}, 
z_{\left\lfloor\!\frac{k}{d}\!\right\rfloor - 1}^{j^{\prime}} \right]$ or
$z_0^{j^{\prime}+l}
\notin \left\langle z_{\left\lfloor\!\frac{k}{d}\!\right\rfloor - 2}^{j^{\prime}}, 
z_{\left\lfloor\!\frac{k}{d}\!\right\rfloor - 1}^{j^{\prime}} \right]$).
We begin by claiming that the set
\begin{equation}
\label{eq: part (iv), J_0 definition with i=k/d-2}
J_0 := \left\{j^{\prime} \in \Z/d \left|\;  
n_{j^{\prime}}\!=\!\textstyle{\left\lfloor\frac{k}{d}\right\rfloor}\!-\!1;\;
z_0^{j+l}\;\text{is inactive in}\;\!
\left\langle z_{\left\lfloor\!\frac{k}{d}\!\right\rfloor-2}^j, z_{\left\lfloor\!\frac{k}{d}\!\right\rfloor-1}^j\right]
\!\;\text{when}\;j=j^{\prime}
\right.\!\right\}
\end{equation}
is nonempty.
Suppose that $J_0$ is empty.
Then $\Z/d$ partitions
into the disjoint union of $J_1$, $J_2$, and $J_3$, defined as follows:
\begin{align}
\label{eq: part (iv), J_1 definition with i=k/d-2}
       J_1
&:= \left\{j^{\prime} \in \Z/d \left|\;  
       n_{j^{\prime}}\!=\!\textstyle{\left\lfloor\frac{k}{d}\right\rfloor}\!-\!1;\;
       z_0^{j+l}\;\text{is active in}\;\!
       \left\langle z_{\left\lfloor\!\frac{k}{d}\!\right\rfloor-2}^j, 
       z_{\left\lfloor\!\frac{k}{d}\!\right\rfloor-1}^j\right]
       \!\;\text{when}\;j=j^{\prime}
       \right.\!\right\},
             \\
\label{eq: part (iv), J_2 definition with i=k/d-2}
       J_2
&:= \left\{j^{\prime} \in \Z/d \left|\;  
       n_{j^{\prime}}\!=\!\textstyle{\left\lfloor\frac{k}{d}\right\rfloor};\;
       z_0^{j+l}\;\text{is inactive in}\;\!
       \left\langle z_{\left\lfloor\!\frac{k}{d}\!\right\rfloor-2}^j, 
       z_{\left\lfloor\!\frac{k}{d}\!\right\rfloor-1}^j\right]
       \!\;\text{when}\;j=j^{\prime}
       \right.\!\right\},
             \\
\label{eq: part (iv), J_3 definition with i=k/d-2}
       J_3
&:= \left\{j^{\prime} \in \Z/d \left|\;  
       n_{j^{\prime}}\!=\!\textstyle{\left\lfloor\frac{k}{d}\right\rfloor};\;
       z_0^{j+l}\;\text{is active in}\;\!
       \left\langle z_{\left\lfloor\!\frac{k}{d}\!\right\rfloor-2}^j, 
       z_{\left\lfloor\!\frac{k}{d}\!\right\rfloor-1}^j\right]
       \!\;\text{when}\;j=j^{\prime}
       \right.\!\right\}.
\end{align}
Note that $J_1$ and $J_2$ are nonempty.  $J_1$ is nonempty because
$n_{j^{\prime}}=\left\lfloor\frac{k}{d}\right\rfloor-1$ only if $j^{\prime} \in J_1$;
$J_2$ is nonempty because $z_0^{j+l}$ is inactive in
$\left\langle z_{\left\lfloor\!\frac{k}{d}\!\right\rfloor-2}^j, 
z_{\left\lfloor\!\frac{k}{d}\!\right\rfloor-1}^j\right]$
at time $j=j^{\prime}$ only if $j^{\prime} \in J_2$.
The question of whether or not $J_3$ is empty depends on whether or not $l=1$.
Since $J_1$ is nonempty, we know that for any $j_1 \in J_1$ we have
$z_0^{j_1+l} \in \left\langle z_{n_{j_1}-1}^{j_1}, z_{n_{j_1}}^{j_1}\right]$,
so that
\begin{align}
      z_0^{j_1+1}-z_0^{j_1+l}
&= \left(z_0^{j_1+1}-z_{n_{j_1}}^{j_1}\right) + \left(z_{n_{j_1}}^{j_1} - z_0^{j_1+l}\right)
                \\ \nonumber
&=\psi + \left(z_{n_{j_1}}^{j_1} - z_0^{j_1+l}\right)
                \\ \nonumber
&\in \left\langle \psi, \psi + dq \right\rangle.
\end{align}
The fact that $0< \psi + [dq]_{k^2} < k^2$ then implies that
$l \neq 1$.  Thus $J_3$ is nonempty, since otherwise
$z_{n_j}^j - z_0^{j+l}$ would be constant in $\Z/d$,
contradicting Corollary \ref{cor: q of positive type: combo of lemma and difference eq}.
The definitions of $J_1$, $J_2$, and $J_3$, along with
line (\ref{eq: main prop (iv), type (1,1), chaperoned R-mobile, L min in}),
then show that
$\maxq_{j\in\Z/d}\left(z_{n_{j+l-1}}^{j+l-1} - z_{n_j}^j\right) \in
\left\langle -dq, 0\right\rangle$ and
$\maxq_{j\in\Z/d}\left(z_0^{j+l} - z_{n_j}^j\right) \in
\left\langle -dq, 0\right\rangle$,
which tell us, respectively, that
$z_{n_{j+l-1}}^{j+l-1}$ is L-mobile in
$\left\langle z_{n_j - 1}^j, z_{n_j}^j\right]$,
and that
$\minq_{j\in\Z/d}\left(z_0^{j+l} - z_{n_j - 1}^j\right) \in
\left\langle -dq, 0\right\rangle$,
so that $\minq_{j\in\Z/d}\left(z_0^{j+l} - z_{n_j - 1}^j\right) \notin
\left\langle 0, dq\right\rangle$,
making $z_{n_{j+l-1}}^{j+l-1}$ L-mobile non-neutralized L-mobile in
$\left\langle z_{n_j - 1}^j, z_{n_j}^j\right]$.
Part (ii') then tells us that the non-neutralized 
R-mobile point $z_0^{j+l}$ in
$\left\langle z_{\left\lfloor\!\frac{k}{d}\!\right\rfloor - 2}^j, 
z_{\left\lfloor\!\frac{k}{d}\!\right\rfloor - 1}^j \right]$ is the mirror of the
non-neutralized L-mobile point $z_{n_{j+l-1}}^{j+l-1}$  in
$\left\langle z_{n_j - 1}^j, z_{n_j}^j\right]$.

Thus $\left\lfloor\!\frac{k}{d}\!\right\rfloor = 2$, and
$z_0^{j+l}$ is non-neutralized R-mobile in
$\left\langle z_0^j, z_1^j \right]$.
We claim $z_0^{j+l}$ is the only non-neutralized mobile point in
$\left\langle z_0^j, z_1^j \right]$.
That is, (\ref{eq: main prop (iv), type (1,1), chaperoned R-mobile, L min in})
implies that
the $\minq$ of $\left\{z^{j+l-1}_{n_{j+l-1}}-z_1^j\right\}_{j\in\Z/d}$
occurs when $j \in J_2$, but for any $j_2 \in J_2$,
(\ref{eq: main prop (iv), type (1,1), chaperoned R-mobile, L min in}) implies that
$z^{j_2+l-1}_{n_{j_2+l-1}}-z_1^{j_2} \in
\left\langle 0, dq\right\rangle$, which, since $2[dq]_{k^2} < k^2$,
has no intersection with
$\left\langle -dq, 0 \right\rangle$.
Thus $z^{j+l-1}_{n_{j+l-1}}$ is not L-mobile in
$\left\langle z_0^j, z_1^j \right]$, and so
$z_0^{j+l}$ is the only non-neutralized mobile point in
$\left\langle z_0^j, z_1^j \right]$.

Suppose that $\alpha = +1$.
Then by Part (i'), we know there is some  $j_* \in \Z/d$ such that
$v_q(z_0^{j_*}, z_1^{j_*}) = k - k^2$, and 
$v_q(z_0^{j^{\prime}}, z_1^{j^{\prime}}) = k$ for all $j^{\prime} \neq j_*$ in $\Z/d$.
This means that there is one more active non-neutralized mobile point in
$\left\langle z_0^j, z_1^j\right]$ at time $j_*$ than at any other time, which
in this case implies that the unique non-neutralized mobile point in
$\left\langle z_0^j, z_1^j\right]$, namely, $z_0^{j+l}$, must be active precisely once,
but this contradicts the fact that both $J_1$ and $J_3$ are nonempty.
Thus $J_0$ cannot be empty when $\alpha = +1$.

If $\alpha = -1$,
then Part (i') tells us there is some  $j_* \in \Z/d$ such that
$v_q(z_0^{j_*}, z_1^{j_*}) = -k + k^2$, and 
$v_q(z_0^{j^{\prime}}, z_1^{j^{\prime}}) = -k$ for all $j^{\prime} \neq j_*$ in $\Z/d$.
This means that there is one fewer active non-neutralized mobile point in
$\left\langle z_0^j, z_1^j\right]$ at time $j_*$ than at any other time,
and so we have $j_* \in J_2$.  But $j_* \in J_2$ implies that
$n_{j_*} = 2$, contradicting the fact that $i_* = \frac{n_{j_*}-1}{2}\in \Z$
implies $n_{j_*}$ is odd.
Thus $J_0$ cannot be empty when $\alpha = -1$,
and so $J_0 \neq \emptyset$.

We next observe that the nonemptyness of $J_0$ implies that the set
\begin{equation}
       J_3
:=   \left\{j^{\prime} \in \Z/d \left|\;  
       n_{j^{\prime}}\!=\!\textstyle{\left\lfloor\frac{k}{d}\right\rfloor};\;
       z_0^{j+l}\;\text{is active in}\;\!
       \left\langle \!z_{\left\lfloor\!\frac{k}{d}\!\right\rfloor-2}^j, 
       z_{\left\lfloor\!\frac{k}{d}\!\right\rfloor-1}^j\right]
       \!\;\text{at time}\;j=j^{\prime}
       \right.\!\right\}
\end{equation}
is empty, because if not, then for
arbitrary elements $j_0 \in J_0$ and $j_3 \in J_3$, we have
\begin{align}
\label{eq: part (iv), J_0 nonempty implies J_3 empty}
      z_0^{j_0+l} - z_{n_{j_0}}^{j_0}
&= \maxq_{j\in\Z/d}\left(z_0^{j+l}-z_{\left\lfloor\!\frac{k}{d}\!\right\rfloor - 2}^j\right) - dq
                          \\ \nonumber
&= \minq_{j\in\Z/d}\left(z_0^{j+l}-z_{\left\lfloor\!\frac{k}{d}\!\right\rfloor - 2}^j\right)
                          \\ \nonumber
&= z_0^{j_3+l}-z_{\left\lfloor\!\frac{k}{d}\!\right\rfloor - 2}^{j_3}
                          \\ \nonumber
&= z_0^{j_3+l} - z_{n_{j_3}}^{j_3} + 2dq,
\end{align}
contradicting Corollary \ref{cor: q of positive type: combo of lemma and difference eq}.

Note that we once again know only that
$i_* = \left\lfloor\frac{k}{d}\right\rfloor - 2$, with
$\left\lfloor\frac{k}{d}\right\rfloor \in \{2, 3\}$.
Writing $i_*$ instead of $\left\lfloor\!\frac{k}{d}\!\right\rfloor - 2$
to simplify notation, we can now express $\Z/d$ is the disjoint union of
\begin{align}
       J_0
&:= \left\{j^{\prime} \in \Z/d \left|\;  
       n_{j^{\prime}}\!=\!\textstyle{\left\lfloor\frac{k}{d}\right\rfloor}\!-\!1;\;
       z_0^{j+l}\;\text{is inactive in}\;\!
       \left\langle z_{i_*}^j, z_{i_* + 1}^j\right]
       \!\;\text{when}\;j=j^{\prime}
       \right.\!\right\},
           \\
       J_1
&:= \left\{j^{\prime} \in \Z/d \left|\;  
       n_{j^{\prime}}\!=\!\textstyle{\left\lfloor\frac{k}{d}\right\rfloor}\!-\!1;\;
       z_0^{j+l}\;\text{is active in}\;\!
       \left\langle z_{i_*}^j, z_{i_* + 1}^j\right]
       \!\;\text{when}\;j=j^{\prime}
       \right.\!\right\},\;\;\mathrm{and}
           \\
       J_2
&:= \left\{j^{\prime} \in \Z/d \left|\;  
       n_{j^{\prime}}\!=\!\textstyle{\left\lfloor\frac{k}{d}\right\rfloor};\;
       z_0^{j+l}\;\text{is inactive in}\;\!
       \left\langle z_{i_*}^j, z_{i_* + 1}^j\right]
       \!\;\text{when}\;j=j^{\prime}
       \right.\!\right\},
\end{align}
all of which are nonempty.  We have already proven that $J_0$ is nonempty.
$J_1$ is nonempty because $z_0^{j+l}$ is active
in $\left\langle z_{i_*}^j, z_{i_* + 1}^j\right]$ at time $j=j^{\prime}$
only if $j^{\prime} \in J_1$, and $J_2$ is nonempty because
$n_{j^{\prime}}= \left\lfloor\!\frac{k}{d}\!\right\rfloor$ only if $j^{\prime}\in J_2$.
The above definitions of $J_0$, $J_1$, and $J_2$ imply
that $z_0^{j+l}$ is R-mobile in
$\left\langle z_{n_j-1}^j, z_{n_j}^j \right]$,
and (\ref{eq: main prop (iv), type (1,1), chaperoned R-mobile, L min in})
implies that $z^{j+l-1}_{n_{j+l-1}}$ is not L-mobile in
$\left\langle z_{n_j-1}^j, z_{n_j}^j \right]$,
so that $z_0^{j+l}$ is in fact non-neutralized R-mobile in
$\left\langle z_{n_j-1}^j, z_{n_j}^j \right]$.
The mirror relation (\ref{eq: mirror relation 2})
then tells us that $z_{n_{j-l}}^{j-l}$ is non-neutralized L-mobile in
$\left\langle z_0^j, z_1^j\right]$.
Thus, $i_* = 0$ and $\left\lfloor\frac{k}{d}\right\rfloor =2$.

We next claim that
\begin{equation}
\label{eq: part (iv), type (1,1), min in (2psi, dq)}
2 \psi < \minq_{j\in\Z/d}\left(z_0^{j+l}-z_0^j\right) < [dq]_{k^2}.
\end{equation}
Line (\ref{eq: main prop (iv), type (1,1), chaperoned R-mobile, L min in})
already implies
$\minq_{j\in\Z/d}\left(z_0^{j+l}-z_0^j\right) 
\in \left\langle \psi, dq\right\rangle$.
Suppose
$\minq_{j\in\Z/d}\left(z_0^{j+l}-z_0^j\right)
\in \left\langle \psi, 2\psi \right\rangle$.
Then, for arbitrary $j_0 \in J_0$, we have
\begin{align}
\label{eq: part (iv), max = min - 2psi}
      \maxq_{j\in\Z/d}\left(z_{n_{j+l-1}}^{j+l-1} - z_0^{j+1}\right)
&=  z_{n_{j_0+l-1}}^{j_0+l-1} - z_0^{j_0+1}
                          \\ \nonumber
&=  \left(z_0^{j_0+l} - \psi\right) - \left(z_0^{j_0} + dq+\psi\right)
                          \\ \nonumber
&=  \maxq_{j\in\Z/d}\left(z_0^{j+l} - z_0^j\right)
       -dq - 2\psi
                          \\ \nonumber
&=  \minq_{j\in\Z/d}\left(z_0^{j+l} - z_0^j\right) - 2\psi
                          \\ \nonumber
&\in \left\langle -\psi, 0\right\rangle,
\end{align}
so that $z_{n_{j+l-1}}^{j+l-1}$ is L-pseudomobile in
$\left\langle z_{n_j}^j, z_0^{j+1} \right]$.
Moreover, for arbitrary $j_1 \in J_1$, we have
\begin{align}
      \minq_{j\in\Z/d}\left(z_0^{j+l}-z_{n_j}^j\right) 
&= z_0^{j_1+l}-z_1^{j_1}
                \\ \nonumber
&= \minq_{j\in\Z/d}\left(z_0^{j+l}-z_0^j\right) + dq
                \\ \nonumber
&\in \left\langle \psi + dq, 2dq\right\rangle,
\end{align}
which, since $2[dq]_{k^2} < k^2$, has no intersection with
$\left\langle 0, \psi \right\rangle$.  Thus 
$z_0^{j+l}$ is not R-pseudomobile in $\left\langle z_{n_j}^j, z_0^{j+1} \right]$,
but this means that $z_{n_{j+l-1}}^{j+l-1}$ is L-pseudomobile in
$\left\langle z_{n_j}^j, z_0^{j+1} \right]$, a contradiction.
Thus (\ref{eq: part (iv), type (1,1), min in (2psi, dq)}) must be true.

This, in turn, implies that
$z_{n_{j+l-1}}^{j+l-1}$ is L-mobile in $\left\langle z_0^{j+1}, z_1^{j+1}\right]$.
That is, 
\begin{align}
      \maxq_{j\in\Z/d}\left(z_{n_{j+l-1}}^{j+l-1} - z_1^{j+1}\right)
&= \maxq_{j\in\Z/d}\left(z_{n_{j+l-1}}^{j+l-1} - z_0^{j+1}\right) - dq
                          \\ \nonumber
&=  \minq_{j\in\Z/d}\left(z_0^{j+l} - z_0^j\right) - 2\psi - dq
                          \\ \nonumber
&\in \left\langle -dq, -2\psi\right\rangle,
\end{align}
where the second line used (\ref{eq: part (iv), max = min - 2psi})
and the third line used (\ref{eq: part (iv), type (1,1), min in (2psi, dq)}).
Moreover, since
\begin{align}
      \minq_{j\in\Z/d}\left(z_0^{j+l}-z_0^{j+1}\right)
&= \maxq_{j\in\Z/d}\left(z_0^{j+l}-z_1^{j+1}\right)
            \\ \nonumber
&= \maxq_{j\in\Z/d}\left(z_{n_{j+l-1}}^{j+l-1} - z_1^{j+1}\right) + \psi
            \\ \nonumber
&\in \left\langle -dq + \psi, -\psi\right\rangle,
\end{align}
we know that $z_0^{j+l}$ is not R-mobile in
$\left\langle z_0^{j+1}, z_1^{j+1}\right]$.
Thus, $z_{n_{j+l-1}}^{j+l-1}$ is non-neutralized
L-mobile in $\left\langle z_0^{j+1}, z_1^{j+1}\right]$,
or equivalently,
$z_{n_{j+l-2}}^{j+l-2}$ is non-neutralized L-mobile in
$\left\langle z_0^j, z_1^j\right]$.  We already know, however, that
$z_{n_{j-l}}^{j-l}$ is non-neutralized L-mobile in
$\left\langle z_0^j, z_1^j\right]$.  Thus,
$l-2 \equiv -l\; (\mod d)$, so that
\begin{equation}
\label{part iv: eq: 2l = 2}
2l \equiv 2\;(\mod d).
\end{equation}

We claim that this implies $c=1$.
That is, since $z_0^{j+l}$ and $z_{n_{j-l}}^{j-l}$ are mobile
in $\left\langle z_0^j, z_1^j\right]$, Part (i') tells us there is a unique
$j_* \in \Z/d$ such that
$v_q\!\left( z_0^{j_*}, z_1^{j_*}\right) = \alpha(k-k^2)$, and
$v_q\!\left( z_0^{j^{\prime}}, z_1^{j^{\prime}}\right) = \alpha(k)$
for all $j^{\prime}\neq j_* \in \Z/k^2$.
Thus, if we define $\chi_R(j^{\prime})$ (respectively 
$\chi_L(j^{\prime})$) to be equal to 1 if $z_0^{j+l}$
(respectively $z_{n_{j-l}}^{j-l}$) is active in 
$\left\langle z_0^j, z_1^j \right]$ at time $j = j^{\prime}$,
and equal to 0 otherwise, then for any $j^{\prime} \in \Z/d$,
\begin{equation}
\label{part iv: eq: how many mobile points are active?}
    \chi_R(j^{\prime}) + \chi_L(j^{\prime})
= \begin{cases}
       1   &   j^{\prime} \neq j_*
              \\
       2    &   j^{\prime} = j_*, \;\alpha = +1
              \\
       0   &   j^{\prime} = j_*,\;\alpha = -1
    \end{cases}.
\end{equation}
This is because a mobile point in
$\left\langle z_0^j, z_1^j \right]$ contributes
$-k^2$ to $v_q\!\left( z_0^{j^{\prime}}, z_1^{j^{\prime}}\right)$
if it is active at time $j=j^{\prime}$ and contributes zero otherwise.
Thus (\ref{part iv: eq: how many mobile points are active?})
implies that
\begin{equation}
\label{part iv: eq: total actives is d+alpha}
\sum_{j\in \Z/d}  \!\chi_R(j) \;+\; \sum_{j\in \Z/d}  \!\chi_L(j) = d+\alpha.
\end{equation}
On the other hand, 
Proposition \ref{prop: mobile point is active [lepsilon] times or [(l-1)epsilon] times}
tells us that
\begin{equation}
\label{part iv: eq: R active le, L active (l-1)e}
\sum_{j\in \Z/d}  \!\chi_R(j) = [l\epsilon]_d,\;\;\;\;\;\;
\sum_{j\in \Z/d}  \!\chi_L(j) = [(l-1)\epsilon]_d.
\end{equation}
Combining 
(\ref{part iv: eq: total actives is d+alpha}) and
(\ref{part iv: eq: R active le, L active (l-1)e}),
we then have
\begin{align}
      d+ \alpha 
&= [l\epsilon]_d + [(l-1)\epsilon]_d
              \\ \nonumber
       \alpha
&\equiv  l\epsilon + (l-1)\epsilon\;\; (\mod d)
              \\ \nonumber
      \gamma\left( \alpha\gamma\,\epsilon^{-1}\right) 
&\equiv  2l-1\;\; (\mod d)
              \\ \nonumber
        \gamma c
&\equiv  (2)-1\;\; (\mod d)
              \\ \nonumber
         c
& = [\gamma]_d,
\end{align}
where the second to last line used the facts that
$c = [\alpha\gamma\,{\epsilon}^{-1}]_d$ and that
$2l \equiv 2\;(\mod d)$, from (\ref{part iv: eq: 2l = 2}).
If $\gamma = -1$, then $c = d-1 \leq \frac{d}{2}$,
implying $d\leq 2$, contradicting the fact that
$\left\lfloor\frac{k}{d}\right\rfloor = 2$.  Thus $\gamma = +1$ and $c=1$,
This, in turn, implies that $(\mu,\gamma) = (1,1)$ and $\epsilon = [\alpha]_d$.

If $\alpha = +1$, then $\epsilon = 1$.  Now,
$n_{j^{\prime}} = \left\lfloor\frac{k}{d}\right\rfloor - {\theta}^{d, \epsilon}(j^{\prime})$
for all $j^{\prime}\in\Z/d$.  Thus, we have
\begin{align}
     1
&= \epsilon
          \\ \nonumber
&= \#\{j^{\prime}\in \Z/d\left\vert\; {\theta}^{d, \epsilon}(j^{\prime})=1  \right.\}
          \\ \nonumber
&= \#\{j^{\prime}\in \Z/d\left\vert\; n_{j^{\prime}}= 
        \textstyle{\left\lfloor\frac{k}{d}\right\rfloor} - 1\right.\}
          \\ \nonumber
&= \#\{J_0 \cup J_1\},
\end{align}
but this contradicts the fact that both $J_0$ and $J_1$ are nonempty.
Thus we are left with the case in which $\alpha = -1$.

The case of $\alpha = -1$ is somewhat more complicated.
We begin by computing $\psi$:
\begin{align}
      \psi
&= \left[(\mu m + \gamma c)k + \alpha - \gamma \textstyle{\frac{ck + \alpha\gamma}{d}}k\right]_d
         \\ \nonumber
&= \left[(m + 1)k -1 - \textstyle{\frac{k -1}{d}}k\right]_d
         \\ \nonumber
&= \left[(m + 1)k -1 - 2k\right]_d
         \\ \nonumber
&= \left[(m - 1)k -1\right]_d,
\end{align}
where the third line used the fact that $\left\lfloor\frac{k}{d}\right\rfloor = 2$.
If $m=1$, then $\psi = k^2-1$, contradicting the fact
that $\psi < [dq]_k^2 < \frac{k^2}{2}$.  Thus $m > 1$ and
$\psi = (m-1)k-1$.
Now,  (\ref{eq: part (iv), type (1,1), min in (2psi, dq)}) also tells us that
$2\psi < [dq]_{k^2}$.  Thus,
\begin{align}
                    \nonumber
         2\psi
&<    (m+1)k -1
                   \\ \nonumber
         2\left((m-1)k-1\right)
&<    (m+1)k -1
                   \\ \nonumber
         mk
&<    3k +1
                   \\ 
         m
&\le   3.
\end{align}
Combining this with the fact that $m>1$ tells us that $m \in \{2,3\}$.

First consider the case in which $m=3$, so that $q = -\frac{k-1}{d}(3k-1) = -2(3k-1)$.
Note that the fact that $\frac{k-1}{d} = 2$ implies that $k$ is odd.
We claim that $q$ is not genus-minimizing, which we shall prove
by showing that $q^{\prime} := -q = 2(3k-1)$ is not genus-minimizing.
Recall that $\tilde{Q}_{q^{\prime}}$ denotes the lift to the integers of the set
$Q_{q^{\prime}} := \{a{q^{\prime}}\left\vert\;a\in\{0, \ldots, k-1\}\right.\}$.
If $k \equiv 0 \; (\mod 3)$, then
\begin{equation}
    \left(\textstyle{\frac{2k}{3}}{q^{\prime}}, \textstyle{\frac{k}{3}}{q^{\prime}}, 0{q^{\prime}}, \textstyle{\frac{5k+3}{6}}{q^{\prime}}, 
    \textstyle{\frac{3k+3}{6}}{q^{\prime}}, \textstyle{\frac{k+3}{6}}{q^{\prime}}\right)
= \left(-\textstyle{\frac{4k}{3}}, -\textstyle{\frac{2k}{3}}, 0, \textstyle{\frac{4k}{3}}\!-\!1, 
    \textstyle{\frac{6k}{3}}\!-\!1, \textstyle{\frac{8k}{3}}\!-\!1\right)
\end{equation}
in $\left(\Z/k^2\right)^6$, and so the interval
$\left\langle -\frac{4k}{3}, \frac{8k}{3} - 1\right]$ contains at least 
5 elements of $\tilde{Q}_{q^{\prime}}$.  Thus
\begin{align}
         -v_{q^{\prime}}\left(\textstyle{\frac{2k}{3}}{q^{\prime}},\textstyle{\frac{k+3}{6}}{q^{\prime}}\right)
&= -\left(\left(\textstyle{\frac{8k}{3}}\!-\!1\right)- \left(-\textstyle{\frac{4k}{3}}\right)\right)k
         \;\;+\;\;\#\!\left(\left\langle -\textstyle{\frac{4k}{3}}, 
          \textstyle{\frac{8k}{3}} \!-\! 1\right]\cap\tilde{Q}_{q^{\prime}}\right) k^2
                     \\ \nonumber
&\ge     -(4k-1)k + 5k^2
                     \\ \nonumber
&=     k(k+1),
\end{align}
and so Proposition \ref{prop: not minimizing if v >= k(k+1)} implies
that ${q^{\prime}}$ is not genus-minimizing.
If $k \equiv -1 \; (\mod 3)$, then
\begin{equation}
    \left(0{q^{\prime}}, \textstyle{\frac{k+1}{6}}{q^{\prime}}, \textstyle{\frac{k+1}{3}}{q^{\prime}}, \textstyle{\frac{k+1}{2}}{q^{\prime}}\right)
= \left(0, \textstyle{\frac{2k-1}{3}}, \textstyle{\frac{4k-2}{3}}, 2k-1\right)
\end{equation}
in $\left(\Z/k^2\right)^4$, and so the interval
$\left\langle 0, 2k-1\right]$ contains at least 
3 elements of $\tilde{Q}_{q^{\prime}}$.  Thus
\begin{align}
         -v_{q^{\prime}}\left(0{q^{\prime}},\textstyle{\frac{k+1}{2}}{q^{\prime}}\right)
&=   -\left((2k-1) - (0)\right)k
         \;\;+\;\;\#\!\left(\left\langle 0, 2k-1\right] \cap \tilde{Q}_{q^{\prime}}\right) k^2
                     \\ \nonumber
&\ge     -(2k-1)k + 3k^2
                     \\ \nonumber
&=     k(k+1),
\end{align}
and so ${q^{\prime}}$ is not genus-minimizing.  Lastly, if $k \equiv 1 \; (\mod 3)$, then
\begin{equation}
    \left(\textstyle{\frac{k-1}{6}}{q^{\prime}}, \textstyle{\frac{5k+1}{6}}{q^{\prime}}, 0{q^{\prime}}, \textstyle{\frac{2k+1}{3}}{q^{\prime}}\right)
= \left(\textstyle{\frac{-4k+1}{3}}, \textstyle{\frac{-2k-1}{3}}, 0,\textstyle{\frac{2k-2}{3}}\right)
\end{equation}
in $\left(\Z/k^2\right)^4$, and so the interval
$\left\langle \textstyle{\frac{-4k+1}{3}}, \textstyle{\frac{2k-2}{3}}\right]$ contains at least 
3 elements of $\tilde{Q}_{q^{\prime}}$.  Thus
\begin{align}
         -v_{q^{\prime}}\left(\textstyle{\frac{k-1}{6}}{q^{\prime}}, \textstyle{\frac{2k+1}{3}}{q^{\prime}}\right)
&=   -\left(\left(\textstyle{\frac{2k-2}{3}}\right) - \left(\textstyle{\frac{-4k+1}{3}}\right)\right)k
         \;\;+\;\;\#\!\left(\left\langle \textstyle{\frac{-4k+1}{3}}, \textstyle{\frac{2k-2}{3}}\right] 
         \cap \tilde{Q}_{q^{\prime}}\right) k^2
                     \\ \nonumber
&\ge     -(2k-1)k + 3k^2
                     \\ \nonumber
&=     k(k+1),
\end{align}
and so ${q^{\prime}}$ is not genus-minimizing.  We have checked all three equivalence classes
of $k$ modulo 3, and so $q= -q^{\prime}$ is not genus-minimizing for any $k$.

This leaves us with the case in which $m=2$, so that
$q = -\frac{k-1}{d}(2k-1)= -2(2k-1)$.
Since $\frac{k-1}{d}=2$, we know that $k \equiv 1\;(\mod 2)$.
Suppose that $k \equiv 3\;(\mod 4)$, and set $q^{\prime} := -q = 4k-2$.
Then
\begin{equation}
    \left(0{q^{\prime}}, \textstyle{\frac{k+1}{4}}{q^{\prime}}, \textstyle{\frac{k+1}{2}}{q^{\prime}}\right)
= \left(0, \textstyle{\frac{k-1}{2}}, k-1 \right)
\end{equation}
in $\left(\Z/k^2\right)^3$, and so the interval
$\left\langle 0, k-1\right]$ contains at least 
2 elements of $\tilde{Q}_{q^{\prime}}$.  Thus
\begin{align}
         -v_{q^{\prime}}\left(0{q^{\prime}},\textstyle{\frac{k+1}{2}}{q^{\prime}}\right)
&=   -\left((k-1) - (0)\right)k
         \;\;+\;\;\#\!\left(\left\langle 0, k-1\right] \cap \tilde{Q}_{q^{\prime}}\right) k^2
                     \\ \nonumber
&\ge     -(k-1)k + 2k^2
                     \\ \nonumber
&=     k(k+1),
\end{align}
and so ${q^{\prime}}$ is not genus-minimizing.

At last, this leaves us with the case in which
$(\mu, \gamma) = (1,1)$, $\alpha = -1$,
$m=2$, $c=1$, and $d=\frac{k-1}{2} \equiv 0\; (\mod 2)$,
so that $q = -2(2k-1)$.  As we shall show later,
in Proposition \ref{prop: bookkeeping for q genus-minimizing},
$q$ is actually genus-minimizing in this case,
so there is no argument we can make to explain it away.

We can still, however, prove that there are no
non-neutralized mobile points in this case.
We do this by showing that
$\left\langle z_0^j, z_1^j\right]$ has only one R-mobile point.
Now, for any $l \neq 0\in\Z/d$, $z_0^{j+l}$ is R-mobile in
$\left\langle z_0^j, z_1^j\right]$ if and only if
$\minq_{j\in\Z/d}\left(z_0^{j+l}-z_0^j\right) \in \left\langle 0, dq\right\rangle$,
We therefore proceed by using
Corollary \ref{cor: q of positive type: combo of lemma and difference eq}
to compute $\minq_{j\in\Z/d}\left(z_0^{j+l}-z_0^j\right) \in \Z/k^2$.
Note that $\epsilon = \left[\alpha\gamma\, c^{-1}\right]_d = d-1$, implying
$[l\epsilon]_d = d-[l]_d$.  We therefore have
\begin{align}
      \minq_{j\in\Z/d}\left(z_0^{j+l}-z_0^j\right)
&= [ml]_d\textstyle{\frac{k^2}{d}} + \left(\textstyle{\frac{[l\epsilon]_d}{d}}-1\right)[dq]_{k^2}
                 \\ \nonumber
&= [ml]_d\textstyle{\frac{k^2}{d}} - \textstyle{\frac{[l]_d}{d}}((m+c)k-1)
                 \\ \nonumber
&= m[l]_dk\left(\textstyle{\frac{k-1}{d}}\right) - [l]_d\textstyle{\frac{ck-1}{d}}
                 \\ \nonumber
&=\left(\textstyle{\frac{k-1}{d}}\right)[l]_d ( mk - 1)
                 \\ \nonumber
&= [l]_d(4k-2).
                 \\ \nonumber
\end{align}
This leaves us with the task of determining which
$[l]_d \in \{1, \ldots, d-1\} = \left\{1, \ldots, \frac{k-3}{2}\right\}$ satisfies
$[l]_d(4k-2) \in \left\langle 0, 3k-1\right\rangle$
(since $[dq]_{k^2} = 3k-1$).
Now, as integers,
\begin{equation}
3k-1 < 4k-2 \leq \;\;(4k-2)x \;\;\leq k^2 - \textstyle{\frac{3k+1}{2}} < k^2
\end{equation}
for all $x \in \left\{1, \ldots, \frac{k+3}{4}-1\right\}$, and 
\begin{equation}
k^2 + 3k-1 < k^2 + 6k + \textstyle{\frac{k-7}{2}} 
\leq \;\;(4k-2)x \;\;\leq  2k^2 - 7k + 3 < 2k^2
\end{equation}
for all $x \in \left\{\frac{k+3}{4}+1, \ldots, \frac{k-3}{2}\right\}$.
Thus, as elements of $\Z/k^2$, $[l]_d (4k-2) \notin \left\langle 0, 3k-1 \right\rangle$
for all $l \in \Z/d \setminus \left\{0, \frac{k+3}{4}\right\}$.
On the other hand, $\frac{k+3}{4} (4k-2) = \frac{5k-3}{2} \in \left\langle 0, 3k-1 \right\rangle$.
Thus $z_0^{j+l}$ is genus-minimizing if and only if
$l = \frac{k+3}{4}$.
This is the answer for $l$ we should have expected, since
by (\ref{part iv: eq: 2l = 2}), we know that
$2l \equiv 2\;(\mod d)$, whose solutions are $l \in \left\{1, \frac{d}{2}+1\right\}$.
We already know that $l\neq 1$, and $\frac{d}{2} + 1 = \frac{k+3}{4}$.
More to the point, we have shown that
$\left\langle z_0^j, z_1^j\right]$ has only one R-mobile point,
and so all mobile points are non-neutralized in this case.

On the other hand, we have shown that
when $\psi < [dq]_{k^2}$ and we do {\em not} have
$(\mu, \gamma) = (1,1)$, $\alpha = -1$,
$m=2$, $c=1$, and $d=\frac{k-1}{2} \equiv 0\; (\mod 2)$,
then $q$ is not genus-minimizing.
             \\

We therefore assume for the remainder of the proof that
$[dq]_{k^2} < \psi < 2[dq]_{k^2}$.
Suppose first that $(\mu,\gamma) \in \{(1,1), (-1,1)\}$, so that $\gamma = +1$.
Then
\begin{equation}
        \psi
\;=\; \left[dq - \gamma\textstyle{\frac{ck+\alpha\gamma}{d}}k\right]_{k^2}
\;=\; \left[dq - \textstyle{\frac{ck+\alpha}{d}}k\right]_{k^2}
\end{equation}
and so the fact that $\psi > [dq]_{k^2}$ implies 
$\psi = [dq]_{k^2} - \frac{ck+\alpha}{d}k + k^2$.
Now, Proposition \ref{prop: properties of parameters d, m, c, alpha, mu, gamma}
tells us that $[dq]_{k^2} < \frac{k^2}{2}$ and $\frac{ck+\alpha}{d} < \frac{k^2}{2}$.
Thus 
\begin{align}
      \psi
&= [dq]_{k^2} - \textstyle{\frac{ck+\alpha}{d}}k + k^2
             \\ \nonumber
&> [dq]_{k^2}  + \textstyle{\frac{k^2}{2}}
             \\ \nonumber
&> 2[dq]_{k^2}.
\end{align}

This leaves us with the case in which 
$[dq]_{k^2} < \psi < 2[dq]_{k^2}$ and
$(\mu,\gamma) = (1,-1)$.
We first claim that $v_q(z_{n_{j-1}}^{j-1}, z_0^j)$ is constant in $j \in \Z/d$.
Suppose this is not the case.
Then by Part (i$\psi$) and Part (ii$\psi$), there exists a non-neutralized 
R-pseudomobile point, say  $z_0^{j+l}$,
in $\left\langle z_{n_{j-1}}^{j-1}, z_0^j\right]$.
If $\minq_{j\in\Z/d}\left(z_0^{j+l} - z_{n_{j-1}}^{j-1}\right) \in
\left\langle \psi - dq, \psi \right\rangle$,
then $z^{j+l-1}_{n_{j+l-1}}$ is L-pseudomobile 
in $\left\langle z_{n_{j-1}}^{j-1}, z_0^j\right]$, contradicting the supposition
that $z_0^{j+l}$ is non-neutralized as an R-pseudomobile point
in $\left\langle z_{n_{j-1}}^{j-1}, z_0^j\right]$.
Thus $\minq_{j\in\Z/d}\left(z_0^{j+l} - z_{n_{j-1}}^{j-1}\right) \in
\left\langle0, \psi - dq \right\rangle$.
This, however, implies
$z^{j+l-1}_{n_{j+l-1}}$ is L-mobile in
$\left\langle z_{n_{j-1} - 1}^{j-1}, z_{n_{j-1}}^{j-1}\right]$, since
$\maxq_{j\in\Z/d}\left(z_{n_{j+l-1}}^{j+l-1} - z_{n_{j-1}}^{j-1}\right)
\in \left\langle -\psi + dq , 0\right\rangle$.
Moreover, since 
$\minq_{j\in\Z/d}\left(z_0^{j+l} - z_{n_{j-1}-1}^{j-1}\right)
= \minq_{j\in\Z/d}\left(z_0^{j+l} - z_{n_{j-1}}^{j-1}\right) + dq
\in \left\langle dq, \psi \right\rangle$,
we know that
$\minq_{j\in\Z/d}\left(z_0^{j+l} - z_{n_{j-1}-1}^{j-1}\right) \notin
\left\langle0, dq\right\rangle$,
and so $z_0^{j+l}$ is {\em not} R-mobile in 
$\left\langle z_{n_{j-1} - 1}^{j-1}, z_{n_{j-1}}^{j-1}\right]$.
Thus $z^{j+l-1}_{n_{j+l-1}}$ is in fact non-neutralized L-mobile in
$\left\langle z_{n_{j-1} - 1}^{j-1}, z_{n_{j-1}}^{j-1}\right]$, implying that
$v_q(z_{n_{j-1}-1}^{j-1}, z_{n_{j-1}}^{j-1})$ is not constant in $j \in \Z/d$.
But this contradicts our initial supposition that
$v_q(z_{n_{j-1}}^{j-1}, z_0^j)$ is not constant in $j \in \Z/d$.
Thus our initial supposition must have been false, and so
$v_q(z_{n_{j-1}}^{j-1}, z_0^j)$ is constant in $j \in \Z/d$.

Part (ii') therefore tells us that there exist unique $l \in \Z/d$ and
$i_* \in \left\{ 0, \ldots, \left\lfloor \frac{k}{d} \right\rfloor - 2 \right\}$
for which $z_0^{j+l}$ is non-neutralized R-mobile in
$\left\langle z_{i_*}^j, z_{{i_*}+1}^j \right]$,
and $z_{n_{j-l}}^{j-l}$ is non-neutralized L-mobile in
$\left\langle z_{n_j - (i_*+1)}^j, z_{n_j - i_*}^j \right]$, with
$\left(z_{i_*}^{j_*}, z_{{i_*}+1}^{j_*}\right) 
= \left(z_{n_j - (i_*+1)}^{j_*}, z_{n_j - i_*}^{j_*}\right)
= (x_*,y_*)$, for some unique $j_* \in \Z/d$,
where $x_*, y_* \in Q_q$ are the unique elements of $Q_q$
satisfying $v_q(x_*,y_*) = \alpha(k-k^2)$.
In particular, $i_* = \frac{n_{j_*}-1}{2}$.
Now, if $\minq_{j\in\Z/d}\left(z_0^{j+l}-z_{i_*}^j\right)
\in \left\langle \psi - dq,\, dq \right\rangle$, then
\begin{align}
      \maxq_{j\in\Z/d}\left(z_{n_{j+l-1}}^{j+l-1} - z_{{i_*}+1}^j\right)
&= dq + \minq_{j\in\Z/d}\left(z_0^{j+l}-z_{i_*}^j\right) - \psi - dq
             \\ \nonumber
&\in \left\langle - dq,\, dq-\psi \right\rangle
             \\ \nonumber
&\subset \left\langle - dq,\, 0\right\rangle,
\end{align}
making $z_{n_{j+l-1}}^{j+l-1}$ L-mobile in
$\left\langle z_{i_*}^j, z_{{i_*}+1}^j \right]$, so that
$z_0^{j+l}$ is in fact neutralized R-mobile in $\left\langle z_{i_*}^j, z_{{i_*}+1}^j \right]$, 
a contradiction.
Thus $\minq_{j\in\Z/d}\left(z_0^{j+l} - z_{i_*}^j\right) \in
\left\langle 0, \psi - dq\right\rangle$, which implies
\begin{align}
\label{eq: part (iv), type (1,-1), max in (-dq, 0)}
      \maxq_{j\in\Z/d}\left(z_{n_{j+l-1}}^{j+l-1} - z_{i_*}^j\right)
&= dq + \minq_{j\in\Z/d}\left(z_0^{j+l}-z_{i_*}^j\right) - \psi
             \\ \nonumber
&\in \left\langle  dq - \psi, 0 \right\rangle
             \\ \nonumber
&\subset \left\langle - dq,\, 0\right\rangle.
\end{align}
If $\left\lfloor \frac{k}{d} \right\rfloor > 2$, so that ${i_*} \neq 0$,
then this makes $z_{n_{j+l-1}}^{j+l-1}$ non-neutralized L-mobile in
$\left\langle z_{{i_*}-1}^j, z_{i_*}^j \right]$ (a contradiction), since
$z_{n_{j+l-1}}^{j+l-1}$ is L-mobile in $\left\langle z_{{i_*}-1}^j, z_{i_*}^j \right]$
and since
$\minq_{j\in\Z/d}\left(z_0^{j+l}-z_{{i_*}-1}^j\right)
=\minq_{j\in\Z/d}\left(z_0^{j+l} - z_{i_*}^j\right) + dq
\in \left\langle dq, \psi \right\rangle$
implies that $z_0^{j+l}$ is {\em{not}} R-mobile in 
$\left\langle z_{{i_*}-1}^j, z_{i_*}^j \right]$.
Thus $\left\lfloor \frac{k}{d} \right\rfloor = 2$ and $i_*=0$.

Equation (\ref{eq: part (iv), type (1,-1), max in (-dq, 0)})
then tells us that
$\maxq_{j\in\Z/d}\left(z_{n_{j+l-1}}^{j+l-1} - z_0^j\right) \in
\left\langle dq-\psi,\, 0\right\rangle
\subset \left\langle -\psi, 0 \right\rangle$,
so that $z_{n_{j+l-1}}^{j+l-1}$ is
L-pseudomobile in
$\left\langle z_{n_{j-1}}^{j-1}, z_0^j \right]$.
We then have
\begin{align}
      \minq_{j\in\Z/d}\left(z_0^{j+l}-z_{n_{j-1}}^{j-1}\right)
&= \minq_{j\in\Z/d}\left(\left(z_{n_{j+l-1}}^{j+l-1}+\psi\right)
      -\left(z_0^j - \psi\right)\right)
             \\ \nonumber
&= \maxq_{j\in\Z/d}\left(z_{n_{j+l-1}}^{j+l-1} -z_0^j\right) + 2\psi - dq
             \\ \nonumber
&\in \left\langle  \psi, 2\psi -dq\right\rangle.
\end{align}
If $2\psi -[dq]_{k^2} < k^2$,
then $\left\langle 0, \psi \right\rangle \cap \left\langle  \psi, 2\psi -dq\right\rangle = \emptyset$,
and so $z_0^{j+l}$ is not R-pseudomobile in
$\left\langle z_{n_{j-1}}^{j-1}, z_0^j \right]$,
which means that 
$z_{n_{j+l-1}}^{j+l-1}$ is
non-neutralized L-pseudomobile in
$\left\langle z_{n_{j-1}}^{j-1}, z_0^j \right]$, a contradiction.
Thus $2\psi -[dq]_{k^2} > k^2$, but this, in turn, implies that
\begin{equation}
k^2 < \psi + \psi - [dq]_{k^2} \;\;<\;\; \psi + [dq]_{k^2} \;\;<\;\; k^2 + [dq]_{k^2},
\end{equation}
so that $\psi + dq \in \left\langle 0, dq \right\rangle$.
Now, since $\left\lfloor \frac{k}{d} \right\rfloor = 2$,
$\minq_{j\in\Z/d}\left(z_0^{j+1}-z_0^j\right) = \psi + dq$,
and so the fact that $\psi + dq \in \left\langle 0, dq \right\rangle$
means that $z_0^{j+1}$ is R-mobile in
$\left\langle z_0^j, z_1^j \right]$.
On the other hand,
$\maxq_{j\in\Z/d}\left(z_{n_j}^j-z_1^j\right) = dq \notin \left\langle -dq, 0 \right\rangle$,
and so $z_{n_j}^j$ is not L-mobile in
$\left\langle z_0^j, z_1^j \right]$.
Thus $z_0^{j+1}$ is non-neutralized R-mobile in
$\left\langle z_0^j, z_1^j \right]$.

Since $z_0^{j+1}-z_{n_j}^j$ is constant in $j\in\Z/d$,
we can then express $\Z/d$ as the disjoint union of $J_1$ and $J_2$, where
\begin{align}
       J_1
&:= \left\{j^{\prime} \in \Z/d \left|\;  
       n_{j^{\prime}}=1;\;
       z_0^{j^{\prime}+1} \in
       \left\langle z_0^{j^{\prime}}, z_1^{j^{\prime}}\right]
       \right.\!\right\},
           \\
       J_2
&:= \left\{j^{\prime} \in \Z/d \left|\;  
       n_{j^{\prime}} = 2;\;
       z_0^{j^{\prime}+1} \in
       \left\langle z_1^{j^{\prime}}, z_2^{j^{\prime}}\right]
       \right.\!\right\}.
\end{align}
In this case, $n_j := 2 - {\theta}^{d, \epsilon}(j)$,
and the definition of ${\theta}^{d, \epsilon}(j)$, or alternatively, the $l=1$
case of Lemma \ref{lemma: q of positive type, xi lemma}, implies that
$\#\left\{ j^{\prime}\in\Z/d\left|\; {\theta}^{d, \epsilon}(j^{\prime})=1\right.\right\} = \epsilon$.
Thus $|J_1| = \epsilon$.
Since $z_{n_{j-1}}^{j-1} - z_0^j$ is constant in $j\in\Z/d$,
$z_{n_{j-1}}^{j-1}$ is not L-mobile in 
$\left\langle z_0^j, z_0^j\right]$.
Thus $z_0^{j+1}$ is the only mobile point in
$\left\langle z_0^j, z_1^j\right]$,
and it is active at time $j=j^{\prime}\in \Z/d$
if and only if $j^{\prime} \in J_1$.
Thus, if $\alpha = +1$, then
$z_0^{j+l}$ is active in $\left\langle z_0^j, z_1^j \right]$
precisely once, so that $\epsilon = |J_1| = 1$.
On the other hand, if $\alpha = -1$, then
$z_0^{j+l}$ is inactive in $\left\langle z_0^j, z_1^j \right]$
precisely once, so that $\epsilon = |J_1| = d-1$;

In either case, 
$\epsilon = [\alpha]_d$, and so
$c = \left[\alpha\gamma\, {\epsilon}^{-1}\right]_d = [\gamma]_d = d-1$.
The fact that $c \leq \frac{d}{2}$ then implies $d=2$,
but this contradicts the fact that
$\left\lfloor \frac{k}{d} \right\rfloor = 2$,
and so $\psi \notin  \left\langle dq, 2dq\right\rangle$.
             \\

Now that we have shown that 
$\psi > 2[dq]_{k^2}$ except in the special case mentioned
in the statement of Part (iv) (in which we have already shown that ${\mathbf{z}}^j$ has
no neutralized mobile points), it remains to show that
$\psi > 2[dq]_{k^2}$ implies
that all mobile points are non-neutralized.

Suppose that $\psi > 2[dq]_{k^2}$, and that
there exist $l \neq 0 \in \Z/d$ and
$i \in \left\{0, \ldots, \left\lfloor\frac{k}{d}\right\rfloor - 2\right\}$
for which $z_0^{j+l}$ is R-mobile in
$\left\langle z_i^j, z_{i+1}^j\right]$.
Then $\minq_{j\in\Z/d}\left(z_0^{j+l}-z_i^j\right) \in \left\langle 0, dq\right\rangle$, and so
\begin{align}
      \maxq_{j\in\Z/d}\left(z_{n_{j+l-1}}^{j+l-1}-z_{i+1}^j\right)
&=  \minq_{j\in\Z/d}\left(\left(z_0^{j+l}-\psi\right)-\left(z_i^j + dq\right)\right)+dq
                   \\ \nonumber
&= \minq_{j\in\Z/d}\left(z_0^{j+l}-z_i^j\right) -\psi
                   \\ \nonumber
&\in \left\langle -\psi, dq - \psi \right\rangle,
\end{align}
which, since $[dq]_{k^2}-\psi < -[dq]_{k^2}$,
has no intersection with $\left\langle -dq, 0 \right\rangle$.
Thus $z_{n_{j+l-1}}^{j+l-1}$ is not L-mobile in
$\left\langle z_i^j, z_{i+1}^j\right]$, and so
$z_0^{j+l}$ is non-neutralized R-mobile in
$\left\langle z_i^j, z_{i+1}^j\right]$.
A completely analogous argument shows that
if $\psi > 2[dq]_{k^2}$, and
$z_0^{j+l}$ is R-mobile in 
$\left\langle z_{n_j-(i+1)}^j, z_{n_j-i}^j \right]$
for some $l \neq 1 \in \Z/d$ and
$i \in \left\{0, \ldots, \left\lfloor\frac{k}{d}\right\rfloor - 2\right\}$,
then $z_0^{j+l}$ is non-neutralized R-mobile in 
$\left\langle z_{n_j-(i+1)}^j, z_{n_j-i}^j \right]$.
Thus ${\mathbf{z}}^j$ has no neutralized R-mobile points,
which, in turn, implies that ${\mathbf{z}}^j$ has no
neutralized L-mobile points,
and so all mobile points are non-neutralized.

\end{proof}

\begin{cor}
\label{cor: if z^j has mobile points, then (k/d)dq < k^2}
$\left\lfloor\frac{k}{d}\right\rfloor\! [dq]_{k^2} < k^2$.
\end{cor}
\begin{proof}
Suppose that $\left\lfloor\frac{k}{d}\right\rfloor \![dq]_{k^2} > k^2$.
Then, since $2[dq]_{k^2} < k^2$, this implies that
$z_0^{j^{\prime}} \in
\left\langle z_0^{j^{\prime}} + 2dq,\,
z_0^{j^{\prime}} + \left\lfloor\frac{k}{d}\right\rfloor\! dq \right\rangle$
for all $j^{\prime} \in \Z/d$.
Thus, if we choose any $j_0 \in \Z/d$ for which 
$n_{j_0} = \left\lfloor\frac{k}{d}\right\rfloor$, then
there exists $i \in \left\{0, \ldots, \left\lfloor\frac{k}{d}\right\rfloor -2\right\}$
for which
$z_0^{j_0} \in
\left\langle z_{n_{j_0}-(i+1)}^{j_0}, z_{n_{j_0}-i}^{j_0} \right]$.
This, in turn, implies that
$\min_{j\in\Z/d}\left(z_0^j - z_{n_j-(i+1)}^j\right)
\in \left\langle 0, dq \right\rangle$, so that 
$z_0^j$ is R-mobile in 
$\left\langle z_{n_j-(i+1)}^j, z_{n_j-i}^j \right]$.
Proposition \ref{prop: positive type, main prop}.(ii)
then implies that
$z_0^j$ is R-mobile rel $z_0^j$, but this is a contradiction,
so our supposition that $\left\lfloor\frac{k}{d}\right\rfloor \![dq]_{k^2} > k^2$
must have been false.
\end{proof}

Before preceding to other results, we pause to introduce one more
item of terminology, which we have postponed until now in order
to keep Proposition \ref{prop: positive type, main prop}
from becoming any more unwieldy than it already is.

So far, we have only described pseudomobile points
in terms of the interval $\left\langle z_{n_{j-1}}^{j-1}, z_0^j\right]$.
In some cases, however,
it turns out to be more convenient to
consider the interval $\left\langle z_0^j, z_{n_{j-1}}^{j-1}\right]$.
We therefore fix
\begin{equation}
   \psibar
:= \left[z_{n_{j-1}}^{j-1} - z_0^j\right]_{k^2}
= k^2 - \psi,
\end{equation}
and introduce the notion of an {\em antipseudomobile} point.
For any $l \neq 0 \in \Z/d$, we say that
$z_0^{j+l}$ is R-antipseudomobile (or R$\psibar$-mobile) in
$\left\langle z_0^j, z_{n_{j-1}}^{j-1} \right]$ if
\begin{equation}
      0
\;<\; \minq_{j\in\Z/d} \left( z_0^{j+l} - z_0^j \right)
\;<\; \psibar,
\end{equation}
and that $z_{n_{j-1-l}}^{j-1-l}$ is
L-pseudomobile (or L$\psibar$-mobile) in
$\left\langle z_0^j, z_{n_{j-1}}^{j-1} \right]$ if
\begin{equation}
      -\psibar
\;<\; \maxq_{j\in\Z/d} \left( z_{n_{j-1-l}}^{j-1-l}- z_{n_{j-1}}^{j-1} \right)
\;<\; 0.
\end{equation}
We say that a point is antipseudomobile (or $\psibar$-mobile) in
$\left\langle z_0^j, z_{n_{j-1}}^{j-1} \right]$ if it is
R$\psibar$-mobile or L$\psibar$-mobile in
$\left\langle z_0^j, z_{n_{j-1}}^{j-1} \right]$.
(To be consistent in our method of abbreviation, we shall also sometimes say
$\psi$-mobile instead of pseudomobile.)
Corollary \ref{cor: q of positive type: combo of lemma and difference eq}
tells us that
\begin{equation}
\label{eq: psibar mirror relation}
    \minq_{j\in\Z/d} \left( z_0^{j+l} - z_0^j \right)
= -\maxq_{j\in\Z/d} \left( z_{n_{j-1-l}}^{j-1-l} - z_{n_{j-1}}^{j-1}\right)
\end{equation}
for all nonzero $l\in\Z/d$.
We therefore say that $z_0^{j+l}$ and
$z_{n_{j-1-l}}^{j-1-l}$ are mirror $\psibar$-mobile points
in $\left\langle z_0^j, z_{n_{j-1}}^{j-1} \right]$.

If $z_0^{j+l}$ is R$\psibar$-mobile in
$\left\langle z_0^j, z_{n_{j-1}}^{j-1} \right]$,
then we say that $z_0^{j+l}$ is active
at time $j=j^{\prime} \in \Z/d$ if 
\begin{equation}
      z_0^{j^{\prime}+l} - z_0^{j^{\prime}}
\;=\; \minq_{j\in\Z/d} \left( z_0^{j+l} - z_0^j\right),
\end{equation}
and is inactive otherwise.
If $z_{n_{j-1-l}}^{j-1-l}$ is L$\psibar$-mobile in
$\left\langle z_0^j, z_{n_{j-1}}^{j-1} \right]$,
then we say $z_{n_{j-1-l}}^{j-1-l}$ is active at time $j = j^{\prime}\in\Z/d$ if
\begin{equation}
      z_{n_{j^{\prime}-1-l}}^{j^{\prime}-1-l} -  z_{n_{j^{\prime}-1}}^{j^{\prime}-1}
\;=\; \maxq_{j\in\Z/d} \left( z_{n_{j-1-l}}^{j-1-l} - z_{n_{j-1}}^{j-1} \right),
\end{equation}
and is inactive otherwise.
We say that an R$\psibar$-mobile point $z_0^{j+l}$ in
$\left\langle z_0^j, z_{n_{j-1}}^{j-1} \right]$
is {\em neutralized} if
$z_{n_{j+l-1}}^{j+l-1}$ is L$\psibar$-mobile in
$\left\langle z_0^j, z_{n_{j-1}}^{j-1} \right]$,
and that an L$\psibar$-mobile point $z_{n_{j+l}}^{j+l}$ in
$\left\langle z_0^j, z_{n_{j-1}}^{j-1} \right]$
is neutralized if
$z_0^{j+l+1}$ is R$\psibar$-mobile in
$\left\langle z_0^j, z_{n_{j-1}}^{j-1} \right]$.
We say a $\psibar$-mobile point is {\em non-neutralized}
if it is not neutralized.

All of our main results for pseudomobile points have analogs for
antipseudomobile points.

\begin{prop}
\label{prop: antipseudomobile point is active [lepsilon] times}
Suppose that $q$ of positive type is genus-minimizing.
If $z_0^{j+l}$ is R-antipseudomobile (and hence
$z_{n_{j-1-l}}^{j-1-l}$ is L-antipseudomobile)
in $\left\langle z_0^j , z_{n_{j-1}}^{j-1} \right]$,
then each of the two antipseudomobile points is active precisely
$[l\epsilon]_d$ times.
\end{prop}
\begin{proof}
For all $j\in \Z/d$, 
(\ref{eq: z_0^j+l - z_0^j = mu ml k^2/d + xi dq}) implies that
\begin{align}
\label{prop: antipseudomobile active le times. eq: explicit equation for z_0^j+l -z_0^j}
      z^{j+l}_0 - z_0^j
&= \left[{\mu}ml\right]_d\frac{k^2}{d}
       \;+\; \left(\Xi^{d, \epsilon}_l(j)\right) [dq]_{k^2},
                  \\ \nonumber
      z_{n_{j-1-l}}^{j-1-l} - z_{n_{j-1}}^{j-1}
&= -\left[{\mu}ml\right]_d\frac{k^2}{d}
       \;-\; \left(\Xi^{d, \epsilon}_l(j-l)\right) [dq]_{k^2},
\end{align}
where, by
Lemma \ref{lemma: q of positive type, xi lemma},
$\Xi^{d, \epsilon}_l(j) \in \left\{ \frac{[l\epsilon]_d}{d},  \frac{[l\epsilon]_d}{d}-1 \right\}$
for all $j\in\Z/d$, with $\frac{[l\epsilon]_d}{d}$ occurring
$[-l\epsilon]_d$ times and $\frac{[l\epsilon]_d}{d} - 1$
occurring $[l\epsilon]_d$ times.
Since (\ref{prop: antipseudomobile active le times. eq: explicit equation for z_0^j+l -z_0^j})
implies $z_0^{j+l}$ (respectively $z_{n_{j-1-l}}^{j-1-l}$)
is active in $\left\langle z_0^j , z_{n_{j-1}}^{j-1} \right]$
at time $j = j^{\prime}$ if and only if
$\Xi^{d, \epsilon}_l(j^{\prime}) = \frac{[l\epsilon]_d}{d}-1$
(respectively $\Xi^{d, \epsilon}_l(j^{\prime}-l) = \frac{[l\epsilon]_d}{d}-1$),
we conclude that each of the two antipseudomobile points is active
precisely $[l\epsilon]_d$ times.
\end{proof}

\begin{prop}
\label{prop: active iff inactive means neutralized for antipseudomobile}
Suppose that $q$ of positive type is genus-minimizing, and
that $z_{n_{j+l_1}}^{j+l_1}$ and $z_0^{j+l_2}\!$ are antipseudomobile in
$\left\langle z_0^j, z_{n_{j-1}}^{j-1} \right]$
for some $l_1, l_2 \in \Z/d$.  Then
$z_{n_{j+l_1}}^{j+l_1}$ and $z_0^{j+l_2}$ form a neutralized pair
({\em i.e.}, $l_1 + 1 = l_2$) if and only if they satisfy
\begin{equation}
\nonumber
z_{n_{j+l_1}}^{j+l_1}\;\text{is active}
\;\;\;\;\;\Leftrightarrow\;\;\;\;\;
z_0^{j+l_2}\;\text{is inactive}
\end{equation}
in $\left\langle z_0^j, z_{n_{j-1}}^{j-1} \right]$ at all times $j=j^{\prime} \in \Z/d$.
\end{prop}
\begin{proof}
The result is proved by adapting
Proposition \ref{prop: active iff inactive means neutralized for pseudomobile}---the
analogous result for pseudomobile points---in an obvious manner.
\end{proof}

\begin{prop}
\label{prop: antipseudomobile analog of main prop}
Suppose $q$ of positive type is genus-minimizing,
and let $x_*, y_*$ denote the unique elements of $Q_q$ for which
$v_q(x_*, y_*) = {\alpha}(k-k^2)$.  Then the following are true:
\begin{itemize}
\item[(i$\psibar$)]
If the interval
$\left\langle z_0^j, z_{n_{j-1}}^{j-1} \right]$
has any non-neutralized antipseudomobile points, then
there exists a unique $j_* \in \Z/d$ such that
$\left(z_0^{j_*}, z_{n_{j_*-1}}^{j_*-1} \right) = \left(y_*, x_*\right)$.

\item[(ii$\psibar$)]
If $v_q(z_0^j, z_{n_{j-1}}^{j-1})$ is nonconstant in $j\in\Z/d$, then
$\left\langle z_0^j, z_{n_{j-1}}^{j-1} \right]$ has precisely one
non-neutralized R$\psibar$-mobile point and precisely one
non-neutralized L$\psibar$-mobile point, namely,
$z_0^{j+l}$ and $z_{n_{j-1-l}}^{j-1-l}$ for some nonzero $l\in \Z/d$.
\end{itemize}
\end{prop}
\begin{proof}
The results of Parts (i$\psibar$) and (ii$\psibar$) are proved by taking the
respective proofs of Proposition \ref{prop: positive type, main prop},
Parts (i$\psi$) and (ii$\psi$), and making the following adaptations.
Mostly, one must replace the word ``pseudomobile'' with the word ``antipseudomobile''
and replace $z_{n_{j-1}}^{j-1}$ with $z_0^j$ and {\em vice versa}.
Since $\psibar \equiv -dq \equiv -\alpha\;(\mod k^2)$,
one must also replace $\alpha$ with $-\alpha$.
Lastly, one must replace all references to
Proposition \ref{prop: active iff inactive means neutralized for pseudomobile}.
with references to
Proposition \ref{prop: active iff inactive means neutralized for antipseudomobile}.
\end{proof}

Now that we have introduced the antipseudomobile point, 
we resume our task of tabulating results useful for the
classfication of genus-minimizing $q$ of positive type.
We begin with a new result about antipseudomobile points.

\begin{prop}
\label{prop: psibar < k^2/2 implies psibar-mobile are nn}
Suppose $q$ of positive type is genus-minimizing.
If $\psibar < \frac{k^2}{2}$, then all antipseudomobile points
are non-neutralized.
\end{prop}
\begin{proof}
We begin by reducing the problem to the question of whether
$\left\langle z_{n_j-1}^j, z_{n_j}^j\right]$
has R-mobile points.

Suppose, for some $l\neq 0\in \Z/d$, that $z_0^{j+l}$ is R$\psibar$-mobile in
$\left\langle z_0^j, z_{n_{j-1}}^{j-1} \right]$, or in other words, that
$\minq_{j\in\Z/d}\left(z_0^{j+l} - z_0^j\right) \in \left\langle 0, \psibar\right\rangle$.
If $\minq_{j\in\Z/d}\left(z_0^{j+l} - z_0^j\right) \in \left\langle 0, \psibar-dq\right\rangle$,
then
\begin{align}
\label{psibar < k^2/2 implies NN, eq: min in <0,psibar-dq> implies NN}
      \maxq_{j\in\Z/d}\left(z_{n_{j+l-1}}^{j+l-1} - z_{n_{j-1}}^{j-1}\right)
&= \minq_{j\in\Z/d}\left(\left(z_0^{j+l} + \psibar\right) - \left(z_0^j + \psibar\right)\right) + dq
                 \\ \nonumber
&= \minq_{j\in\Z/d}\left(z_0^{j+l} - z_0^j \right) + dq
                 \\ \nonumber
&\in \left\langle dq, \psibar \right\rangle,
\end{align}
which, since $2\psibar < k^2$, has no intersection with
$\left\langle -\psibar, 0\right\rangle$, so that
$z_{n_{j+l-1}}^{j+l-1}$ is not L$\psibar$-mobile in
$\left\langle z_0^j, z_{n_{j-1}}^{j-1} \right]$.
Thus $z_0^{j+l}$ is non-neutralized R$\psibar$-mobile in
$\left\langle z_0^j, z_{n_{j-1}}^{j-1} \right]$ if
$\minq_{j\in\Z/d}\left(z_0^{j+l} - z_0^j\right) \in \left\langle 0, \psibar-dq\right\rangle$.
This leaves us with the case in which
$\minq_{j\in\Z/d}\left(z_0^{j+l} - z_0^j\right) \in 
\left\langle \psibar-dq, \psibar\right\rangle$,
but this implies that
$\minq_{j\in\Z/d}\left(z_0^{j+l}-z_{n_{j-1}-1}^{j-1}\right) \in
\left\langle 0, dq \right\rangle$,
so that
$z_0^{j+l}$ is R-mobile in
$\left\langle z_{n_{j-1}-1}^{j-1}, z_{n_{j-1}}^{j-1} \right]$.
Thus, if we can show that 
$\left\langle z_{n_j-1}^j, z_{n_j}^j \right]$
has no R-mobile points when $\psibar < \frac{k^2}{2}$,
then we shall have proved the proposition.
           \\

We therefore assume, for the remainder of the proof,
that $\left\langle z_{n_j-1}^j, z_{n_j}^j \right]$ has an R-mobile point.
It is important to note that this implies
\begin{equation}
\label{prop: psibar NN, eq: floor(k/d) = 2}
\textstyle{\left\lfloor \frac{k}{d} \right\rfloor} = 2.
\end{equation}
That is, the mirror relation (\ref{eq: mirror relation 1}) tells us that
$\left\langle z_0^j, z_1^j \right]$ has an L-mobile point, and so
Proposition \ref{prop: positive type, main prop}.(i)
implies that there exists $j_* \in \Z/d$ for which
$\left\langle z_0^{j_*}, z_1^{j_*} \right] = \left\langle z_{n_{j_*}-1}^{j_*}, z_{n_{j_*}}^{j_*} \right]$.
In particular, such $j_*$ must satisfy $n_{j_*} = 1$, implying
$\left\lfloor \frac{k}{d} \right\rfloor = 2$.
Note as well that Proposition \ref{prop: positive type, main prop}.(i)
tells us that the R-mobile point in $\left\langle z_{n_j-1}^j, z_{n_j}^j \right]$,
say, $z_0^{j+l}$, for some $l \neq 1 \in \Z/d$, must be R-mobile rel $z_0^j$,
and so (\ref{prop: psibar NN, eq: floor(k/d) = 2}) implies that such
$z_0^{j+l}$ is also R-mobile in $\left\langle z_0^j, z_1^j \right]$.

We next claim that
$\psibar \in \left\langle [dq]_{k^2}, 2[dq]_{k^2} \right\rangle$.
Recalling that $\psibar :\equiv z_{n_{j-1}}^{j-1}-z_0^j$ for all $j \in \Z/d$,
suppose first that
$\psibar \in \left\langle 0, dq \right\rangle$.
We then have
\begin{align}
      \minq_{j\in\Z/d}\left(z_0^{j+1}-z_0^j\right)
&= \minq_{j\in\Z/d}\left(\left(z_{n_j}^j - \psibar\right) - z_0^j\right)
                  \\ \nonumber
&= z_1^j - \psibar - z_0^j
                  \\ \nonumber
&= dq - \psibar
                  \\ \nonumber
&\in \left\langle 0, dq \right\rangle,
\end{align}
where the second line used the fact that $\left\lfloor \frac{k}{d} \right\rfloor = 2$.
This means that $z_0^{j+1}$ is R-mobile in $\left\langle z_0^j, z_1^j \right]$,
but in the preceding paragraph, we showed that
$\left\langle z_0^j, z_1^j \right]$ already has an R-mobile point
$z_0^{j+l}$ with $l \neq 1$.
Since Proposition \ref{prop: positive type, main prop}.(ii) precludes the existence of
two distinct R-mobile points in $\left\langle z_0^j, z_1^j \right]$,
our supposition that $\psibar \in \left\langle 0, dq \right\rangle$
must have been false.
On the other hand, if $2[dq]_{k^2} < \psibar < \frac{k^2}{2}$,
then this R-mobile point 
$z_0^{j+l}$ in $\left\langle z_0^j, z_1^j \right]$ must satisfy
$\minq_{j\in\Z/d}\left(z_0^{j+l}-z_0^j\right) 
\in \left\langle 0, dq \right\rangle
\subset \left\langle 0, \psibar - dq \right\rangle$,
and so the argument surrounding
(\ref{psibar < k^2/2 implies NN, eq: min in <0,psibar-dq> implies NN})
implies that $z_0^{j+l}$ is non-neutralized
R$\psibar$-mobile in $\left\langle z_0^j, z_{n_{j-1}}^{j-1}\right]$,
another contradiction.  Thus $\psibar \in \left\langle [dq]_{k^2}, 2[dq]_{k^2} \right\rangle$.

We next attempt to determine $m$ by interpreting $\mu m$ as a sort of winding number
of $\maxq_{j\in\Z/d}\left(z_{n_{j-l}}^{j-l}-z_0^j\right)$ around $\Z/k^2$ as $l$ varies.
To make this notion more precise, we define the function
$M : \Z \rightarrow \Z$,
\begin{align}
\label{eq: M(l) definition}
       M(l) 
&:= -\mu m(l-1)\textstyle{\frac{k^2}{d}}
       -\left(\textstyle{\frac{(l-1)\epsilon}{d} - \left\lceil\!\frac{(l-1)\epsilon}{d}\!\right\rceil}\right)[dq]_{k^2}
       + \psibar
                 \\ \nonumber
                &\;\;\;\;\;\;\;\;\;\;\;\;\;\;\;\;\;
       + l\left(\psibar - \left(\textstyle{\frac{\gamma c k + \alpha}{d}}k - [dq]_{k^2}\right)\right).
\end{align}
Since $\psibar = \left[\frac{\gamma c k + \alpha}{d}k - [dq]_{k^2}\right]_{k^2}$,
the term on the second line vanishes when $\gamma = +1$ and is equal to either zero
or $lk^2$ when $\gamma = -1$.
Thus, Corollary \ref{cor: q of positive type: combo of lemma and difference eq}
tells us that
$M(l)$ is an integer lift of
$\maxq_{j\in\Z/d}\left(z_{n_{j-l}}^{j-l}-z_0^j\right) \in \Z/k^2$
whenever $l \not\equiv 1 \; (\mod d)$,
and $M(l)$ is an integer lift of $z_{n_{j-1}}^{j-1}-z_0^j \in \Z/k^2$
when $l \equiv 1 \; (\mod d)$.
For any $l \in \Z$, we can calculate $M(l+1)-M(l)$:
\begin{align}
\label{eq: M(l+1) - M(l)}
        M(l+1)-M(l)
&= -\mu m\textstyle{\frac{k^2}{d}}
       +\left(-\textstyle{\frac{\epsilon}{d} + \left\lceil\frac{l\epsilon}{d}\right\rceil
                                                  - \left\lceil\!\frac{(l-1)\epsilon}{d}\!\right\rceil}\right)[dq]_{k^2}
                 \\ \nonumber
                &\;\;\;\;\;\;\;\;\;\;\;\;\;\;\;\;\;\;\;\;\;\;\;\;\;\;\;\;
       + \psibar - \left(\textstyle{\frac{\gamma c k + \alpha}{d}}k - [dq]_{k^2}\right)
                                      \\ \nonumber
&= -\textstyle{\frac{k+\epsilon}{d}}\left((\mu m + \gamma c)k+\alpha\right)
        + \left( \left\lceil\frac{l\epsilon}{d}\right\rceil 
       - \left\lceil\!\frac{(l-1)\epsilon}{d}\!\right\rceil + 1\right)[dq]_{k^2} + \psibar
                                      \\ \nonumber
&= \psibar  +  \textstyle{\left( \left\lceil\frac{l\epsilon}{d}\right\rceil 
       - \left\lceil\!\frac{(l-1)\epsilon}{d}\!\right\rceil\;-\;2\right)} [dq]_{k^2}
                                      \\ \nonumber
&\in \left\{ \psibar - 2[dq]_{k^2}, \psibar - [dq]_{k^2} \right\}.
\end{align}
Since $[dq]_{k^2} < \psibar < 2[dq]_{k^2}$, (\ref{eq: M(l+1) - M(l)})
implies in particular that
\begin{equation}
\left\vert M(l+1)-M(l)\right\vert \;<\; [dq]_{k^2}\;\;\text{for all}\; l \in \Z.
\end{equation}
Since $[M(1)]_{k^2} = \psi \notin \left\langle 0, [dq]_{k^2} \right\rangle$,
this means that the constraint
\begin{equation}
[M(l)]_{k^2} \in \left\langle 0, [dq]_{k^2} \right\rangle
\end{equation}
has at least $n$ solutions
in $l \neq 1 \in \{0, \ldots, d-1\} \subset \Z$,
where $\left| M(d) - M(0) \right| = nk^2$.
Since $z_{n_{j-l}}^{j-l}$ is L-mobile in
$\left\langle z_0^j, z_1^j\right]$ if and only if
$[M(l)]_{k^2} \in \left\langle 0, [dq]_{k^2} \right\rangle$,
this means that
$M$ must satisfy
$\left|M(d)-M(0)\right| \leq (1)k^2$.
Using (\ref{eq: M(l) definition}), we calculate that
\begin{align}
\label{eq: M(d) - M(0)}
        M(d)-M(0)
&= -\mu m(d)\textstyle{\frac{k^2}{d}}
       -\left(\textstyle{\frac{(d)\epsilon}{d}
       - \left\lceil\!\frac{(d-1)\epsilon}{d}\!\right\rceil
       + \left\lceil\!\frac{(0-1)\epsilon}{d}\!\right\rceil}\right)[dq]_{k^2}
                 \\ \nonumber
                &\;\;\;\;\;\;\;\;\;\;\;\;\;\;\;\;\;
       + (d)\left(\psibar - \left(\textstyle{\frac{\gamma c k + \alpha}{d}}k - [dq]_{k^2}\right)\right)
                                      \\ \nonumber
&= -\mu mk^2  - (\epsilon - \epsilon + 0)[dq]_{k^2}
                 \\ \nonumber
                &\;\;\;\;\;\;\;\;\;\;\;\;\;\;\;\;\;
       + d\left(\psibar - \left(\textstyle{\frac{\gamma c k + \alpha}{d}}k - [dq]_{k^2}\right)\right)
                                      \\ \nonumber
&=  -\mu mk^2 + d\left(\psibar - \left(\textstyle{\frac{\gamma c k + \alpha}{d}}k - [dq]_{k^2}\right) \right)
                                      \\ \nonumber
&= \begin{cases}
         -\mu m k^2
              & \psibar = \textstyle{\frac{\gamma c k + \alpha}{d}}k - [dq]_{k^2}
                        \\
          (-\mu m + d)k^2
              & \psibar = \textstyle{\frac{\gamma c k + \alpha}{d}}k - [dq]_{k^2} + k^2
      \end{cases}.
\end{align}
Thus, if $\gamma = +1$, then
$\psibar = \frac{\gamma c k + \alpha}{d}k - [dq]_{k^2}$, and so
$|\mu m| \leq 1$, implying $m=1$.
If $\gamma = -1$, so that $\mu = +1$, then $m > c > 0$ implies $m \neq 1$, 
so we must have $|d - m| \leq 1$, implying
$m \in \{d-1, d, d+1\}$.  If $m = d$, however, then
$\maxq_{j\in\Z/d}\left(z_{n_{j-l}}^{j-l}-z_0^j\right)
= \frac{[-(l-1)\epsilon]_d}{d}[dq]_{k^2} + \psibar
\in \left\langle 2dq, 3dq \right\rangle$ for all $l \neq 1 \in \Z/d$.
(Here, again, we must first interpret $\frac{[-(l-1)\epsilon]_d}{d}[dq]_{k^2} + \psibar$
as an integer---noting that in this case,
$[dq]_{k^2} = dk - (ck + \alpha\gamma)$ is divisible by $d$---and
only then take its image in $\Z/k^2$.)
Since Proposition \ref{prop: positive type, main prop}.(iv)
tells us that $\psi > 2[dq]_{k^2}$, we know that
$[dq]_{k^2} < \psibar < k^2 - 2[dq]_{k^2}$, implying
$3[dq]_{k^2} < k^2$, so that
$\left\langle 2dq, 3dq \right\rangle \cap  \left\langle 0, dq \right\rangle = \emptyset$,
which means that $\left\langle z_0^j, z_1^j \right]$ has no L-mobile points,
and so $\left\langle z_{n_j-1}^j, z_{n_j}^j \right]$ has no R-mobile points,
a contradiction.
Thus,
\begin{equation}
\label{eq: m = 1 or m in d-1, d+1}
m \in \begin{cases}
             \{1\}
             & \gamma = +1
                     \\
             \{d-1, d+1\}
             & \gamma = -1
          \end{cases},
\end{equation}
and for either value of $\gamma$, we have
\begin{equation}
\label{eq: (mu m)^2 equiv 1(mod d)}
(\mu m)^2 \equiv 1\; (\mod d).
\end{equation}

Now that we have almost determined $m$,
we turn our attention to $c$.
We shall momentarily do away with the case in which $\gamma = -1$,
but when $\gamma = +1$, we can find a lower bound for $c$
that helps us to prove the following claim.

\begin{claim}
\label{claim: k^2/d < dq and l not 1,-1, 2, -2}
Suppose that $\gamma = +1$.
Then $\frac{k^2}{d} < [dq]_{k^2}$, and if
$z_0^{j+l}$ is R-mobile in
$\left\langle  z_{n_j-1}^j, z_{n_j}^j\right]$,
then $l \notin \{0, \pm 1, \pm 2\}$.
\end{claim}

\begin{proof}
Suppose, for some $l \neq 1 \in \Z/d$, that
$z_0^{j+l}$ is R-mobile in
$\left\langle z_{n_j-1}^j, z_{n_j}^j\right]$.
As discussed in the paragraph surrounding (\ref{prop: psibar NN, eq: floor(k/d) = 2}),
this implies both that $\left\lfloor \frac{k}{d} \right\rfloor = 2$ and that
$z_0^{j+l}$ is R-mobile in
$\left\langle z_0^j, z_1^j \right]$.
We can therefore partition $\Z/d$ as the disjoint union of $J_0$, $J_1$, and $J_2$,
where 
\begin{align}
       J_0
&:= \left\{j^{\prime} \in \Z/d \left|\;  
       n_{j^{\prime}} = 1;\;
       z_0^{j+l}\;\text{is inactive in}\;\!
       \left\langle z_0^j, z_1^j\right]
       \!\;\text{when}\;j=j^{\prime}
       \right.\!\right\}\!,
           \\
       J_1
&:= \left\{j^{\prime} \in \Z/d \left|\;  
       n_{j^{\prime}} = 1;\;
       z_0^{j+l}\;\text{is active in}\;\!
       \left\langle z_0^j, z_1^j\right]
       \!\;\text{when}\;j=j^{\prime}
       \right.\!\right\}\!,
           \\
       J_2
&:= \left\{j^{\prime} \in \Z/d \left|\;  
       n_{j^{\prime}} = 2;\;
       z_0^{j+l}\;\text{is inactive in}\;\!
       \left\langle z_0^j, z_1^j\right]
       \!\;\text{when}\;j=j^{\prime}
       \right.\!\right\}\!.
\end{align}
Note that we omitted the only other possibility,
\begin{equation}
       J_3
:= \left\{j^{\prime} \in \Z/d \left|\;  
       n_{j^{\prime}} = 2;\;
       z_0^{j+l}\;\text{is active in}\;\!
       \left\langle z_0^j, z_1^j\right]
       \!\;\text{when}\;j=j^{\prime}
       \right.\!\right\}\!,
\end{equation}
because the nonemptiness of $J_3$ would imply that 
$z_0^{j+l}$ was not R-mobile in 
$\left\langle z_{n_j-1}^j, z_{n_j}^j\right]$.
Observe that $z_0^{j+l}$ is active in
$\left\langle z_0^j, z_1^j\right]$
when and only when $j \in J_1$.  Thus,
Proposition \ref{prop: mobile point is active [lepsilon] times or [(l-1)epsilon] times}
tells us that
$|J_1| = [l\epsilon]_d$.
Similarly, $z_0^{j+l}$ is active in
$\left\langle z_{n_j-1}^j, z_{n_j}^j\right]$
when and only when $j \in \left\{J_1 \cup J_2 \right\}$.  Thus, by
Proposition \ref{prop: mobile point is active [lepsilon] times or [(l-1)epsilon] times},
we have
\begin{align}
\label{eq: [(l-1)e] = [le] + d-e}
      [(l-1)\epsilon]_d 
&= |J_1| + |J_2|
           \\ \nonumber
&= [l\epsilon]_d + d-\epsilon,
\end{align}
where the second line can be deduced either from the fact that
$[(l-1)\epsilon]_d > [l\epsilon]_d$, or from the fact that
$d-\epsilon = \#\left\{j^{\prime}\in\Z/d\left\vert\, 
n_{j^{\prime}} = \left\lfloor\frac{k}{d}\right\rfloor\right.\right\} = |J_2|$.

The fact that $z_0^{j+l}$ is R-mobile in 
$\left\langle z_{n_j-1}^j, z_{n_j}^j\right]$ implies
that its mirror mobile point,
$z_{n_{j-l}}^{j-l}$, is L-mobile in
$\left\langle z_0^j, z_1^j \right]$.
Now, by Proposition \ref{prop: positive type, main prop}, we know that
there is a unique $j_* \in \Z/d$ such that
$v_q\!\left(z_0^{j_*}, z_1^{j_*}\right) = \alpha(k-k^2)$ and
$v_q\!\left(z_0^{j^{\prime}}, z_1^{j^{\prime}}\right) = \alpha(k)$
for all $j^{\prime} \neq j_* \in \Z/d$.
Thus, if we define $\chi_R(j^{\prime})$ (respectively 
$\chi_L(j^{\prime})$) to be equal to 1 if $z_0^{j+l}$
(respectively $z_{n_{j-l}}^{j-l}$) is active in 
$\left\langle z_0^j, z_1^j \right]$ at time $j = j^{\prime}$,
and equal to 0 otherwise, then for any $j^{\prime} \in \Z/d$, we must have
\begin{equation}
\label{eq: chi for mobile points in < z_0^j, z_1^j ]}
    \chi_R(j^{\prime}) + \chi_L(j^{\prime})
= \begin{cases}
        1  &   j^{\prime} \neq j_*
              \\
        2   &   j^{\prime} = j_*, \;\alpha = +1
              \\
        0 &   j^{\prime} = j_*,\;\alpha = -1
    \end{cases}.
\end{equation}
Here, the relative values of $\chi_R(j^{\prime}) + \chi_L(j^{\prime})$ are determined
by the fact that a mobile point in
$\left\langle z_0^j, z_1^j \right]$ contributes
$-k^2$ to $v_q\!\left( z_0^{j^{\prime}}, z_1^{j^{\prime}} \right)$
if it is active at time $j=j^{\prime}$ and contributes zero otherwise.
The exact values are then determined by the fact that
$\chi_R(j^{\prime}), \chi_L(j^{\prime}) \in \{0,1\}$, implying 
$\chi_R(j^{\prime}) + \chi_L(j^{\prime}) \in \{0,1,2\}$.
Equation (\ref{eq: chi for mobile points in < z_0^j, z_1^j ]})
then implies that
\begin{equation}
\label{eq: total active mobile points is d+alpha}
\sum_{j\in \Z/d}  \!\chi_R(j) \;+\; \sum_{j\in \Z/d}  \!\chi_L(j) = d+\alpha.
\end{equation}
On the other hand, Proposition \ref{prop: mobile point is active [lepsilon] times or [(l-1)epsilon] times}
tells us that
\begin{equation}
\label{eq: one active le times, the other active (l-1)e times}
\sum_{j\in \Z/d}  \!\chi_R(j) = [l\epsilon]_d,\;\;\;\;
\sum_{j\in \Z/d}  \!\chi_L(j) = [(l-1)\epsilon]_d.
\end{equation}
Combining 
(\ref{eq: total active mobile points is d+alpha}) and
(\ref{eq: one active le times, the other active (l-1)e times}) yields
\begin{align}
\label{eq: le + (l-1)e = d+ alpha}
      d + \alpha
&= [l\epsilon]_d + [(l-1)\epsilon]_d 
         \\ \nonumber
               \alpha 
&\equiv (2l-1)\epsilon\; (\mod d)
                    \\ \nonumber
               \alpha\gamma \,\epsilon^{-1} 
&\equiv \gamma(2l-1)\; (\mod d)
                     \\
\label{eq: c = 2l-1}
                c
&= [2l-1]_d.
\end{align}
(Recall that $\gamma = +1$.)
Alternatively, we could use (\ref{eq: [(l-1)e] = [le] + d-e})
to solve (\ref{eq: le + (l-1)e = d+ alpha}) for $[l\epsilon]_d$,
obtaining
\begin{align}
            \nonumber
      [l\epsilon]_d + [(l-1)\epsilon]_d 
&= d + \alpha
         \\ \nonumber
      [l\epsilon]_d + [l\epsilon]_d + d-\epsilon
&= d + \alpha
          \\
\label{eq: le = (e+a)/2}
       [l\epsilon]_d
& = \textstyle{\frac{\epsilon + \alpha}{2}}.
\end{align}

Now that we have derived 
(\ref{eq: c = 2l-1})
and
(\ref{eq: le = (e+a)/2}),
we proceed with the task of finding a lower bound for $c$.
Since $\gamma = +1$, we have
$\psibar = \frac{ck+\alpha}{d}k - [dq]_{k^2}$, which means that the constraint
$\psibar \in \left\langle [dq]_{k^2}, 2[dq]_{k^2}\right\rangle$
implies $\frac{ck+\alpha}{d}k \in \left\langle 2[dq]_{k^2}, 3[dq]_{k^2}\right\rangle$.
We can then reexpress $\frac{ck+\alpha}{d}$ as
$\frac{ck+\alpha}{d}
= \left(\frac{k+\epsilon}{d}\right)c - \frac{c\epsilon - \alpha}{d}
= 3c - \frac{c\epsilon - \alpha}{d}$,
where the fact that $\frac{ck+\alpha}{d} \in \Z$ implies
$\frac{c\epsilon - \alpha}{d} \in \Z$, with
$0 \leq \frac{c\epsilon - \alpha}{d} \leq c$.
Since (\ref{eq: m = 1 or m in d-1, d+1}) tells us $m=1$,
implying $[dq]_{k^2} = (c+ \mu)k + \alpha$, we then have
\begin{equation}
\label{eq: 2dq < psibar < 3dq}
(2c+2\mu)k + 2\alpha
\;\;<\;\;
\left(3c - \textstyle{\frac{c\epsilon - \alpha}{d}}\right)k
\;\;<\;\;
(3c+3\mu)k + 3\alpha,
\end{equation}
which, when $\mu = +1$, implies
$c > 2\mu + \frac{c\epsilon - \alpha}{d} + \frac{2\alpha}{k} >1$,
and when $\mu = -1$, implies
$\frac{c\epsilon - \alpha}{d} > -3\mu -\frac{3\alpha}{k} > 2$,
so that $c \geq \frac{c\epsilon - \alpha}{d} > 2$.
Thus, in either case, $c> 1$. 
This means that $\frac{cd-1}{c} > d-1 \geq \epsilon$,
so that $c >  \frac{c\epsilon + 1}{d}  \geq \frac{c\epsilon - \alpha}{d}$,
implying $c \geq \frac{c\epsilon - \alpha}{d} + 1$.
In addition, $c > 1$ implies $\frac{c\epsilon - 1}{d} > 0$,
which, since $\frac{c\epsilon - \alpha}{d} \in \Z$, implies
$\frac{c\epsilon - \alpha}{d} \geq 1$.
Thus, when $\mu = +1$, the left-hand inequality in (\ref{eq: 2dq < psibar < 3dq})
tells us that
\begin{align}
\label{prop: psibar NN, claim 2, eq: mu = +1, c > 2mu + ...}
      c
&> 2\mu + \textstyle{\frac{c\epsilon - \alpha}{d}} \;+\; \textstyle{\frac{2\alpha}{k}}
                 \\ \nonumber
&\geq 3 - \textstyle{\frac{2}{k}},
\end{align}
which, since $\frac{2}{k} < 1$, implies $c \geq 3$.
When $\mu = -1$, the right-hand inequality in (\ref{eq: 2dq < psibar < 3dq}),
along with the fact that $c \geq \frac{c\epsilon - \alpha}{d} + 1$ (proved a few lines ago),
tells us that
\begin{align}
\label{prop: psibar NN, claim 2, eq: mu = -1, c >= 4}
            c
&\geq \textstyle{\frac{c\epsilon - \alpha}{d}} + 1
                 \\ \nonumber
&>       \left( - 3\mu \;-\; \textstyle{\frac{3\alpha}{k}}\right) + 1
                 \\ \nonumber
&\geq       4 \;-\; \textstyle{\frac{3}{k}},
\end{align}
which, since $\frac{3}{k} < 1$, implies $c \geq 4$.
Thus, for either value of $\mu$, we have
\begin{align}
            [dq]_{k^2}
&=       (c+ \mu)k + \alpha
                 \\ \nonumber
&\geq \begin{cases}
                     (3 + 1)k + \alpha     & \mu = +1
                                    \\
                     (4  - 1)k + \alpha     & \mu = -1
            \end{cases}
                 \\ \nonumber
&\geq 3k + \alpha
                 \\ \nonumber
&= \textstyle{\frac{k+\epsilon}{d}}k + \alpha
                 \\ \nonumber
&>         \textstyle{\frac{k^2}{d}}.
\end{align}

It therefore remains to show that $l \notin \{0, \pm 1, \pm 2\}$,
recalling that $z_0^{j+l}$ is R-mobile in
$\left\langle  z_{n_j-1}^j, z_{n_j}^j\right]$.
The R-mobility of  $z_0^{j+l}$ in
$\left\langle  z_{n_j-1}^j, z_{n_j}^j\right]$
requires, by definition, that $l \neq 1$.
It also implies, as shown in (\ref{prop: psibar NN, eq: floor(k/d) = 2}), that
$\left\lfloor\frac{k}{d}\right\rfloor = 2$, which, together
with Proposition \ref{prop: positive type, main prop}.(ii), implies
that $z_0^{j+l}$ is R-mobile in
$\left\langle z_0^j, z_1^j \right]$, requiring that $l \neq 0$.
Since we are taking $k > 100$, the fact that
$\left\lfloor\frac{k}{d}\right\rfloor = 2$ also implies that
$d > \frac{k}{3} > 33$.  If $l = -1$, then by (\ref{eq: c = 2l-1}),
$c =  [2l-1]_d = 2(-1)-1 + d = d-3 > \frac{d}{2}$, a contradiction.
If $l = -2$, then $c = 2(-2)-1 + d = d-5 > \frac{d}{2}$, another contradiction.
Thus, $l \notin  \{0, 1,-1, -2\}$.

This leaves the case of $l = 2$,
with $c = 2(2) -1 = 3$.
As shown in (\ref{prop: psibar NN, claim 2, eq: mu = -1, c >= 4}),
$c \geq 4$ when $\mu = -1$, so we must have $\mu = +1$.
The first line of (\ref{prop: psibar NN, claim 2, eq: mu = +1, c > 2mu + ...})
then gives
\begin{align}
      \textstyle{\frac{c\epsilon - \alpha}{d}}
&< c - 2\mu - \textstyle{\frac{2\alpha}{k}}
                 \\ \nonumber
&= 1 - \textstyle{\frac{2\alpha}{k}},
\end{align}
which, since $\frac{c\epsilon - \alpha}{d} > 0$ is an integer,
implies $\frac{c\epsilon - \alpha}{d} = 1$ and $\alpha = -1$.
This, in turn, implies that $\epsilon = \frac{d-1}{3}$, so that
$[l\epsilon]_d = \frac{2d-2}{3}$ and $\frac{\epsilon+\alpha}{2} = \frac{d-4}{6}$,
contradicting (\ref{eq: le = (e+a)/2}).  Thus, $l \neq 2$.

\end{proof}

Lastly, we attempt to further constrain the values of 
$l \notin  \{0, \pm1, \pm2\} \in \Z/d$ for which
$z_0^{j+l}$ can be R-mobile in both
$\left\langle z_0^j, z_1^j \right]$
and $\left\langle z_{n_j-1}^j, z_{n_j}^j \right]$.
Let $s$ denote the smallest positive integer such that
\begin{equation}
\label{eq: q++, k/d = 2, definition of s}
s\textstyle{\frac{k^2}{d}} > [dq]_{k^2}.
\end{equation}
Since $[dq]_{k^2} < \frac{k^2}{2}$, we know that
$\left\lceil\frac{d}{2}\right\rceil \frac{k^2}{d} > [dq]_{k^2}$,
and so $s \leq \left\lceil\frac{d}{2}\right\rceil$.

\begin{claim}
\label{claim: m mu t_1 and m mu t_2 not R-mob implies m mu(t_1+t_2) R-mob}
Suppose that $t_1$ and $t_2$ are positive integers such that
$t_1< s$, $t_2 < s$, and $s \leq t_1+t_2$.  (Note that this implies $t_1 + t_2 < d$.)
If neither $z_0^{j+\mu mt_1}$
nor $z_0^{j + \mu mt_2}$ is R-mobile in 
$\left\langle z_0^j, z_1^j \right]$, then
$z_0^{j + \mu m(t_1+t_2)}$ is R-mobile in 
$\left\langle z_0^j, z_1^j \right]$.
\end{claim}

\begin{proof}
Suppose that $t_1$ and $t_2$ are positive integers such that
$t_1< s$, $t_2 < s$, and $s \leq t_1+t_2$, and that
neither $z_0^{j+\mu m t_1}$
nor $z_0^{j + \mu m t_2}$ is R-mobile in 
$\left\langle z_0^j, z_1^j \right]$.
Recalling that (\ref{eq: (mu m)^2 equiv 1(mod d)})
tells us that $(\mu m)^2 \equiv 1\; (\mod d)$,
we fix the lift
\begin{align}
      \minq_{j\in\Z/d}\left(z_0^{j+\mu mt}-z_0^j\right)
&= [\mu m(\mu m t)]_d \frac{k^2}{d} + \left(\frac{[(\mu mt)\epsilon]_d}{d}-1\right)[dq]_{k^2}
             \\ \nonumber
&= t \frac{k^2}{d} - \frac{[-(\mu mt)\epsilon]_d}{d}[dq]_{k^2}
\end{align}
of $\minq_{j\in\Z/d}\left(z_0^{j+\mu mt}-z_0^j\right)$ to the integers
for any positive integer $t < d$.  Thus, for any $t < s$,
so that $t\frac{k^2}{d} < [dq]_{k^2}$, we have
$-[dq]_{k^2} <  t \frac{k^2}{d} - \frac{[-(\mu mt)\epsilon]_d}{d}[dq]_{k^2} < [dq]_{k^2}$,
and for any $t\geq s$, so that $t\frac{k^2}{d} > [dq]_{k^2}$, we have
$0 <   t \frac{k^2}{d} - \frac{[-(\mu mt)\epsilon]_d}{d}[dq]_{k^2} < k^2$.

For $i \in \{1,2\}$, we know in addition that $z_0^{j+\mu mt_i}$
is not R-mobile in $\left\langle z_0^j, z_1^j \right]$, and so
\begin{equation}
-[dq]_{k^2} < t_i \frac{k^2}{d} - \frac{[-(\mu mt_i)\epsilon]_d}{d}[dq]_{k^2} < 0.
\end{equation}
Using the fact that
\begin{equation}
   [\mu m(t_1\!+t_2)\epsilon]_d
= \begin{cases}
        [-(\mu mt_1)\epsilon]_d \!+\!  [-(\mu mt_2)\epsilon]_d
           & [-(\mu mt_1)\epsilon]_d \!+\!  [-(\mu mt_2)\epsilon]_d < d
                          \\
        [-(\mu mt_1)\epsilon]_d \!+\!  [-(\mu mt_2)\epsilon]_d \!-\!d
           & [-(\mu m t_1)\epsilon]_d \!+\!  [-(\mu mt_2)\epsilon]_d \geq d
   \end{cases},
\end{equation}
we then obtain that
\begin{align}
          (t_1 + t_2)\frac{k^2}{d} - \frac{[-(\mu m(t_1+t_2))\epsilon]_d}{d}[dq]_{k^2}
&\leq   \left( t_1 \frac{k^2}{d} - \frac{[-(\mu mt_1)\epsilon]_d}{d}[dq]_{k^2}\right)
                     \\ \nonumber
                     &\;\;\;\;
          +\left( t_2 \frac{k^2}{d} - \frac{[-(\mu mt_2)\epsilon]_d}{d}[dq]_{k^2}\right)
          + \frac{d}{d}[dq]_{k^2}
                      \\ \nonumber
&<     0 \;+\; 0 \;+\; 1\!\cdot\![dq]_{k^2}
                      \\ \nonumber
&=    [dq]_{k^2}.
\end{align}
Thus, $\minq_{j\in\Z/d}\left(z_0^{j+\mu m( t_1+t_2)}-z_0^j\right) \in
\left\langle 0, dq\right\rangle$, and so
$z_0^{j + \mu m(t_1+t_2)}$ is R-mobile in 
$\left\langle z_0^j, z_1^j \right]$.
\end{proof}

Suppose that $s \geq 4$.  Then
$\left\lceil \frac{s}{2} \right\rceil  < s$,
$s \leq 2\!\left\lceil \frac{s}{2} \right\rceil < d$,
$\left\lceil \frac{s}{2} \right\rceil + 1 < s$, and
$s \leq 2\!\left(\left\lceil \frac{s}{2} \right\rceil+1\right) < d$,
and so Claim \ref{claim: m mu t_1 and m mu t_2 not R-mob implies m mu(t_1+t_2) R-mob}
implies that there exist
$l_0 \in \left\{ \mu m\! \left\lceil \frac{s}{2} \right\rceil,\, 
2\mu m\! \left\lceil \frac{s}{2} \right\rceil \right\}$
and
$l_1 \in \left\{ \mu m\!\left(\left\lceil \frac{s}{2} \right\rceil +1\right),\, 
2\mu m\!\left(\left\lceil \frac{s}{2} \right\rceil+1\right) \right\}$
for which both
$z_0^{j+l_0}$ and $z_0^{j+l_1}$ are R-mobile in
$\left\langle z_0^j, z_1^j \right]$, a contradiction.
Thus $s \leq 3$.
This, in turn, implies that $3\frac{k^2}{d} > [dq]_{k^2}$.
Thus, if $\gamma = -1$, so that $\mu = +1$ and $[dq]_{k^2} = (m-c)k + \alpha$, 
then we have
\begin{align}
\label{prop: psibar NN, claim 3, eq: mk < (d-1)k}
      mk 
&= [dq]_{k^2} + ck - \alpha
             \\ \nonumber
&< 3\textstyle{\frac{k}{d}k} + ck - \alpha
             \\ \nonumber
&= \left(3\left(3 - \textstyle{\frac{\epsilon}{d}}\right) + c\right)k - \alpha
             \\ \nonumber
&< \left(9 + \textstyle{\frac{d}{2}}\right)k
             \\ \nonumber
&< (d-1)k.
\end{align}
Here, the third line uses the fact that
$\left\lfloor \frac{k}{d} \right\rfloor = 2$, implying $\frac{k+\epsilon}{d} = 3$;
the fourth line uses the fact that $c < \frac{d}{2}$; and the fifth line
uses the fact that $\left\lfloor \frac{k}{d} \right\rfloor = 2$ and $k > 100$
imply $d > 33$.
Since (\ref{prop: psibar NN, claim 3, eq: mk < (d-1)k})
contradicts
(\ref{eq: m = 1 or m in d-1, d+1}),
our supposition that $\gamma = -1$ must have been false,
and so we are left with the case in which $\gamma = +1$,
which, by (\ref{eq: m = 1 or m in d-1, d+1}), implies $m = 1$.

If $s=3$, then since
Claim \ref{claim: k^2/d < dq and l not 1,-1, 2, -2} tells us that
neither $z_0^{j + \mu}$ nor $z_0^{j+2\mu}$ is R-mobile in
$\left\langle z_0^j, z_1^j \right]$,
Claim \ref{claim: m mu t_1 and m mu t_2 not R-mob implies m mu(t_1+t_2) R-mob}
implies that both
$z_0^{j+3\mu}$ and $z_0^{j+4\mu}$ are R-mobile in 
$\left\langle z_0^j, z_1^j \right]$, a contradiction.
If $s=2$, then the fact that $z_0^{j+\mu}$ is not R-mobile in
$\left\langle z_0^j, z_1^j \right]$ implies that
 $z_0^{j+2\mu}$ is R-mobile in
$\left\langle z_0^j, z_1^j \right]$, again a contradiction.
Thus $s=1$, implying that $\frac{k^2}{d} > [dq]_{k^2}$,
but this contradicts Claim \ref{claim: k^2/d < dq and l not 1,-1, 2, -2}.

Thus, it is impossible for 
$\left\langle z_{n_j-1}^j, z_{n_j}^j \right]$
to have R-mobile points, and so,
by the argument at the beginning of the
proof of this proposition,
all $\psibar$-mobile points are non-neutralized.

\end{proof}

\begin{prop}
\label{prop: q of positive type. If mobile points exist, then l=1}
Suppose $q$ of positive type is genus-minimizing,
and $\psi > 2[dq]_{k^2}$.
If $z_0^{j+l}$, and hence $z_{n_{j-l}}^{j-l}$, are mobile in ${\mathbf{z}}^j$
for some nonzero $l \in \Z/d$, then $\psibar < \frac{k^2}{2}$,
$l=1$, $c=1$, and $2 \!\!\not\vert\, \frac{k+\alpha}{d}$,
and either $(\mu,\gamma)=(1,-1)$ with $d=2$ and $\alpha=+1$,
or $(\mu,\gamma)=(1,1)$.
\end{prop}
\begin{proof}
Parts (i) and (ii) of Proposition \ref{prop: positive type, main prop}
tell us that $z_0^{j+l}$ and $z_{n_{j-l}}^{j-l}$ are respectively mobile in
$\left\langle z_{i_*}^j, z_{i_*+1}^j\right]$ and
$\left\langle z_{n_j-(i_*+1)}^j, z_{n_j-i_*}^j \right]$,
where $i_*$ is the unique element of 
$\left\{ 0, \ldots, \left\lfloor \frac{k}{d} \right\rfloor - 2 \right\}$
satisfying
$\left(z_{i_*}^{j_*}, z_{{i_*}+1}^{j_*}\right) 
= \left(z_{n_j - (i_*+1)}^{j_*}, z_{n_j - i_*}^{j_*}\right)
= (x_*,y_*)$, for some unique $j_* \in \Z/d$,
where $x_*, y_* \in Q_q$ are the unique elements of $Q_q$
satisfying $v_q(x_*,y_*) = \alpha(k-k^2)$.
Note that this implies $i_* = \frac{n_{j_*}-1}{2}$.
We begin by showing that $\psibar < \frac{k^2}{2}$.

Suppose that $\psi < \frac{k^2}{2}$.
Since $z_0^{j+l}$ is R-mobile in
$\left\langle z_{i_*}^j, z_{i_*+1}^j\right]$,
we know that
$z_0^j$ is R-mobile in
$\left\langle z_{i_*}^{j-l}, z_{i_*+1}^{j-l}\right]$.
Thus, 
$\minq_{j\in\Z/d}\left(z_0^j - z_0^{j-l}\right) \in
\left\langle i_* dq, (i_* + 1)dq \right\rangle \subset
\left\langle 0, (i_* + 1)dq \right\rangle$,
since Corollary \ref{cor: if z^j has mobile points, then (k/d)dq < k^2}
tells us that
$\left\lfloor\frac{k}{d}\right\rfloor \! [dq]_{k^2} < k^2$.
We therefore have       
\begin{align}
      \minq_{j\in\Z/d}\left(z_0^{j-l}-z_{n_{j-1}}^{j-1}\right)
&= \minq_{j\in\Z/d}\left(z_0^{j-l}-z_0^j\right) + \psi
                \\ \nonumber
&= -\maxq_{j\in\Z/d}\left(z_0^j - z_0^{j-l}\right) + \psi
                \\ \nonumber
&= -\minq_{j\in\Z/d}\left(z_0^j - z_0^{j-l}\right) -dq + \psi
                \\ \nonumber
&\in \left\langle  - \left(i_*+1\right)dq,\, 0 \right\rangle + \psi-dq
                \\ \nonumber
&= \left\langle \psi - \left(i_*+2\right)dq,\, \psi-dq \right\rangle.
\end{align}
This means that if
$\minq_{j\in\Z/d}\left(z_0^{j-l}-z_{n_{j-1}}^{j-1}\right)
\in \left\langle 0, \psi \right\rangle$,
so that $z_0^{j-l}$ is R$\psi$-mobile in
$\left\langle z_{n_{j-1}}^{j-1}, z_0^j \right]$,
then $\minq_{j\in\Z/d}\left(z_0^{j-l}-z_{n_{j-1}}^{j-1}\right)
\in \left\langle 0, \psi - dq \right\rangle$, implying that
\begin{align}
      \maxq_{j\in\Z/d}\left(z_{n_{j-l-1}}^{j-l-1}-z_0^j\right)
&= \maxq_{j\in\Z/d}\left(\left(z_0^{j-l} - \psi\right)- \left(z_{n_{j-1}}^{j-1}+\psi\right)\right)
               \\ \nonumber
&= \minq_{j\in\Z/d}\left(z_0^{j-l} - z_{n_{j-1}}^{j-1}\right) -2\psi +dq
               \\ \nonumber
&\in \left\langle -2\psi +dq, -\psi \right\rangle,
 \end{align}
which, since $2\psi < k^2$, has no intersection with
$\left\langle -\psi, 0\right\rangle$, so that
$z_{n_{j-l-1}}^{j-l-1}$ is not L$\psi$-mobile in
$\left\langle z_{n_{j-1}}^{j-1}, z_0^j \right]$, making 
$z_0^{j-l}$ non-neutralized R$\psi$-mobile in
$\left\langle z_{n_{j-1}}^{j-1}, z_0^j \right]$, a contradiction.
Thus
$\minq_{j\in\Z/d}\left(z_0^{j-l}-z_{n_{j-1}}^{j-1}\right) \in
\left\langle \psi - \left(i_*+2\right)dq,\, \psi-dq \right\rangle
\setminus \left\langle 0, \psi \right\rangle$.
The fact that $\psi > 2[dq]_{k^2}$ implies that
$\psi - [dq]_{k^2} > 0$, which means that
$\left\langle \psi - \left(i_*+2\right)dq,\, \psi-dq \right\rangle
\subset \left\langle 0, \psi \right\rangle$
unless $\psi - \left(i_*+2\right)[dq]_{k^2} < 0$.
Thus 
\begin{equation}
\psi - \left(i_*+2\right)[dq]_{k^2}
\;<\; \minq_{j\in\Z/d}\left(z_0^{j-l}-z_{n_{j-1}}^{j-1}\right) \;<\; 0.
\end{equation}
This, in turn, implies that
\begin{align}
      \maxq_{j\in\Z/d}\left(z_{n_{j-1}}^{j-1}-z_0^{j-l}\right)
&= -\minq_{j\in\Z/d}\left(z_0^{j-l}-z_{n_{j-1}}^{j-1}\right)
                \\ \nonumber
&\in \left\langle  0, \left(i_*+2\right)[dq]_{k^2} - \psi \right\rangle
                \\ \nonumber
&\subset \left\langle 0, i_*[dq]_{k^2} \right\rangle,
\end{align}
where the last line uses the fact that
$\psi > 2[dq]_{k^2}$.
Thus $z_{n_{j-1}}^{j-1}$ is L-mobile in
$\left\langle z_i^{j-l}, z_{i+1}^{j-l}\right]$
for some $i \in \{0, \ldots, i_*-1\}$,
but this means that both
$\left\langle z_i^j, z_{i+1}^j\right]$ and
$\left\langle z_{i_*}^j, z_{i_*+1}^j\right]$
have mobile points, even though $i \neq i_*$.
This contradicts Proposition \ref{prop: positive type, main prop}.(i),
and so our supposition that $\psi < \frac{k^2}{2}$ must have been false.

Thus $\psibar < \frac{k^2}{2}$,
and so Proposition \ref{prop: psibar < k^2/2 implies psibar-mobile are nn}
tells us that all $\psibar$-mobile points are non-neutralized.
Since $z_0^{j+l}$ is R-mobile rel $z_0^j$, we know that
\begin{equation}
\minq_{j\in\Z/d}\left(z_0^{j+l} - z_0^j\right)
\in \left\langle 0, \left(\textstyle{\left\lfloor\frac{k}{d}\right\rfloor} -1\right)[dq]_{k^2} \right\rangle.
\end{equation}
Thus, if $\psibar > \left(\left\lfloor\frac{k}{d}\right\rfloor -1\right)[dq]_{k^2}$,
then $\minq_{j\in\Z/d}\left(z_0^{j+l} - z_0^j\right) \in
\left\langle 0, \psibar\right\rangle$, so that
$z_0^{j+l}$ is non-neutralized R$\psibar$-mobile in
$\left\langle z_0^j, z_{n_{j-1}}^{j-1} \right]$, a contradiction.
Thus, $\psibar < \left(\left\lfloor\frac{k}{d}\right\rfloor -1\right)[dq]_{k^2}$,
but this implies that
$z_{n_{j^{\prime}-1}}^{j^{\prime}-1} \in 
\left\langle z_0^{j^{\prime}}, z_{\left\lfloor\!\frac{k}{d}\!\right\rfloor-1}^{j^{\prime}} \right]$
for all $j^{\prime} \in \Z/d$,
and so
$z_{n_{j-1}}^{j-1}$ is L-mobile rel $z_{n_j}^j$, and $l=1$.

Since $z_0^{j+1}-z_{n_j}^j$ is constant in $j\in\Z/d$,
we can express $\Z/d$ as the disjoint union of $J_1$ and $J_2$, where
\begin{align}
       J_1
&:= \left\{j^{\prime} \in \Z/d \left|\;  
       n_{j^{\prime}}\!=\!\textstyle{\left\lfloor\frac{k}{d}\right\rfloor}\!-\!1;\;
       z_0^{j^{\prime}+1} \in
       \left\langle z_{i_*}^{j^{\prime}}, z_{i_* + 1}^{j^{\prime}}\right]
       \right.\!\right\},
           \\
       J_2
&:= \left\{j^{\prime} \in \Z/d \left|\;  
       n_{j^{\prime}}\!=\!\textstyle{\left\lfloor\frac{k}{d}\right\rfloor};\;
       z_0^{j^{\prime}+1} \in
       \left\langle z_{i_*+1}^{j^{\prime}}, z_{i_* + 2}^{j^{\prime}}\right]
       \right.\!\right\}.
\end{align}
Recall that $n_j := \left\lfloor\frac{k}{d}\right\rfloor - {\theta}^{d, \epsilon}(j)$,
and that the definition of ${\theta}^{d, \epsilon}(j)$, or alternatively, the $l=1$
case of Lemma \ref{lemma: q of positive type, xi lemma}, implies that
$\#\left\{ j^{\prime}\in\Z/d\left|\; {\theta}^{d, \epsilon}(j^{\prime})=1\right.\right\} = \epsilon$.
Thus $|J_1| = \epsilon$.
Now, since $z_{n_{j-1}}^{j-1} - z_0^j$ is constant in $j\in\Z/d$,
$z_{n_{j-1}}^{j-1}$ is not L-mobile in 
$\left\langle z_{i_*}^j, z_{i_*+1}^j\right]$.
Thus $z_0^{j+1}$ is the only mobile point in
$\left\langle z_{i_*}^j, z_{i_*+1}^j\right]$,
and it is active at time $j=j^{\prime}\in \Z/d$
if and only if $j^{\prime} \in J_1$.
Thus, if $\alpha = +1$, then
$z_0^{j+l}$ is active in $\left\langle z_{i_*}^j, z_{i_*+1}^j \right]$
precisely once, and so $\epsilon = |J_1| = 1$; moreover, since $j_* \in J_1$, we have
$n_{j_*} =   \left\lfloor\frac{k}{d}\right\rfloor - 1 = \left(\frac{k+\epsilon}{d}-1\right) -1
= \frac{k+1}{d}-2$.
On the other hand, if $\alpha = -1$, then
$z_0^{j+l}$ is inactive in $\left\langle z_{i_*}^j, z_{i_*+1}^j \right]$
precisely once, and so $\epsilon = |J_1| = d-1$;
moreover, since $j_* \in J_2$, we have
$n_{j_*} =   \left\lfloor\frac{k}{d}\right\rfloor = \frac{k+\epsilon}{d}-1
= \frac{k-1}{d}$.
Lastly, note that the equation $i_* = \frac{n_{j_*} - 1}{2}$ implies that
$n_{j_*} \equiv 1\;(\mod 2)$.  Thus, regardless of the value of $\alpha$,
we have $\epsilon = [\alpha]_d$ and $2 \!\!\not\vert\, \frac{k+\alpha}{d}$.

Since $\epsilon = [\alpha]_d$,
we know that $c = \left[\alpha\gamma\, {\epsilon}^{-1}\right]_d = [\gamma]_d$.
If $(\mu,\gamma) = (-1,1)$,
then $c>m>0$ implies $c>1$, contradicting the fact that
$c = [\gamma]_d = 1$.  Thus $(\mu,\gamma) \neq (-1,1)$.
If $(\mu,\gamma) = (1,-1)$, then
$c = [\gamma]_d = d-1$ contradicts the fact that $c \leq \frac{d}{2}$
unless $d=2$.  If $d=2$, then $c=d-1=1$, and the fact that
$\frac{ck-\alpha}{d} < \frac{k}{2}$ implies $\alpha = +1$.
If $(\mu,\gamma)=(1,1)$, then $c=[\gamma]_d = 1$.

\end{proof}

\begin{prop}
\label{prop: type (1,1), k/d = 2 or 3}
Suppose that $q$ of positive type is genus-minimizing
and $\psi > 2[dq]_{k^2}$.
If $(\mu,\gamma) = (1,1)$,
and $\left\lfloor\frac{k}{d}\right\rfloor \in \{2,3\}$, 
then $m=c=1$.
\end{prop}
\begin{proof}
Since $(\mu, \gamma) = (1,1)$,
Proposition \ref{prop: positive type, main prop}
tells us that ${\mathbf{z}}^j$ has mobile points, and so
the conclusions of
Proposition \ref{prop: q of positive type. If mobile points exist, then l=1} hold.
As we just saw in the last paragraph of the proof of 
Proposition \ref{prop: q of positive type. If mobile points exist, then l=1},
$n_{j_*}$ is odd and is equal to
$\left\lfloor\frac{k}{d}\right\rfloor - 1$ (respectively $\left\lfloor\frac{k}{d}\right\rfloor$)
when $\alpha = +1$ (respectively $\alpha = -1$).
This means that
$\left\lfloor\frac{k}{d}\right\rfloor \neq 3$ when
$\alpha = +1$, and $\left\lfloor\frac{k}{d}\right\rfloor \neq 2$ when
$\alpha = -1$.  Thus if $\alpha = +1$, then
$\left\lfloor\frac{k}{d}\right\rfloor = 2$, $n_{j_*}=1$, and $i_*:= \frac{n_{j_*}-1}{2} = 0$,
whereas if $\alpha = -1$, then
$\left\lfloor\frac{k}{d}\right\rfloor = 3$, $n_{j_*}=3$, and  $i_*:= \frac{n_{j_*}-1}{2} = 1$.
Note that in both cases, $\left\lfloor\frac{k}{d}\right\rfloor - i_* = 2$ and
$\frac{k+\alpha}{d} = 3$.
Proposition \ref{prop: q of positive type. If mobile points exist, then l=1} 
also tells us that $c=1$, and that
$z_0^{j+1}$ is R-mobile in $\left\langle z_{i_*}^j, z_{i_*+1}^j\right]$.

We begin by calculating  $\minq_{j\in\Z/d}\left(z_0^{j+1} - z_{i_*}^j\right)$.
We could use Corollary \ref{cor: q of positive type: combo of lemma and difference eq}
to do this, but it is easier to perform the calculation directly:
\begin{align}
      \minq_{j\in\Z/d}\left(z_0^{j+1} - z_{i_*}^j\right)
&= \left(\textstyle{\left\lfloor\frac{k}{d}\right\rfloor}\!-\!1 - i_*\right)dq \;+\; \psi
               \\ \nonumber
&= \left(\textstyle{\left\lfloor\frac{k}{d}\right\rfloor}\!-\!1 - i_*\right)dq
      \;+\; dq - \textstyle{\frac{ck+\alpha}{d}}k
               \\ \nonumber
&= \left(\textstyle{\left\lfloor\frac{k}{d}\right\rfloor} - i_*\right)dq
       - \textstyle{\frac{k+\alpha}{d}}k
               \\ \nonumber
&= 2dq - 3k.
\end{align}
Now, the fact that $[dq]_{k^2}= (m+c)k+\alpha \geq 2k+\alpha$ implies that
$2[dq]_{k^2}-3k > 0$, and the fact that $[dq]_{k^2} < \frac{k^2}{2}$
implies that $2[dq]_{k^2}-3k < k^2$.  Thus, we in fact have
$\left[\minq_{j\in\Z/d}\left(z_0^{j+1} - z_{i_*}^j\right)\right]_{k^2}
=2[dq]_{k^2} - 3k$.
The fact that $z_0^{j+l}$ is
R-mobile in $\left\langle z_{i_*}^j, z_{i_*+1}^j\right]$ then implies that
\begin{align}
       \left[\minq_{j\in\Z/d}\left(z_0^{j+1} - z_{i_*}^j\right)\right]_{k^2}
&<  [dq]_{k^2}
               \\ \nonumber
      2[dq]_{k^2} - 3k
&<  [dq]_{k^2}
               \\ \nonumber
      [dq]_{k^2}
&<  3k
               \\ \nonumber
     (m+1)k+\alpha
&<  3k
               \\ \nonumber
     m
&<  2 - \textstyle{\frac{\alpha}{k}}.
\end{align}
Thus, $m=1$ unless $\alpha = -1$ and $m=2$.

Suppose that $\alpha = -1$ and $m=2$.  Then
\begin{align}
      q 
&= \alpha\gamma\mu \textstyle{\frac{ck+\alpha\gamma}{d}}(mk + \alpha\mu)
                \\ \nonumber
&= - \textstyle{\frac{k-1}{d}}(2k -1)
                \\ \nonumber
&= -3(2k-1)
                \\ \nonumber
&= -6k+3.
\end{align}
Consider the case in which $k \equiv 1\;(\mod 2)$.  Since $\frac{k-1}{3} = d \in \Z$,
we know that $k \equiv 1\; (\mod 6)$, and so
$\frac{2k+1}{3}, \frac{k-1}{6}, \frac{5k+1}{6} \in \Z$.  Observe that
\begin{equation}
    \left(0q, \textstyle{\frac{2k+1}{3}}q, \textstyle{\frac{k-1}{6}}q, \textstyle{\frac{5k+1}{6}}q \right)
=  \left(0, 1, \textstyle{\frac{3k-1}{2}}, \textstyle{\frac{3k+1}{2}}\right) \in \left(\Z/k^2\right)^4,
\end{equation}
which implies
\begin{align}
       \left| v_q \left(0q, \textstyle{\frac{5k+1}{6}}q\right) \right|
&=  \left|\left(\textstyle{\frac{3k+1}{2}} - 0\right)k \;\;-\;\;
       \#\left({\tilde{Q}}_q \cap \left\langle 0,  \textstyle{\frac{3k+1}{2}}\right] \right)k^2\right|
           \\ \nonumber
&\geq  -\textstyle{\frac{3k+1}{2}}k\;\;+\;\; 3k^2
           \\ \nonumber
&>     k(k+1),
\end{align}
and so Proposition \ref{prop: not minimizing if v >= k(k+1)} tells us that
$q$ is not genus-minimizing.
This leaves us with the case in which $k \equiv 0\; (\mod 2)$.
In this case, the fact that  $\frac{k-1}{3} \in \Z$ implies that
$k \equiv 4\; (\mod 6)$, and so
$\frac{k+2}{6}, \frac{5k+4}{6}, \frac{2k+1}{3} \in \Z$.  Observe that
\begin{equation}
    \left( \textstyle{\frac{k+2}{6}}q, \textstyle{\frac{5k+4}{6}}q, 0q,  \textstyle{\frac{2k+1}{3}}q\right)
=  \left(-\textstyle{\frac{3k}{2}}+1, -\textstyle{\frac{3k}{2}}+2, 0, 1\right) \in \left(\Z/k^2\right)^4,
\end{equation}
which implies
\begin{align}
       \left| v_q \left(\textstyle{\frac{k+2}{6}}q, \textstyle{\frac{2k+1}{3}}q\right) \right|
&=  \left|\left(1 - \left(-\textstyle{\frac{3k}{2}}+1\right)\right)k \;\;-\;\;
       \#\left({\tilde{Q}}_q \cap \left\langle -\textstyle{\frac{3k}{2}}+1, 1\right] \right)k^2\right|
           \\ \nonumber
&\geq  -\textstyle{\frac{3k}{2}}k\;\;+\;\; 3k^2
           \\ \nonumber
&>     k(k+1),
\end{align}
and so $q$ is not genus-minimizing.
Since $q$ is not genus-minimizing when $\alpha = -1$ and $m=2$,
we must have $m=1$.

\end{proof}

\begin{prop}
\label{prop: mobile points implies mu m + gamma c = 2}
Suppose that $q$ of positive type is genus-minimizing.
If ${\mathbf{z}}^j$ has mobile points and 
$\left\lfloor\frac{k}{d}\right\rfloor > 3$, then
$[dq]_{k^2}=2k+\alpha$, so that $\mu m + \gamma c = 2$.
\end{prop}
\begin{proof}
Since $\left\lfloor\frac{k}{d}\right\rfloor > 3$,
Proposition \ref{prop: positive type, main prop}.(iv) tells us that
$\psi > 2[dq]_{k^2}$.
Since, in addition, $q$ of positive type is genus-minimizing
and ${\mathbf{z}}^j$ has mobile points, we know that the conclusions of
Proposition \ref{prop: q of positive type. If mobile points exist, then l=1} hold,
so that $z_0^{j+1}$ is R-mobile in $\left\langle z_{i_*}^j, z_{i_*+1}^j\right]$
and $z_{n_{j-1}}^{j-1}$ is L-mobile in $\left\langle z_{n_j - (i_*+1)}^j, z_{n_j - i_*}^j \right]$,
where $i_* = \frac{n_{j_*}-1}{2}$, and $j_*\in \Z/d$ is the unique element of $\Z/d$
satisfying $\left( z_{i_*}^{j_*}, z_{i_*+1}^{j_*}\right) = 
\left( z_{n_{j_*} - (i_*+1)}^{j_*}, z_{n_{j_*} - i_*}^{j_*} \right) = (x_*,y_*)$,
where $x_*,y_* \in Q_q$ are the unique elements of $Q_q$ satisfying
$v_q(x_*,y_*) = \alpha(k-k^2)$.
In particular, the last paragraph of the proof of 
Proposition \ref{prop: q of positive type. If mobile points exist, then l=1}
tells us that $n_{j_*}$ is odd and is equal to
$\left\lfloor\frac{k}{d}\right\rfloor - 1$ (respectively $\left\lfloor\frac{k}{d}\right\rfloor$)
when $\alpha = +1$ (respectively $\alpha = -1$).
Thus, if $\left\lfloor\frac{k}{d}\right\rfloor = 4$, then
$\alpha = +1$ and $n_{j_*} = 3$, so that
$i_*:= \frac{n_{j_*}-1}{2} = 1$ and
$\left\lfloor\frac{k}{d}\right\rfloor - i_* = 3$.
On the other hand, if $\left\lfloor\frac{k}{d}\right\rfloor \geq 5$,
then $n_{j_*} \geq 5$, and so
$i_*:= \frac{n_{j_*}-1}{2} \geq \frac{5-1}{2} = 2$, and
$\left\lfloor\frac{k}{d}\right\rfloor - i_*
\geq n_{j_*} - i_* = \frac{n_{j_*}+1}{2} \geq \frac{5+1}{2} = 3$.
Thus, in all cases,
$i_* \in \left\{1, \ldots, \left\lfloor\frac{k}{d}\right\rfloor - 3\right\}$.
In particular, $i_* \notin \left\{0, \left\lfloor\frac{k}{d}\right\rfloor - 2\right\}$.
The fact that $z_0^{j+1} - z_{n_j}^j$ and $z_{n_{j-1}}^{j-1} - z_0^j$
are constant in $j\in \Z/d$ then tells us that $z_0^{j+1}$ is not R-mobile rel $z_{n_j}^j$
and $z_{n_{j-1}}^{j-1}$ is not L-mobile rel $z_0^j$.  Thus, 
there are no R-mobile points rel $z_{n_j}^j$, and there are no L-mobile points rel $z_0^j$.

Since $i_* > 0$, the interval
$\left\langle z_0^j, z_1^j\right]$ has no mobile points.
Thus, for any $j^{\prime} \in \Z/d$,
Proposition \ref{prop: positive type, main prop}.(i) tells us that
$v_q\!\left(z_0^{j^{\prime}}, z_1^{j^{\prime}}\right) = \alpha(k)$.
It is therefore sufficient to show that
\begin{equation}
\label{eq: mu m + gamma c = 2, Q_q cap blah has 2 elts}
\#\left({\tilde{Q}}_q \cap 
\left\langle z_0^{j^{\prime}}, z_1^{j^{\prime}} \right]\right)
= 2
\end{equation}
for some $j^{\prime}\in\Z/d$,
because this would imply that
\begin{align}
                       \nonumber
      v_q\!\left(z_0^{j^{\prime}}, z_1^{j^{\prime}}\right)
&= \alpha(k)
                     \\ \nonumber
      \left[ z_1^{j^{\prime}} - z_0^{j^{\prime}} \right]_{k^2}\!\!k
      \;-\;
      \#\left({\tilde{Q}}_q \cap 
     \left\langle z_0^{j^{\prime}}, z_1^{j^{\prime}} \right]\right) k^2
&= \alpha(k)
                     \\ \nonumber
     [dq]_{k^2}k - 2k^2
&= \alpha(k)
                     \\
\label{eq: mu m + gamma c = 2, 2 elts of Q_q implies dq = 2k + alpha}
      [dq]_{k^2} 
&= 2k + \alpha.
\end{align}
We therefore devote the remainder of the proof to showing that
(\ref{eq: mu m + gamma c = 2, Q_q cap blah has 2 elts}) is true.

First, observe that
Corollary \ref{cor: if z^j has mobile points, then (k/d)dq < k^2},
which says that
$\left\lfloor \frac{k}{d}\right\rfloor \! [dq]_{k^2} < k^2$,
implies that
\begin{equation}
\label{eq: mu m + gamma c = 2,  z_i^j notin <z_0, z_1> for all i}
z_i^{j^{\prime}} \notin
\left\langle z_0^{j^{\prime}}, z_1^{j^{\prime}} \right\rangle
\text{for all}\; j^{\prime} \in \Z/d,\;
i \in \left\{0, \ldots, n_{j^{\prime}}\right\}.
\end{equation}

We next claim that
\begin{equation}
\label{eq: mu m + gamma c = 2, z_0^j+l notin for any l or j}
z_0^{j^{\prime}+l} \notin
\left\langle z_0^{j^{\prime}}, z_1^{j^{\prime}} \right]
\text{for all}\; l, j^{\prime} \in \Z/d.
\end{equation}
First, since $z_0^{j+1}$ is R-mobile in
$\left\langle z_{i_*}^j, z_{i_*+1}^j\right]$ and $i_* > 0$,
Corollary \ref{cor: if z^j has mobile points, then (k/d)dq < k^2} implies that
$z_0^{j^{\prime}+1} \notin
\left\langle z_0^{j^{\prime}}, z_1^{j^{\prime}} \right]$ for all
$j^{\prime} \in \Z/d$.  This, in turn, implies that
$z_0^{j^{\prime}-1} \notin
\left\langle z_0^{j^{\prime}}, z_1^{j^{\prime}} \right]$ for all
$j^{\prime} \in \Z/d$.  Combining these two facts with
(\ref{eq: mu m + gamma c = 2,  z_i^j notin <z_0, z_1> for all i}),
we then have that
\begin{equation}
z_0^{j^{\prime}+l} \notin
\left\langle z_0^{j^{\prime}}, z_1^{j^{\prime}} \right]
\text{for all}\; l\in\{-1,0,1\},\; j^{\prime} \in \Z/d.
\end{equation}
In addition, we know that
\begin{equation}
\minq_{j\in\Z/d}\left(z_0^{j+l} - z_0^j\right) \notin \left\langle 0, dq\right\rangle\;
\text{for all}\; l \in \Z/d \setminus \{-1,0,1\},\; j^{\prime} \in \Z/d,
\end{equation}
since $z_0^{j+1}$ is the only R-mobile point in ${\mathbf{z}}^j$.
Suppose there exists $l\in \Z/d$, $l \notin \{-1,0,1\}$, for which 
$\minq_{j\in\Z/d}\left(z_0^{j+l}-z_0^j\right) \in \left\langle -dq, 0\right\rangle$.
Since $z_0^{j+1}$ is R-mobile in
$\left\langle z_{i_*}^j, z_{i_*+1}^j\right]$, we also know that
$\minq_{j\in\Z/d}\left(z^{j+1}_0 - z^j_0\right)
\in \left\langle i_* dq, (i_* + 1)dq\right\rangle$.  Thus, for some $x \in \{0, dq\}$,
we have
\begin{align}
\label{eq: mu m + gamma c = 2, z^{j+l+1}_0 is a new R-mobile point}
      \minq_{j\in\Z/d}\left(z^{j+l+1}_0 - z^j_0\right)
&= \minq_{j\in\Z/d}\left(\left(z^{j+l+1}_0 - z^{j+1}_0\right) 
       + \left(z^{j+1}_0 - z^j_0\right)\right)
                   \\ \nonumber
&= \minq_{j\in\Z/d}\left(z^{j+l}_0 - z^j_0\right)
      + \minq_{j\in\Z/d}\left(z^{j+1}_0 - z^j_0\right) + x
                   \\ \nonumber
&\in \left\langle -dq + i_* dq + x,\; 0 + (i_*+1)dq + x\right\rangle 
                   \\ \nonumber
&\subset \left\langle (i_*-1) dq,\; (i_*+2)dq \right\rangle
                   \\ \nonumber
&\subset \left\langle 0,\; \left(\textstyle{\left\lfloor\frac{k}{d}\right\rfloor}-1\right)dq \right\rangle,
\end{align}
where the last line uses the fact that $1 \leq i_* \leq \left\lfloor\frac{k}{d}\right\rfloor-3$.
But (\ref{eq: mu m + gamma c = 2, z^{j+l+1}_0 is a new R-mobile point})
is impossible, because
it implies $z_0^{j+l+1}$ is R-mobile rel $z_0^j$, contradicting the
fact that $z_0^{j+1}$ is the unique R-mobile point rel $z_0^j$.
Thus, $\maxq_{j\in\Z/d}\left(z_0^{j+l}-z_0^j\right) \notin \left\langle 0, dq\right\rangle$
for all $l \in \Z/d$, $l \notin \{-1,0,1\}$, and so
(\ref{eq: mu m + gamma c = 2, z_0^j+l notin for any l or j}) must be true.
This, in turn, implies that
\begin{equation}
\label{eq: mu m + gamma c = 2, z_{n_{j-l}}^{j-l} not in z_n_j-1, z_n_j}
z_{n_{j^{\prime}-l}}^{j^{\prime}-l} \notin
\left\langle z_{n_{j^{\prime}}-1}^{j^{\prime}}, z_{n_{j^{\prime}}}^{j^{\prime}}\right]
\;\;\text{for all}\;l, j^{\prime}\in\Z/d.
\end{equation}

We next claim, for all $j^{\prime} \in \Z/d$ and $i \in \{0, \ldots, n_{j^{\prime}}\}$, that
\begin{equation}
\label{eq: mu m + gamma c = 2, l=1 implies just one element contained}
        z_{n_{j^{\prime}-1}-i}^{j^{\prime}-1}
\in \left\langle z_0^{j^{\prime}}, z_1^{j^{\prime}} \right]\;
        \text{if and only if}\,\; i = i_* + \textstyle{\frac{\alpha-1}{2}}.
\end{equation}
When $\alpha = +1$, $z_{n_{j-1}}^{j-1}$ is active in
$\left\langle z_{n_j-(i_*+1)}^j, z_{n_j-i_*}^j\right]$ at time $j=j_*$, and so
$z_{n_{j_*-1}}^{j_*-1} \in 
\left\langle z_{n_{j_*} - (i_*+1)}^{j_*}, z_{n_{j_*} - i_*}^{j_*} \right]
= \left\langle z_{i_*}^{j_*}, z_{i_*+1}^{j_*}\right]$,
implying that
$z_{n_{j_*-1}-i_*}^{j_*-1} \in \left\langle z_0^{j_*}, z_1^{j_*}\right]$.
When $\alpha = -1$, $z_{n_{j-1}}^{j-1}$ is inactive in
$\left\langle z_{n_j-(i_*+1)}^j, z_{n_j-i_*}^j\right]$ at time $j=j_*$, and so
$z_{n_{j_*-1}}^{j_*-1} \in 
\left\langle z_{i_*-1}^{j_*}, z_{i_*}^{j_*}\right]$,
implying that
$z_{n_{j_*-1}-(i_*-1)}^{j_*-1} \in \left\langle z_0^{j_*}, z_1^{j_*}\right]$.
We can summarize these two facts by saying that
$z_{n_{j_*-1}-\left(i_*+\frac{\alpha-1}{2}\right)}^{j_*-1}
 \in \left\langle z_0^{j_*}, z_1^{j_*}\right]$, which, since
 $z_{n_{j-1}}^{j-1} - z_0^j$ is constant in $j\in\Z/d$, implies that
 $z_{n_{j^{\prime}-1}-\left(i_*+\frac{\alpha-1}{2}\right)}^{j^{\prime}-1}
 \in \left\langle z_0^{j^{\prime}}, z_1^{j^{\prime}}\right]$ for all $j^{\prime} \in \Z/d$.
The fact that $\left\lfloor\frac{k}{d}\right\rfloor [dq]_{k^2} < k^2$ then implies that
$z_{n_{j^{\prime}-1}-i}^{j^{\prime}-1}
\notin \left\langle z_0^{j^{\prime}}, z_1^{j^{\prime}}\right]$ for all $j^{\prime} \in \Z/d$
and $i \in \left\{0, \ldots, n_{j^{\prime}}\right\} \setminus 
\left\{i_* + \textstyle{\frac{\alpha-1}{2}}\right\}$.

Next, we claim that
\begin{equation}
\label{eq: mu m + gamma c = 2, z_{n_{j-l}-k/d-1}^{j-l} not in <z_0,z_1>}
        z_{n_{j^{\prime}-l}-\left(\left\lfloor\frac{k}{d}\right\rfloor-1\right)}^{j^{\prime}-l}
\notin \left\langle z_0^{j^{\prime}}, z_1^{j^{\prime}} \right]\;
\text{for all}\; l \neq 1, j^{\prime} \in \Z/d.
\end{equation}
Suppose
(\ref{eq: mu m + gamma c = 2, z_{n_{j-l}-k/d-1}^{j-l} not in <z_0,z_1>}) fails
for some $l \in \Z/d$, $l \notin \{0,1\}$.  Then
$\minq_{j\in\Z/d}\left(z_{n_{j-l}-\left(\left\lfloor\frac{k}{d}\right\rfloor-1\right)}^{j-l}-z_0^j\right)
\notin \left\langle 0,dq\right\rangle$,
since otherwise
(\ref{eq: mu m + gamma c = 2, z_0^j+l notin for any l or j}) is contradicted,
and so $\maxq_{j\in\Z/d}\left(z_{n_{j-l}-\left(\left\lfloor\frac{k}{d}\right\rfloor-1\right)}^{j-l}-z_0^j\right)
\in \left\langle 0,dq\right\rangle$, implying that
$\maxq_{j\in\Z/d}\left(z_{n_{j-l}}^{j-l}-z_{\left\lfloor\frac{k}{d}\right\rfloor-1}^j\right)
\in \left\langle 0,dq\right\rangle$.  But this, in turn, implies that
$\maxq_{j\in\Z/d}\left(z_{n_{j-l}}^{j-l}-z_{n_j}^j\right)
\in \left\langle -dq,dq\right\rangle$,
contradicting (\ref{eq: mu m + gamma c = 2, z_{n_{j-l}}^{j-l} not in z_n_j-1, z_n_j}).
Thus (\ref{eq: mu m + gamma c = 2, z_{n_{j-l}-k/d-1}^{j-l} not in <z_0,z_1>})
must be true.

Lastly, we claim that
\begin{equation}
\label{eq: mu m + gamma c = 2, z_{n_{j-l}-i} notin <z_0,z_1>}
        z_{n_{j^{\prime}-l}-i}^{j^{\prime}-l}
\notin \left\langle z_0^{j^{\prime}}, z_1^{j^{\prime}} \right]\;
\text{for all}\; l \neq 1, j^{\prime} \in \Z/d,\;
i \in \left\{0, \ldots, \textstyle{\left\lfloor\frac{k}{d}\right\rfloor}-2\right\}.
\end{equation}
First of all, we know that
\begin{equation}
\label{eq: mu m + gamma c = 2, max z_{n_{j-l}-i}-z_0^j notin}
\maxq_{j\in\Z/d}\left(z_{n_{j-l}-i}-z_0^j\right)
\notin \left\langle 0,dq \right\rangle\;
\text{for all}\; l\neq 1 \in \Z/d,\;
i \in \left\{0, \ldots, \textstyle{\left\lfloor\frac{k}{d}\right\rfloor}-2\right\},
\end{equation}
since otherwise we would have
$\maxq_{j\in\Z/d}\left(z_{n_{j-l}}-z_{i+1}^j\right)
\in \left\langle -dq,0 \right\rangle$,
making $z_{n_{j-l}}$ L-mobile in 
$\left\langle z_i^j, z_{i+1}^j\right]$.
Line (\ref{eq: mu m + gamma c = 2, max z_{n_{j-l}-i}-z_0^j notin}),
in turn, implies that
\begin{equation}
\minq_{j\in\Z/d}\left(z_{n_{j-l}-i}-z_0^j\right)
\notin \left\langle 0,dq \right\rangle\;
\text{for all}\; l\neq 1 \in \Z/d,\;
i \in \left\{0, \ldots, \textstyle{\left\lfloor\frac{k}{d}\right\rfloor}-3\right\},
\end{equation}
and  (\ref{eq: mu m + gamma c = 2, z_{n_{j-l}-k/d-1}^{j-l} not in <z_0,z_1>})
implies that
$\minq_{j\in\Z/d}\left(z_{n_{j-l}-\left(\left\lfloor\frac{k}{d}\right\rfloor-2\right)}-z_0^j\right)
\notin \left\langle 0,dq \right\rangle$
for all $l\neq 1 \in \Z/d$.  Thus
(\ref{eq: mu m + gamma c = 2, z_{n_{j-l}-i} notin <z_0,z_1>})
must be true.

Together, lines
(\ref{eq: mu m + gamma c = 2, z_0^j+l notin for any l or j}),
(\ref{eq: mu m + gamma c = 2, z_{n_{j-l}-k/d-1}^{j-l} not in <z_0,z_1>}),
(\ref{eq: mu m + gamma c = 2, z_{n_{j-l}-i} notin <z_0,z_1>}), and
(\ref{eq: mu m + gamma c = 2, l=1 implies just one element contained})
tell us that
\begin{equation}
Q_q \cap \left\langle z_0^{j^{\prime}}, z_1^{j^{\prime}}\right\rangle
=  z_{n_{j^{\prime}-1}-\left(i_*+\frac{\alpha-1}{2}\right)}^{j^{\prime}-1}\;
\text{for every}\;j^{\prime} \in \Z/d.
\end{equation}
Since $Q_q \cap \left\langle z_0^{j^{\prime}}, z_1^{j^{\prime}}\right]$
also contains $z_1^{j^{\prime}}$, this implies that
\begin{equation}
\#\left({\tilde{Q}}_q \cap \left\langle z_0^{j^{\prime}}, z_1^{j^{\prime}}\right] \right)
=2
\end{equation}
for every $j^{\prime} \in \Z/d$.
Thus (\ref{eq: mu m + gamma c = 2, 2 elts of Q_q implies dq = 2k + alpha})
tells us that $[dq]_{k^2} = 2k + \alpha$,
so that ${\mu}m + {\gamma}c = 2$.

\end{proof}


This concludes our study of the case in which
${\mathbf{z}}^j$ has mobile points,
since we have now learned everything we need to know to
classify when such $q$ are genus-minimizing.
For the remainder of this section, we therefore focus
on the case in there are no mobile points, which means that
$v_q\!\left(z_{n_{j-1}}^{j-1}, z_0^j\right)$
is nonconstant in $j\in \Z/d$,
and so there are non-neutralized pseudomobile points
and non-neutralized antipseudomobile points.

\begin{prop}
\label{prop: no mobile points, psibar < k^2/2, implies type (-1,1), m=1, c=2, d odd}
Suppose that $q$ of positive type is genus-minimizing.
If ${\mathbf{z}}^j$ has no mobile points
and $\psibar < \frac{k^2}{2}$, then
$(\mu, \gamma) = (-1,1)$, $m=1$, $c=2$, and $2\!\! \not\vert \,d$.
\end{prop}

\begin{proof}
Suppose that $q$ is genus-minimizing and of positive type,
that ${\mathbf{z}}^j$ has no mobile points,
and that $\psibar < \frac{k^2}{2}$,
so that, by
Proposition \ref{prop: psibar < k^2/2 implies psibar-mobile are nn},
all antipseudomobile points are non-neutralized.

We begin by claiming that if
$z_0^{j+l}$ (hence $z_{n_{j-1-l}}^{j-1-l}$) is
antipseudomobile in
$\left\langle z_0^j, z_{n_{j-1}}^{j-1} \right]$, then
$[l\epsilon]_d = \frac{d-\alpha}{2}$.
Suppose that $z_0^{j+l}$ and $z_{n_{j-1-l}}^{j-1-l}$
are $\psibar$-mobile in $\left\langle z_0^j, z_{n_{j-1}}^{j-1} \right]$.
Then the preceding paragraph tells us that they
are non-neutralized $\psibar$-mobile, and
Proposition \ref{prop: antipseudomobile analog of main prop}.(ii$\psibar$)
tells us that they are the only non-neutralized
$\psibar$-mobile points in $\left\langle z_0^j, z_{n_{j-1}}^{j-1} \right]$.
Now, by
Propositions \ref{prop: unique v_q = alpha(k - k^2), and the rest are v_q = alpha (k)}
and \ref{prop: antipseudomobile analog of main prop}.(i$\psibar$),
we know that there is a unique $j_* \in \Z/d$ such that
$v_q\!\left( z_0^{j_*}, z_{n_{j_*-1}}^{j_*-1}\right) = -\alpha(k-k^2)$ and
$v_q\!\left( z_0^{j^{\prime}}, z_{n_{j^{\prime}-1}}^{j^{\prime}-1}\right) = -\alpha(k)$
for all $j^{\prime}\neq j_* \in \Z/k^2$.
Thus, if we define $\chi_R(j^{\prime})$ (respectively 
$\chi_L(j^{\prime})$) to be equal to 1 if $z_0^{j+l}$
(respectively $z_{n_{j-1-l}}^{j-1-l}$) is active in 
$\left\langle z_0^j, z_{n_{j-1}}^{j-1} \right]$ at time $j = j^{\prime}$,
and equal to 0 otherwise, then for any $j^{\prime} \in \Z/d$,
\begin{equation}
\label{prop l=-1 in pseudomobile: eq: how many psibars are active?}
    \chi_R(j^{\prime}) + \chi_L(j^{\prime})
= \begin{cases}
       1   &   j^{\prime} \neq j_*
              \\
       2    &   j^{\prime} = j_*, \;\alpha = -1
              \\
       0   &   j^{\prime} = j_*,\;\alpha = +1
    \end{cases}.
\end{equation}
This is because a $\psibar$-mobile point in
$\left\langle z_0^j, z_{n_{j-1}}^{j-1} \right]$ contributes
$-k^2$ to $v_q\!\left( z_0^{j^{\prime}}, z_{n_{j^{\prime}-1}}^{j^{\prime}-1}\right)$
if it is active at time $j=j^{\prime}$ and contributes zero otherwise.
Line (\ref{prop l=-1 in pseudomobile: eq: how many psibars are active?})
then implies that
\begin{equation}
\label{prop l=-1 in pseudomobile: eq: total actives is d-alpha}
\sum_{j\in \Z/d}  \!\chi_R(j) \;+\; \sum_{j\in \Z/d}  \!\chi_L(j) = d-\alpha.
\end{equation}
Proposition \ref{prop: antipseudomobile point is active [lepsilon] times}
then tells us that
\begin{equation}
\label{prop l=-1 in pseudomobile: eq: each is active le times}
\sum_{j\in \Z/d}  \!\chi_R(j) =  \sum_{j\in \Z/d}  \!\chi_L(j) = [l\epsilon]_d.
\end{equation}
The combination of 
(\ref{prop l=-1 in pseudomobile: eq: total actives is d-alpha}) and
(\ref{prop l=-1 in pseudomobile: eq: each is active le times}) then yields
\begin{equation}
\label{prop l=-1 in pseudomobile: eq: le = (d-alpha)/2}
[l\epsilon]_d = \textstyle{\frac{d-\alpha}{2}}.
\end{equation}
Note that this implies $2 \!\! \not\vert \, d$.

We next claim that $z_{n_j}^j$ is L-antipseudomobile in
$\left\langle z_0^j, z_{n_{j-1}}^{j-1} \right]$.
Suppose not, so that
\begin{equation}
\label{eq: (k/d-1)dq < psibar < k/d dq}
\left(\textstyle{\left\lfloor\frac{k}{d}\right\rfloor}-1\right)[dq]_{k^2}
\;<\; \psibar \;<\; \textstyle{\left\lfloor\frac{k}{d}\right\rfloor}  [dq]_{k^2}.
\end{equation}
That is, if $\psibar > \left\lfloor\frac{k}{d}\right\rfloor  [dq]_{k^2}$,
then $z_{n_{j^{\prime}}}^{j^{\prime}}
\in \left\langle z_0^{j^{\prime}}, z_{n_{j^{\prime}-1}}^{j^{\prime}-1} \right]$
for all $j^{\prime} \in\Z/d$, making
$z_{n_j}^j$ L$\psibar$-mobile in
$\left\langle z_0^j, z_{n_{j-1}}^{j-1} \right]$.
On the other hand, if
$\left(\left\lfloor\frac{k}{d}\right\rfloor-1\right)[dq]_{k^2} > \psibar$, then
$z_{n_{j^{\prime}-1}}^{j^{\prime}-1} \in
\left\langle z_0^{j^{\prime}},
z_{\left\lfloor\!\frac{k}{d}\!\right\rfloor-1}^{j^{\prime}} \right]
\subset
\left\langle z_0^{j^{\prime}},
z_{n_{j^{\prime}}}^{j^{\prime}} \right]$ for all $j^{\prime} \in \Z/d$,
making $z_{n_{j-1}}^{j-1}$ L-mobile in ${\mathbf{z}}^j$, a contradiction,
so (\ref{eq: (k/d-1)dq < psibar < k/d dq}) must hold.
By Proposition \ref{prop: antipseudomobile analog of main prop}.(ii$\psibar$),
we know there exists,
for some $l \neq 0 \in \Z/d$, an L$\psibar$-mobile point 
$z_{n_{j-1-l}}^{j-1-l}$ in
$\left\langle z_0^j, z_{n_{j-1}}^{j-1} \right]$.
Now,  (\ref{eq: (k/d-1)dq < psibar < k/d dq}) implies that
$z_{n_{j^{\prime}}}^{j^{\prime}}
\in \left\langle z_{n_{j^{\prime}-1}}^{j^{\prime}-1} - dq,
z_{n_{j^{\prime}-1}}^{j^{\prime}-1} + dq \right\rangle$
for all $j^{\prime} \in \Z/d$.  Thus, if
$\maxq_{j\in\Z/d}\left(z_{n_{j-1-l}}^{j-1-l} - z_{n_{j-1}}^{j-1}\right)
\in \left\langle -\psibar, -dq\right\rangle$, then
$z_{n_{j^{\prime}-1-l}}^{j^{\prime}-1-l} \in
\left\langle z_0^{j^{\prime}}, z_{n_{j^{\prime}}}^{j^{\prime}} \right]$
for all $j^{\prime} \in \Z/d$, implying
that $z_{n_{j-1-l}}^{j-1-l}$ is L-mobile rel $z_{n_j}^j$, a contradiction.
This means that instead, we must have
$\maxq_{j\in\Z/d}\left(z_{n_{j-1-l}}^{j-l-l} - z_{n_{j-1}}^{j-1}\right)
\in \left\langle -dq, 0\right\rangle$.
The mirror relation, (\ref{eq: psibar mirror relation}), for
$\psibar$-mobile points
then tells us that the mirror $\psibar$-mobile point,
$z_0^{j+l}$ R$\psibar$-mobile in $\left\langle z_0^j, z_{n_{j-1}}^{j-1} \right]$,
satisfies 
\begin{align}
      \minq_{j\in\Z/d}\left(z_0^{j+l}-z_0^j\right)
&= -\maxq_{j\in\Z/d}\left(z_{n_{j-1-l}}^{j-1-l} - z_{n_{j-1}}^{j-1}\right)
                 \\ \nonumber
&\in \left\langle 0, dq\right\rangle,
\end{align}
so that $z_0^{j+l}$ is R-mobile in
$\left\langle z_0^j, z_1^j \right]$,
a contradiction.
Thus, $z_{n_j}^j$ is L$\psibar$-mobile in
$\left\langle z_0^j, z_{n_{j-1}}^{j-1} \right]$.
Note that this implies
$\left\lfloor\frac{k}{d}\right\rfloor \![dq]_{k^2} < \psibar$.
Since $\psibar < \frac{k^2}{2}$, we then have
$\left\lfloor\frac{k}{d}\right\rfloor \![dq]_{k^2} < \frac{k^2}{2}$,
implying
\begin{equation}
\label{prop l=-1 in pseudomobile: eq: dq < dk}
[dq]_{k^2} < dk.
\end{equation}

Now, since  $z_{n_j}^j$ and its mirror $z_0^{j-1}$ are $\psibar$-mobile in
$\left\langle z_0^j, z_{n_{j-1}}^{j-1} \right]$,
(\ref{prop l=-1 in pseudomobile: eq: le = (d-alpha)/2})
tells us that
\begin{equation}
(-1)\epsilon \equiv \textstyle{\frac{d-\alpha}{2}}\; (\mod d).
\end{equation}
Thus, since
$c = \left[\alpha\gamma {\epsilon}^{-1}\right]_d$, we have
\begin{align}
      c 
&= \left[\alpha\gamma (-1)\left(\textstyle{\frac{d-\alpha}{2}}\right)^{-1}\right]_d
                \\ \nonumber
&= \left[\alpha\gamma(-1)(-2\alpha)\right]_d
                \\ \nonumber
&= [2\gamma]_d
                \\ \nonumber
&= \begin{cases}
             2   & \gamma = +1
                       \\
             d-2   & \gamma = -1
       \end{cases}.
\end{align}
By Proposition \ref{prop: properties of parameters d, m, c, alpha, mu, gamma},
we know that $c \leq \frac{d}{2}$, with equality if and only if $d=2$ (which
does not occur here, since $d$ must be odd).  Thus if $\gamma = -1$,
then $d-2 < \frac{d}{2}$ implies $d=3$ and $c=1$.
Since $\gamma = -1$ implies $(\mu, \gamma) = (1,-1)$, we then have
\begin{align}
\label{prop l=-1 in pseudomobile: eq: psibar > k^2/2}
      \psibar
&= \left[-dq - \textstyle{\frac{ck-\alpha}{d}}k\right]_{k^2}
            \\ \nonumber
&= k^2 -\left([dq]_{k^2} + \textstyle{\frac{ck-\alpha}{d}}k\right)
            \\ \nonumber
&> k^2 -\left(dk + \textstyle{\frac{ck-\alpha}{d}}k\right)
            \\ \nonumber
&= k^2 -\left(3k + \textstyle{\frac{k-\alpha}{3}}k\right)
            \\ \nonumber
&> \textstyle{\frac{k^2}{2}},
\end{align}
where the second line uses the facts that
$[dq]_{k^2} < \frac{k^2}{2}$ and $\frac{ck-\alpha}{d} < \frac{k}{2}$,
the third line uses
(\ref{prop l=-1 in pseudomobile: eq: dq < dk}),
and the last line uses the fact that $k > 100$.
Since (\ref{prop l=-1 in pseudomobile: eq: psibar > k^2/2})
contradicts our supposition that $\psibar < \frac{k^2}{2}$,
we conclude that $\gamma \neq -1$.  Thus $\gamma = 1$ and $c=2$.
Moreover, since there are no mobile points, Proposition \ref{prop: positive type, main prop}
tells us that $(\mu, \gamma) \neq (1,1)$.
Thus $(\mu, \gamma) = (-1, 1)$, and so $c > m > 0$
implies that $m=1$.

\end{proof}

\begin{prop}
\label{prop: no mobile points, psi < k^2/2, implies type (1,-1), m=2, c=1, d odd}
Suppose that $q$ of positive type is genus-minimizing.
If $\psi < \frac{k^2}{2}$ and $\psi > 2[dq]_{k^2}$, then
${\mathbf{z}}^j$ has no mobile points,
all pseudomobile points are non-neutralized,
$(\mu, \gamma) = (1,-1)$, $m=2$, $c=1$, and $2\!\! \not\vert \,d$.
\end{prop}

\begin{proof}
Suppose that $q$ is genus-minimizing and of positive type,
and that $2[dq]_{k^2}< \psi < \frac{k^2}{2}$.
Proposition \ref{prop: q of positive type. If mobile points exist, then l=1}
tells us that if $\psi > 2[dq]_{k^2}$ and mobile points exist, then
$\psibar < \frac{k^2}{2}$.  Thus, the fact that $\psi < \frac{k^2}{2}$
implies that there are no mobile points.
Proposition \ref{prop: positive type, main prop} then tells us that
$v_q\!\left(z_{n_{j-1}}^{j-1}, z_0^j\right)$ is nonconstant in $j\in\Z/d$,
and that $\left\langle z_{n_{j-1}}^{j-1}, z_0^j \right]$ has precisely one
non-neutralized R$\psi$-mobile point and one
non-neutralized L$\psi$-mobile point, namely,
$z_0^{j+l}$ and $z_{n_{j-1-l}}^{j-1-l}$ for some $l\neq 0\in \Z/d$.

We begin by showing that
\begin{equation}
\label{prop: type (1,-1), c=1, eq: all of z^{j+l} fits in < z_0^j, z_n>}
z_0^{j^{\prime}+l}, \ldots,  
z_{n_{j^{\prime}+l}}^{j^{\prime}+l}
\in\left\langle z_{n_{j^{\prime}-1}}^{j^{\prime}-1}, z_0^{j^{\prime}} \right]\;
\text{for all}\; j^{\prime} \in \Z/d.
\end{equation}
First, note that if $\minq_{j\in\Z/d}\left(z_0^{j+l}-z_{n_{j-1}}^{j-1}\right)
\in \left\langle \psi-dq, \psi \right\rangle$, then
\begin{align}
      \maxq_{j\in\Z/d}\left(z_{n_{j+l-1}}^{j+l-1}-z_0^j\right)
&= \minq_{j\in\Z/d}\left(\left(z_0^{j+l}-\psi\right)-\left(z_{n_{j-1}}^{j-1}+\psi\right)\right) + dq
             \\ \nonumber
&= \minq_{j\in\Z/d}\left(z_0^{j+l} - z_{n_{j-1}}^{j-1}\right) + dq - 2\psi
             \\ \nonumber
&\in \left\langle -\psi, dq-\psi \right\rangle,
\end{align}
so that $z_0^{j+l}$ is neutralized by
$z_{n_{j+l-1}}^{j+l-1}$ L$\psi$-mobile in $\left\langle z_{n_{j-1}}^{j-1}, z_0^j \right]$,
a contradiction.  Thus
$\minq_{j\in\Z/d}\left(z_0^{j+l}-z_{n_{j-1}}^{j-1}\right)
\in \left\langle 0, \psi-dq \right\rangle$, which implies that
$z_0^{j^{\prime}+l} \in
\left\langle z_{n_{j^{\prime}-1}}^{j^{\prime}-1}, z_0^{j^{\prime}} \right]$
for all $j^{\prime} \in \Z/d$.
In addition, observe that if 
$\minq_{j\in\Z/d}\left(z_0^j - z_0^{j+l}\right)
\in \left\langle 0, \left(\left\lfloor\frac{k}{d}\right\rfloor -1 \right)dq \right\rangle$,
then $z_0^j$ is R-mobile in $\left\langle z_i^{j+l}, z_{i+1}^{j+l} \right]$,
for some $i \in \left\{0, \ldots, \left\lfloor\frac{k}{d}\right\rfloor-2\right\}$,
a contradiction.
Thus, for all $j^{\prime} \in \Z/d$, we have
$\left\langle z_0^{j^{\prime}+l} , 
z_0^{j^{\prime}+l} + \left(\left\lfloor\frac{k}{d}\right\rfloor -1 \right)dq \right\rangle
\subset
\left\langle z_{n_{j^{\prime}-1}}^{j^{\prime}-1}, z_0^{j^{\prime}} \right]$,
implying
$z_0^{j^{\prime}+l}, \ldots,  
z_{\left\lfloor\!\frac{k}{d}\!\right\rfloor -1}^{j^{\prime}+l}
\in\left\langle z_{n_{j^{\prime}-1}}^{j^{\prime}-1}, z_0^{j^{\prime}} \right]$.
Now, the fact that $z_{\left\lfloor\!\frac{k}{d}\!\right\rfloor -1}^{j^{\prime}+l}
\in\left\langle z_{n_{j^{\prime}-1}}^{j^{\prime}-1}, z_0^{j^{\prime}} \right]$
for all $j^{\prime} \in \Z/d$ and that $\psi,\, [dq]_{k^2} < \frac{k^2}{2}$
implies that $z_{n_{j^{\prime}+l}}^{j^{\prime}+l} \neq z_{n_{j^{\prime}-1}}^{j^{\prime}-1}$,
and so $l\neq -1$.  It is therefore safe to make our final observation that if
$\minq_{j\in\Z/d}\left(z_0^j - z_{n_{j+l}-1}^{j+l}\right)
\in \left\langle 0, dq \right\rangle$, then
$z_0^j$ is R-mobile in $\left\langle z_{n_{j+l}-1}^{j+l}, z_{n_{j+l}}^{j+l}\right]$,
a contradiction.  Thus 
$z_{n_{j^{\prime}+l}}^{j^{\prime}+l}
\in\left\langle z_{n_{j^{\prime}-1}}^{j^{\prime}-1}, z_0^{j^{\prime}} \right]$
for all $j^{\prime} \in \Z/d$, and so
(\ref{prop: type (1,-1), c=1, eq: all of z^{j+l} fits in < z_0^j, z_n>}) holds.

One implication of (\ref{prop: type (1,-1), c=1, eq: all of z^{j+l} fits in < z_0^j, z_n>})
is that $z_{n_{j+l}}^{j+l}$ is L$\psi$-mobile in
$\left\langle z_{n_{j-1}}^{j-1}, z_0^j \right]$.
We claim that in fact, $z_{n_{j+l}}^{j+l}$ is non-neutralized
L$\psi$-mobile in $\left\langle z_{n_{j-1}}^{j-1}, z_0^j \right]$.
That is, since $\psi < \frac{k^2}{2}$,
(\ref{prop: type (1,-1), c=1, eq: all of z^{j+l} fits in < z_0^j, z_n>})
implies
\begin{equation}
\label{prop: type (1,-1), c=1, k/d dq < psi}
\textstyle{\left\lfloor\frac{k}{d}\right\rfloor} [dq]_{k^2} <
\maxq_{j\in\Z/d}\left(z_{n_{j+l}}^{j+l}-z_{n_{j-1}}^{j-1}\right) < \psi,
\end{equation}
and so, recalling that $l\neq -1$, we have
\begin{align}
      \minq_{j\in\Z/d}\left(z_0^{j+l+1} - z_{n_{j-1}}^{j-1}\right)
&= \maxq_{j\in\Z/d}\left(\left(z_{n_{j+l}}^{j+l} + \psi\right) - z_{n_{j-1}}^{j-1}\right) - dq
                     \\ \nonumber
&= \maxq_{j\in\Z/d}\left(z_{n_{j+l}}^{j+l} - z_{n_{j-1}}^{j-1} \right) - dq + \psi
                     \\ \nonumber
&= \left\langle \left(\textstyle{\left\lfloor\frac{k}{d}\right\rfloor} - 1\right) dq  + \psi, 2\psi  - dq \right\rangle,
\end{align}
which has no intersection with $\left\langle 0, \psi \right\rangle$.
Thus $z_0^{j+l+1}$ is not R$\psi$-mobile in $\left\langle z_{n_{j-1}}^{j-1}, z_0^j \right]$,
and so $z_{n_{j+l}}^{j+l}$ is non-neutralized
L$\psi$-mobile in $\left\langle z_{n_{j-1}}^{j-1}, z_0^j \right]$.

Since $z_0^{j+l}$ and $z_{n_{j+l}}^{j+l}$ are non-neutralized
$\psi$-mobile in $\left\langle z_{n_{j-1}}^{j-1}, z_0^j \right]$,
Proposition \ref{prop: positive type, main prop}.(ii$\psi$) tells us that
$z_0^{j+l}$ and $z_{n_{j+l}}^{j+l}$ must be mirror $\psi$-mobile
in $\left\langle z_{n_{j-1}}^{j-1}, z_0^j \right]$, or in other words,
$z_{n_{j+l}}^{j+l} = z_{n_{j-1-l}}^{j-1-l}$.
Thus $j+l = j-1-l \in \Z/d$, and so
\begin{equation}
\label{prop: type (1,-1), c=1, eq: l = (d-1)/2}
l = \textstyle{\frac{d-1}{2}}\;\in\Z/d.
\end{equation}
Note that this requires that $2 \!\! \not\vert \, d$.

On the other hand, we can also show that
\begin{equation}
\label{prop: type (1,-1), c=1, eq: le = (d+alpha)/2}
[l\epsilon]_d = \textstyle{\frac{d+\alpha}{2}}.
\end{equation}
By Propositions \ref{prop: unique v_q = alpha(k - k^2), and the rest are v_q = alpha (k)}
and \ref{prop: positive type, main prop}.(i$\psi$),
we know that there is a unique $j_* \in \Z/d$ such that
$v_q\!\left(z_{n_{j_*-1}}^{j_*-1},  z_0^{j_*} \right) = \alpha(k-k^2)$ and
$v_q\!\left( z_{n_{j^{\prime}-1}}^{j^{\prime}-1}, z_0^{j^{\prime}} \right) = \alpha(k)$
for all $j^{\prime}\neq j_* \in \Z/k^2$.
Thus, if we define $\chi_R(j^{\prime})$ (respectively 
$\chi_L(j^{\prime})$) to be equal to 1 if $z_0^{j+l}$
(respectively $z_{n_{j-1-l}}^{j-1-l}$) is active in 
$\left\langle z_{n_{j-1}}^{j-1},  z_0^j \right]$ at time $j = j^{\prime}$,
and equal to 0 otherwise, then for any $j^{\prime} \in \Z/d$,
\begin{equation}
\label{prop: type (1,-1), c=1, eq: how many psis are active?}
    \chi_R(j^{\prime}) + \chi_L(j^{\prime})
= \begin{cases}
       1   &   j^{\prime} \neq j_*
              \\
       2    &   j^{\prime} = j_*, \;\alpha = +1
              \\
       0   &   j^{\prime} = j_*,\;\alpha = -1
    \end{cases}.
\end{equation}
This is because a $\psi$-mobile point in
$\left\langle z_{n_{j-1}}^{j-1}, z_0^j \right]$ contributes
$-k^2$ to $v_q\!\left(  z_{n_{j^{\prime}-1}}^{j^{\prime}-1}, z_0^{j^{\prime}}\right)$
if it is active at time $j=j^{\prime}$ and contributes zero otherwise.
Line (\ref{prop: type (1,-1), c=1, eq: how many psis are active?})
then implies that
\begin{equation}
\label{prop: type (1,-1), c=1, eq: total actives is d+alpha}
\sum_{j\in \Z/d}  \!\chi_R(j) \;+\; \sum_{j\in \Z/d}  \!\chi_L(j) = d+\alpha,
\end{equation}
and Proposition \ref{prop: pseudomobile point is active [lepsilon] times}
tells us that
\begin{equation}
\label{prop: type (1,-1), c=1, eq: each is active le times}
\sum_{j\in \Z/d}  \!\chi_R(j) =  \sum_{j\in \Z/d}  \!\chi_L(j) = [l\epsilon]_d.
\end{equation}
Lines
(\ref{prop: type (1,-1), c=1, eq: total actives is d+alpha}) and
(\ref{prop: type (1,-1), c=1, eq: each is active le times}) then tell us that
(\ref{prop: type (1,-1), c=1, eq: le = (d+alpha)/2}) holds, and so
lines 
(\ref{prop: type (1,-1), c=1, eq: l = (d-1)/2}) and
(\ref{prop: type (1,-1), c=1, eq: le = (d+alpha)/2})
tell us that 
\begin{equation}
\epsilon = [-\alpha]_d.
\end{equation}
This, in turn, allows us to compute $c$.
\begin{align}
      c
&= \left[\alpha \gamma \, {\epsilon}^{-1}\right]_d
               \\ \nonumber
&= \left[-\gamma\right]_d
               \\ \nonumber
&= \begin{cases}
           1     & \gamma = -1
                          \\
          d-1  & \gamma = +1
      \end{cases}.
\end{align}
Now, Proposition \ref{prop: properties of parameters d, m, c, alpha, mu, gamma}
tells us that $c \leq \frac{d}{2}$.  Thus, if $\gamma = +1$, then
$d-1 \leq \frac{d}{2}$, implying $d \leq 2$, but the fact that 
$d > 1$ and $2\!\!\not\vert\, d$ implies that $d \geq 3$.
Thus $\gamma = -1$, and so $(\mu,\gamma) = (1,-1)$
and $c=1$.

We have now proven everything we wanted to prove except that
all $\psi$-mobile points are non-neutralized and that $m=2$,
the latter of which is now equivalent to showing that $[dq]_{k^2} = k+ \alpha$.
Toward these ends, we claim that
\begin{equation}
\label{prop: type (1,-1), c=1, eq: z_0^j+l_0 not in <z_0, z_1>}
z_0^{j^{\prime}+l_0} \notin 
\left\langle z_0^{j^{\prime}}, z_1^{j^{\prime}} \right\rangle\;
\text{for all}\; l_0, j^{\prime} \in \Z/d.
\end{equation}
Suppose not, so that there exist $l_0, j_0 \in \Z/d$
for which $z_0^{j_0+l_0} \in 
\left\langle z_0^{j_0}, z_1^{j_0} \right\rangle$.
Then the fact that
$z_0^{j^{\prime}} \notin 
\left\langle z_0^{j^{\prime}}, z_1^{j^{\prime}} \right\rangle$
implies $l_0 \neq 0$.
Moreover, $l+l_0 \neq 0$, because otherwise,
setting $j_1 = j_0-l$, we would have
$z_0^{j_1} = z_0^{j_1 + l + l_0} \in 
\left\langle z_0^{j_1 +l}, z_1^{j_1 + l} \right\rangle$,
contradicting (\ref{prop: type (1,-1), c=1, eq: all of z^{j+l} fits in < z_0^j, z_n>}).
We also know that 
$\minq_{j\in \Z/d}\left(z_0^{j+l_0}-z_0^j\right) \notin \left\langle 0, dq\right\rangle$,
because otherwise $z_0^{j+l_0}$ would be R-mobile in
$\left\langle z_0^j, z_1^j\right]$.
Thus $\minq_{j\in \Z/d}\left(z_0^{j+l_0}-z_0^j\right) \in \left\langle -dq, 0\right\rangle$,
and so $\minq_{j\in \Z/d}\left(z_0^{j+l+l_0}-z_0^{j+l}\right) \in \left\langle -dq, 0\right\rangle$.
On the other hand, because of (\ref{prop: type (1,-1), c=1, eq: all of z^{j+l} fits in < z_0^j, z_n>}),
we know that
\begin{equation}
0 < \minq_{j\in \Z/d}\left(z_0^{j+l}-z_{n_{j-1}}^{j-1} \right)
< \psi -\textstyle{\left\lfloor\frac{k}{d}\right\rfloor} [dq]_{k^2}.
\end{equation}
Thus, for some $x \in \{0, dq\}$, we have
\begin{align}
      \minq_{j\in \Z/d}\left(z_0^{j+l+l_0}-z_{n_{j-1}}^{j-1} \right)
&= \minq_{j\in \Z/d}\left(\left(z_0^{j+l+l_0}-z_0^{j+l}\right) +
                 \left(z_0^{j+l}- z_{n_{j-1}}^{j-1}\right) \right)
                            \\ \nonumber
&= \minq_{j\in \Z/d}\left(z_0^{j+l+l_0}-z_0^{j+l}\right) +
      \minq_{j\in \Z/d}\left(z_0^{j+l}- z_{n_{j-1}}^{j-1}\right) + x
                            \\ \nonumber
&\in \left\langle -dq + x,\,  \psi -\textstyle{\left\lfloor\frac{k}{d}\right\rfloor} dq + x \right\rangle
                            \\ \nonumber
&\subset \left\langle -dq,\,  
       \psi -\left(\textstyle{\left\lfloor\frac{k}{d}\right\rfloor}\!-\!1\right)\!dq \right\rangle.
\end{align}
If
$\minq_{j\in \Z/d}\left(z_0^{j+l+l_0}-z_{n_{j-1}}^{j-1}\right) \in \left\langle -dq, 0\right\rangle$,
then 
$\minq_{j\in \Z/d}\left(z_0^{j+l+l_0}-z_{n_{j-1}-1}^{j-1}\right) \in \left\langle 0, dq\right\rangle$,
so that $z_0^{j+l+l_0}$ is R-mobile in
$\left\langle z_{n_{j-1}-1}^{j-1}, z_{n_{j-1}}^{j-1} \right]$, a contradiction.
Thus instead, we have $\minq_{j\in \Z/d}\left(z_0^{j+l+l_0}-z_{n_{j-1}}^{j-1}\right) 
\in \left\langle 0,\, \psi -\left(\textstyle{\left\lfloor\frac{k}{d}\right\rfloor}\!-\!1\right)\!dq \right\rangle$,
so that $z_0^{j+l+l_0}$ is R-pseudomobile in
$\left\langle z_{n_{j-1}}^{j-1}, z_0^j\right]$.
We then have
\begin{align}
      \maxq_{j\in\Z/d}\left(z_{n_{j+l+l_0-1}}^{j+l+l_0-1} - z_0^j\right)
&= \minq_{j\in\Z/d}\left(\left(z_0^{j+l+l_0}-\psi\right) - \left(z_{n_{j-1}}^{j-1}+\psi\right)\right) + dq
                        \\ \nonumber
&= \minq_{j\in\Z/d}\left(z_0^{j+l+l_0} - z_{n_{j-1}}^{j-1}\right) + dq -2\psi
                        \\ \nonumber
&=  \left\langle -2\psi + dq,\, 
        -\psi -\left(\textstyle{\left\lfloor\frac{k}{d}\right\rfloor}\!-\!2\right)\!dq \right\rangle,
\end{align}
which, since $\left\lfloor\frac{k}{d}\right\rfloor\!-\!2 \geq 0$ and $\psi, [dq]_{k^2} < \frac{k^2}{2}$,
has no intersection with $\left\langle -\psi, 0\right\rangle$.
Thus $z_{n_{j+l+l_0-1}}^{j+l+l_0-1}$ is not L$\psi$-mobile in
$\left\langle z_{n_{j-1}}^{j-1}, z_0^j\right]$, and so
$z_0^{j+l+l_0}$ is non-neutralized R-pseudomobile in
$\left\langle z_{n_{j-1}}^{j-1}, z_0^j\right]$,
contradicting the uniqueness of $z_0^{j+l}$ as a non-neutralized
R$\psi$-mobile point in $\left\langle z_{n_{j-1}}^{j-1}, z_0^j\right]$.
Thus (\ref{prop: type (1,-1), c=1, eq: z_0^j+l_0 not in <z_0, z_1>})
must be true.

We next claim that (\ref{prop: type (1,-1), c=1, eq: z_0^j+l_0 not in <z_0, z_1>})
implies that all pseudomobile points are non-neutralized.
Suppose that for some nonzero $l_0  \in \Z/d$,
we know that $z_0^{j+l_0}$ is R$\psi$-mobile in 
$\left\langle z_{n_{j-1}}^{j-1}, z_0^j\right]$.
If $\minq_{j\in\Z/d}\left(z_0^{j+l_0}-z_{n_{j-1}}^{j-1}\right)
\in \left\langle \psi - 2dq, \psi\right\rangle$, then there exists
$j^{\prime} \in \Z/d$ for which
$z_0^{j^{\prime}+ l_0} \in
\left\langle z_0^{j^{\prime}} - dq, z_0^{j^{\prime}} \right\rangle$, implying that
$z_0^{j^{\prime}} \in \left\langle z_0^{j^{\prime}+l_0}, z_1^{j^{\prime}+l_0} \right\rangle$,
but this contradicts (\ref{prop: type (1,-1), c=1, eq: z_0^j+l_0 not in <z_0, z_1>}).
Thus $\minq_{j\in\Z/d}\left(z_0^{j+l_0}-z_{n_{j-1}}^{j-1}\right)
\in \left\langle 0, \psi - 2dq \right\rangle$, and so
\begin{align}
      \maxq_{j\in\Z/d}\left(z_{n_{j+l_0-1}}^{j+l_0-1}-z_0^j\right)
&= \minq_{j\in\Z/d}\left(z_0^{j+l_0}-z_{n_{j-1}}^{j-1}\right) + dq - 2\psi
              \\ \nonumber
&\in \left\langle -2\psi + dq, -\psi - dq \right\rangle,
\end{align}
which, since $\psi, [dq]_{k^2}< \frac{k^2}{2}$, has no intersection with
$\left\langle -\psi, 0\right\rangle$.
Thus $z_{n_{j+l_0-1}}^{j+l_0-1}$ is not L$\psi$-mobile in
$\left\langle z_{n_{j-1}}^{j-1}, z_0^j\right]$,
and so $z_0^{j+l_0}$ is non-neutralized R$\psi$-mobile in 
$\left\langle z_{n_{j-1}}^{j-1}, z_0^j\right]$.
The fact that there are no neutralized R$\psi$-mobile points in
$\left\langle z_{n_{j-1}}^{j-1}, z_0^j\right]$ implies that there are no
neutralized L$\psi$-mobile points in
$\left\langle z_{n_{j-1}}^{j-1}, z_0^j\right]$.
Thus all pseudomobile points are non-neutralized.

Lastly, we claim that $[dq]_{k^2} = k+ \alpha$.
To prove this, we shall show that
$Q_q \cap \left\langle z_0^{j^{\prime}}, z_1^{j^{\prime}} \right\rangle = \emptyset$
for all $j^{\prime} \in\Z/d$.
First, since
(\ref{prop: type (1,-1), c=1, k/d dq < psi}) implies both that
$\left\lfloor \frac{k}{d} \right\rfloor [dq]_{k^2} < \psi$
and that $\left\lfloor \frac{k}{d} \right\rfloor [dq]_{k^2} + \psi < k^2$,
and since
(\ref{prop: type (1,-1), c=1, eq: z_0^j+l_0 not in <z_0, z_1>})
tells us in particular that
$z_0^{j^{\prime}-1} \notin \left\langle z_0^{j^{\prime}}, z_1^{j^{\prime}} \right]$,
we deduce that
\begin{equation}
\label{prop: type (1,-1), c=1, eq: l=0 and l=1 not in}
z_i^{j^{\prime}-l_0}
\notin \left\langle z_0^{j^{\prime}}, z_1^{j^{\prime}} \right\rangle\;
\text{for all}\; l_0 \in \{0, 1\},\;
j^{\prime} \in \Z/d,\;
i \in \left\{0, \ldots, n_{j^{\prime}}\right\}\!.
\end{equation}
Next, note that (\ref{prop: type (1,-1), c=1, eq: z_0^j+l_0 not in <z_0, z_1>}),
along with the mirror relation (\ref{eq: mirror relation 1}),
implies that
\begin{equation}
\label{prop: type (1,-1), c=1, eq: z_nj-l notin z_nj}
z_{n_{j^{\prime}-l_0}}^{j^{\prime}-l_0} \notin
\left\langle z_{n_{j^{\prime}}-1}^{j^{\prime}}, z_{n_{j^{\prime}}}^{j^{\prime}}\right\rangle\;
\text{for all}\;l_0, j^{\prime}\in\Z/d,
\end{equation}
which is useful for showing that
\begin{equation}
\label{prop: type (1,-1), c=1, eq: z_n_j - (k/d -1) not in}
z_{n_{j^{\prime}-l_0} - \left(\left\lfloor\!\frac{k}{d}\!\right\rfloor - 1\right)}^{j^{\prime}-l_0}
\notin \left\langle z_0^{j^{\prime}}, z_1^{j^{\prime}} \right\rangle\;
\text{for all}\; l_0, j^{\prime} \in \Z/d.
\end{equation}
Suppose
(\ref{prop: type (1,-1), c=1, eq: z_n_j - (k/d -1) not in}) fails
for some $l_0 \in \Z/d$, $l_0 \notin \{0,1\}$.  Then
$\minq_{j\in\Z/d}\left(z_{n_{j-l_0}-\left(\left\lfloor\frac{k}{d}\right\rfloor-1\right)}^{j-l_0}-z_0^j\right)
\notin \left\langle 0,dq\right\rangle$,
since otherwise
(\ref{prop: type (1,-1), c=1, eq: z_0^j+l_0 not in <z_0, z_1>}) is contradicted,
and so $\maxq_{j\in\Z/d}\left(z_{n_{j-l_0}-\left(\left\lfloor\frac{k}{d}\right\rfloor-1\right)}^{j-l_0}-z_0^j\right)
\in \left\langle 0,dq\right\rangle$, implying that
$\maxq_{j\in\Z/d}\left(z_{n_{j-l_0}}^{j-l_0}-z_{\left\lfloor\frac{k}{d}\right\rfloor-1}^j\right)
\in \left\langle 0,dq\right\rangle$.  But this, in turn, implies that
$\maxq_{j\in\Z/d}\left(z_{n_{j-l_0}}^{j-l_0}-z_{n_j}^j\right)
\in \left\langle -dq,dq\right\rangle$,
contradicting (\ref{prop: type (1,-1), c=1, eq: z_nj-l notin z_nj}).
Thus 
(\ref{prop: type (1,-1), c=1, eq: z_n_j - (k/d -1) not in})
must be true.
We next claim that
\begin{equation}
\label{prop: type (1,-1), c=1, eq: max n_j-i notin for  i  0, ..., k/d-2}
\maxq_{j\in\Z/d}\left(z_{n_{j-l_0} - i}^{j-l_0}
- z_1^j\right) \notin \left\langle -dq, 0\right\rangle\;
\text{for all}\;
i \in \left\{0, \ldots, \textstyle{\left\lfloor\frac{k}{d}\right\rfloor} - 2\right\}\!,\;
l_0 \neq 1 \in \Z/d.
\end{equation}
Suppose there exist
$i \in \left\{0, \ldots, \left\lfloor\frac{k}{d}\right\rfloor - 2\right\}$ and $l_0 \notin \{0,1\} \in \Z/d$
for which (\ref{prop: type (1,-1), c=1, eq: max n_j-i notin for  i  0, ..., k/d-2})
does not hold.  Then, since 
$\maxq_{j\in\Z/d}\left(z_{n_{j-l_0}}^{j-l_0}
- z_{i+1}^j\right) \in \left\langle -dq, 0\right\rangle$ and 
$n_j \in \left\{\left\lfloor\frac{k}{d}\right\rfloor -1, \left\lfloor\frac{k}{d}\right\rfloor \right\}$
for all $j \in \Z/d$, we know that either
$\maxq_{j\in\Z/d}\left(z_{n_{j-l_0}}^{j-l_0}
- z_{n_j - \left(\left\lfloor\!\frac{k}{d}\!\right\rfloor - (i+2)\right)}^j\right) 
\in \left\langle -dq, 0\right\rangle$, so that $z_{n_{j-l_0}}^{j-l_0}$
is L-mobile in
$\left\langle z_{n_j - \left(\left\lfloor\!\frac{k}{d}\!\right\rfloor - (i+1)\right)}^j,
z_{n_j - \left(\left\lfloor\!\frac{k}{d}\!\right\rfloor - (i+2)\right)}^j \right]$, or
$\maxq_{j\in\Z/d}\left(z_{n_{j-l_0}}^{j-l_0}
- z_{n_j - \left(\left\lfloor\!\frac{k}{d}\!\right\rfloor - (i+1)\right)}^j\right) 
\in \left\langle -dq, 0\right\rangle$, so that $z_{n_{j-l_0}}^{j-l_0}$
is L-mobile in
$\left\langle z_{n_j - \left(\left\lfloor\!\frac{k}{d}\!\right\rfloor - i\right)}^j,
z_{n_j - \left(\left\lfloor\!\frac{k}{d}\!\right\rfloor - (i+1)\right)}^j \right]$,
either of which is a contradiction.
Thus (\ref{prop: type (1,-1), c=1, eq: max n_j-i notin for  i  0, ..., k/d-2})
must be true.  This, combined with (\ref{prop: type (1,-1), c=1, eq: z_n_j - (k/d -1) not in}),
then implies that
\begin{equation}
\minq_{j\in\Z/d}\left(z_{n_{j-l_0} - i}^{j-l_0}
- z_1^j\right) \notin \left\langle -dq, 0\right\rangle\;
\text{for all}\;
i \in \left\{0, \ldots, \textstyle{\left\lfloor\frac{k}{d}\right\rfloor} - 2\right\}\!,\;
l_0 \neq 1 \in \Z/d,
\end{equation}
and so, taking (\ref{prop: type (1,-1), c=1, eq: l=0 and l=1 not in})
into account, we now know that
\begin{equation}
\label{prop: type (1,-1), c=1, eq: z_n_j-i notin for  i  0, ..., k/d-2}
z_{n_{j^{\prime}-l_0} - i}^{j^{\prime}-l_0}
\notin \left\langle z_0^{j^{\prime}}, z_1^{j^{\prime}} \right\rangle\;
\text{for all}\;
i \in \left\{0, \ldots, \textstyle{\left\lfloor\frac{k}{d}\right\rfloor} - 2\right\}\!,\;
l_0 , j^{\prime} \in \Z/d.
\end{equation}
Thus, the combination of
(\ref{prop: type (1,-1), c=1, eq: z_0^j+l_0 not in <z_0, z_1>}),
(\ref{prop: type (1,-1), c=1, eq: z_n_j - (k/d -1) not in}), and
(\ref{prop: type (1,-1), c=1, eq: z_n_j-i notin for  i  0, ..., k/d-2})
tells us that
\begin{equation}
z_i^{j^{\prime}+l_0}
\notin \left\langle z_0^{j^{\prime}}, z_1^{j^{\prime}} \right\rangle\;
\text{for all}\;
l_0 , j^{\prime} \in \Z/d,\;
i \in \left\{0, \ldots, n_{j^{\prime}} \right\},
\end{equation}
or in other words,
\begin{equation}
Q_q \cap  \left\langle z_0^{j^{\prime}}, z_1^{j^{\prime}} \right\rangle = \emptyset\;
\text{for all}\;
j^{\prime} \in \Z/d.
\end{equation}

Thus
$Q_q \cap  \left\langle z_0^{j^{\prime}}, z_1^{j^{\prime}} \right] = z_1^{j^{\prime}}$
for any $j^{\prime} \in \Z/d$,
and so we know in particular that
\begin{equation}
\#\left( {\tilde{Q}}_q \cap  \left\langle z_0^{j^{\prime}}, z_1^{j^{\prime}} \right] \right) = 1.
\end{equation}
Moreover, since $v_q\!\left(z_{n_{j-1}}^{j-1}, z_0^j\right)$ is not constant in $j\in\Z/d$,
we know that $v_q\!\left(z_0^{j^{\prime}}, z_1^{j^{\prime}} \right) = \alpha(k)$
for any $j^{\prime} \in \Z/d$.
Thus, for any $j^{\prime} \in \Z/d$, we have
\begin{align}
                       \nonumber
      v_q\!\left(z_0^{j^{\prime}}, z_1^{j^{\prime}} \right)
&= \alpha(k)
                     \\ \nonumber
      \left[z_1^{j^{\prime}} - z_0^{j^{\prime}}\right]_{k^2}\!\! k
       \;-\; \#\left( {\tilde{Q}}_q \cap  \left\langle z_0^{j^{\prime}}, z_1^{j^{\prime}} \right] \right) k^2
&= \alpha(k)
                     \\ \nonumber
      [dq]_{k^2}k - (1)k^2
&= \alpha(k)
                     \\
      [dq]_{k^2} 
&= k + \alpha.
\end{align}

\end{proof}

In addition to concluding our study of the properties of
genus-minimizing $q$ of positive type,
Proposition \ref{prop: no mobile points, psi < k^2/2, implies type (1,-1), m=2, c=1, d odd}
also completes our survey of the ``non-neutralizedness''
of (anti)(pseudo)mobile points, with the interesting result that all
(anti)(pseudo)mobile points are as non-neutralized as possible.
That is, Proposition \ref{prop: positive type, main prop}.(iv)
tells us that all mobile points are non-neutralized,
Proposition \ref{prop: psibar < k^2/2 implies psibar-mobile are nn}
states that all antipseudomobile points are non-neutralized
when $\psibar < \frac{k^2}{2}$, and 
Proposition \ref{prop: no mobile points, psi < k^2/2, implies type (1,-1), m=2, c=1, d odd}
tells us that all pseudomobile points are non-neutralized when
$2[dq]_{k^2}< \psi < \frac{k^2}{2}$.  (It is easy to check that in
the exceptional case in which $\psi < 2[dq]_{k^2}$, no pseudomobile points are present.)
Since it is algebraically impossible
for a pseudomobile (respectively antipseudomobile) point to be non-neutralized
when $\psi < \frac{k^2}{2}$ (respectively $\psibar < \frac{k^2}{2}$),
this is the most non-neutralizedness that could have occurred.
The classification of genus-minimizing $q$ does not make use of
this observation, but it seems an interesting observation nonetheless.

\subsection{Classification of Genus-Minimizing Solutions for $q$}

We have now done all the work necessary to
say what the genus-minimizing solutions for $q$ are.
It is mainly a matter of bookkeeping to collect them all.

\begin{prop}
\label{prop: bookkeeping for q genus-minimizing}
Suppose that $k$ is an integer $\geq 2$, and that
$q \in \Z/k^2$ is primitive.
Then the triple $(p=k^2, q, k)$ is genus-minimizing if and only if
$q \in \Z/k^2$ can be expressed in one or more of the following forms,
\begin{align*}
\text{0.}\;\;
&
\begin{cases}
   q = ik \pm 1, \;
       & \gcd (i,k)  \in \{1, 2\}
\end{cases}
                  \\
\text{1.}\;\;
&
\begin{cases}
   q = \pm \textstyle{\frac{k+1}{d}}(k+1), \; 
       & 2\!\! \not\vert\, \textstyle{\frac{k+1}{d}}
              \\
   q = \pm \textstyle{\frac{k-1}{d}}(k-1), \; 
       & 2\!\! \not\vert\, \textstyle{\frac{k-1}{d}}
\end{cases}
                  \\
\text{2.}\;\;
&
\begin{cases}
   q = \pm \textstyle{\frac{k-1}{d}}(2k+1), \; 
       & 2\!\! \not\vert\, d
            \\
   q = \pm \textstyle{\frac{k+1}{d}}(2k-1), \; 
       & 2\!\! \not\vert\, d
\end{cases}
                  \\
\text{3.}\;\;
&
\begin{cases}
   q = \pm \textstyle{\frac{2k+1}{d}}(k-1), \;
       & 2\!\! \not\vert\, d
             \\
   q = \pm \textstyle{\frac{2k-1}{d}}(k+1), \;
       & 2\!\! \not\vert\, d
\end{cases}
\end{align*}
for any positive integer $d$, where all fractions shown represent integers.
In case `3', the redundant condition $2 \!\! \not\vert \, d$ is listed for
aesthetic reasons.
\end{prop}

\begin{proof}
For $k \leq 100$, it is easy to check the proposition by computer.

We therefore take $k >100$ for the remainder of the proof.
For such $k$, Definition \ref{def: parameters assigned to q}
parameterizes all primitive $q$ in $\Z/k^2$, so that
\begin{equation}
\xi q = \alpha\gamma\mu \textstyle{\frac{ck+\alpha\gamma}{d}}(mk + \alpha\mu).
\end{equation}
If $q \equiv \pm 1\; (\mod k)$, or equivalently, if
the parameter $c$ satisfies $c=0$,
we say that $q$ is of type 0.
The sign $\xi \in \{\pm1\}$ is chosen in such a way as to make
$\xi q$ satisfy $[d\xi q]_{k^2} < \frac{k^2}{2}$.
When $q$ is not of type 0, we say that
$q$ is of positive type if $q = +\xi q$,
and that $q$ is of negative type if $q = -\xi q$.
In order to avoid carrying around an extra $\xi$ everywhere,
we restricted
Sections \ref{ss: Notation and Definitions for q of Positive Type}
and \ref{ss: Properties of z^j for Genus-Minimizing q of Positive Type}
to the case in which $q$ is of positive type.
However, it is easy to see that the definitions and results of those
sections also hold for $\xi q$, for $q$ of negative type.

Proposition \ref{prop: type 0 genus-minimizing classification}
shows that if $q$ is of type 0, then
$q$ is genus-minimizing if and only if
$q$ is of the form shown in ``0'' above.
This leaves us with the case in which
$q$ is of positive or negative type,
or, for brevity, of {\em nonzero type}.
For the reader's convenience, we pause to restate the
propositions we shall use to classify genus-minimizing
$q$ of nonzero type.  Propositions
\ref{prop: q of positive type. If mobile points exist, then l=1},
\ref{prop: type (1,1), k/d = 2 or 3}, and 
\ref{prop: mobile points implies mu m + gamma c = 2}
deal only with cases in which ${\mathbf{z}}^j$ has mobile points,
whereas Propositions
\ref{prop: no mobile points, psibar < k^2/2, implies type (-1,1), m=1, c=2, d odd} and
\ref{prop: no mobile points, psi < k^2/2, implies type (1,-1), m=2, c=1, d odd}
deal only with cases in which mobile points are not present.
\begin{itemize}

\item[\ref{prop: positive type, main prop}.(iv)]
{Suppose $q$ of positive type is genus-minimizing.
Then all mobile points are non-neutralized.
Moreover, $\psi > 2\left[dq\right]_{k^2}$
unless $(\mu,\gamma) = (1,1)$, $\alpha = -1$, $m=2$, $c=1$,
and $d= \frac{k-1}{2} \equiv 0\; (\mod 2)$.}

\item[\ref{prop: q of positive type. If mobile points exist, then l=1}]
{Suppose $q$ of positive type is genus-minimizing,
and $\psi > 2[dq]_{k^2}$.
If $z_0^{j+l}$, and hence $z_{n_{j-l}}^{j-l}$, are mobile in ${\mathbf{z}}^j$
for some nonzero $l \in \Z/d$, then $\psibar < \frac{k^2}{2}$,
$l=1$, $c=1$, and $2 \!\!\not\vert\, \frac{k+\alpha}{d}$,
and either $(\mu,\gamma)=(1,-1)$ with $d=2$ and $\alpha=+1$,
or $(\mu,\gamma)=(1,1)$.}

\item[\ref{prop: type (1,1), k/d = 2 or 3}]
{Suppose $q$ of positive type is genus-minimizing
and $\psi > 2[dq]_{k^2}$.
If $(\mu,\gamma) = (1,1)$,
and $\left\lfloor\frac{k}{d}\right\rfloor \in \{2,3\}$, 
then $m=c=1$.}

\item[\ref{prop: mobile points implies mu m + gamma c = 2}]
{Suppose $q$ of positive type is genus-minimizing.
If ${\mathbf{z}}^j$ has mobile points and 
$\left\lfloor\frac{k}{d}\right\rfloor > 3$, then
$[dq]_{k^2}=2k+\alpha$, so that $\mu m + \gamma c = 2$.}

\item[\ref{prop: no mobile points, psibar < k^2/2, implies type (-1,1), m=1, c=2, d odd}]
{Suppose $q$ of positive type is genus-minimizing.
If ${\mathbf{z}}^j$ has no mobile points
and $\psibar < \frac{k^2}{2}$, then
$(\mu, \gamma) = (-1,1)$, $m=1$, $c=2$, and $2\!\! \not\vert \,d$.}

\item[\ref{prop: no mobile points, psi < k^2/2, implies type (1,-1), m=2, c=1, d odd}]
{Suppose $q$ of positive type is genus-minimizing.
If $\psi < \frac{k^2}{2}$ and $\psi > 2[dq]_{k^2}$, then
${\mathbf{z}}^j$ has no mobile points,
all pseudomobile points are non-neutralized,
$(\mu, \gamma) = (1,-1)$, $m=2$, $c=1$, and $2\!\! \not\vert \,d$.}
\end{itemize}

Suppose that $q$ of nonzero type is genus-minimizing---so that
$\xi q$ is of positive type and, by Proposition \ref{prop:G properties},
is also genus-minimizing---and consider the case in which
$\psibar < \frac{k^2}{2}$ and 
${\mathbf{z}}^j$ has mobile points.
Then Proposition \ref{prop: q of positive type. If mobile points exist, then l=1}
tells us that $2 \!\!\not\vert\, \frac{k+\alpha}{d}$, and
either $(\mu,\gamma)=(1,1)$, or
$(\mu,\gamma)=(1,-1)$ with $c=1$, $d=2$, and $\alpha=+1$.
If $(\mu,\gamma) = (1,1)$, then either 
$\left\lfloor\frac{k}{d}\right\rfloor \leq 3$,
so that Proposition \ref{prop: type (1,1), k/d = 2 or 3} tells us that
$m=c=1$, or $\left\lfloor\frac{k}{d}\right\rfloor > 3$, so that
Proposition \ref{prop: mobile points implies mu m + gamma c = 2}
tells us that $\mu m + \gamma c = m+c = 2$, implying $m=c=1$.
Thus, in either case, $(\mu,\gamma)=(1,1)$ implies
$m=c=1$, so that
$\xi q = \alpha \frac{k + \alpha}{d}(k + \alpha)$, with
$3 \leq d \leq \frac{k+1}{3}$ when $\alpha = +1$,
and $2 \leq d \leq \frac{k-1}{3}$ when $\alpha = -1$.
(Note that the fact that $2 \!\!\not\vert\, \frac{k+\alpha}{d}$ implies that
$d \neq \frac{k+\alpha}{2}$.)
On the other hand, if $(\mu,\gamma)=(1,-1)$ with $c=1$, $d=2$, and $\alpha=+1$,
then Proposition \ref{prop: mobile points implies mu m + gamma c = 2}
tells us that $\mu m + \gamma c = m - c = 2$, implying $m = 3$, so that
$\xi q = - \frac{k - 1}{2}(3k + 1) =  +\frac{k+1}{2}(k+1)$.
Thus, if $\psibar < \frac{k^2}{2}$ and ${\mathbf{z}}^j$ has mobile points, then
\begin{equation}
q = \xi\alpha \textstyle{\frac{k + \alpha}{d}}(k + \alpha),\;\;
\text{with}\;
\textstyle{\frac{k+\alpha}{d}}\in \Z,\; 2 \!\!\not\vert\, \textstyle{\frac{k+\alpha}{d}},\;
2 \leq d \leq \frac{k+\alpha}{3}.
\end{equation}
These values of $q \in \Z/k^2$ constitute a subset of the
solutions listed in ``1'' above.  The complement of this subset consists of the
cases in which $d \in \{1, k+\alpha\}$, which are simply the cases in which
forms ``0'' and ``1'' intersect.  We already classified them as
genus-minimizing solutions of type 0.

Next, suppose that $q$ of nonzero type is genus-minimizing---so that
$\xi q$ is of positive type and is genus-minimizing---and consider the case in which
$\psi < \frac{k^2}{2}$.
Proposition \ref{prop: q of positive type. If mobile points exist, then l=1}
then tells us that ${\mathbf{z}}^j$ has no mobile points unless
$\psi < 2[dq]_{k^2}$, which only occurs as the exceptional case of 
Proposition \ref{prop: positive type, main prop}.(iv), in which
$(\mu,\gamma) = (1,1)$, $\alpha = -1$, $m=2$, $c=1$,
and $d= \frac{k-1}{2} \equiv 0\; (\mod 2)$.
In other words, 
$\xi q = -\frac{k-1}{\left(\frac{k-1}{2}\right)}(2k-1) = -2(2k-1)$,
with $4 \vert k-1$.
For reasons that will soon become clear, we choose to re-express this as
$\xi q = -\frac{k-\alpha}{\left(\frac{k-\alpha}{2}\right)}(2k+\alpha)$, with $\alpha=-1$ and
$2 \!\! \not\vert \frac{k-\alpha}{2}$.
This leaves us with the case in which $2[dq]_{k^2}< \psi < \frac{k^2}{2}$,
so that, by 
Proposition \ref{prop: no mobile points, psi < k^2/2, implies type (1,-1), m=2, c=1, d odd},
${\mathbf{z}}^j$ has no mobile points,
$(\mu, \gamma) = (1,-1)$, $m=2$, $c=1$, and $2\!\! \not\vert \,d$.
Combining this and the special case just mentioned yields
\begin{equation}
q = -\xi\alpha \textstyle{\frac{k - \alpha}{d}}(2k + \alpha),\;\;
\text{with}\;
\textstyle{\frac{k-\alpha}{d}}\in \Z,\; 2 \!\!\not\vert\, d,\;
3 \leq d \leq \frac{k-\alpha}{2}.
\end{equation}
These values of $q \in \Z/k^2$ constitute a subset of the
solutions listed in ``2'' above.  The complement of this subset consists of the
cases in which $d \in \{1, k-\alpha\}$, which are simply the cases in which
forms ``0'' and ``2'' intersect.  We already classified them as
genus-minimizing solutions of type 0.

Lastly, suppose that $q$ of nonzero type is genus-minimizing---so that
$\xi q$ is of positive type and is genus-minimizing---and
consider the only case that remains, namely, in which
$\psibar < \frac{k^2}{2}$ and ${\mathbf{z}}^j$ has no mobile points.
Proposition \ref{prop: no mobile points, psibar < k^2/2, implies type (-1,1), m=1, c=2, d odd}
then tells us that
$(\mu, \gamma) = (-1,1)$, $m=1$, $c=2$, and $2\!\! \not\vert \,d$, so that
\begin{equation}
q = -\xi\alpha \textstyle{\frac{2k + \alpha}{d}}(k - \alpha),\;\;
\text{with}\;
\textstyle{\frac{2k+\alpha}{d}}\in \Z,\; 2 \!\!\not\vert\, d,\;
3 \leq d \leq \frac{2k+\alpha}{3}.
\end{equation}
These values of $q \in \Z/k^2$ constitute a subset of the
solutions listed in ``3'' above.  The complement of this subset consists of the
cases in which $d \in \{1, 2k+\alpha\}$, which are simply the cases in which
forms ``0'' and ``3'' intersect.  We already classified them as
genus-minimizing solutions of type 0.
                 \\

We have now shown that all genus-minimizing $q$ can be expressed in
one or more of forms ``0'', ``1'', ``2'', or ``3'', as shown above.
It remains to show that all such forms of $q$ are genus-minimizing.
It is straightforward but tedious to use the tools so far introduced to
show, using the original description for each $q$,
that each of the intervals of length $[dq]_{k^2}$ or of length
$\psi$ contains the correct number of elements of $q$.
Fortunately, we are not obligated to perform this task, because
Berge has already provided us with a topological proof
\cite{Berge}
that all of the above forms of $q$ are genus-minimizing.

\end{proof}

We are almost done with this section, but it turns out that
we need to know the form of $q^{-1} \in \Z/k^2$,
rather than of $q \in \Z/k^2$,
for genus-minimizing $q$.

\begin{prop}
\label{prop: classification of q^-1 for genus-minimizing q}
Suppose that $k$ is an integer $\geq 2$,
and that $p \in \Z/k^2$ is primitive.
The triple $(k^2, p^{-1}, k)$ is genus-minimizing if and only if
$p \in \Z/k^2$ can be expressed in one or more of the following forms,
\begin{align*}
\text{0.}\;\;
&
\begin{cases}
   p = ik \pm 1, \;
       & \gcd (i,k) \in \{1, 2\}
\end{cases}
                  \\
\text{1.}\;\;
&
\begin{cases}
   p = \pm d(2k+1), \; 
       & d\vert k-1,\; 2\!\! \not\vert\, \textstyle{\frac{k-1}{d}}
            \\
   p = \pm d(2k-1), \; 
       & d\vert k+1,\; 2\!\! \not\vert\, \textstyle{\frac{k+1}{d}}
\end{cases}
                  \\
\text{2.}\;\;
&
\begin{cases}
   p = \pm d(k+1), \; 
       & d \vert k+1,\; 2\!\! \not\vert\, d
              \\
   p = \pm d(k-1), \; 
       & d \vert k-1,\; 2\!\! \not\vert\, d
\end{cases}
                  \\
\text{3.}\;\;
&
\begin{cases}
   p = \pm d(k-1), \;
       & d \vert 2k+1,\; 2\!\! \not\vert\, \textstyle{\frac{2k+1}{d}}
             \\
   p = \pm d(k+1), \;
       & d \vert 2k-1,\; 2\!\! \not\vert\, \textstyle{\frac{2k-1}{d}}
\end{cases}
\end{align*}
for any positive integer $d$ satisfying the above divisibility constraints.
The redundant oddness condition in case `3' is listed for
aesthetic reasons.
\end{prop}

\begin{proof}
The triple $(k^2, p^{-1},k)$ is genus-minimizing if and only if
$p^{-1}=q \in \Z/k^2$ (or equivalently, if and only if $p = q^{-1}$),
for one of the forms of $q$
listed in Proposition \ref{prop: bookkeeping for q genus-minimizing}.
If $q =  ik \pm 1$, for some 
$i \in \Z/k^2$ with $\gcd(i,k) \in \{1,2\}$,
then $q^{-1}=-ik \pm 1 = (k-i)k \pm 1$, with $\gcd(k-i,k) \in \{1,2\}$.
If, for some
$m, c \in \{1,2\}$, $\alpha, \gamma, \mu, \xi \in \{1,-1\}$,
and primitive $d \in \Z/k^2$, we have
\begin{equation}
\xi dq = \mu m k + \gamma c k + \alpha \;\in \Z/k^2,
\end{equation}
then
\begin{align}
      q^{-1}
&= \xi d \left(\xi d q\right)^{-1}
              \\ \nonumber
&= -\xi d ((\mu m + \gamma c)k - \alpha).
\end{align}
These rules for inverting $q$ establish a bijection
between the forms of $q$ listed in 
Proposition \ref{prop: bookkeeping for q genus-minimizing}
and the correspondingly numbered forms of $p$ listed above.
\end{proof}

The observant reader might notice that the set of solutions listed in 
Proposition \ref{prop: classification of q^-1 for genus-minimizing q}
coincides with the set of solutions listed in
Proposition \ref{prop: bookkeeping for q genus-minimizing}.
This is because
the set of genus-minimizing solutions for $q$ is closed under the
operation of taking inverses in $\Z/k^2$.

\section{Case $q \equiv k^{-2}(\mod p)$}
\label{s:q=k-2}

We have now classified all 
genus-minimizing triples of the form $(k^2, p^{-1}, k)$,
but in order to determine all simple knots in lens spaces
that have L-space homology sphere surgeries,
we need to classify all genus-minimizing triples of the form
$(p, k^{-2}, k)$.  The following proposition tells us
that these two classifications coincide when $p > k^2$.

\begin{prop}
\label{prop:section q=k^-2, (p,k^-2,k) is like (k^2, p^-1, k)}
If $p > k^2$, then the triple $(p, k^{-2}, k)$ is genus-minimizing if and only if
the triple $(k^2, p^{-1}, k)$ is genus-minimizing.
\end{prop}
\begin{proof}
By Proposition \ref{prop:G properties}, we know that
\begin{equation}
  \bar{G}\!\left(k^2, p^{-1}, k\right)
= \bar{G}\!\left(k^2, p, p^{-1}k\right)
= \bar{G}\!\left(k^2, -p, -p^{-1}k\right).
\end{equation}
Likewise, Proposition \ref{prop:G properties} tells us that
\begin{equation}
  \bar{G}\!\left(p, k^{-2}, k\right)
= \bar{G}\!\left(p, k^2, k^{-2}\cdot k\right)
= \bar{G}\!\left(p, k^2, k^{-1}\right).
\end{equation}
Thus, it is sufficient to show that
the triple $\left(k^2, -p, -p^{-1}k\right)$ is genus-minimizing
if and only if the triple
$\left(p, k^2, k^{-1}\right)$ is genus-minimizing.

For brevity, let $A$ and $B$ denote the triples
\begin{align}
\label{eq: section q=k^-2, def of triple A}
A&:= \left(k^2, -p, -p^{-1}k\right) = \left(k^2, \epsilon, nk\right),
         \\
\label{eq: section q=k^-2, def of triple B}
B&:= \left(p, k^2, k^{-1}\right) = \left(p, k^2, \frac{np+1}{k}\right),
\end{align}
where $\epsilon := \left[-p\right]_{k^2}$ and $n := \left[-p^{-1}\right]_k$.
Note that the $n$ in
(\ref{eq: section q=k^-2, def of triple A})
is the same as the $n$ in
(\ref{eq: section q=k^-2, def of triple B}),
since $\left[k^{-1}\right]_p = \frac{\left[\!-\!p^{-1}\!\right]_{\!k} p \;+ 1}{k}$.

Furthermore, define
\begin{equation}
{\bar{v}}_A := \frac{1}{k}v_A
\;\;\;\;\mathrm{and}\;\;\;\;
{\bar{v}}_B := \frac{k^2}{p}\cdot \frac{1}{k}v_B,
\end{equation}
so that for any $x, y \in \Z$ with $x \leq y$, we have
\begin{align}
    {\bar{v}}_A(x,y) 
&:= \#\left(\Z\cap \left\langle x, y\right]\right) n
    \;\; - \;\;
    \#\left({\tilde{Q}}_A \cap \left\langle x, y\right]\right) k,
            \\
    {\bar{v}}_B(x,y) 
&:= \#\left(\Z\cap \left\langle x, y\right]\right) \left(n + 
     \textstyle{\frac{1}{p}}\right)
    \;\; - \;\;
    \#\left({\tilde{Q}}_B \cap \left\langle x, y\right]\right) k,
\end{align}
where
\begin{align}
Q_A &:= \left\{0\epsilon, \ldots, (nk-1)\epsilon\right\} \subset \Z/k^2,
         \\ 
Q_B &:= \left\{0k^2, \ldots, \left(\frac{np+1}{k}-1\right)\!k^2\right\}
        \subset \Z/p,
\end{align}
and ${\tilde{Q}}_A := {\pi}_A^{-1}(Q_A)$ and 
${\tilde{Q}}_B := {\pi}_B^{-1}(Q_B)$ are the preimages of
$Q_A$ and $Q_B$ under the respective quotient maps
$\Z \stackrel{{\pi}_A}{\rightarrow} \Z/k^2$ and
$\Z \stackrel{{\pi}_B}{\rightarrow} \Z/p$.
When $x > y$, one could take ${\bar{v}}_A$ and ${\bar{v}}_B$
to be defined by the identities 
${\bar{v}}_A(x,y) = -{\bar{v}}_A(y,x)$ and
${\bar{v}}_B(x,y) = -{\bar{v}}_B(y,x)$.

Recall that according to Corollary 
\ref{cor: intro defs, genus-minimizing if v < p + k},
an arbitrary triple $(p_0, q_0, k_0)$ is genus-minimizing
if and only if
\begin{equation}
v_{(p_0, q_0, k_0)}(x,y) < p_0 + k_0\;\;\;\;
\text{for all}\;x,y \in Q_0,
\end{equation}
where $Q_0 := \{0q_0, \ldots, (k_0-1)q_0)\} \subset \Z/{p_0}$.
Equivalently, this condition can be phrased in terms of
lifts of $x$ and $y$ to $\Z$. That is, $(p_0, q_0, k_0)$
is genus-minimizing if and only if
$v_{(p_0, q_0, k_0)}(x,y) < p_0 + k_0$
for all $x, y \in {\tilde{Q}}_0$, where
${\tilde{Q}}_0 := {\pi}_0^{-1}(Q_0)$ is the preimage of
$Q_0$ under the quotient map
$\Z \stackrel{{\pi}_0}{\rightarrow} \Z/{p_0}$.

We can therefore phrase the respective conditions for
the triples $A$ and $B$ to be genus-minimizing as follows:
\begin{align}
   A\;\text{is genus-minimizing}
   \;\;\;\;\Leftrightarrow\;\;\;\;
   {\bar{v}}_A(x,y)
&< k+n
   \;\;\;\;\;\;\;\;\;\:\forall\; x,y \in {\tilde{Q}}_A;
           \\
   B\;\text{is genus-minimizing}
   \;\;\;\;\Leftrightarrow\;\;\;\;
   {\bar{v}}_B(x,y)
&< k+n + \textstyle{\frac{1}{p}}
   \;\;\;\forall\; x,y \in {\tilde{Q}}_B.
\end{align}
Now, ${\bar{v}}_A \in \Z$, so ${\bar{v}}_A < k + n$ if and only if
${\bar{v}}_A \leq k + n - 1$.  On the other hand,
${\bar{v}}_B \in \frac{1}{p}\Z$, so
${\bar{v}}_B < k + n + \frac{1}{p}$ if and only if
${\bar{v}}_B \leq k + n$.   However, ${\bar{v}}_B(x,y) \in \Z$
only if $y-x$ is a multiple of $p$, which in turn implies that
${\bar{v}}_B(x,y) = 0$.  Thus it is impossible to have
${\bar{v}}_B = k + n$.  We therefore have the revised conditions,
\begin{align}
   A\;\text{is genus-minimizing}
   \;\;\;\;\Leftrightarrow\;\;\;\;
   {\bar{v}}_A(x,y)
&\leq k+n - 1
   \;\;\;\;\;\;\;\;\;\;\;\;\;\:\forall\; x,y \in {\tilde{Q}}_A;
           \\
   B\;\text{is genus-minimizing}
   \;\;\;\;\Leftrightarrow\;\;\;\;
   {\bar{v}}_B(x,y)
&\leq k+n - 1 + \textstyle{\frac{p-1}{p}}
   \;\;\;\forall\; x,y \in {\tilde{Q}}_B.
\end{align}
        \\

Let us next turn our attention to $Q_A$ and $Q_B$,
recalling that 
$Q_A = \left\{0\epsilon, \ldots, (nk-1)\epsilon\right\} \subset \Z/k^2$,
and
$Q_B = \left\{0k^2, \ldots, (k^{-1} -1)k^2\right\} \subset \Z/p$.
In a manner somewhat reminiscent of the construction
of the tuples $\{{\bf{z}}^j\}$ in Section \ref{s:p=k2},
we arrange the elements of $\left[ 0,p \right\rangle\cap {\tilde{Q}}_B$ into tuples
$\{{\bf{w}}^j\}$, defined by 
$w_i^j := \left[j\epsilon \right]_{k^2} + ik^2$,
where $\epsilon = \left[-p\right]_{k^2}$,
and for each $j$, we restrict $i$ to lie in 
$\{ i^{\prime} \in \Z_{\geq 0}\left\vert\; w_{i^{\prime}}^j < p \right.\}$.
Thus, for some $j_* \in \Z$, we have
\begin{equation}
  \left({\bf{w}}^0, \ldots, {\bf{w}}^{j_*-1}\right)
= \left(\left[0k^2\right]_p, \ldots,  \left[(k^{-1}-1)k^2\right]_p\right) 
\end{equation}
Of course, this also requires that when $j = j_*-1$,
we restrict $i$ to lie in $\{0, \ldots, i_*-1\}$,
where $w_{i_*-1}^{j_*-1} = \left[(k^{-1}-1)k^2\right]_p$.
Let us pause to determine $i_*$ and $j_*$.
\begin{align}
                 \nonumber
    w_{i_*-1}^{j_*-1}
&=  \left[(k^{-1}-1)k^2\right]_p
              \\ \nonumber
    \left[(j_*-1)\epsilon \right]_{k^2} + (i_*-1) k^2
&=  p + k - k^2
              \\ \nonumber
    i_*
&=  \frac{p + k - \left[(j_*-1)\epsilon \right]_{k^2}}{k^2}
              \\
    i_*
&=  \left\lceil\frac{p - \left[(j_*-1)\epsilon \right]_{k^2}}{k^2}\right\rceil,
\end{align}
where the second line uses the fact that $p > k^2$.
Taking the second line modulo $k^2$, we next determine $j_*$.
Recalling that $n = \left[-p^{-1}\right]_k$, we have
\begin{align}
                 \nonumber
        (j_*-1)\epsilon + (i_* - 1) k^2
&\equiv p + k - k^2\;\;(\mod k^2)
              \\ \nonumber
        (j_*-1)(-p)
&\equiv p + k \;\;(\mod k^2)
              \\ \nonumber
        j_*
&\equiv -p^{-1}k \;\;(\mod k^2)
              \\
        j_*
&=      nk.
\end{align}
This allows us to write out the tuples ${\bf{w}}^j$ as follows:
\begin{align}
   {\bf{w}}^0
&= \left(\;  \left[0\epsilon\right]_{k^2} + 0k^2\!,\;\;
             \left[0\epsilon\right]_{k^2} + 1k^2\!,\;\; \ldots,\;\;
             \left[0\epsilon\right]_{k^2}
              + \left({\left\lceil\!\frac{p - \left[0\epsilon\right]_{k^2}\!}{k^2}
                       \right\rceil} - 1 \right)\! k^2
   \;\right),
         \\ \nonumber
&\;\;\vdots
         \\ \nonumber
   {\bf{w}}^j
&= \left(\;  \left[j\epsilon\right]_{k^2} + 0k^2\!,\;\;
             \left[j\epsilon\right]_{k^2} + 1k^2\!,\;\; \ldots,\;\;
             \left[j\epsilon\right]_{k^2}
              + \left({\left\lceil\!\frac{p - \left[j\epsilon\right]_{k^2}\!}{k^2}
                       \right\rceil} - 1 \right)\! k^2
   \;\right),
         \\ \nonumber
&\;\;\vdots
         \\ \nonumber
   {\bf{w}}^{nk-1}
&= \left(\;  \left[(nk\!-\!1)\epsilon\right]_{k^2} + 0k^2\!,\;\;
             \left[(nk\!-\!1)\epsilon\right]_{k^2} + 1k^2\!,\;\; \ldots,
   \right.
                         \\ \nonumber
   &\left.\;\;\;\;\;\;\;\;\;\;\;\;\;\;\;\;\;\;\;
          \;\;\;\;\;\;\;\;\;\;\;\;\;\;\;\;\;\;\;
             \left[(nk\!-\!1)\epsilon\right]_{k^2}
              + \left({\left\lceil\!\frac{p - \left[(nk\!-\!1)\epsilon\right]_{k^2}\!}{k^2}
                       \right\rceil} - 1 \right)\! k^2
   \;\right).
\end{align}
Thus, for each
$i \in \left\{0, \ldots, \left\lfloor \frac{p}{k^2} \right\rfloor - 1\right\}$,
we have
\begin{align}
   \left\{w_i^j \left\vert\; j\in \{0, \ldots, nk-1 \} \right. \right\}
&= ik^2 + \left\{ \left[0\epsilon\right]_{k^2}, \ldots, 
                 \left[(nk-1)\epsilon\right]_{k^2} \right\}
            \\ \nonumber
&= \left[ ik^2, (i+1)k^2 \right\rangle \;\cap\; {\tilde{Q}}_A.
\end{align}
This, in turn, implies that
\begin{align}
   \left[ 0, \left\lfloor\frac{p}{k^2}\right\rfloor\! k^2 \right\rangle \cap {\tilde{Q}}_A
&= \left\{w_i^j \left\vert\;
   \begin{array}{c} j\in \{0, \ldots, nk-1 \},
                       \\
                    i \in \left\{0,\ldots, \left\lfloor\frac{p}{k^2}\right\rfloor - 1\right\}
   \end{array}
                    \right.\right\}
                             \\ \nonumber
&= \left\{\left[ 0, p \right\rangle \cap {\tilde{Q}}_B\right\}
   \setminus
   \left\{w_{\left\lfloor\!\frac{p}{k^2}\!\right\rfloor}^j \left\vert\;
          j \in J \right.\right\},
\end{align}
where
\begin{equation}
J = \left\{ j \in \{0, \ldots, nk-1\} \left\vert\;
    \left[j\epsilon\right]_{k^2} < [p]_{k^2} \right.\right\}.
\end{equation}
But this definition of $J$ implies that
\begin{equation}
  \left\{w_{\left\lfloor\!\frac{p}{k^2}\!\right\rfloor}^j \left\vert\;
          j \in J \right.\right\}
= \left[ \left\lfloor \frac{p}{k^2}\right\rfloor\! k^2, p \right\rangle 
  \cap {\tilde{Q}}_A,
\end{equation}
and thus, we have
\begin{equation}
  \left[ 0 , p \right\rangle \cap {\tilde{Q}}_A
= \left[ 0 , p \right\rangle \cap {\tilde{Q}}_B.
\end{equation}
            \\

We now return to the question of genus-minimization.
First, for brevity, set
\begin{equation}
  \tilde{Q}
:= \left[ 0 , p \right\rangle \cap {\tilde{Q}}_A
 = \left[ 0 , p \right\rangle \cap {\tilde{Q}}_B,
\end{equation}
so that, for any $x, y \in \left[ 0 , p \right\rangle$ with $x \leq  y$,
we have
\begin{align}
    {\bar{v}}_A(x,y) 
&:= \#\left(\Z\cap \left\langle x, y\right]\right) n
    \;\; - \;\;
    \#\left(\tilde{Q} \cap \left\langle x, y\right]\right) k,
            \\
    {\bar{v}}_B(x,y) 
&:= \#\left(\Z\cap \left\langle x, y\right]\right) \left(n + 
     \textstyle{\frac{1}{p}}\right)
    \;\; - \;\;
    \#\left(\tilde{Q} \cap \left\langle x, y\right]\right) k,
\end{align}
with ${\bar{v}}_A(y,x) := -{\bar{v}}_A(x,y)$ and ${\bar{v}}_B(y,x) := -{\bar{v}}_B(x,y)$.
Thus, for any $x, y \in \left[0, p\right\rangle$, we have
\begin{equation}
\label{eq: section p = k^-2, compare v_A to v_B}
    {\bar{v}}_B(x,y) = {\bar{v}}_A(x,y) + \frac{y-x}{p}.
\end{equation}

Suppose that $A$ is genus-minimizing.
Then for any $x, y \in \tilde{Q}$, we have
\begin{align}
      {\bar{v}}_B(x,y) 
&= {\bar{v}}_A(x,y) + \frac{y-x}{p}
           \\ \nonumber
&\leq (n + k - 1) + \frac{y-x}{p}
           \\ \nonumber
&\leq n + k - 1 + \frac{p-1}{p},
\end{align}
so that $B$ is genus-minimizing.

Conversely, suppose that $B$ is genus-minimizing.
Then for any $x, y \in \tilde{Q}$,
(\ref{eq: section p = k^-2, compare v_A to v_B}) implies
\begin{equation}
{\bar{v}}_B(x,y) \equiv \frac{y-x}{p}\;(\mod \Z).
\end{equation}
We therefore have
\begin{align}
      {\bar{v}}_A(x,y)
&= {\bar{v}}_B(x,y) - \frac{y-x}{p}
           \\ \nonumber
&\leq \left(n + k - 1 + \frac{y-x}{p}\right) -  \frac{y-x}{p}
           \\ \nonumber
&= n + k - 1,
\end{align}
so that $A$ is genus-minimizing.

\end{proof}

Combining Propositions 
\ref{prop: classification of q^-1 for genus-minimizing q} and
\ref{prop:section q=k^-2, (p,k^-2,k) is like (k^2, p^-1, k)}
then gives the result we have been seeking.

\begin{theorem}
When $p > k^2$, Conjecture \ref{conj: berge} is true.
\end{theorem}

\bibliographystyle{plain}
\bibliography{pqk}

\begin{thebibliography}{1}

\bibitem{BGH}
Kenneth~L. Baker, J.~Elisenda Grigsby, and Matthew Hedden.
\newblock Grid diagrams for lens spaces and combinatorial knot {F}loer
  homology.
\newblock {\em Int. Math. Res. Not. IMRN}, 10(24):39pp, 2008.
\newblock arXiv:0710.0359.

\bibitem{Berge}
John Berge.
\newblock {Some knots with surgeries yielding lens spaces}.
\newblock unpublished manuscript.

\bibitem{BergeTorus}
John Berge.
\newblock The knots in {$D\sp 2\times S\sp 1$} which have nontrivial {D}ehn
  surgeries that yield {$D\sp 2\times S\sp 1$}.
\newblock {\em Topology Appl.}, 38(1):1--19, 1991.

\bibitem{GabaiTorus}
David Gabai.
\newblock Surgery on knots in solid tori.
\newblock {\em Topology}, 28(1):1--6, 1989.

\bibitem{GreeneBerge}
Joshua Greene.
\newblock {The lens space realization problem}.
\newblock {\em Annals of Mathematics}, 177:449--511, 2013.
\newblock math/1010.6257.

\bibitem{OSLens}
Peter Ozsv{\'a}th and Zolt{\'a}n Szab{\'o}.
\newblock On knot {F}loer homology and lens space surgeries.
\newblock {\em Topology}, 44:1281--1300, 2005.
\newblock math.GT/0303017.

\bibitem{Rasmussen}
Jacob Rasmussen.
\newblock Lens space surgeries and {L}-space homology spheres.
\newblock arXiv:0710.2531, 2007.

\end{thebibliography}
\end{document}